%% file: thesis.tex
\documentclass[twoside,11pt,vi]{mitthesis}

\usepackage{latexsym}
\usepackage{epsfig}
\usepackage{makeidx}

\def\proof{\medskip\noindent{\sc Proof. }}
\def\note{\medskip\noindent{\sc Note: }}
\def\EOP{\hfill$\Box$\medskip}
\def\Ex{\mbox{E}}
\def\real{\hbox{\rm I\kern-.18em R}}
\def\rati{\hbox{\rm Q\kern-.5em
       \vrule depth 0ex height 1.4ex width .05em\kern.41em}}
\def\natu{\hbox{\rm I\kern-.17em N}}
\def\inte{\hbox{\rm Z\kern-.3em Z}}

\def\flu#1{\tilde{#1}}
\def\go{\rightarrow}
\def\nogo{\not\rightarrow}

\def\choose#1#2{ \left( \begin{array}{c} #1 \\ #2 \end{array} \right) }

\def\st#1{\left(\framebox{$#1$}\right)}
\def\stq#1#2{\left(
\!\!
\begin{array}{c} #1 \\ \framebox{$#2$} \end{array} 
\!\!
\right)}

\def\rh{\hat{r}}
\def\eps#1{\epsilon_{#1}}

\def\ssx{\mathcal{X}}  
\def\ssa{\mathcal{A}}  
\def\sss{\mathcal{S}}  
\def\ssr{\mathcal{R}}  
\def\ssm{\mathcal{M}}  
\def\tot#1{\clubsuit#1}

\newtheorem{theorem}{Theorem}
\newtheorem{lemma}{Lemma}

\newtheorem{definition}{Definition}

\newtheorem{corollary}{Corollary}

\newtheorem{hypothesis}{Hypothesis}
\renewcommand\baselinestretch{1.0}

\makeindex

\begin{document}

\bibliographystyle{plain}


\include{cover}

\tableofcontents


 \include{statement}


 \include{exact}

 \include{bounds}

 \include{fluid}

 \include{analytic}
 \include{ringlike}

 %

 \appendix
 \include{probability.appendix}

 \include{dummies.appendix}

 \include{examples.appendix}

 \include{computers.appendix}

\bibliography{thesis}

\printindex

\end{document}

%% file: cover.tex
%
%
%
%
%
%
%
\title{Running in Circles:\\ Packet Routing on Ring Networks}

\author{William F. Bradley}
\department{Department of Mathematics}
\degree{Doctor of Philosophy}
\degreemonth{June}
\degreeyear{2002}
\thesisdate{April 26, 2002}

\copyrightnoticetext{\copyright \,\, 2002 William F. Bradley.  All rights reserved}

\supervisor{F. T. Leighton}{Professor of Applied Mathematics}

\chairman{R. B. Melrose}{Chairman, Committee on Pure Mathematics}

\maketitle



\cleardoublepage
\setcounter{savepage}{\thepage}
\begin{abstractpage}
\input{abstract}

\end{abstractpage}


\cleardoublepage

\begin{center}
\textbf{For my father}
\end{center}

%


%% file: abstract.tex
%
%
%
I analyze packet routing on unidirectional ring networks, with an eye towards
establishing bounds on the expected length of the queues.
Suppose we route packets by a greedy ``hot potato'' protocol.
If packets are inserted by a Bernoulli process and
have uniform destinations around the ring, and if the
nominal load is kept fixed, then I can construct an upper bound on the 
expected queue length per node that is independent of 
the size of the ring.  If the packets only travel one or two steps, I can
calculate the exact expected queue length for rings of any size.

I also show some stability
results under more general circumstances.  If the
packets are inserted by any ergodic hidden Markov process
with nominal loads less than one, and routed by any greedy
protocol, I prove that the ring is ergodic.

%% file: statement.tex
\chapter{Statement of the Problem} \label{statement chapter}

\section{The Problem} \label{problem section}

What is packet routing?  In a packet routing network, we populate the
nodes of a directed graph with a collection of discrete objects called
packets.  As time passes, these packets occasionally travel across
edges, or depart the network.  Sometimes, new packets are inserted.  A
typical question to ask is: what is the expected number of packets in
the system?  

This thesis is inspired by the following packet routing
problem on the ring:

Suppose we have a directed graph in the form of a
cycle with the edges directed clockwise.  
Let's label the nodes 1 through $N$, where  for $i<N$, we have
a directed edge from node $i$ to $i+1$, and an additional
edge from $N$ to 1. See Figure~\ref{basic ring figure}.
\begin{figure}[ht]
\centerline{\includegraphics[height=2in]{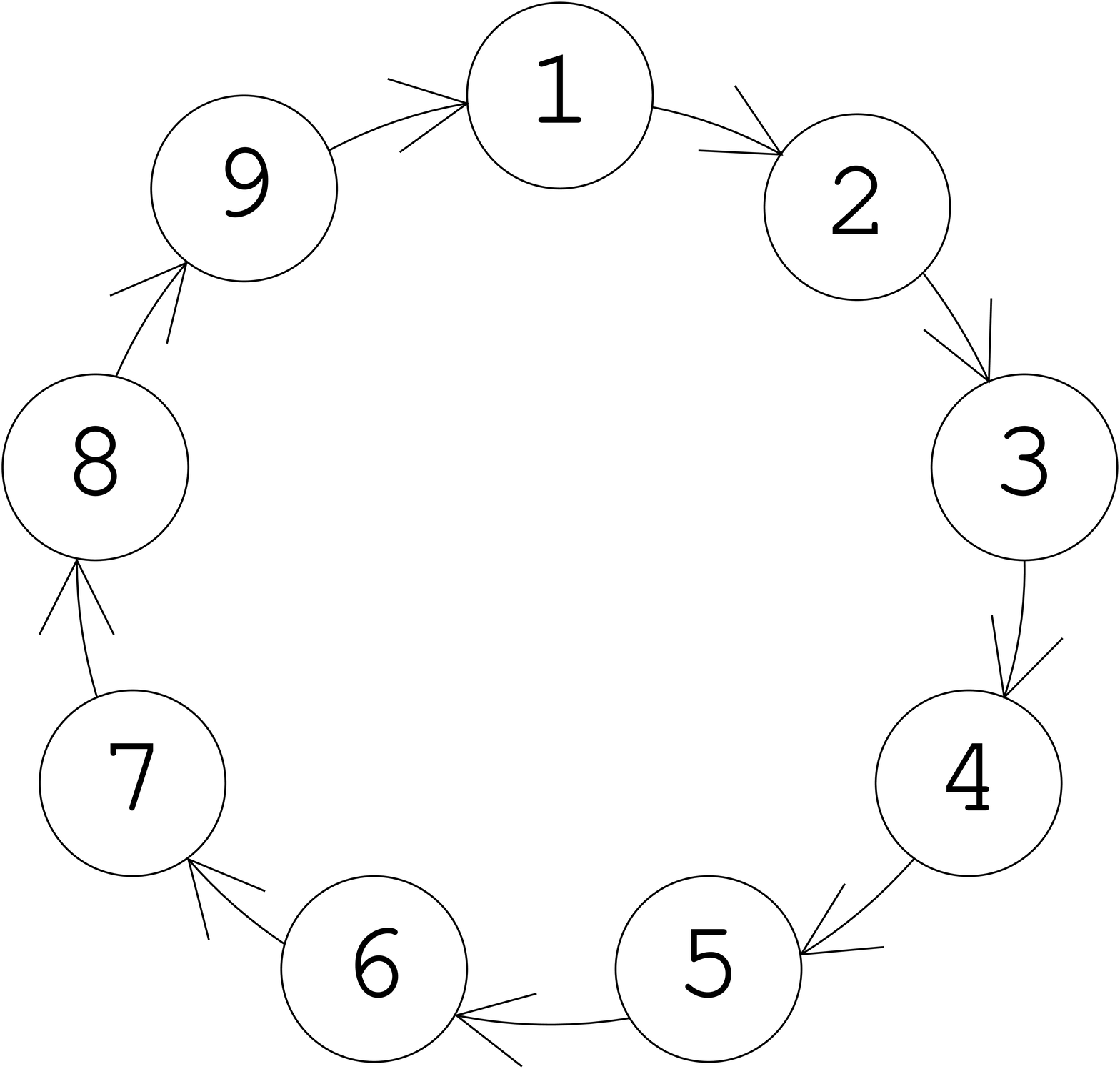} }
\caption{An $N=9$ node unidirectional ring }
\label{basic ring figure}
\end{figure}

We are going to analyze the network's behavior as it evolves in time,
where time is measured in discrete steps.  First, we have to specify
how packets enter the ring.  Let us suppose that with probability $p$,
the probability that a new packet arrives at a node on one time step.
With probability $1-p$, no packet arrives.  This event occurs
independently at every node, on every time step.

Next, we must specify how packets travel along the ring.  A packet
arriving at node $i$ chooses its destination uniformly from the other
$N-1$ nodes.  We will allow at most one packet to depart from
a node in
one time step.  When a packet arrives at its destination, it is
immediately removed; that is, a packet waiting in queue can be
inserted into the ring on the same time step.  Finally, if there
is more than one packet at a node, we must
specify which packet advances next.  We'll use the Greedy Hot
Potato protocol.

\begin{definition} \index{Greedy Hot Potato}\index{GHP protocol}
In the \emph{Greedy Hot Potato} protocol, packets travelling
in the network have priority over packets waiting in queue.
Nodes with non-empty queues always route packets.
\end{definition}

This protocol for determining packet priority is called Greedy Hot Potato
because a packet being passed along the ring is a ``hot
potato'' that never stops moving until it reaches its destination.
It is ``greedy'' in the sense that whenever a node has the opportunity
to route a packet, it always takes it.
This protocol resolves all contentions over which packet
gets to depart from a node.

By specifying the number of packets waiting at each of the $N$ nodes,
and the destination of each packet travelling in the ring, 
we completely specify
the state of the system, and we have a discrete-time Markov chain.

Consider the number of packets in the system that need to use a node
$n$.  At most 1 packet can depart from $n$ on each time step.
Therefore, if too many new packets arrive, the system is unstable (the
Markov chain is not ergodic).  In practical terms, this means that the
mean total number of packets in the system will diverge to infinity
with time.  Let us calculate what value of $p$ corresponds to this
unstable regime.

\begin{lemma} \label{basic ring instability lemma}
\index{Lemma \ref{basic ring instability lemma}}
Given the ring network described above, the system is unstable if
$p>\frac{2}{N}$.
\end{lemma}

\proof
Consider the node $N-1$.  (By symmetry of the ring,
the expected number of packets that need to cross this node is the same as any other node.)  If a packet
arrives at node 1, it has a $1/(N-1)$ chance of needing to cross $n$.
More generally, a packet arriving at node $i$ has an $i/(N-1)$ chance
of needing to cross $n$, for $i<N$.  Summing, the increase in congestion
on $n$ by new arrivals is:
\[\sum_{i=1}^{N-1} p\frac{i}{N-1} = p\frac{N(N-1)}{2(N-1)} = p\frac{N}{2}\]
Therefore, if $p>\frac{2}{N}$, the expected number of new packets that
need to cross $n$ increases by more than 1.  However, the maximum
possible number of packets that can cross per time step is 1.
Therefore, the expected number of packets waiting to cross will
increase without bound, so the system is unstable.
\EOP

If our system is stable, then, we must have $0\leq p \leq 2/N$.  To make this
value appear somewhat less dependent on $N$, it's useful to define
$r=p\frac{N}{2}$.  We can then fix some $0\leq r \leq 1$ 
and study the system as $N$ gets large.  This $r$ is called
the \emph{nominal load}.\index{nominal load!standard Bernoulli ring}

I will call this system, as described above, the 
\emph{standard Bernoulli ring}.  
\begin{definition} \index{standard Bernoulli ring}
An \emph{$N$-node standard Bernoulli ring} is an $N$-node directed
cycle.  Packet arrivals occur according to a Bernoulli arrival process
at each node.  Packet destinations are uniformly distributed.  Packets
are routed by the Greedy Hot Potato protocol.  The nominal load $r$ is
$\frac{N}{2}p$.
\index{nominal load!standard Bernoulli ring}
\end{definition}
Coffman et al~\cite{Leighton95} made the following natural hypothesis:
\begin{hypothesis} \label{ring hypothesis}
\index{Hypothesis 1}
The expected queue length per node of the standard Bernoulli ring,
for any fixed nominal load $0\leq r<1$, is $\Theta(1/N)$.
\end{hypothesis}
The authors performed extensive computer simulations that seemed to support
the hypothesis.  Then, in Coffman et al~\cite{Leighton98}, the authors
partially proved this result:
\begin{theorem}[Coffman et al.] \label{leighton's theorem}
\index{Theorem \ref{leighton's theorem}}
The expected queue length per node of the standard Bernoulli ring,
for any fixed nominal load $0\leq r<\frac{1}{2}$, is $\Theta(1/N)$.
\end{theorem}

Although Coffman et al.~\cite{Leighton98} established impressive
results in the $r<\frac{1}{2}$ case, the $\frac{1}{2} \leq r<1$ regime was
left wide open.  It wasn't even clear that the network was ergodic for
\emph{any} $N>2$.\footnote{$N=2$ is trivially ergodic; no packets ever wait
in queues.}  This thesis began as an attempt to determine the stability
of the ring for values of $r$ greater than $\frac{1}{2}$, and find
asymptotic bounds for the expected queue length as a function of $N$
(for a fixed $r$).  As I began exploring ring networks more, I
discovered that a number of interesting theorems could be proved for
much more general arrival processes.  This document is the result of
my investigations.

\section{What's in this Thesis}

\subsection{A General Overview}
In the earlier chapters of this thesis, I begin by examining simple
ring networks.  As the chapters progress, I analyze increasingly more
general rings.  

I begin in Chapter~\ref{exact chapter} by considering a ring network
where each packet travels either 1 or 2 nodes.  This type of
restriction can be considered a kind of localness\footnote{Not a
``local ring'' in the commutative algebra sense!}, where nodes only
need to communicate with their nearest few neighbors.  I consider a
number of different routing protocols and calculate their (exact)
expected queue lengths.  I also calculate the stationary
distribution under the GHP protocol.

In Chapter~\ref{bounds chapter}, I consider the standard Bernoulli
ring.  I prove that it is ergodic (for any $r<1$ and sufficiently
large $N_{r}$), and I construct an $O(1)$ upper bound on the expected
queue length per node as $N\rightarrow \infty$.  This result isn't as
tight as the $O(1/N)$ upper bound postulated in Hypothesis~\ref{ring
hypothesis}, but is a first step towards achieving it.  The same
techniques can be applied to a host of other rings, and I discuss some
of these possibilities at the end of the chapter.

In Chapter~\ref{fluid chapter}, I examine the fluid limit method
introduced by Dai~\cite{Dai}.  The chapter is divided into two halves.
In the first half, I translate the fluid limit theorems to discrete
time.  This half is sufficient to prove the stability of the standard
Bernoulli ring whenever $r<1$, not merely for large $N$.  In the
second half, I generalize the result a bit, so that (for instance)
arrivals can be generated by a hidden Markov process, rather than just
by arrival processes with i.i.d.\ interarrival times.  This change
leads to proofs of the stability and finiteness of expected queue
length on rings much more general than the standard Bernoulli ring.

In Chapter~\ref{analytic chapter}, I translate a theorem of
Zazanis~\cite{zazanis} to the discrete time case.  This result shows
that in the face of Bernoulli arrivals, the expected queue length of
an ergodic network is an analytic function of the nominal load $r$,
for $r\in[0,1)$.  This means that light traffic calculations of the
expected queue length are actually well defined.  I can then make some
explicit light traffic calculations and draw various conclusions.

The final chapter, Chapter~\ref{ringlike chapter}, concerns itself
with ringlike networks, rather than rings themselves.  The wrapped
butterfly is an example: a $d$-dimensional wrapped butterfly shares
certain features in common with a $d$-node ring, as both are regular,
layered graphs with very high degrees of symmetry.  I extend several
of the results of the earlier sections to these more complicated
topologies.  I begin by proving a fluid-style stability result on all
networks that use convex routing.  I continue with a result about the
graph structure of butterfly networks.  I show that, under various
conditions, a concatenated pair of butterfly graphs forms a
superconcentrator.  This means that we can lock down node-disjoint
paths between any subset of input and output nodes (of the same size).
This result is of a different flavor than the other proofs, being more
graph theoretical than probabilistic.

There are also several appendices.  Probability and queueing theory
foundations are reviewed in Appendix~\ref{probability chapter}.  For
the reader unfamiliar with fluid limits, I include a complete proof of
the stability results applicable for packet routing in
Appendix~\ref{dummies chapter}.  The results are the same as those of
Dai~\cite{Dai} (actually, weaker), but the proofs are much shorter and
simpler, and the Appendix is self-contained.  In Appendices~\ref{fluid
limit examples chapter} and~\ref{fluid limit counterexamples chapter},
I list a number of useful examples and counter-examples from the world
of fluid limits.  Finally, I wrote many computer programs to help me
calculate stationary distributions.  I discuss some of the more
interesting details of this process in Appendix~\ref{computer
chapter}.

For ease of reference, I have included an index.  It lists
the locations of definitions and main theorems.

\subsection{The New Results}
For the reader curious about which of these results are new, 
here is a brief list. 

In Chapter~\ref{exact chapter}:
\begin{itemize}
\item I calculate the exact expected queue length on 
an $N>1$ node nonstandard Bernoulli ring with parameter
$L=2$, for protocols GHP, EPF, SIS, CTO, and FTG.
\item I calculate the stationary distribution for 
the nonstandard Bernoulli ring with parameter $L=2$
for all $N$ under GHP.  This result allows an exact solution
for a 3 node standard Bernoulli ring, and
a 5 node bidirectional standard Bernoulli ring.
\end{itemize}

In Chapter~\ref{bounds chapter}:
\begin{itemize}
\item The number of packets in one queue of a standard Bernoulli ring
is bound by the number of packets in a single server queue with
Bernoulli arrivals and geometric service times.  An $O(1)$ bound on
the expected queue length per node follows for nominal load
$r<\frac{1}{2}+\epsilon$, for an explicit (but very small) $\epsilon$.
\item For any $r<1$, on all sufficiently
large standard Bernoulli rings, the network is ergodic.  (But
see the stronger results of Chapter~\ref{fluid chapter}.)
\item 
For any $r<1$, the expected queue length per node
on a standard Bernoulli ring has an $O(1)$ bound.
\item 
For any $r<1$, the expected delay of a packet on
an $N$ node standard Bernoulli ring is $\Theta(N)$.
\item I briefly discuss how to extend these techniques to
other Bernoulli rings:
\begin{itemize}
\item For a (standard) bidirectional ring,
the expected queue length per node has an $O(1)$ upper bound.
\item For an $N$ node nonstandard Bernoulli ring with parameter $L$,
if $N$ is constant and $L\rightarrow \infty$,
the expected queue length per node has a tight $\Theta(1)$ bound.
\item For an $N$ node nonstandard Bernoulli ring with parameter $L$,
if $L$ is constant and $N\rightarrow \infty$,
the expected queue length per node
is lower bounded by $\Omega(1)$ and upper bounded
by $O(\log (N))$.  
\item Suppose that queues are finite, with some bound
$B_{N}$.  As $N\rightarrow \infty$, we may let
$B_{N}\rightarrow \infty$.  
The expected queue length per node has an $O(1)$ bound.
\end{itemize}
\end{itemize}

In Chapter~\ref{fluid chapter}, determining the novelty of the results
is a little bit more complicated; the proofs are very
closely tied to a paper by Dai~\cite{Dai}.  My own contributions
amount to the following:
\begin{itemize}
\item
I prove a discrete time fluid limit theorem.  (This result is a 
modification of a theorem of Dai's.)
\item
A corollary of the previous result is the stability of
any ring, under any greedy protocol, for any maximum nominal
load $r<1$.
\item The fluid limit technique holds when the arrival,
service, and routing processes are hidden Markov chains.  This
generalization of Dai's results requires very little proof, because
the hard work has already been done by Dai; only some careful
definitions and reflection are needed.
\end{itemize}

In Chapter~\ref{analytic chapter}:
\begin{itemize}
\item I provide a rigorous justification of light traffic limits
on Bernoulli rings.
\item The stationary distributions for
standard Bernoulli rings with $N>3$ nodes are \emph{not} product form.
\item The stationary distributions for
geometric Bernoulli rings are \emph{not} product form, except for
a finite number of exceptions.
\item Computer-aided calculations show that
the expected queue length of a 4 node standard Bernoulli
ring is not a rational function of degree less than 18.
\item Consider the expected total number of packets in queue in a
single-class network with rate $p$ Bernoulli arrivals.  The expected
value is an absolutely monotonic function of $p$.
\end{itemize}

In Chapter~\ref{ringlike chapter}:
\begin{itemize}
\item On any network with convex routing and nominal loads less
than one, with any greedy protocol, the network is ergodic.
\item Suppose we have two $d$ dimensional butterflies.  Choose
two permutations $\pi_{1}$ and $\pi_{2}$ on the $d$ dimensions.
Then if we permute the layers of the first butterfly by
$\pi_{1}$ and the second butterfly by $\pi_{2}$, and concatenate
the graphs, the resulting graph is a superconcentrator.
\item Suppose we take two graphs, each isomorphic to a butterfly,
and concatenate them.  The resulting graph concentrates subsets
whose cardinality is a power of two.
\end{itemize}

The appendices are mostly abbreviated versions of material that can be
found elsewhere.  There are a few exceptions.  Although the results in
Appendix~\ref{dummies chapter} are similar to (in fact, weaker than)
those of Dai~\cite{Dai}, the proofs are fairly different.  Several of
the stability proofs from Appendix~\ref{fluid limit examples chapter}
appear to be new, namely the theorems in Sections~\ref{LIS stability
section} and~\ref{SIS stability section}, and the corollaries from
Section~\ref{round robin stability section}.  Finally, in
Appendix~\ref{computer chapter}, Theorem~\ref{computer theorem} is
new.

\section{Ring Details}
I still have to specify a few more picayune details about the ring.  
As mentioned before, I will be using a non-blocking model of the
ring, so that if a packet departs at node $i$, then a new
packet can be inserted on the same time step.  

If we look at the packets waiting at a node, we will consider the packet
that is about to move to be \emph{in the ring}; the other packets are
\emph{in queue} at that node. I sometimes refer to a packet travelling
in the ring as a \emph{hot potato} packet.\index{hot potato}

It's important to distinguish between the packets ``at a node''
and those ``in queue''.  The queue doesn't include
the packet (if any) in the ring, so there may be one fewer
packet in queue than at the node.

In discrete time, there's a non-zero probability that arrivals,
departures and routing occur at the same time.  Therefore, we have to
settle on the order in which these events occur.  Let us specify that
one time step consists of routing current packets, possibly inducing
some of them to depart, and \emph{then} inserting new arrivals.
On a standard Bernoulli ring, the choice of ``route, then arrive'' or
``arrive, then route'' only amounts to an $O(1/N)$ difference 
in the expected queue length per node, so it doesn't really matter
much which model we use.

Finally, packet routing theorists and queueing theorists tend to model
packet routing problems slightly differently.
Packet routing researchers like to view edges of a network as wires,
and allow only one message to cross a wire at a time.  Therefore, queues wait
on edges.  Queueing theorists, on the other hand, prefer to view packets as 
waiting at nodes.  I will be adopting the queueing theorists' point
of view. To translate from the first perspective to the second, we
can simply consider the edge graph of the packet routing network.

\section{The Bidirectional Ring} \label{bidirectional section}
\index{bidirectional ring}
Most of my analysis in this thesis will be directed towards the
unidirectional ring, where all the packets travel in a fixed
direction, e.g.\ clockwise.  It is natural to wonder what happens if we
have a bidirectional ring, where packets travel either clockwise or
counterclockwise along the shortest path to their destinations.  After
all, this change halves the expected travel distance on the ring.  In
certain circumstances, we can reduce these problems to questions about
the unidirectional ring.

To make this reduction, we need a slightly more general model than the
standard Bernoulli ring:

\begin{definition} \index{nonstandard Bernoulli ring}
A \emph{nonstandard Bernoulli ring} with parameter L is identical to a
standard Bernoulli ring, except that rather than choosing destinations
uniformly from the $N-1$ nodes downstream, the destinations are chosen
uniformly from the $L$ nodes downstream.  (If we select $L=N-1$, we
regain a standard Bernoulli ring.)  The nominal load $r$ is
$\frac{L+1}{2}p$.\index{nominal load!nonstandard Bernoulli ring}
\end{definition}

Suppose we have an $N$-node bidirectional ring with Bernoulli
arrivals.  (For simplicity, imagine that $N$ is odd, so that there
exists a unique shortest path between any pair of nodes.)  Suppose
further that there are two edges between adjacent nodes: a clockwise
edge and a counterclockwise edge.  That way, node $i$ can send a
packet to node $i+1$ at the same time that node $i+1$ sends a packet
to node $i$.  Consider only the packets that travel in a clockwise
direction.  These packets form an $N$-node nonstandard Bernoulli ring
with parameter $L=(N-1)/2$.  The counterclockwise packets form the same
system.  

These two networks are highly dependent (after all, if a clockwise
packet arrives at a node, then a counterclockwise packet cannot).
However, by the linearity of expectation, the expected queue length at
a node in the bidirectional ring is exactly twice the expected queue
length at that node on the nonstandard Bernoulli ring with the $L$
given above.  Therefore, the solutions to nonstandard Bernoulli rings
in Chapters~\ref{exact chapter} and~\ref{bounds chapter} translate to
results about bidirectional rings.


\section{Standardized Notation}
As a kindness to the reader, I have tried to make my notation
uniform throughout this thesis.  In particular,
\begin{itemize}
\item The number of nodes in a network is $N$.
\item The maximum lifespan of a packet, i.e.\ the longest
path in the network, is $L$.  (For the standard Bernoulli ring,
$L=N-1$.)
\item The probability of a packet arriving at a node 
on one time step in a Bernoulli network is $p$.
\item The nominal load of a node is $r$.  (For a standard
Bernoulli ring, $r=\frac{N}{2}p$.  For a nonstandard Bernoulli
ring, $r=\frac{L+1}{2}p$.)
\end{itemize}

\section{A Little History}


There is a large literature pertaining to packet routing on ring
networks.  I survey some of the results that bear more directly on
this thesis below.\footnote{A very popular model of packet routing on
a ring is a \emph{token exchange ring}, where one node (the one with
the ``token'') is allowed to broadcast unimpeded to all the other
nodes.  Although this network's name has the word ``ring'' in it, its
topology is really more of a complete graph, so it doesn't relate to
this thesis.\index{token ring}}

\begin{itemize}
\item 
Coffman et al, \cite{Leighton95} and \cite{Leighton98}, analyze
the geometric Bernoulli ring:
\begin{definition} \index{geometric Bernoulli ring}
An \emph{$N$-node geometric Bernoulli ring} is an $N$-node directed
cycle.  Packet arrivals occur according to a Bernoulli arrival process
at each node.  Packet destinations are geometrically distributed.  Packets
are routed in a greedy fashion.
\end{definition}
(Unlike the standard Bernoulli ring, there is essentially
only one greedy protocol on a geometric Bernoulli ring.)

Through very careful and clever arguments, they show that a geometric
Bernoulli ring has $\Theta(1/N)$ expected queue length for any nominal
load $r<1$.  Their argument relies on showing that the greedy protocol
is optimal on geometric Bernoulli ring across a wide class of
protocols, and then finding another protocol with $O(1/N)$ expected
queue length.\footnote{A careful reader might note that there is a
slight error in both papers: the authors fail to prove the ergodicity
of the protocol that provides the upper bound.  Since the protocol is
not greedy, it's not possible simply to quote the standard results.
However, the generalized fluid limit techniques of Chapter~\ref{fluid
chapter} should be applicable, with some effort.}
(The $\Omega(1/N)$ lower bound follows easily; see
Section~\ref{lower bound section}.)

Coffman et al. observe that the expected queue length of a geometric Bernoulli
ring with nominal load $r<1$ is an upper bound on the expected
queue length of a standard Bernoulli ring with nominal load $2r$.
(This fact follows readily from a stochastic dominance argument.)
It follows that the expected queue length of a standard
Bernoulli ring is $\Theta(1/N)$ if $0\leq r < \frac{1}{2}$.

Why can't we use the same techniques on the standard Bernoulli ring as
we do on the geometric Bernoulli ring?  Well, all the packets on a
geometric Bernoulli ring are essentially indistinguishable; because of
the geometric distribution on travel distances, the past history of a
packet doesn't effect its future probability of leaving the ring.
This property makes stochastic dominance arguments straightforward, so
it's easy to find other, more analytically tractable protocols that
can bound the expected queue length of the greedy protocol.  On the
other hand, the conditional probability that a packet departs the
standard Bernoulli ring is very much dependent on how far it's
travelled.  It is correspondingly very, very difficult to find
networks that could stochastically dominate all these conditional
probabilities.

Both papers mention Hypothesis~\ref{ring hypothesis} as a 
vexing open question.

\item The Greedy Hot Potato protocol may be the most natural to
use on the ring, but it's certainly not the easiest to analyze. 
Kahale and Leighton~\cite{kahale95} use generating functions
to calculate a bound on the expected packets per node under the Farthest
First protocol (where the packet with the most distant destination 
gets precedence over other packets.)  The bound is:
\[\frac{4r}{N}\left(\frac{1}{2(1-r)^{2}}-\frac{1}{2}\right)=O(1/N)\]
These arguments depend very heavily on the protocol, and 
don't translate to GHP.

\item There are some fairly impressive and general results on stability
and expected queue length on Markovian networks.
\begin{definition} \index{Markovian}\index{single class}\index{classless}
A network is \emph{Markovian} if the behavior of any two packets at a
queue is stochastically identical.  Thus, to specify a Markov chain,
it is sufficient to specify how many packets are at each node (as
opposed to specifying the class of each packet).  A network with this
property is also called \emph{classless}, or \emph{single classed}.
\end{definition}
The geometric Bernoulli ring is an example of a Markovian network.

The first breakthrough in the subject came from Stamoulis and 
Tsitsiklis~\cite{kn:j7}.  They showed how to bound the expected
queue length under a First In, First Out (FIFO) protocol and 
(continuous time) deterministic
service by a processor sharing protocol with exponential service times.
It's easy to calculate the expected queue length of the latter network,
so the method provides explicit upper bounds on expected total
queue length in the network.  

Stamoulis and Tsitsiklis used their results on hypercubes and
butterfly graphs, but their proofs clearly apply to any layered
network.  Mitzenmacher~\cite{kn:j6} used these results to analyze the
$N\times N$ array, for instance.  However, the technique broke down on
networks with loops, such as rings or tori.

This problem was very nicely resolved by Harchol-Balter~\cite{mor} in
her dissertation.  She showed how to construct the same simple upper
bounds for any Markovian network, including those with loops.

If we applied these results naively to a standard Bernoulli ring, we
would get an $O(1)$ bound on the expected queue length per node.  This
result is akin to the $O(1)$ bound in Theorem 2 from Coffman et
al.~\cite{Leighton95} on the geometric Bernoulli ring.  Unfortunately,
a standard Bernoulli ring is emphatically not Markovian, and the
analysis fails.

\item Since the standard Bernoulli ring model runs in discrete time,
and each packet needs only one unit of time to cross an edge,
it is tempting to imagine that there should be some 
very general solutions for the stationary probabilities,
analogous to the solutions to a Kelly network in continuous
time.  One successful result along those lines is due
to Modiano and Ephremides~\cite{modiano}.  They
show exact solutions for expected queue length on a tree network
where all paths lead back to the root node.

Can this result be extended for arbitrary layered graphs?  
Modiano believes that this is true, but the proof is non-obvious,
to say the least.  (If true, this would resolve an open question
in Stamoulis and Tsitsiklis~\cite{kn:j7} concerning the expected
queue length per node on a butterfly.)  Extending it 
to networks with feedback, like a ring, seems impossible.

\item Rene Cruz~\cite{Cruz1},~\cite{Cruz2} 
developed a model of packet routing with ``burstiness'' constraints.
These constraints boil down to the following: for each edge, fix
$r,s>0$.  Then in $T$ time steps, at most $\lfloor Tr+s \rfloor$
packets can arrive.  In Cruz~\cite{Cruz2}, he proves a stability
result on a model of a 4 node ring.  

Georgiadis and Tassiulas~\cite{georgiadis1} show that Cruz's model
of the ring is stable under a greedy protocol, on a ring of 
any size, so long as the nominal loads are less than one.

For stochastic arrival processes like the Bernoulli process, Cruz's
burstiness assumptions are too restrictive, so his stability theorems
don't apply.
\item Cruz can be considered one of the forefathers of adversarial
queueing theory.\index{adversarial queueing theory}  The intent of
adversarial queueing theory is to prove that even in the face of
maliciously planned packet insertions, certain networks and protocols
are still stable.

More specifically, fix an integer $T$ and some $0\leq r \leq 1$.
Imagine that an adversary injects packets such that for any fixed edge
$e$, the number of packets injected during the previous $T$ time steps
that need to use $e$ is less than $\lfloor rT\rfloor$.
A network and protocol is stable with load $r$ if
for any $T$, there is a maximum number of packets $M_{T}$ that
can appear in the network simultaneously.  (Thus, the adversary
``wins'' if he can make the number of packets in the system
grow unboundedly.)

Adversarial queueing theory was originally introduced by
Borodin et al.~\cite{kn:j8}.  The result of interest to
us is from Andrews et al.~\cite{Andrews}, where the authors
show that the ring is adversarially stable under any greedy protocol,
for any $r<1$.  A very interesting converse was proved by
Goel~\cite{Goel}, who showed that any network containing more than
one ring is adversarially unstable for some protocol and some
$r_{0}<1$.  An equivalent result for stochastic stability
is unknown but desirable.

Almost any stochastic arrival process (like the Bernoulli)
has a potential for unbounded ``burstiness''.  This
fact prevents the adversarial results from applying to a 
standard Bernoulli ring in any obvious way.
\item 
Around 1995, a major advance was made in the general study of
stability on queueing networks.  Dai~\cite{Dai} introduced fluid limit
models, a method of rescaling a stochastic system to reduce it to a
deterministic one.  One of the consequences of this theory was a proof
that in continuous time, the (generalized Kelly) ring is stable under
any greedy protocol, so long as the maximum nominal load on any node
is less than one (see Dai and Weiss~\cite{Dai_and_Weiss}).  Further
refinements of the theory allowed proofs of the finiteness of the
expected queue length assuming bounded variance of the arrival and
service times of the network (see Dai and Meyn~\cite{Dai_and_Meyn}).
I'll be looking at fluid limits in greater detail in
Chapter~\ref{fluid chapter}.

\item
Gamarnik~\cite{Gamarnik} managed to prove an adversarial fluid
limit theorem, providing a way to prove adversarial stability
by analyzing a more complicated fluid limit.  As an example,
he provided yet another proof of the adversarial stability of the ring.

\end{itemize}

%% file: exact.tex
\chapter{Exact Solutions} \label{exact chapter}

\section{Introduction} \label{exact intro section}

In this chapter, I'm going to perform exact calculations of the
expected queue length and stationary distribution of several families
of rings.  For a brief review of stationary distributions and discrete
time Markov chains, please see Section~\ref{basic probability section}.

Recall the nonstandard Bernoulli ring with parameter $L$ introduced in
Section~\ref{bidirectional section}.  
A nonstandard Bernoulli ring can be specified by it's size $N$
and its maximum path length $L$. If $L$ is fairly small relative
to $N$, then we can imagine that packets only need to communicate in a small
local neighborhood of themselves.  If, on the other hand, 
$L\geq N$, then a packet can cross the same node more than once.

I can only hope to calculate exact solutions in the simplest
cases; even then, some of the proofs are fairly involved.
I will exactly calculate the expected queue length for
the case of $N=1$ for arbitrary $L$, and $L=1$ or $2$ for arbitrary
$N$.
The results hold for several different protocols.
I will also find the stationary distribution for $L=2$ and
all $N$ under the GHP protocol.

\section{The One Node Ring}

Remember that the standard routing protocol for a ring is 
\emph{Greedy Hot Potato} (GHP), where packets travelling in the ring have
precedence over packets in queue.  In a one node ring, this means that
the packet which is being serviced remains in service until it leaves
the queue (i.e.\ no pre-emptions occur.)  Observe that this protocol is
the same as \emph{First In, First Out} (FIFO):
\begin{definition} \index{FIFO protocol} \index{FCFS protocol}
The \emph{First In, First Out (FIFO)} protocol, as its name suggests,
gives priority to earlier arrivals at a node.  That is, the $n$th
packet arriving at the node will be the $n$th packet departing.  
(Simultaneous arrivals are numbered randomly.)  This protocol is
also called \emph{First Come, First Served (FCFS)}.
\end{definition}
For an $N$ node ring with $N>1$, FIFO and GHP are \emph{not} the same.
Note that since a packet doesn't really ``travel'' anywhere on a $N=1$
node ring, some people find it might be more natural to view a packet
as having an amount of work associated with it.  (So, for example,
rather than ``travelling'' in place for $k$ time steps, we say the
packet has $k$ units of work.)  However, I will stick with the
``travel'' metaphor.

\begin{theorem}
Suppose we have a 1 node nonstandard Bernoulli ring with parameter $L$,
and we are routing using GHP.  Suppose that the arrival rate is
$p=\frac{2}{L+1}r$.  Then the expected queue length is:
\[E[\mbox{queue length}] = \frac{L-1}{L+1}\frac{2r^{2}}{3(1-r)} \]
\end{theorem}
\note Therefore, for a fixed $r$, the expected queue length is
$O(1)$ in $L$.

\proof
Since $N=1$, we have a single server queue, and can apply standard
tools from queueing theory.  The ergodicity of a single server queue
for nominal loads less than one follows from typical arguments (e.g.\
Gallager~\cite{kn:b3}, Chapter 7).  For the expected queue length,
recall the discrete time version of the Pollaczek-Khinchin formula
(Theorem~\ref{p-k theorem}):
\[E[\mbox{queue length}] = 
\frac{\lambda^{2}(E[Z^{2}] - E[Z])}{2(1-\lambda E[Z])}\]
where $\lambda$ is the arrival rate (i.e.\ $p$), and 
$Z$ is the distribution of service times (i.e.\ uniform
between 1 and $L$.) So, since
\[E[Z] = \sum_{i=1}^{L} \frac{1}{L}i = \frac{L+1}{2}\]
\[E[Z^{2}] = \sum_{i=1}^{L} \frac{1}{L}i^{2} = \frac{(2L+1)(L+1)}{6}\]
we can plug in and get
\[E[\mbox{queue length}] = \frac{L-1}{L+1}\frac{2r^{2}}{3(1-r)} \]
as desired.
\EOP

We can also make some qualitative comparisons of expected queue length.

\begin{lemma} \label{single node lemma}
\index{Lemma \ref{single node lemma}}
Suppose we are comparing the expected queue length of greedy
protocols
$A$ and $B$ on a single node network.  Suppose that the mean
work of a packet in queue under $A$ is strictly \emph{greater}
than under $B$.  Suppose also that the queue length is 
independent of the expected work in each packet in the queue.
Then it follows than the expected queue length
under $A$ is strictly \emph{shorter} than under $B$.  
\end{lemma}

\proof
Observe that since we are in the single server regime, 
the total amount of work in the queue is constant for all greedy 
protocols.  Also, the expected amount of work of the packet in 
service is also invariant over the protocols (because it's
the mean work per packet).  Now,
\[\Ex[\mbox{work in queue}]=\Ex[\mbox{work per packet $\times$ length of queue}]\]
so, by our independence assumption,
\[=\Ex[\mbox{work per packet}]\Ex[\mbox{length of queue}]\]
Since $\Ex[\mbox{work in queue}]$ is constant, we have the result.
\EOP

Consider, then, the \emph{Farthest To Go}\index{FTG protocol}
(FTG) protocol,
where packets with the greatest distance left to travel
have priority over packets with nearer destinations.  
If a packet arrives with a greater distance to travel than all
the other packets in the system, I allow it to serviced immediately
(so it spends no time in queue.) On a ring, FTG is a well-defined 
protocol.\footnote{Generally, though, FTG does not completely
specify a protocol, since packets from different classes might have the 
same distance to their destinations.}  We can deduce the following
corollary:

\begin{corollary} \label{FTG corollary}
\index{Corollary \ref{FTG corollary}}
Suppose we have a 1 node non-standard Bernoulli ring with parameter
$L=2$.  Then for any arrival rate greater than zero, the expected
queue length under GHP is shorter than under FTG.
\end{corollary}
\note A queueing theorist would probably express this result 
by saying that the Least Remaining Work protocol is worse than FIFO.

\proof
By inducting on time, we can show that under FTG,
all packets in queue need
one unit of service time.  At time $t=0$, there are no packets in
queue, so the result holds.  At time $t$, by induction all packets
in queue need one unit of service, so if a packet arrives needing 
2 units, it will be immediately serviced, and thus removed from the 
queue.  Thus, at time $t+1$, all 
the packets in queue will need one unit of service.

Therefore, under FTG, the mean work per packet in queue is 1,
independent of the queue length.  Under GHP (which is identical to
FIFO), it's $3/2$, independent of the queue length.  By
Lemma~\ref{single node lemma}, we're done.
\EOP

This result may give some plausible hint that GHP has shorter
expected queues than FTG on larger rings.  Nevertheless,
in section~\ref{stationary solution section}, I show that 
GHP and FTG have identical expected queue length if $L=2$ and $N>1$, so 
Corollary~\ref{FTG corollary} is somewhat surprising.

\section{Fixing $L$ in the Nonstandard Bernoulli Ring}

First off, let us consider the case of $L=1$, for a ring of any size.
Since our model
of packet-routing is non-blocking, the only node that a packet
blocks is the node that it arrives at.  Since at most one packet arrives
on each time step, and (with any greedy protocol) 
at least one packet is emitted on each time step
from a non-empty queue, it follows that there are never any packets in
queue.  Therefore, the stationary distribution is of product
form, where the probability of a node being empty is $1-p$; the probability
of there being one packet at that node is $p$.  (``Product form'' is
defined in Section~\ref{basic probability section}.)  The expected queue
length is identically zero.

The case of $L=2$ is much more interesting.  I am going to analyze a
number of different protocols in the following sections, but the
marginal stationary distributions (per node) will all be essentially
the same.  Because the different protocols have slightly different
state spaces, the distributions are formally incomparable, but the
probability that a particular node has $i$ packets in it is the same
across all the protocols.  In particular, the expected queue length
per node (as a function of $r$) is constant across all these
protocols.  Even more surprisingly, the marginal distribution per node
is independent of $N$, for $N>1$.  That is, the expected queue length
per node is independent not only of which of these protocols are
chosen, but also of the size of the ring.

The protocols (which will be defined in Section~\ref{protocols section}) are
Exogenous Packets First (EPF), 
Closest To Origin (CTO), Farthest To Go (FTG), Shortest In System (SIS),
and Greedy Hot Potato (GHP).  GHP is the protocol specified in the 
standard Bernoulli ring in Section~\ref{problem section}, and hence
is of particular interest.  I calculate its full stationary distribution
(not just the marginal distribution per node).
This latter proof is substantially longer than
any of the other proofs, taking up the majority of this chapter.

\section{The Stationary Distribution and Consequences} 
 \label{stationary solution section}

As mentioned in Section~\ref{exact intro section}, the distributional
values (expected queue length, and so forth) are the same for all the
protocols I examine.  In advance of the proofs of the marginal
stationary distributions, I preview the results in this section.

\begin{theorem} \label{main exact theorem}
\index{Theorem \ref{main exact theorem}}
For the GHP, SIS, CTO, FTG, and EPF protocols, on an $N>1$ node ring,
with maximum destination $L=2$, and packet arrival probability $p$,
the stationary probability that a fixed node has $n$ packets in it is:
\[ \Pr (0 \mbox{ packets}) = 1 - \frac{3}{2}p \]
\[ \Pr (1 \mbox{ packet}) = 
\left( 1-\frac{3}{2}p \right)
\frac{3p-p^{2}}{(1-p)(2-p)} \]
and for any $n>1$,
\[ \Pr (n \mbox{ packets}) = 
\left( 1 - \frac{3}{2}p \right)
\frac{2p^{2(n-1)}}{[(1-p)(2-p)]^{n}}   \]
Under GHP, this result also holds if $N=1$.
\end{theorem}

\proof
The proofs follow in the remainder of the chapter.
\EOP

We can use this theorem to calculate various interesting quantities.
The expected queue length per processor is:
\[\sum_{n=1}^{\infty} n\left[\Pr\left(n\mbox{ packets in queue}
\right)\right] 
=\sum_{n=1}^{\infty} n\left[\Pr\left(n+1\mbox{ packets at node}
\right)\right] \]
\[=\sum_{n=1}^{\infty}n\left[2\left(1-\frac{3}{2}p\right)
\frac{1}{(1-p)(2-p)}
\left(\frac{p^{2}}{(1-p)(2-p)}\right)^{n}\right] \]
\[=\frac{2-3p}{(1-p)(2-p)}\sum_{n=1}^{\infty}n\left(
\frac{p^{2}}{(1-p)(2-p)}\right)^{n}\]
\[=\frac{2-3p}{(1-p)(2-p)}
\frac{\frac{p^{2}}{(1-p)(2-p)}}{\left(1-\frac{p^{2}}{(1-p)(2-p)}
\right)^{2}}\]
\[=\frac{2-3p}{(1-p)(2-p)}\frac{p^{2}{(1-p)(2-p)}}{\left( (1-p)(2-p)
-p^{2}\right)^{2}}\]
\[=(2-3p)\frac{p^{2}}{(2-3p)^{2}}\]
\[=\frac{p^{2}}{2-3p}\]
By Section~\ref{queueing theory section}, the expected number of packets 
per processor is equal to:
\[\mbox{(Expected queue length)} +(1-\Pr\mbox{(empty processor))}\]
\[=\frac{p^{2}}{2-3p}+(1-(1-(3/2)p))\]
\[=\frac{p^{2}}{2-3p}+\frac{3}{2}p\]
The expected variance of the queue length per processor (for any
$N$) is equal to:
\[E[(\mbox{packets in queue})^{2}]-
(E[\mbox{packets in queue}])^{2}\]
\[=\left(\sum_{n=1}^{\infty} n^{2}
(\Pr(n \mbox{ packets in queue}))
 \right) -\left(\frac{p^{2}}{2-3p}\right)^{2}\] 
\[=\frac{2-3p}{(1-p)(2-p)}\left[ \sum_{n=1}^{\infty}n^{2}\left(
\frac{p^{2}}{(1-p)(2-p)}\right)^{n}\right]
-\left(\frac{p^{2}}{2-3p}\right)^{2}\] 
\[=\frac{2-3p}{(1-p)(2-p)}
\left[\frac{\frac{p^{2}}{(1-p)(2-p)}}{\left(1-\frac{p^{2}}
{(1-p)(2-p)}\right)^{2}} +
\frac{\frac{2p^{4}}{[(1-p)(2-p)]^{2}}}{\left(1-\frac{p^{2}}
{(1-p)(2-p)}\right)^{3}}\right]
-\left(\frac{p^{2}}{2-3p}\right)^{2}\] 
\[=\frac{p^{2}}{2-3p} + \frac{2p^{4}}{(2-3p)^{2}} -
\left(\frac{p^{2}}{2-3p}\right)^{2}\] 
\[=\left(\frac{p^{2}}{2-3p}\right)^{2} +
\frac{p^{2}}{2-3p}\]
Finally, just for fun, we can calculate the entropy of the queue
length per processor:
\[H \choose{\mbox{queue length}}{\mbox{per processor}}\]
\[=- \sum_{n=0}^{\infty} \Pr(n \mbox{ packets in queue})
\log (\Pr(n \mbox{ packets in queue}))\]
Plugging in, we get
\[-\sum_{n=0}^{\infty}
\frac{2-3p}{(1-p)(2-p)}\left(\frac{p^{2}}{(1-p)(2-p)}\right)^{n}
\log\left[
\frac{2-3p}{(1-p)(2-p)}\left(\frac{p^{2}}{(1-p)(2-p)}\right)^{n}
\right] \]
\[=-\left[\sum_{n=0}^{\infty}
\frac{2-3p}{(1-p)(2-p)}\left(\frac{p^{2}}{(1-p)(2-p)}\right)^{n}
\log\left(\frac{2-3p}{(1-p)(2-p)}\right)\right]
\]
\[-\left[
\sum_{n=0}^{\infty}
n\frac{2-3p}{(1-p)(2-p)}\left(\frac{p^{2}}{(1-p)(2-p)}\right)^{n}
\log\left(\frac{p^{2}}{(1-p)(2-p)}\right)\right] \]
\[=-\log\left(\frac{2-3p}{(1-p)(2-p)}\right)
\left[\sum_{n=0}^{\infty}
\frac{2-3p}{(1-p)(2-p)}\left(\frac{p^{2}}{(1-p)(2-p)}\right)^{n}
\right] \]
\[
-\log\left(\frac{p^{2}}{(1-p)(2-p)}\right)
\left[\sum_{n=0}^{\infty}
n\frac{2-3p}{(1-p)(2-p)}\left(\frac{p^{2}}{(1-p)(2-p)}\right)^{n}
\right]\]
Observe that the first sum in square brackets is the sum of
the probability from all states, which equals 1.  The second
sum in square brackets is the expected queue length, which
we know is $p^{2}/(2-3p)$. So,
\[=-\log\left(\frac{2-3p}{(1-p)(2-p)}\right) -
\frac{p^{2}}{2-3p}\log\left(\frac{p^{2}}{(1-p)(2-p)}\right)\]
It's pretty easy to verify that this entropy function equals 0
when $p=0$, diverges to positive infinity at $p=2/3$, 
is continuous on $0\leq p<2/3$, and is monotonically increasing.
For GHP, since the distribution has product form, the entropy of all
$N$ processors is $N$ times the entropy per processor.

These results also allow exact analysis of two cases of special
interest.
\subsection{The 3 Node Standard Bernoulli Ring}

If $N=3$, then any processor can send a packet to any other processor.
Observe that this network is a 3 node standard Bernoulli ring.  The previous
section allows us to calculate exactly the expected delay, expected
queue length, variance, etc.

\subsection{The 5 Node Bidirectional Ring}
\index{bidirectional ring!5 node ring}
Suppose we have a 5-node bidirectional ring, where packets take
the (unique) shortest path to their destination, the destinations
are distributed uniformly over the other processors, and
packets arrive with probability $p$.  Suppose that a processor
can send out 2 packets in 1 turn as long as the packets are
using different edges.  (There are two edges between adjacent
nodes, so that node $i$ can send node $i+1$ a packet at the same
time that $i+1$ sends $i$ a packet.)  Packets arrive at a node
according to a Bernoulli process, per usual.
	
Then, as described in Section~\ref{bidirectional section},
we can decompose the ring into two unidirectional rings (in
opposite directions), each operating with an effective arrival
rate of $\hat{p}=p/2$.  The arrival processes into these
two rings are correlated, but since expectation is linear,
this correlation doesn't effect the expected queue length.
The expected queue length is then
\[2\frac{\hat{p}^{2}}{2-3\hat{p}} =
\frac{p^{2}}{4-3p}\]
Note that $p\leq1 \Rightarrow \hat{p}\leq1/2$, yet the
critical point is $\hat{p}=2/3$.  Therefore, the system is 
always stable.  The largest expected queue length occurs when
$p=1$, giving $E(\mbox{queue length})=\frac{1}{4-3}=1$.

\section{The EPF, SIS, CTO, and FTG Protocols}
 \label{protocols section}

As I'll show below, for the $L=2$ case, all four 
of these protocols can be viewed as
functionally identical.  (For larger $L$, this is not
necessarily true, and for non-ring networks, it's almost never
true.) I will now define each of these protocols in turn.

The Exogenous Packets First (EPF)\index{EPF protocol} 
protocol always prefers an
exogenous arrival to an internal arrival.  (A packet arrives exogenously
if it has just been inserted from the Bernoulli arrival process;
an internal arrival is a packet that has been routed from another
node in the network.)  Simply specifying the priority of exogenous
arrivals over internal arrivals does not usually fully specify a
protocol for an arbitrary graph.  But when the maximum path 
length is 2 and there are only two classes at each node (exogenous
arrivals and internal arrivals), then
everything is well defined.  Note
that, since there is at most one exogenous arrival to a node on
each time step, and it has priority, the exogenous packets never
wait in queue; a packet is only (possibly) queued after its
first step, at which time it has become an internal packet.

The Shortest In System (SIS)\index{SIS protocol}
protocol dictates that if two packets
are contending for an edge, the packet with the most recent insertion
into the network gets precedence.  This means that if a packet
is injected into a node, it is guaranteed to move on the next 
time step.  The only packets that can wait in queues are packets that
have already moved one step but have a second step left to take.
Therefore, exogenous packets 
have priority.  Thus, SIS is the functionally the same protocol
as Exogenous Packets First (EPF).  

The Closest To Origin (CTO)\index{CTO protocol}
protocol gives priority to the packet that
is closest to its own origin (i.e.\ point of arrival to the ring).
Since we're on a ring, this specifies
a unique class of packets.  Since packets travel only one or two
spaces, then the packet closest to its origin is the packet that
has just been exogenously inserted.  In other words, CTO is identical to EPF.

The Farthest To Go (FTG)\index{FTG protocol} protocol looks at the
destination of the packets in the system and gives priority to the
packets that have the greatest distance left to go.  Suppose, however,
that an exogenous and an internal packet both arrive at a node, and
both have exactly one edge left to cross.  Which gets precedence?  In
some sense it doesn't matter; the two packets are interchangeable, so
whichever choice we make, the behavior of the system (number of
packets in queues) is identical regardless of which packet advances.
Therefore, we might as well specify that the exogenous packet advances
first.  So, if an exogenous arrival has a destination two nodes away,
it has priority because it is travelling farther than any other packet
at that node; if it has a destination one node away, by the previous
observation, it has precedence over internal packets.  Thus, FTG is
identical to EPF.

SIS, CTO, and FTG are all well-defined on any ring network (not just
with $L=2$), but are not necessarily well-defined on networks with
arbitrary topology.  They are meaningful if and only if the
probability of a packet choosing any particular path is a function of
its total path length.  EPF can be defined on a network
with arbitary topology so long as there are only two
classes of packets present at any node: exogenous and
internal.  (In other words, all internal packets behave identically.)
If the maximum path length is two, then EPF is a somewhat
natural protocol to use.

We have reduced the problem of understanding SIS, CTO, and FTG to 
understanding EPF.  For ease of reference, I will state this formally:

\begin{lemma} \label{equivalent protocols lemma}
\index{Lemma \ref{equivalent protocols lemma}}
The stationary distributions on a nonstandard Bernoulli ring with
$L=2$ are identical under the protocols SIS, CTO, FTG, and EPF.
\end{lemma}

Next, I'll introduce a lemma that hinges on the fact that the 
maximum path length $L$ is 2.

\begin{lemma} \label{EPF lemma}
\index{Lemma \ref{EPF lemma}}
Suppose we have an arbitrary network with $N$ nodes. Suppose that
\begin{itemize}
\item Packets arrive at node $i$ as a $p_{i}$-rate Bernoulli process.
\item The maximum path length is two.
\item Packets are routed according to EPF.
\item No path crosses itself.
\item If node $i$ has $j$ outgoing edges, then an (exogenous) packet leaving
node $i$ crosses edge $j$ with probability $q_{i,j}$.  It departs
the system with probability $1-\sum_{j} q_{i,j}$.  (If a packet
is not exogenous, then it has already crossed an edge, and must
necessarily depart on its next move.)
\end{itemize}
Then the stationary distribution of internal packets waiting in queue
at node $i$ is
stochastically identical to the total number of packets
at a single server where the arrival process
is a sum of Bernoulli arrivals, and the service time is exponentially
distributed.
(The particular arrival and service distributions are spelled out below.)
\end{lemma}

\proof
Consider node $i$.  Because we are using EPF, the only packets that
queue are internal packets.  An internal packet arrives at node $i$
only if it arrived exogenously at node $j$ on the previous time step,
received priority (because it was exogenous), and then with
probability $q_{j,i}$ elected to travel to node $i$.  This event is a
Bernoulli arrival process with rate $p_{j}q_{j,i}$.  Since these
arrivals at each $j$ are independent of each other, then the total
internal arrivals to the queue at node $i$ consist of a sum of
independent Bernoulli arrival processes.

Suppose that there is a queue of internal packets waiting at node $i$.
We will be able to remove a packet from the queue, unless there is a new
exogenous arrival at node $i$.  Imagining an internal packet waiting at 
the head of the line at node $i$, it has a $1-p_{i}$ chance of leaving
on each time step.  This behavior is identical to giving each
packet an exponentially distributed service time.

(In order to insure that the arrival process and the service times are
independent, we needed to assume that no path crosses itself.)
\EOP

We can also conclude that:

\begin{corollary} \label{EPF ergodic lemma}
\index{Corollary \ref{EPF ergodic lemma}}
If the assumptions in Lemma~\ref{EPF lemma} are true and the nominal
loads are less than one at each node, then the system is ergodic.
\end{corollary}

\proof
The nominal loads are less than one iff the expected number of packets
that arrive on each step that need to use node $i$ is less than one,
for all $i$.  In that case, Lemma~\ref{EPF lemma} implies that the 
marginal distribution of packets queued at each individual node
converges to a (marginal) stationary distribution.  It follows
that the whole system is ergodic.  
\EOP

We can draw another interesting corollary from this lemma:
\begin{corollary}
Suppose that the assumptions of Lemma~\ref{EPF lemma} hold.
Suppose further that we can partition the network's nodes into
disjoint sets $A_{1}, \ldots, A_{k}$ such that no two nodes in the
same partition share an edge.  (For instance, if $k=2$, we have a 
bipartite graph.)  
Finally, suppose that for any node $x \not\in A_{i}$, there is 
at most one edge from $x$ to nodes in $A_{i}$.  Then the
marginal distribution of the state of all the nodes in $A_{i}$ is the
product of the marginal distribution of each
node in $A_{i}$ (which is given in Lemma~\ref{EPF lemma}).
\end{corollary}

\proof
This follows by observing that the arrival and service times of
nodes in the same partition are independent of each other, since
the partition has no internal edges.
\EOP

It seems quite likely that the stationary distribution itself is
of product form, but I will not investigate that idea at the moment.
Instead, let us use Lemma~\ref{EPF lemma} to calculate the marginal stationary
distribution of a node on a ring.

\begin{theorem}
Suppose we have an $N$ node ring, and we are routing packets using
either SIS, CTO, FTG, or EPF. Then the system is ergodic if $p<2/3$
(i.e.\ if the nominal load $r<1$), and the marginal stationary
probability of having $n$ packets in the node is:
\[ \Pr (0 \mbox{ packets}) = 1 - \frac{3}{2}p \]
\[ \Pr (1 \mbox{ packet}) = 
\left( 1-\frac{3}{2}p \right)
\frac{3p-p^{2}}{(1-p)(2-p)} \]
and for any $n>1$,
\[ \Pr (n \mbox{ packets}) = 
\left( 1 - \frac{3}{2}p \right)
\frac{2p^{2(n-1)}}{[(1-p)(2-p)]^{n}}   \]
It follows that all the expected queue length calculations from
Section~\ref{stationary solution section} hold for these
protocols.
\end{theorem}

\proof
From Lemma~\ref{equivalent protocols lemma}, these four protocols are
all interchangeable, so I need only prove the result for EPF.  By
Corollary~\ref{EPF ergodic lemma}, the system is ergodic if $p + p/2 <
1$, i.e.\ if $p<2/3$.  Therefore, there exists a
stationary distribution whenever $p<2/3$.  Since the system is
unstable if $r>1$ by an argument analogous to Lemma~\ref{basic ring
instability lemma}, we have pretty well characterized stability.
(Although I won't prove it, if $r=1$ we get a system that is not
ergodic, but is null-recurrent.)

By Lemma~\ref{EPF lemma}, we can calculate the marginal stationary
distributions when $p<2/3$.  Throughout, we consider some fixed node.  
New internal packets arrive as a rate $p/2$ Bernoulli process.  (That
is, they arrive as a rate $p$ Bernoulli process at the previous
node, and half of them remain in the system.)  An internal packet
departs the node (and the system) iff an exogenous packet does not
arrive.  A non-arrival occurs with probability $1-p$.  

This description gives us a fairly standard birth-death process.  I've
worked out the details of the stationary distribution in
Section~\ref{queueing theory section} (and remember that I'm assuming
that on each time step we route old packets, then insert new arrivals,
and then measure the state).  Let $\pi_{n}$ be the stationary
probability that there are $i$ internal packets at the node.  Then the
result is:
\[\pi_{0} = \frac{1-\frac{3}{2}p}{1-p}\]
\[\pi_{n} = \frac{1-\frac{3}{2}p}{1-p}
\frac{1}{p}
\left[ \frac{p^{2}}{(1-p)(2-p)} \right]^{n}
\]

We want to calculate the stationary distribution for all the
packets, not just the internal packets.  Now, the probability 
of there being $n>1$ packets in the system is the probability of
$n$ internal packets and no exogenous packet, plus $n-1$
internal packets and 1 exogenous packet.  So,
\[ \Pr (n \mbox{ packets at the node})
 = \Pr (n \mbox{ internal packets})\Pr (\mbox{0 exogenous})\]
\[ + \Pr (n-1 \mbox{ internal packets})\Pr (\mbox{1 exogenous})\]
\[ = \pi_{n}(1-p) + \pi_{n-1}p \]
\[ = \frac{1-\frac{3}{2}p}{1-p}
\frac{1-p}{p}
\left[ \frac{p^{2}}{(1-p)(2-p)} \right]^{n}
+ \frac{1-\frac{3}{2}p}{1-p}
\left[ \frac{p^{2}}{(1-p)(2-p)} \right]^{n-1}
\]
\[ = 
\frac{1-\frac{3}{2}p}{1-p}
\left[ \frac{p^{2}}{(1-p)(2-p)} \right]^{n-1}
\left[\frac{1-p}{p}\frac{p^{2}}{(1-p)(2-p)} +1\right] 
\]
\[ = 
\frac{1-\frac{3}{2}p}{1-p}
\left[ \frac{p^{2}}{(1-p)(2-p)} \right]^{n-1}
\left[\frac{p}{2-p} +1\right] 
\]
\[ = 
\frac{1-\frac{3}{2}p}{1-p}
\left[ \frac{p^{2}}{(1-p)(2-p)} \right]^{n-1}
\left[\frac{2}{2-p}\right] 
\]
\[ = 
\left(1-\frac{3}{2}p\right)
\left[ \frac{2p^{2(n-1)}}{[(1-p)(2-p)]^{n}} \right]
\]

For the $n=1$ case, we have
\[ \Pr (1 \mbox{ packet at the node}) 
 = \pi_{1}(1-p) + \pi_{0}p \]
\[ = \frac{1-\frac{3}{2}p}{1-p}
\left[(1-p)\frac{p/2}{(1/2)(2-p)(1-p)}  + p\right] \]
\[ = \frac{1-\frac{3}{2}p}{1-p}
\left[\frac{p}{2-p}  + p\right] \]
\[ = 
\left( 1-\frac{3}{2}p \right)
\frac{3p-p^{2}}{(1-p)(2-p)}
\]
The probability of there being no packets in the system is
\[\Pr (0 \mbox{ packets at the node}) = \pi_{0}(1-p) = 1-\frac{3}{2}p\]
and we are done. \EOP

Observe that the marginal stationary probability of there being
$n$ packets in a queue is identical to the GHP case.

\section{The GHP Protocol}

The remainder of this chapter is dedicated to calculating the
stationary distribution of the GHP protocol (not just the marginal
stationary distribution per node, as with the other protocols).
Let us begin with a description of the stationary distribution.
The information from Section~\ref{stationary solution section}
does not give us quite enough information to specify a Markov chain,
so I will need to refine the state description.

There are a number of ways of specifying the state of the Markov
chain.  For instance, we could specify the destination of every packet
in the system (including packets in queue).  Since the packets waiting
in queue are stochastically interchangeable, though, we only really
need to specify the destinations of the packets travelling in the
ring, and the number (but not the destinations) of the packets in
queue.  This is the model I will use in this chapter.  On the other
hand, it is sufficient to know the origin of each packet in the ring,
rather than its destination, because the probability of a packet
departing on the next step is a function of the number of steps the
packet has already travelled.  I'll use that model in Chapter~\ref{bounds
chapter}.  However, all the models are essentially equivalent,
e.g.\ the expected queue lengths are identical regardless of the model.

Let us begin with some notation.  The state of the 
ring is determined by the state of each of its processors.
I will denote a processor with $n$ packets in its queue and
a hot potato with $t$ steps left to travel as:
\[\stq{n}{t}\]
and the ground state (no queue, no hot potato) as:
\[\st{X}\]
Note that on our parameter $L=2$ ring, $t=1, 2$ or $X$, that
$n\in\natu$, and that if $t=X$ then $n=0$.

My guess for the probability distribution is that it is
of product form (so we can calculate the probability of 
the state of all $N$ processors by multiplying the probability
of the state of each processor), and the probability per
processor is:
\begin{equation} \label{stationary equations}
\Pr\st{X}
=\left(1-\frac{3}{2}p\right)
\end{equation}
\[
\Pr\st{1}
=\left(1-\frac{3}{2}p\right)
\frac{p}{1-p}
\]
\[
\Pr \stq{n}{1}
=\left(1-\frac{3}{2}p\right)
\frac{p^{2n}}{[(1-p)(2-p)]^{n+1}}(2-p)\mbox{\,\,\,\,\,(for $n\geq1$)}
\]
\[
\Pr \stq{n}{2}
=\left(1-\frac{3}{2}p\right)
\frac{p^{2n}}{[(1-p)(2-p)]^{n+1}}p\mbox{\,\,\,\,\,(for $n\geq0$)}
\]

Assuming that our guess is correct, it shouldn't be too difficult 
in principle to verify it-- we just check the balance equations:
\[\pi(\sigma)=\sum_{\tau}\pi_{\tau} \Pr(\tau \rightarrow\sigma)\]
where $\sigma$ and $\tau$ are states of the system, $\pi(\sigma)$
is our guess for the stationary probability of state $\sigma$, and  
$\Pr(\tau \rightarrow \sigma)$ is the probability of travelling from
$\tau$ to $\sigma$ in one step.  Now, calculating $\pi(\sigma)$ is
fairly simple, and calculating $\Pr(\tau \rightarrow\sigma)$ isn't 
too bad either, assuming that $\tau$ actually precedes $\sigma$ with
non-zero probability.  However, finding the $\tau$s that precede
$\sigma$ (i.e.\ figuring out what states precede any given state)
appears to be very difficult to do in general.  I'll use a number of 
tricks to reduce the problem to checking a finite number of states
(actually, classes of states), and then verify that the balance
equations hold on them.

In general outline, I will begin by verifying the claim for the $N=1$
case.  I will continue by induction on $N$.  For fixed $N$, however,
there are still an infinite number of cases, so I will reduce the
problem to one with bounded queues (all queues of length $\leq 2$.)
At this point, we can cut the ring at two points and rejoin them to
form two smaller subrings and use induction on the smaller rings.
Cutting the ring is a fairly delicate operation in some cases, and
takes up the body of the proof.

\section{N=1, L=2}
I want to verify that the guessed stationary distribution for the ring
(Equations~\ref{stationary equations}, page~\pageref{stationary
equations}) satisfies the balance equations for the 1-node ring.
This  verification is straightforward.

\begin{lemma} \label{N=1,L=2 lemma}
\index{Lemma \ref{N=1,L=2 lemma}}
The stationary distribution for a 1-node nonstandard ring with 
parameter $L=2$ is given by Equations~\ref{stationary equations}.
\end{lemma}

There are 5 cases to consider.

\begin{itemize}
\item
The ground state, $\st{X}$.
By Little's theorem (or the ``Utilization law''), the
probability that the processor is empty is $1-r$, where
$r$ is the fraction of loading, in this case $(3/2)p$. 
(See Section~\ref{queueing theory section} for details.)  This
matches our guess for the stationary probability.
\item
The state $\st{1}$.
If we write down the balance equation for the ground state,
we get
\[\Pr \st{X}
=
(1-p)\left[
\Pr\st{X} +
\Pr\st{1}
\right]\]
Since we now know $\Pr\st{X}$,
we can solve and find that
\[\Pr\st{1}
=\frac{p}{1-p}\Pr\st{X}
=\left(1-\frac{3}{2}p\right)\frac{p}{1-p}\]
\item
The state $\st{2}$.
The probability flowing in is
\[\frac{p}{2}
\Pr\st{X} +
\frac{p}{2}
\Pr\st{1}
+\frac{1-p}{2}
\Pr\stq{1}{1}
\]
\[=\left(1-\frac{3}{2}\right)\left[
\frac{p}{2}\left(1+\frac{p}{1-p}\right) +
\frac{(1-p)p^{2}}{(1-p)^{2}(2-p)}\right]\]
\[=\left(1-\frac{3}{2}\right)\frac{2p}{(1-p)(2-p)}\]
\item
The state $\stq{n}{2}$, for
$n>0$. The probability flowing in is
\[\frac{p}{2}
\Pr\stq{n}{1}
+\frac{1-p}{2}
\Pr\stq{n+1}{1}\]
\[=\frac{1}{2}\left(1-\frac{3}{2}p\right)\frac{p^{2n}}{[(1-p)(2-p)]^{n+1}}
(2-p)\left[p +\frac{(1-p)p^{2}}{(1-p)(2-p)}\right]\]
\[=(1-\frac{3}{2}p)\frac{p^{2n+1}}{[(1-p)(2-p)]^{n+1}}\]
\item
The state $\stq{n}{1}$, for $n>0$. The probability flowing in is
\[\left[\frac{p}{2}
\Pr\stq{n}{1}
+\frac{1-p}{2}
\Pr\stq{n+1}{1}
\right]
+(1-p)
\Pr\stq{n}{2}
+p
\Pr\stq{n}{2}
\]
Note that the bracketed term is equal to the probability flowing in to
$\stq{n}{2}$, which we've just shown is equal to our guessed
probability.  Plugging this in, we get
\[=
(2-p)
\Pr \stq{n}{2}
+p
\Pr \stq{n-1}{2}\]
Calculating the probabilities of $\stq{n}{2}$ and $\stq{n-1}{2}$ in
terms of the probability of $\stq{n}{1}$, we get
\[=p
\Pr\stq{n}{1}+
(1-p)
\Pr\stq{n}{1}=
\Pr\stq{n}{1}\]
\end{itemize}

This covers all states for the $N=1$ case.
\EOP

\section{Proof for All $N$}
It's pretty easy to exhaustively verify Equations~\ref{stationary equations}
for $N=2$ and $3$, but it's not clear how to prove it for any $N$.
This section (and its subsections) are devoted to a proof of that fact.
I'll prove that the stationary distribution for any $N$ is of
product form, where each processor's distribution matching that of 
equations~\ref{stationary equations}.

\begin{theorem} \label{L=2 GHP theorem}
\index{Theorem \ref{L=2 GHP theorem}}
The stationary distribution for any $N$-node nonstandard ring with
parameter $L=2$ is product form and given by Equations~\ref{stationary
equations}.
\end{theorem}
I proceed by induction on $N$.

If $N=1$, we're done, by Lemma~\ref{N=1,L=2 lemma}.

Assume that $N>1$.  I will begin by arguing that it is sufficient to
analyze the cases where all the queues are of length $\leq2$.  Suppose
for a moment that processor $i$ has more than two packets in its
queue, that is, the processor is in state $\stq{n}{h}$ for $h=1$ or 2
and $n\geq3$.  What is the shortest queue length that the processor
could have had on the preceding turn?

	If a packet arrived from the preceding processor, and a 
new packet arrived to the queue, then the preceding queue would
have had a length of $n-1$.  This is the shortest it could be.  Therefore,
for any state $\tau$ that has a non-zero probability of preceding
our current state $\sigma$, the queue length in processor $i$ of
state $\tau$ is $\geq n-1\geq2$.  

	Suppose now that we removed a packet from the queue of
processor $i$ in state $\sigma$.  (Let's call this new state
$\hat{\sigma}$.)  Suppose that we also remove a packet from the
queue of the $i$th processor in $\tau$, forming $\hat{\tau}$.
Observe that $\hat{\tau}$ precedes $\hat{\sigma}$ with non-zero
probability-- in fact, the transition probability is exactly
the same as $\tau$ becoming $\sigma$.  (It's necessary that
$n\geq2$ for this to hold.) Moreover, any state
$\hat{\tau}$ that precedes $\hat{\sigma}$ with non-zero probability
can also be translated back into a state $\tau$ preceding $\sigma$.

Observe that if processor $i$ has $n\geq 3$ packets in queue, and we
remove a packet, the stationary probability of the resulting state is
multiplied by $\frac{(1-p)(2-p)}{p^{2}}$.  The argument in the
preceding paragraph shows that the preceding states will all also lose
a packet in processor $i$.  Since the minimal queue length of
processor $i$ is 2 in any preceding state $\tau$, then it is at least
1 in any state $\hat{\tau}$.  Thus, the balance equations for $\sigma$
and $\hat{\sigma}$ differ by exactly a factor of
$\frac{(1-p)(2-p)}{p^{2}}$ in every term.  Therefore, if we can show
that the balance equations hold when processor $i$'s queue is $\leq
2$, we're done.  This holds for any $i$, so we are reduced to showing
that the balance equations hold when all queues are of length $\leq
2$.

Next, I'll reduce the possible configurations of packets
travelling in the ring (i.e.\ hot potatoes), which will ultimately
reduce the number of equations we need to check.

\begin{definition} 
Suppose that the current state of the ring
is $\sigma$ and the preceding state was $\tau$.  Consider the
edge $e$ between processors $i$ and $i+1$.  If processor
$i$ in state $\tau$ was holding a hot potato equal to 2, we 
say that the edge $e$ in $\sigma$ was crossed, denoted
\[\stq{n_{i}}{h_{i}}
\go
\stq{n_{i+1}}{h_{i+1}}
\]
If this did not occur, we say that $e$ was blocked, denoted
\[\stq{n_{i}}{h_{i}}
\nogo
\stq{n_{i+1}}{h_{i+1}}
\]
An unspecified edge is denoted
\[\stq{n_{i}}{h_{i}}
-
\stq{n_{i+1}}{h_{i+1}}
\]
\end{definition}
This definition might sound a bit odd, in that I don't consider
a packet to cross an edge if it's arriving at its destination.
However, since I'm analyzing
a non-blocking model of the ring (i.e.\ a packet
can arrive at its destination at the same time that a new
packet gets dropped from the destination's queue), this  
definition proves useful.

Note that an edge from processor $i$ to processor $i+1$ can only have
been crossed if processor $i+1$ currently contains a hot potato, and
the hot potato equals 1.

Suppose that we are in state $\sigma$.  Suppose that neither 
processor $i$ nor $j$ ($i\neq j$) 
contains the hot potato 1.  Let $e_{i}$ and $e_{j}$
be the edges preceding processors $i$ and $j$, respectively.  Then 
note that both $e_{i}$ and $e_{j}$ are blocked.  

Let us perform the following operation: we cut edges 
$e_{i}$ and $e_{j}$ and form two smaller unidirectional rings:
ring $R_{i}$ will consist of processors $i$ through $j-1$,
and ring $R_{j}$ will consist of processors $j$ through $i-1$.

Observe that in $R_{i}$, the edge between processor $j-1$ and 
processor $i$ is blocked (and similarly the edge between
processor $i-1$ and $j$ in $R_{j}$ is blocked, too).  Let us refer to the 
state of $R_{i}$ as $\sigma_{i}$ (and similarly for $R_{j}$
and $\sigma_{j}$.)  ($\sigma_{i} $ and $\sigma_{j}$ are determined
by $\sigma$.)
Suppose that some states
$\tau_{i}$ and $\tau_{j}$ preceded $\sigma_{i}$
and $\sigma{j}$, respectively, on the subrings.  
If we glue $\tau_{i}$ and
$\tau_{j}$ together (by reversing the process that gave us
$R_{i}$ and $R_{j}$ originally), we get a state $\tau$ that
precedes $\sigma$, and the probability that $\tau$ becomes
$\sigma$ is found by multiplying the respective probabilities
on $R_{i}$ and $R_{j}$.  
This surprising state of affairs occurs because $e_{i}$ and
$e_{j}$ aren't crossed.  In some sense, no information about
the preceding state arrives at processors $i$ and $j$.  This
allows us to view the two parts of the ring (namely $i$ to $j-1$ 
and $j$ to $i-1$) independently.

The balance equations now follow easily by induction, since the
subrings are smaller than $N$.  By our inductive hypothesis, the sum
of the probabilities into $\sigma_{i}$ is $\Pr(\sigma_{i})$, and the
probability into $\sigma_{j}$ is $\Pr(\sigma_{j})$.  Therefore, the
sum of the probabilities into $\sigma$ is
\[\Pr(\sigma_{i})\Pr(\sigma_{j})\]
Since our distributions are all product form, this is precisely
$\Pr(\sigma)$, as desired.

What states remain to deal with?  We can assume that all 
queues are of length $\leq 2$, and at least $N-1$ processors contains
1 as a hot potato.  I'm going to split the remaining cases into
finitely many classes and then verify the balance equations on
each class.

First of all, let us choose a processor $i$. Suppose the
state of the system, $\sigma$, is
\[\cdots
-\stq{n_{i-2}}{h_{i-2}}
-\stq{n_{i-1}}{h_{i-1}}
\stackrel{e_{1}}{-}
\underbrace{
\stq{n_{i}}{h_{i}}
}_{proc \ i}
\stackrel{e_{2}}{-}
\stq{n_{i+1}}{h_{i+1}}
-\cdots\] Let $e_{1}$ be the edge from processor $i-1$ to processor
$i$, and $e_{2}$ be the edge from processor $i$ to processor $i+1$.
Each of these edges may be crossed or blocked.  By specifying if
$e_{1}$ and $e_{2}$ are crossed or blocked, we partition the states
that precede $\sigma$ into 4 disjoint classes.  Of course, as we saw
above, if $h_{i}=2$ or X, then $e_{1}$ must be blocked-- in other
words, some of the partitions may be empty.

Once we know whether $e_{1}$ or $e_{2}$ are crossed, we can (with 
some manipulation) reduce the possible prior states on the 
processors $i+1$ through $i-1$ to an $N-1$ node ring, and use
induction.  Then we plug the values in, sum over the 4 partitions,
and end up with the balance equation.  I will first calculate 
the probability flowing into processor $i$,
then the probability flowing into the remaining 
$N-1$ processors, and finally check all the balance equations in one fell
swoop.  Here we go.

\subsection{Probability of Processor $i$}

The probability of the possible prior states to 
$\nogo \st{X} \nogo$,
weighted by the probability of travelling from that state to 
$\nogo \st{X} \nogo$,
is:
\[(1-p)\left(
\Pr\st{X} +
\Pr\st{1}
\right)
=(1-p)
\Pr\st{X}+
\frac{p}{1-p}\Pr\st{X}\]
\[=\Pr\st{X}\]
Probability into $\nogo\st{X}\go$ is:
\[(1-p)
\Pr\st{X}\]
\[=\frac{p}{2-p}
\Pr\st{X}\]
Probability into $\nogo\st{2}\nogo$ is:
\[ \frac{p}{2} \left(
\Pr\st{X}+
\Pr\st{1}
\right) +\frac{1-p}{2}
\Pr\stq{1}{1}\]
\[=\Pr \st{2}\]
Probability into $\nogo\st{2}\go$ is:
\[\frac{p}{2}
\Pr\st{2}+
\frac{1-p}{2}
\Pr\stq{1}{2}\]
\[=\frac{p}{2-p}
\Pr\st{2}\]
Probability into $\nogo\stq{1}{2}\nogo$ is:
\[\frac{p}{2}
\Pr\stq{1}{1}+
\frac{1-p}{2}
\Pr\stq{2}{1}\]
\[=\Pr\stq{1}{2}\]
Probability into $\nogo\stq{1}{2}\go$ is:
\[\frac{p}{2}
\Pr\stq{1}{2}+
\frac{1-p}{2}
\Pr\stq{2}{2}\]
\[=\frac{p}{2-p}
\Pr\stq{1}{2}\]
Probability into $\nogo\stq{2}{2}\nogo$ is:
\[\frac{p}{2}
\Pr\stq{2}{1}+
\frac{1-p}{2}
\Pr\stq{3}{1}\]
\[=\Pr\stq{2}{2}\]
Probability into $\nogo\stq{2}{2}\go$ is:
\[\frac{p}{2}
\Pr\stq{2}{2} +
\frac{1-p}{2}
\Pr\stq{3}{2}\]
\[=\frac{p}{2-p}
\Pr\stq{2}{2}\]
Probability into $\nogo\st{1}\nogo$ is:
\[\frac{p}{2}\left(
\Pr\st{0}+
\Pr\st{1}
\right)\]
\[=\frac{1}{2-p}
\Pr\st{1}\]
Probability into $\nogo\st{1}\go$ is:
\[\frac{p}{2}
\Pr\st{2}
+\frac{1-p}{2}
\Pr\stq{1}{2}\]
\[=\frac{p}{(2-p)^{2}}
\Pr\st{1}\]
Probability into $\go\st{1}\nogo$ is:
\[(1-p)\left(
\Pr\st{X}+
\Pr\st{1}
\right)\]
\[=\frac{1-p}{p}
\Pr\st{1}\]
Probability into $\go\st{1}\go$ is:
\[(1-p)
\Pr\st{2}
\]
\[=\frac{1-p}{2-p}
\Pr\st{1}\]
Probability into $\nogo\stq{1}{1}\nogo$ is:
\[\frac{p}{2}
\Pr\stq{1}{1}
+\frac{1-p}{2}
\Pr\stq{2}{1}\]
\[=\frac{p}{2-p}
\Pr\stq{1}{1}\]
Probability into $\nogo\stq{1}{1}\go$ is:
\[\frac{p}{2}
\Pr\stq{1}{2}+
\frac{1-p}{2}
\Pr\stq{2}{2}\]
\[=\frac{p^{2}}{(2-p)^{2}}
\Pr\stq{1}{1}\]
Probability into $\go\stq{1}{1}\nogo$ is:
\[(1-p)
\Pr\stq{1}{1}
+p
\Pr\st{1}
+p
\Pr\st{X}\]
\[=2\frac{1-p}{p}
\Pr\stq{1}{1}\]
Probability into $\go\stq{1}{1}\go$ is:
\[(1-p)
\Pr\stq{1}{2}
+p
\Pr\st{2}\]
\[=2\frac{1-p}{2-p}
\Pr\stq{1}{1}\]
Probability into $\nogo\stq{2}{1}\nogo$ is:
\[\frac{p}{2}
\Pr\stq{3}{1}
+\frac{1-p}{2}
\Pr\stq{2}{1}\]
\[=\frac{p}{2-p}
\Pr\stq{2}{1}\]
Probability into $\nogo\stq{2}{1}\go$ is:
\[\frac{p}{2}
\Pr\stq{2}{2}
+\frac{1-p}{2}
\Pr\stq{3}{2}\]
\[=\frac{p^{2}}{(2-p)^{2}}
\Pr\stq{2}{1}\]
Probability into $\go\stq{2}{1}\nogo$ is:
\[(1-p)
\Pr\stq{2}{1}
+p
\Pr\stq{1}{1}\]
\[=2\frac{1-p}{p}
\Pr\stq{2}{1}\]
Probability into $\go\stq{2}{1}\go$ is:
\[(1-p)
\Pr\stq{2}{2}+
p
\Pr\stq{1}{2}\]
\[=2\frac{1-p}{2-p}
\Pr\stq{2}{1}\]

\subsection{Probability of the Other Processors}

We now have to deal with the somewhat more complicated problem of the other 
$N-1$ processors.
The key to finding the possible preceding states of processors
$i+1$ through $i-1$ is the state of processor $i+1$.  Recall that 
at most one
processor does not have a hot potato equal to one-- therefore,
we can assume that the hot potato in processor $i+1$ is 1.  The
queue can be 0, 1, or 2, and the edges $e_{1}$ and $e_{2}$ can
each be crossed or blocked, so there are 12 possibilities.  I calculate them 
below.

To begin, if the queue in processor $i+1$ is empty, and neither edge
$e_{1}$ nor $e_{2}$ is crossed, i.e.\
\[\stackrel{e_{2}}{\nogo}
\underbrace{
\st{1}
}_{proc \ i+1}
- \cdots
-\underbrace{
\stq{n_{i-1}}{h_{i-1}}
}_{proc \ i-1}
\stackrel{e_{1}}{\nogo}\]
then the prior states of processors $i+1$ through $i-1$ are identical
to the prior states of an $N-1$ node ring obtained by removing
node $i$, fusing edges $e_{1}$ and $e_{2}$ into a single edge (call it
$e$), and not allowing any packets to cross $e$. 

If no packets cross, then the ``1'' hot potato that appears in
processor $i+1$ is newly minted, and with equal probability could have
been a ``2''.  But if it were a ``2'', we would have a guarantee that
no packets crossed.  Therefore, the sum of the probabilities of the
prior states (weighted by transition probabilities) for processors
$i+1$ through $i-1$ on the original ring is equal to the sum of the
probabilities of the prior states (weighted by transition
probabilities) of an $N-1$ node ring, where processor $i$ is removed,
and processor $i+1$'s state is changed to $\st{2}$.  By induction,
this latter weighted sum is equal to the product form probability
distribution from equations~\ref{stationary equations}.  Shifting
processor $i+1$ from $\st{1}$ to $\st{2}$ divides the probability by
$(2-p)$, so
\[\Pr\left(\stackrel{e_{2}}{\nogo}
\underbrace{
\st{1}
}_{proc \ i+1}
- \cdots-
\underbrace{
\stq{n_{i-1}}{h_{i-1}}
}_{proc \ i-1}
\stackrel{e_{1}}{\nogo}\right)\]
\[=\frac{1}{2-p}
\Pr\left(
\stackrel{e}{-}
\underbrace{
\st{1}
}_{proc \ i+1}
- \cdots-
\underbrace{
\stq{n_{i-1}}{h_{i-1}}
}_{proc \ i-1}
\stackrel{e}{-}
\right)\]

Next, suppose that the situation is
\[\stackrel{e_{2}}{\go}
\underbrace{
\st{1}
}_{proc \ i+1}
- \cdots \stackrel{e_{1}}{\nogo}\]
We can use the same kind of reasoning as above, but there's a twist:
if we try to view processors $i+1$ through $i-1$ as an independent 
$N-1$ node ring, where did the packet currently in processor 
$i+1$ come from?  Since edge $e_{1}$ is blocked, the packet at
node $i+1$ seems to have arrived out of the fog.
However, we can take this behavior into account in determining the possible
preceding states to these $N-1$ processors.  The possible preceding states
for nodes $i+1$ through $i-1$  are the same as those on a $N-1$ node 
ring such that no packets
cross edge $e$ ($e$ is the new edge between node $i-1$ and $i+1$)
and where the state of processor $i+1$ is now 
$\st{X}$ instead of $\st{1}$.
(In other words, we replace processor $i+1$'s state with the value
it would have had if processor $i$ hadn't sent its packet over.)
So,
\[\Pr\left( \go
\underbrace{
\st{1}
}_{proc \ i+1}
-\cdots \nogo \right)\]
\[=\Pr \left(-
\underbrace{
\st{X}
}_{proc \ i+1}
-\cdots-  \right)\]
\[=\frac{1-p}{p}\Pr \left(-
\underbrace{
\st{1}
}_{proc \ i+1}
-\cdots-
\right)\]

Next, suppose that the situation is
\[\nogo
\st{1}
- \cdots
\go\]
Again, we can use the same kind of reasoning as above.  In this case,
the $N-1$ node ring crosses at $e$, even though no packet arrives
at processor $i+1$.  Therefore, to account for the packet absorption at
processor $i$, we pad an extra packet onto the state of processor $i+1$.
To make sure that we force a crossing at edge $e$, we calculate
\[\Pr \left(
\stackrel{e}{-}
\stq{1}{1}
-\cdots
\stackrel{e}{-}
\right) -\Pr \left(
\stackrel{e}{-}
\stq{1}{2}
-\cdots
\stackrel{e}{-}
\right)\]  
There is one new wrinkle, though.  Since the packet
which remains in queue in our $N-1$ node ring actually enters the
ring and gets a destination (of 1) in the real $N$-node ring, 
we must multiply the probability by $1/2$.  Thus,
\[\Pr\left(\nogo
\st{1}
- \cdots
\go\right)\]
\[=\frac{1}{2}\left[
\Pr \left(
-\stq{1}{1}
-\cdots-
\right) -\Pr \left(
-\stq{1}{2}
-\cdots-
\right) \right] \]
\[=\frac{p}{(2-p)^{2}}
\Pr\left(-
\st{1}
- \cdots-
\right)\]

Suppose that the situation is
\[\go
\st{1}
- \cdots
\go\]
Then, using the above arguments,
\[\Pr\left(
\go
\st{1}
- \cdots
\go\right)\]
\[=\Pr\left(
-\st{1}
- \cdots-
\right)
-\Pr\left(
-\st{2}
- \cdots-
\right)\]
\[=\frac{1-p}{2-p} \Pr \left(
-\st{1}
- \cdots-
\right)\]

Next, suppose that processor $i+1$ has 1 packet in queue.  Suppose that
the state of edges $e_{1}$ and $e_{2}$ is
\[\nogo
\stq{1}{1}
- \cdots
\nogo\]
Then
\[\Pr\left(
\nogo
\stq{1}{1}
- \cdots
\nogo \right)\]
\[=\Pr\left(-
\stq{1}{2}
- \cdots-
\right)\]
\[=\frac{p}{2-p}
\Pr\left(-
\stq{1}{1}
- \cdots-
\right)\]

Next, suppose that $e_{1}$ and $e_{2}$ are
\[\go
\stq{1}{1}
- \cdots
\nogo\]
Then
\[\Pr\left(
\go
\stq{1}{1}
- \cdots
\nogo \right)\]
\[=2\Pr\left(-
\stq{1}{1}
- \cdots-
\right)\]
(The ``2'' is caused by a packet that doesn't drop in the induced
$N-1$ node ring.)
\[=2\frac{1-p}{p}\Pr\left(-
\stq{1}{1}
- \cdots-
\right)\]

Next, suppose that $e_{1}$ and $e_{2}$ are
\[\nogo
\stq{1}{1}
- \cdots
\go\]
Then
\[\Pr\left(
\nogo
\stq{1}{1}
- \cdots
\go \right)\]
\[=\frac{1}{2}\left[
\Pr \left(-
\stq{2}{1}
- \cdots-
\right)- \Pr \left(-
\stq{2}{2}
- \cdots-
\right) \right]\]
\[=\frac{p^{2}}{(2-p)^{2}}
\Pr \left(-
\stq{1}{1}
- \cdots-
\right)\]

Next, suppose that $e_{1}$ and $e_{2}$ are
\[\go
\stq{1}{1}
- \cdots
\go\]
Then
\[\Pr\left(
\go
\stq{1}{1}
- \cdots
\go \right)\]
\[=
\Pr \left(-
\stq{1}{1}
- \cdots-
\right)- \Pr \left(-
\stq{1}{2}
- \cdots-
\right) \]
\[= 2\frac{1-p}{2-p}
\Pr \left( -
\stq{1}{1}
- \cdots-
\right)\]

Next, suppose that processor $i+1$ has 2 packets in queue.  Suppose that
the state of edges $e_{1}$ and $e_{2}$ is
\[\nogo
\stq{2}{1}
- \cdots
\nogo\]
Then
\[\Pr\left(
\nogo
\stq{2}{1}
- \cdots
\nogo \right)\]
\[=\Pr\left(
-
\stq{2}{2}
- \cdots-
 \right)\]
\[=\frac{p}{2-p}\Pr \left(-
\stq{2}{1}
- \cdots-
\right) \]

Next, suppose that the edges $e_{1}$ and $e_{2}$ are
\[\go
\stq{2}{1}
- \cdots
\nogo\]
Then
\[\Pr\left(
\go
\stq{2}{1}
- \cdots
\nogo \right)\]
\[=2\Pr\left(-
\stq{1}{2}
- \cdots-
\right)\]
\[=2 \frac{1-p}{p} \Pr \left(-
\stq{2}{1}
- \cdots-
\right)\]

Next, suppose that the edges $e_{1}$ and $e_{2}$ are
\[\nogo
\stq{2}{1}
- \cdots
\go\]
Then
\[\Pr\left(
\nogo
\stq{2}{1}
- \cdots
\go \right)\]
\[=\frac{1}{2}\left[
\Pr \left(-
\stq{3}{1}
- \cdots-
\right) -
\Pr \left(-
\stq{3}{2}
- \cdots-
\right) \right] \]
\[=\frac{p^{2}}{(2-p)^{2}} 
\Pr \left(-\stq{2}{1}
- \cdots-
\right) \]

Next, suppose that the edges $e_{1}$ and $e_{2}$ are
\[\go
\stq{2}{1}
- \cdots
\go\]
Then
\[\Pr\left(
\go
\stq{2}{1}
\go \right)\]
\[=
\Pr \left(-\stq{2}{1}
- \cdots-
\right) -
\Pr \left(-\stq{2}{2}
- \cdots-
\right) \]
\[=2\frac{1-p}{2-p}
\Pr \left(-
\stq{2}{1}
- \cdots-
\right) \]

\subsection{The Balance Equations}
We're all set to verify the balance equations now.  Suppose that
we are in state $\sigma$, which is:
\[\cdots
\stackrel{e_{1}}{-}
\underbrace{
\stq{n_{i}}{h_{i}}
}_{i}
\stackrel{e_{2}}{-}
\underbrace{
\stq{n_{i}}{h_{i}}
}_{i+1}-\cdots\]
i.e.\ we are looking at processors $i$ and $i+1$, with preceding edges
labelled $e_{1}$ and $e_{2}$, respectively.
As I've argued above, it is sufficient to consider the cases where
$n_{i}$ and $n_{i+1}$ are $\leq2$, and we can assume that $h_{i+1}=1$.
If we specify whether or not $e_{1}$ and $e_{2}$ are open, we split
the possible preceding states into 4 disjoint sets.  Therefore,
the probability flowing into $\sigma$ is
\begin{eqnarray*}
& &
\left[
\left(\mbox{Probability into}\left(
\stackrel{e_{1}}{\nogo}
\stq{n_{i}}{h_{i}}
\stackrel{e_{2}}{\nogo}
\right)\right) \times \right.
\\ 
& & \left.
\left(\mbox{Probability into}\left(
\stackrel{e_{2}}{\nogo}
\stq{n_{i+1}}{h_{i+1}}
-\cdots-
\stq{n_{i-1}}{h_{i-1}}
\stackrel{e_{1}}{\nogo}
\right)\right)
\right]
\\
& + &
\left[
\left(\mbox{Probability into}\left(
\stackrel{e_{1}}{\nogo}
\stq{n_{i}}{h_{i}}
\stackrel{e_{2}}{\go}
\right)\right) \times
\right.
\\
& &
\left.
\left(\mbox{Probability into}\left(
\stackrel{e_{2}}{\go}
\stq{n_{i+1}}{h_{i+1}}
-\cdots-
\stq{n_{i-1}}{h_{i-1}}
\stackrel{e_{1}}{\nogo}
\right)\right)
\right]
\\
& + &
\left[
\left(\mbox{Probability into}\left(
\stackrel{e_{1}}{\go}
\stq{n_{i}}{h_{i}}
\stackrel{e_{2}}{\nogo}
\right)\right) \times
\right.
\\
& &
\left.
\left(\mbox{Probability into}\left(
\stackrel{e_{2}}{\nogo}
\stq{n_{i+1}}{h_{i+1}}
-\cdots-
\stq{n_{i-1}}{h_{i-1}}
\stackrel{e_{1}}{\go}
\right)\right)
\right]
\\
& + &
\left[
\left(\mbox{Probability into}\left(
\stackrel{e_{1}}{\go}
\stq{n_{i}}{h_{i}}
\stackrel{e_{2}}{\go}
\right)\right) \times
\right.
\\
& &
\left.
\left(\mbox{Probability into}\left(
\stackrel{e_{2}}{\go}
\stq{n_{i+1}}{h_{i+1}}
-\cdots-
\stq{n_{i-1}}{h_{i-1}}
\stackrel{e_{1}}{\go}
\right)\right)
\right]
\end{eqnarray*}
In the preceding two sections, I calculated all the values
we need to evaluate the above equation.  Moreover, I expressed
the values as multiples of 
\[\Pr\left(
\stackrel{e_{1}}{-}
\stq{n_{i}}{h_{i}}
\stackrel{e_{2}}{-}
\right)\]
and \[\Pr\left(
\stackrel{e_{2}}{-}
\stq{n_{i+1}}{h_{i+1}}
-\cdots-
\stq{n_{i-1}}{h_{i-1}}
\stackrel{e_{1}}{-}\right)\]
(Since the probabilities are product form, I trust that the
preceding notation makes sense.)  Therefore, we can immediately
factor out a factor of 
\[\Pr
\stq{n_{i}}{h_{i}}
\Pr\left(-
\stq{n_{i+1}}{h_{i+1}}
-\cdots-
\stq{n_{i-1}}{h_{i-1}}
-\right)\]
\[=\Pr(\sigma)\]
I only need to verify that the 4 factored terms  sum to 1 in all
cases.  (I will work out the first case with extra details to illustrate
what I'm talking about.)  The verification of the cases follows:

Suppose that $h_{i}=X$.  Suppose that $h_{i+1}=1$ and 
$n_{i+1}$=0.  Then the probability flowing into $\sigma$ is
\[
\left(\mbox{Prob. into}\left(
\nogo
\st{X}
\nogo
\right)\right)
\left(\mbox{Prob. into}\left(
\nogo
\st{1}
-\cdots-
\stq{n_{i-1}}{h_{i-1}}
\nogo
\right)\right)\]
\[+\]
\[
\left(\mbox{Prob. into}\left(
\nogo
\st{X}
\go
\right)\right)
\left(\mbox{Prob. into}\left(
\go
\st{1}
-\cdots-
\stq{n_{i-1}}{h_{i-1}}
\nogo
\right)\right)\]
\[+\]
\[
\left(\mbox{Prob. into}\left(
\go
\st{X}
\nogo
\right)\right)
\left(\mbox{Prob. into}\left(
\nogo
\st{1}
-\cdots-
\stq{n_{i-1}}{h_{i-1}}
\go
\right)\right)\]
\[+\]
\[
\left(\mbox{Prob. into}\left(
\go
\st{X}
\go
\right)\right)
\left(\mbox{Prob. into}\left(
\go
\st{1}
-\cdots-
\stq{n_{i-1}}{h_{i-1}}
\go
\right)\right)\]
Now, $\Pr\left(\go \st{X}\nogo\right)$ and $\Pr\left(\go \st{X}\go\right)$
both equal zero, so the third and fourth terms of the sum go away.
Plugging in from our previous calculations, we get:
\[=
\left(\Pr\left(
-
\st{X}
-
\right)\right)
\frac{1}{2-p}\left(\Pr\left(
-
\st{1}
-\cdots
-
\stq{n_{i-1}}{h_{i-1}}
-
\right)\right)\]
\[+
\frac{p}{2-p}\left(\Pr\left(
-
\st{X}
-
\right)\right)
\frac{1-p}{p}\left(\Pr \left( -
\st{1}
-\cdots
-
\stq{n_{i-1}}{h_{i-1}}
-
\right)\right)\]
\[=
\left(\Pr\left(
-
\st{X}
-
\right)\right)
\left(\Pr\left(
-
\st{1}
-\cdots
-
\stq{n_{i-1}}{h_{i-1}}
-
\right)\right)
\left[\frac{1}{2-p} +\frac{p}{2-p}\frac{1-p}{p}\right]\]
\[=\Pr(\sigma)\left[\frac{2-p}{2-p}\right]=\Pr(\sigma)\]
as desired.

Next, suppose that $h_{i}=X$, $h_{i+1}=1$ and $n_{i+1}$=1.  If we
repeat the reasoning above, we find that the probability flowing in to
$\sigma$ is
\[\Pr(\sigma)\left[
\frac{p}{2-p}+\frac{p}{2-p}\frac{2(1-p)}{p}\right] \]
\[=\Pr(\sigma)\]

Note that the coefficients that arise from the 
\[\left(-
\stq{1}{1}
-\cdots-
\stq{n_{i-1}}{h_{i-1}}
-\right)\]
situations (regardless of how we set the edges $e_{1}$ and $e_{2}$)
are identical to those in the 
$\left(-
\stq{2}{1}
-\cdots-
\stq{n_{i-1}}{h_{i-1}}
-\right)$
case.  For example, if we deal with the $h_{i}=X$, $h_{i+1}=1$,
and $n_{i+1}=2$ case,  we find that the probability flowing in is
\[\Pr(\sigma)\left[
\underbrace{\frac{p}{2-p}}_{coef}+\frac{p}{2-p}
\underbrace{\frac{2(1-p)}{p}}_{coef}\right] \]
\[=\Pr(\sigma)\]
where the terms marked $coef$ are determined by the state of
processor $i+1$ (i.e.\ independent of the state of processor $i$).
Therefore, we only need to test if the balance equations work
for $n_{i+1}=0$ or 1; the $n_{i+1}=2$ case follows from 
$n_{i+1}=1$.

Next, observe that if processor $i$ is in state $-\stq{n_{i}}{2}-$ for
$n_{i}=$0, 1, or 2, then the coefficients that we calculated are
identical to those when the state of $i$ is $-\st{X}-$, and we just
verified that the balance equations hold for that case.

Therefore, we can assume, that $h_{i}=1$ for the remaining cases.
Suppose that $n_{i}=0$ and $n_{i+1}=0$.  (We are assuming
that $h_{i+1}=1$ in all these cases.)
Then the probability flowing in to $\sigma$ is
\[\Pr(\sigma)\left[
\left(
\frac{1}{2-p}\right)\left(\frac{1}{2-p}\right) 
+\left(\frac{p^{2}}{(2-p)^{2}}\right)\left(\frac{1-p}{p}\right)
\right.\]
\[\left.+\left(\frac{1-p}{p}\right)\left(\frac{p}{(2-p)^{2}}\right) 
+\left(\frac{1-p}{2-p}\right)\left(\frac{1-p}{2-p}\right) 
\right]\]
\[=\Pr(\sigma)\left[\frac{(1+(1-p))^{2}}{(2-p)^{2}}\right]
=\Pr(\sigma)\]

Suppose that $n_{i}=0$ and $n_{i+1}=1$. 
Then the probability flowing in to $\sigma$ is
\[\Pr(\sigma)\left[
\left(
\frac{1}{2-p}\right)\left(\frac{p}{2-p}\right)
+\left(\frac{p^{2}}{(2-p)^{2}}\right)\left(\frac{2(1-p)}{p}\right)\right.\]
\[\left.+\left(\frac{1-p}{p}\right)\left(\frac{p^{2}}{(2-p)^{2}}\right)
+\left(\frac{1-p}{2-p}\right)\left(\frac{2(1-p)}{2-p}\right)
\right]\]
\[=\Pr(\sigma)\left[\frac{p+2-2p+p-p^{2}+2-4p+2p^{2}}{(2-p)^{2}}
\right]
\]
\[
=\Pr(\sigma)\left[\frac{4-4p+p^{2}}{(2-p)^{2}}\right]
=\Pr(\sigma)\]

As observed above, the fact that the $n_{i+1}=1$ case holds
implies that the $n_{i+1}=2$ case holds, too. Suppose
$n_{i}=1$.  Now, if $n_{i+1}=0$, we can just perform
this whole procedure on processor $i+1$ instead of $i$, and
we are reduced to a prior case.  So we are left with
$n_{i+1}=1$.
Then the probability flowing in to $\sigma$ is
\[\Pr(\sigma)\left[
\left(\frac{p}{2-p}\right)\left(\frac{p}{2-p}\right)
+\left(\frac{2(1-p)}{p}\right)\left(\frac{p^{2}}{(2-p)^{2}}\right)\right.\]
\[\left.+\left(\frac{p^{2}}{(2-p)^{2}}\right)\left(\frac{2(1-p)}{p}\right)
+\left(\frac{1-p}{2-p}\right)\left(\frac{2(1-p)}{2-p}\right)
\right]\]
\[=\Pr(\sigma)\left[
\frac{(p+2(1-p))^{2}}{(2-p)^{2}}\right]
=\Pr(\sigma)\]

We have now accounted for all cases, completing the proof. \EOP

\section{Future Work, and a Warning} \label{exact future section}
Given the surprising number of different protocols present in the
statement of Theorem~\ref{main exact theorem}, it's natural to surmise
that the result holds for any greedy protocol on the ring.  Somewhat
more optimistically, Lemma~\ref{EPF lemma} suggests that the
distribution might hold with any greedy protocol on any network,
assuming that the maximum path length is 2.  However, there does not
seem to be any simple proof along these lines.

I should insert a note of caution at this stage.  After noting the
exact solution to the $N=3$ node ring, it's tempting to imagine that
the stationary distribution for any $N$ product form, and the
stationary probability of a particular state is a rational function of
$p$.  After we have some more results about Bernoulli arrivals and
analytic functions, I'll be able to show in Section~\ref{light traffic
section} that the distributions are not product form, and probably not
rational.

%% file: bounds.tex
\chapter{Bounds on Queue Length} \label{bounds chapter}

\section{Introduction}
In this chapter, I analyze stability and expected queue length
for standard 
Bernoulli rings.  In order
to deal with rings where both the number of nodes $N$ and
the maximum path length $L$ are large, I can no longer
make exact calculations of the expected queue length,
as I did in Chapter~\ref{exact chapter}.
Instead, I offer various upper and lower bounds.

Recall from Theorem~\ref{leighton's theorem} that, for a fixed
nominal load $r<1/2$, the expected queue length of an $N$ node
standard ring is known to be $\Theta(1/N)$.  The case of interest is
$r\geq 1/2$.

I begin by generating a series of lower bounds on expected
queue length.  The most interesting bounds are $\Omega(1/N)$
for the standard Bernoulli ring, and $\Omega(1)$ if either $N$
or $L$ is constant in a non-standard Bernoulli ring.

I start the upper bounds in Section~\ref{51 percent section} by
showing that if $r<1/2 + \epsilon$, then the ring is stable and has an
$O(1)$ upper bound on the expected queue length if $r<1/2 + \epsilon$.
(The exact value of $\epsilon$ can be determined by an equation
specified in the proof.)  As the improvement in $r$ is so small, this
result is mainly interesting in that there are no hidden constants in
the upper bound, and in the novelty of the technique.

Then, we get down to brass tacks.  In Section~\ref{main section}, I
construct a potential function for the standard Bernoulli ring, and
prove a number of useful lemmas about the function.  I use this
potential function in Section~\ref{ergodicity results section} to show
that for any $r<1$, the ring is stable, and the expected queue length
is $O(1)$.  A $\Theta(N)$ bound on expected delay per packet follows.
Finally, in Section~\ref{nonstandard ring bounds section}, I discuss
related results on the expected queue lengths of other rings with
Bernoulli arrival processes.

\section{Lower Bounds} \label{lower bound section}
\begin{lemma} \label{lower bound big N lemma}
\index{Lemma \ref{lower bound big N lemma}}
Fix the nominal load $0\leq r < 1$.  Consider a family of nonstandard
Bernoulli rings of size $N(i)$, with packet lifespans uniformly
distributed from 1 to $L(i)$ (where $L(i)\geq 2$, to make it
non-trivial), for $i=0,1,2,\ldots$.  Then the expected queue length
per node is $\Omega(1/L)$
\end{lemma}

\proof
I will calculate a bound at node 1; by
symmetry, the same bound applies at any node.

In Corollary~\ref{ergodic ring corollary}, I will show that all rings are
stable.  Assuming this result for the moment, we can use Little's Theorem 
(Theorem~\ref{little's theorem}) to conclude that the
probability that there's a packet at node 1 is $r$.
Since we're using a ``route, then arrive'' method of sampling
the state space, and since the probability of a packet
arriving at node 1 on any time step is $p=\frac{2}{L(i)+1}r$,
then the probability of there being a packet at node 1
after routing, but before exogenous arrivals, is at least
\[r-p=r\left(1-\frac{2}{L(i)+1}\right)\]
Since we assumed that $L(i)\geq 2$, then
\[\geq \frac{r}{3}\]
So, the probability that there is at least one packet in queue at node
1 \emph{after} arrivals is at least
\[\frac{r}{3}p = \frac{2r^{2}}{3}\frac{1}{L(i)+1} = \Omega(1/L)\]
\EOP

If $L\leq O(N)$, Lemma~\ref{lower bound big N lemma} is probably tight.
But if $N = o(L)$, this is not always the case, as demonstrated
by the next lemma.

\begin{lemma} \label{lower bound big L lemma}
\index{Lemma \ref{lower bound big L lemma}}
Fix a nominal load $r$ on a family of nonstandard Bernoulli rings,
labelled as in Lemma~\ref{lower bound big N lemma}.  Assume that
$N=o(L)$, and that $L(i)$ is increasing.  Then there exist constants
$0<\alpha_{r}, \beta_{r} < 1$, depending only on $r$, such that for
all sufficiently large $L(i)$,
\[\Ex [\mbox{queue length}] \geq \frac{1}{4}\beta_{r} 
\left[\frac{\alpha_{r}}{4}\right]^{N(i)} \]
so, if $\hat{\alpha}_{r}=\alpha_{r}/4$, then
\[\Ex [\mbox{queue length}] = \Omega\left(\hat{\alpha}_{r}^{N}\right)\]
\end{lemma}

\proof
The probability that every packet now in the ring departs in 
$(L(i)/3)-N(i)$
time steps is at least
\[\left[\frac{(L(i)/3) - N(i)}{L(i)}\right]^{N(i)}\geq (1/4)^{N(i)}\]
for sufficiently large $L(i)$ (here, we're using $N=o(L)$).
The probability that there is at least 1 exogenous packet arrival
in each queue during (the same) $(L(i)/3)-N(i)$ time steps is at least:
\[
\left(1-(1-p)^{(L(i)/3)-N(i)}\right)^{N(i)}
\]
For sufficiently large $L(i)$,
\begin{equation} \label{lower bound equation}
\leq \left(1-(1-p)^{L(i)/4}\right)^{N(i)}
=\left(1-\left(1-\frac{2r}{L(i)+1}\right)^{L(i)/4}\right)^{N(i)}
\end{equation}
Now, 
\[\lim_{i\rightarrow \infty} 1-\left(1-\frac{2r}{L(i)+1}\right)^{L(i)/4}=
1-e^{-\frac{2}{4}r}\]
So for some fixed $0<\alpha_{r}<1$ and all sufficiently large $i$ (and hence
$L(i)$),
we can lower bound Equation~\ref{lower bound equation} by
\[\geq \alpha_{r}^{N(i)}\]
Note that since there are at least $N(i)$ exogenous packet arrivals and $N(i)$
departures from the ring in $(L(i)/3) -N(i)$ time steps, then by the
$L(i)/3$ time step, there will be $N(i)$ packets inserted, each 
having travelled less than $L(i)/3$ steps.
The probability that the $N(i)$ packets newly injected into the ring during the
first $(L(i)/3)$ time steps survive at least $L(i)/3$ steps is
\[ \left(\frac{2}{3} \right)^{N(i)} \]
In this event, on time steps $L(i)/3$ through $2L(i)/3$, the entire
ring remains full of the same $N(i)$ packets.  The probability of at
least 1 exogenous packet arriving  at node 1 during the
time steps $L(i)/3$ through $L(i)/2$ is
\[1-(1-p)^{L(i)/6} > \beta_{r}\]
for sufficiently large $i$ and some fixed $0<\beta_{r}<1$,
since $\lim_{i\rightarrow \infty}1-(1-p)^{L(i)/3}=1-e^{-r/3}$.
In this case, the packet arriving at node 1 will remain there
for at least $L(i)/6$ time steps.
Therefore, with probability at least
\[ \beta_{r}\left( \frac{\alpha_{r}}{4} \right)^{N(i)} \]
node 1 has at least one packet in queue for $L/6$ time steps out of
$2L/3$ time steps.  Since $(L/6)/(2L/3)=1/4$, the expected
queue length at node 1 is at least
\[ \frac{\beta_{r}}{4}\left( \frac{\alpha_{r}}{4} \right)^{N(i)} \]
\EOP

Lemma~\ref{lower bound big N lemma} gives a much tighter (larger)
bound on the expected queue length than Lemma~\ref{lower bound big L
lemma} unless $L$ is very large relative to $N$. Specifically, if
$\hat{\alpha}_{r}^{-N}=o(L)$, then Lemma~\ref{lower bound big L lemma} is
tighter.

We are really interested in certain special cases:
\begin{corollary} \label{lower bound corollary}
\index{Corollary \ref{lower bound corollary}}
Let $\Ex[Q]$ be the expected queue length.
From Lemma~\ref{lower bound big N lemma}, we get:
\begin{itemize}
\item If $L=\Theta(N)$ (e.g.\ if $L=N-1$ on a 
standard Bernoulli ring), then $\Ex[Q]=\Omega(1/N)$.
\item If $L$ is constant, then $\Ex[Q]=\Omega(1)$.
\end{itemize}
From Lemma~\ref{lower bound big L lemma}, we get:
\begin{itemize}
\item If $N$ is constant, then $\Ex[Q]=\Omega(1)$.
\end{itemize}
\end{corollary}

What really happens in the regime where $N=o(L)$?  Is Lemma~\ref{lower
bound big N lemma} tight, until Lemma~\ref{lower bound big L lemma}
takes over?  It's not clear what to expect.  For the purposes of this
thesis, though, Corollary~\ref{lower bound corollary} suffices.

\section{Load of $1/2 + \epsilon$} \label{51 percent section}
In Coffman et al.~\cite{Leighton95} and~\cite{Leighton98}, the authors
show how to analyze a standard Bernoulli ring in the case where
loading is strictly less than $50\%$ (i.e.\ $r<1/2$).  They are able to
prove $\Theta(1/N)$ bounds on the expected queue length per node.  In
this section, I'll show how to prove stability and $O(1)$ upper bounds
for a slightly larger range of loads, namely $r<1/2 + \epsilon$, where
the $\epsilon$ can be explicitly calculated.

\begin{theorem} \label{51 percent theorem}
\index{Theorem \ref{51 percent theorem}}
Suppose we have an $N$ node standard Bernoulli ring in any state at
time $t=0$, with load $r<1$.  Choose a node $i$.  Then for any
$\delta$ there exists $N_{\delta}$ such that for any $N \geq
N_{\delta}$, at any time $t>N$, the probability of an empty cell
arriving at node $i$ is at least
\begin{eqnarray} \label{51 percent equation}
 \lefteqn{\frac{1}{N}[  1 -\delta } \nonumber\\
	& +A(A+B)C(C+D)\left(1-e^{-2rB}\right) \left(1-e^{-2rD}\right) 
	\left(\frac{1}{A+B+C+D}-1 \right) ] &
\end{eqnarray}
for any $A,B,C,D$ such that $A,B,C,D,(A+B+C+D) \in (0,1)$.
Moreover, the bound holds independently for all $t>N$.
\end{theorem}

\begin{note}
Observe that
for any fixed $A,B,C,D,r$, we can always choose $\delta$ small enough that
Equation~\ref{51 percent equation} is greater than $(1 + \epsilon)/N$
for a sufficiently small $\epsilon$, and all sufficiently large $N$.
\end{note}

\proof
Let $j$ be the node that is $\lfloor AN \rfloor$ nodes downstream of
$i$ (so $j=i + \lfloor AN \rfloor \bmod N$).  Let $k$ be the node that
is $\lfloor BN \rfloor$ nodes downstream of $j$, $l$ be the node
$\lfloor CN \rfloor$ nodes downstream of $k$, and $m$ the node that is
$\lfloor DN \rfloor$ nodes downstream of $e$.  For sufficiently large
$N$, these nodes are all distinct.  Throughout, I'm going to treat
$\lfloor AN \rfloor$, $\lfloor BN \rfloor$,$\lfloor CN \rfloor$ and
$\lfloor DN \rfloor$ as integers; the extent to which they are not
leads to the $\delta$ error term in the theorem.  Please see
Figure~\ref{epsilon figure}.
\begin{figure}[ht]
\centerline{\includegraphics[height=2in]{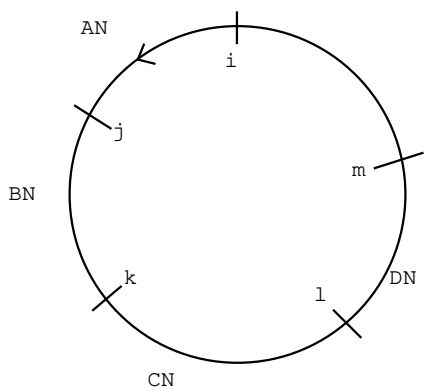} \hspace{.2in}  }
\caption{Arrangement of nodes for Theorem~\ref{51 percent theorem}}
\label{epsilon figure}
\end{figure}

Let's follow the slot in the ring that starts out under node $i$ at
time 0, and see what packets enter and leave as the slot travels
around the ring.  What's the probability that any packet in $i$ at
time $t=0$ departs before reaching node $j$?  Well, suppose that
there's a packet in $i$.  Regardless of the packet's point of
insertion, the probability that it departs in the next $AN$ steps is
at least $A$.  (If there is no packet in $i$, the event occurs with
probability 1.)

Given that the original packet (if any) has departed by node $i$, 
what's the probability that a new packet will arrive in that slot
by node $j$?  If the slot passes under any non-empty queue,
it will pick up a packet with probability 1.  If not, there's
a probability $p$ of a new arrival on each step.  Therefore,
the probability of a new packet arriving by node $k$ is at least:
\[1-(1-p)^{BN} = 1 - \left(1-\frac{2r}{N}\right)^{BN}
\rightarrow 1-e^{-2rB}\]
where the limit is taken as $N \rightarrow \infty$.

What's the probability that this first arrival lasts until node $k$?
Well, the earliest it could have arrived is node $i+1$, so 
the probability is at least $1-A-B=C+D$.

The probability that this packet leaves by node $l$ is $C$, since
the latest it could have arrived is node $k$.  

The probability that a second packet arrives by node $m$ is
$1-e^{-2rD}$, by the same arguments as above. 

The chance that this second packet lasts until it reaches node $i$
is $1-C-D=A+B$.  

Putting all of these (independent) probabilities together, we find 
their joint probability is:
\[J=A(A+B)C(C+D)\left(1-e^{-2rB}\right) \left(1-e^{-2rD}\right) \]
The probability of an empty cell arriving at node $i$ is then
\[\frac{1}{(A+B+C+D)N}J + \frac{1}{N}(1-J)\]
(The probability of a packet leaving after one step is always
at least $1/N$, which gives the $1/N$ factor in the second term.)
Expanding this equation, we get Equation~\ref{51 percent equation}.
\EOP

We can evaluate the theorem with some fortuitously chosen values.
\begin{corollary}
Set $A=C=.217300$, $B=D=.196640$ in Theorem~\ref{51 percent
equation}. Then for any $\delta>0$, there exists $N_{\delta}$ such
that for any $N>N_{\delta}$, the probability of an empty slot arriving
at node $i$ at any time after $t=N$ is at least
\[\frac{2}{N}\left[.500026802248 - \frac{\delta}{2}\right] \]
\end{corollary}

We can translate this result into a statement about queue lengths.
\begin{theorem}
Consider a node in any network.  Suppose it has Bernoulli arrivals at
rate $p$, and the chance that no internal packet arrives at the node
is at least $\mu$, independently on every step.  Suppose that $\mu>p$,
and no more than one internal packet can arrive on each time step.
Then the time expected queue length at the node is bounded by
\begin{equation} \label{life/death equation}
\frac{p^{2}(1-\mu)}{\mu(\mu-p)}
\end{equation}
If this equation (with possibly different values of $p$ and $\mu$)
holds at every node in the network, then the network is ergodic.
\end{theorem}

\proof
Compare the number of packets in the queue in the network to a
Bernoulli arrival, geometric service time single server queue with
rates $p$ and $\mu$.  We can relate the stochastic processes so that
arrivals occur at the same time, and if there is a departure from the
single server queue, then there is a departure from the original queue,
if it is non-empty.  (The original network may, possibly, have more
departures.)

The original network has all the same arrivals, and possibly
more departures, than the single server queue does.
Therefore, the number of packets in the former is 
bounded by the number of packets in the latter.  The
expected queue length for the single queue case is worked out in 
Theorem~\ref{birth-death theorem}, giving Equation~\ref{life/death equation}.

If the bounds hold at every queue, then the total number of
packets in the system is bounded by the sum of these Bernoulli
queues.  It follows that the original stochastically dominated network 
is ergodic.
\EOP

\note
This basic argument appears in a number of places, including
Coffman et alia~\cite{Leighton95}.

Putting together the results so far, we have the following corollary:
\begin{corollary}
There exists an $\epsilon \geq .000026802248$
such that for all sufficiently large $N$, an $N$ node
ring is stable for loads of $r<1/2 + \epsilon$, and
the expected queue length per processor is $O(1)$.
\end{corollary}

\section{Lyapunov Lemmas} \label{main section}
In this section, I will construct a function $\Phi$ on the state space
of the standard Bernoulli ring and show that as the ring evolves in
time, $\Phi$ tends to decrease on average (with exponentially tight
bounds on the probability that it increases).  This kind of decaying
function is sometimes called a Lyapunov or potential function.  In
the next section, I will use these lemmas to prove that for any $r$
and sufficiently large $N$, the system is ergodic, and the expected
queue length per node is $O(1)$.

First, some definitions.
\begin{definition}
If the probability that a packet crosses (blocks) a node $i$ is
greater than zero, then we say that the packet \emph{can reach node
$i$}.
\end{definition}

Next, I will define a function $\Phi$ from the state space to the
positive reals, and various helper functions.
\begin{definition} \index{$\Phi(\sigma)$} \index{$\phi(i,\sigma)$}
Suppose we have a nominal load $r$, with $0\leq r<1$ on an $N$ node
ring.  Fix $\delta>0$ such that
$r\left(1+\frac{\delta}{1+\delta}\right)<1$.  (If $0<r<1$, then there
is always a sufficiently small $\delta$ such that this inequality
holds.)  Define $\rh= r\left(1+\frac{\delta}{1+\delta}\right)<1$.  Let
us suppose that $\delta N$ is an integer, to simplify notation,
and that $\delta<1$.

Suppose we are in state $\sigma$.  Choose a node $i$ and a packet $z$.
Suppose, for a moment, that packet life times in the ring were
uniformly distributed between 1 and $(1+\delta)N$ time steps, rather
than between 1 and $N-1$.  Let $f(i, z, \sigma)$ be the probability
that packet $z$ can reach node $i$ (at least once) if $z$ had a
$(1+\delta)N$ distribution on its life span.  For instance, if $z$ is
from node $k$, at node $j$, and we label the nodes such that $k\leq
j\leq i$, then
\begin{equation} \label{hot potato definition equation}
f(i,z,\sigma)=\frac{(1+\delta)N - (i-k)}{(1+\delta)N - (j-k)}
\end{equation}

Then the sum of $f$ over all packets that can reach $i$ (under the
the $N-1$ distribution of life spans) is:
\[\phi (i, \sigma) = \sum_{z} f(i,z,\sigma)\]
and our non-negative function on the state space is:
\[\Phi (\sigma) = \max_{i} \phi(i,\sigma)\]
It's often clear from context what $\sigma$ is (namely, the
current state of the system), in which case I'll drop it
from the notation, and write $\Phi, \phi(i), f(i,z)$.
\end{definition}

Expressed in English, $\phi(i,\sigma)$ is the expected congestion at
node $i$ if all the packets had a uniform $(1+\delta)N$ distribution
on their life spans.

One other piece of notation I'll want to use:
\begin{definition}
Let $Q_{i}(\sigma)$ be the number of packets waiting in queue
in state $\sigma$.  If we are talking about a fixed state and 
the $\sigma$ is implicit, I'll just write $Q_{i}$.
\end{definition}

To motivate why our definition of $\Phi$ might
be useful, consider the following lemma:

\begin{lemma}[Mean Drift Downward] \label{drift motivation lemma}
\index{Lemma \ref{drift motivation lemma}}
Fix any state $\sigma$.  If $Q_{i}(\sigma)>0$, then the expected
change in $\phi(\sigma,i)$ in one time step is less than $\rh-1<0$.
That is, if the random variable $\tau$ is the state of the system on
the next time step,
\begin{equation} \label{drift motivation equation}
E [\phi(\tau, i)] - \phi(\sigma,i) < \rh -1 < 0
\end{equation}
\end{lemma}

\proof
Since $Q_{i}(\sigma)>0$, then we are guaranteed that a packet will leave node
$i$.  This will reduce $\phi(i)$ by 1.  

A new, exogenous packet arrives $j$ nodes upstream with probability
$p$, and increases $\phi(i)$ by $1 - [j/(1+\delta)N]$.  Summing over
all $j$, we get an expected increase of
\[p\sum_{j=0}^{N-1}\left(1-\frac{j}{(1+\delta)N}\right)\]
\[=p\left(N-1 -\frac{1}{(1+\delta)N}\sum_{j=0}^{N-1}j\right)\]
\[=\frac{2r}{N}\left(N-1-\frac{1}{(1+\delta)N}\frac{N(N-1)}{2}\right)\]
\[=2r\left( 1 - \frac{1}{N} - \frac{1}{2(1+\delta)} 
 			+\frac{1}{2(1+\delta)N} \right)\]
\[<2r\left( 1 - \frac{1}{2(1+\delta)} \right) 
=r\left(1+\frac{\delta}{1+\delta}\right)=\rh\]

Finally, for any packet $z$ in a cell, $f(\tau, z,i)$ is precisely the
probability of remaining in the system in $\tau$, times the
(increased) probability of needing to cross $i$ in $\tau$.  If packets
had a $(1+\delta)N$ distribution on life spans, the expected change in
$\phi(i)$ from any packet $z$ would be zero; since the actual life
span distribution is stochastically less (i.e.\ the probability of
$z$'s departure is strictly greater), then its contribution to the
expected change in $\phi(i)$ is negative.

Adding these three factors together,
we get Equation~\ref{drift motivation equation}.
\EOP

Lemma~\ref{drift motivation equation} is useful for motivating us, but
it doesn't directly prove anything about the drift of $\Phi$.  It has
two failings.  First, we need $Q_{i}$ to be greater than zero.
Second, we need the drift of the \emph{maximum} $\phi(i)$ to be
negative, which requires a bound on the simultaneous decay of all
large $\phi(i)$.  Fortunately, we can dispose of these two problems.
First, we need a trick to guarantee that $Q_{i}>0$ when we want it to
be.

\begin{lemma}[Trick Lemma] \label{trick lemma}
\index{Lemma \ref{trick lemma}}
Let $\zeta = 1+\frac{1}{(1+\delta)N-1}$. (Note that $\zeta>1$.)

Suppose we are in a fixed state $\sigma$. Then
a lower bound on $\phi(i-1)$ in terms of $\phi(i)$
and $Q_{i}$ is:
\begin{equation}
\phi(i-1) \geq \zeta [\phi(i) - Q_{i}] -\zeta
\end{equation}
\end{lemma}

\note
The reason this equation is useful is that, rearranging, we get
a lower bound on $Q_{i}$:
\begin{equation} \label{trick main equation}
Q_{i} \geq 
\phi(i)  
-\frac{1}{\zeta} \phi(i-1) -1 
\end{equation}

\proof
Define $C_{j}$ as the contribution to $\phi(j)$ from 
hot potatoes (packets in cells), so
\[C_{j} = \sum_{z \in \mbox{\footnotesize cell}} f(j,z)\]
Then we can write $\phi(i)$ as the contribution from packets
in cells, plus the contribution from packets in queue.
\begin{equation} \label{trick 2 equation}
\phi(i) = C_{i} + \sum_{j=2}^{N}
	\left(\frac{j+ \delta N}{(1+\delta)N}\right)Q_{i+j}
\end{equation}
We take the index of $Q_{i+j}$ modulo $N$ so that it always
falls between $1$ and $N$ (inclusively).
(The packets in queue $i+1$ can't reach and block node $i$, so
we start the sum with $j=2$ instead of $j=1$.)

We can write $\phi(i-1)$ in the same way:
\[
\phi(i-1) = C_{i-1} + \sum_{j=1}^{N-1}
	\left(\frac{j+ 1+ \delta N}{(1+\delta)N}\right)Q_{i+j}
\]
\begin{equation} \label{trick 3 equation}
\geq C_{i-1}+ \sum_{j=2}^{N-1}\left(\frac{j+ 1+ \delta N}{(1+\delta)N}\right)
  Q_{i+j+1}
\end{equation}
Let us compare the sums in Equations~\ref{trick 2 equation}
and~\ref{trick 3 equation}.  Ignoring the $j=N$ term, the $j$th term
in Equation~\ref{trick 3 equation} is larger than the $j$th term in
Equation~\ref{trick 2 equation} by a factor of
\begin{equation} \label{trick 4 equation}
\frac{j+\delta N + 1}{j+\delta N}=1+\frac{1}{j+\delta N}
\end{equation}
Equation~\ref{trick 4 equation} is minimized when $j=N-1$, so every
term is larger by a factor of at least
$\frac{(1+\delta)N}{(1+\delta)N-1}=\zeta$.  So,
\[ \phi(i-1)\geq C_{i-1}+ \zeta\left[\phi(i) - C_{i} - Q_{i}\right] \]
\begin{equation} \label{trick 1 equation}
 = \zeta\phi(i) 
	- \zeta Q_{i} 
	+ (C_{i-1}-\zeta C_{i}) 
\end{equation}
Consider the $C_{i-1}-\zeta C_{i}$ term.  Take any packet in the ring 
at node $j$, from node $k$, that can reach node $i$, but isn't there
yet.  Label the nodes so that $k\leq j < i$.  Then observe that $z$ is
more likely to cross node $i-1$ than $i$, so $f(i,z)<f(i-1,z)$.  More
precisely, define
\[g(i,j,k)=\frac{f(i-1,z)}{f(i,z)}=
\frac{ \left(\frac{(1+\delta)N-([i-1]-k)}{(1+\delta)N-(j-k)}\right)}
     { \left(\frac{(1+\delta)N-(i-k)}{(1+\delta)N-(j-k)}    \right)} \]
\[=1+\frac{1}{(1+\delta)N-(i-k)}\]
Therefore $g$ is only really dependent on the difference
between $i$ and $k$, i.e.\ we can write $g(i,j,k)=g(i-k)$.
Note that $g(i-k)$ is strictly increasing with $i-k$.  In particular,
since $g(1)=\zeta$, then if $k\leq i-1$, we have 
$f(i-1,z)-\zeta f(i,z)\geq 0$.  Therefore, $C_{i-1}-\zeta C_{i}\geq -\zeta$,
where the $\zeta$ comes from the $k=i$ term.
Plugging back in to Equation~\ref{trick 1 equation}, we get
\[ \phi(i-1) \geq
 \zeta \phi(i) 
	- \zeta Q_{i} -\zeta \]
as desired.
\EOP

\note 
Actually, the queue length at $\phi(i)=\Phi$
is greater than or equal to the mean of all the other queue lengths.
This result follows by looking at the preceding theorem a little more
carefully.  

Lemma~\ref{drift motivation lemma} illustrates the three parts of the
drift we have to analyze: 
\begin{itemize}
\item
the increase in $\phi(i)$ from new,
exogenous arrivals;
\item
the increase in $\phi(i)$ from packets
in the ring that remain in the ring (so that their probability
of using node $i$ increases);
\item
and the decrease in $\phi(i)$ from packets that depart from node $i$.
\end{itemize}
Let's look at each of these three contributions to the drift in turn.

\begin{lemma}[Exogenous Arrivals] \label{exogenous arrivals lemma}
\index{Lemma \ref{exogenous arrivals lemma}}
Fix $r$ (and a corresponding $\delta$ and $\rh$), 
and choose any $\eps{0} >0$.  
Suppose we start in some fixed state $\sigma$ on an $N$ node
ring.
Let $B(\gamma, i,t)$ be the event that in the next $t$ time
steps, the increase in $\phi(i)$ from exogenous
arrivals is greater than $(\rh+\eps{0})\gamma t$,
for any $\gamma\geq 1$.
Then there exist $N_{0}, T_{0}, K_{0}$ such that if $N\geq N_{0}$
and $t\geq T_{0}$, then for any $\gamma \geq 1$,
\begin{equation} \label{exogenous equation}
\Pr[\exists i \mbox{ such that } B(\gamma, i,t)] < e^{-K_{0}\gamma t}
\end{equation}
\end{lemma}

\proof
Fix $\eps{1}$ such that 
\begin{equation} \label{exo 1 equation}
0 < \eps{1} < \frac{1+\delta}{3r}\frac{\eps{0}}{2}
\end{equation}
Divide the ring into $D=1/\eps{1}$ segments.  Each segment is of
length $L=\eps{1} N$ nodes.  Label the segments $1,\ldots,D$.  For
simplicity, assume that $\eps{1} N, 1/\eps{1}$ and $\delta/\eps{1}$
are integral; the analysis for arbitrary values is basically
identical.

Let $\beta_{0}=\frac{\rh+\eps{0}}{\rh+\eps{0}/2}$.  Note that $\beta_{0}>1$.
Let $\beta = \beta_{0}\gamma$.
Fix a node $i$.  Suppose without loss of generality that 
$i$ is in segment $D$.  
Consider the contribution to $\phi(i)$ from a packet arriving in segment 
$J < D$.  Since the packet must cross $D-J-1$ entire
segments between $J$ and $D$, then the contribution is at most 
\[1-\frac{D-J-1}{(1+\delta)D}=
\frac{J+1+\delta D}{(1+\delta)D} = 
\frac{J+1}{(1+\delta)D} + \frac{\delta}{1+\delta}\]
The contribution from a packet arriving in segment $D$
is at most 1, of course, but to maintain consistency, I'll
just bound it by 
\[\frac{D+1}{(1+\delta)D} + \frac{\delta}{1+\delta}\]
which is greater than 1.

Next, we will bound the number of arrivals to each segment.
Let $A(J,t)$ be the total number of arrivals to segment $J$
in $t$ steps.
Since the total number of arrivals to segment $J$ in $t$
time steps is 
a sum of Bernoulli processes with mean $2r\eps{1}t$,
and since $\beta\geq \beta_{0} >1$,
we can use Lemma~\ref{sums of bernoulli lemma} from 
Appendix~\ref{probability chapter} to conclude that there
exists $K_{1}$ such that
\[
\Pr[A(J,t) \geq 2r\beta\eps{1} t] \leq 
e^{-K_{1}\beta t}\]
Since $J$ ranges over finitely many values (namely $D$),
we can select a $T_{0}$ and $K_{0}$ such that for any
$t\geq T_{0}$,
\[
\Pr[\exists J \mbox{ such that }A(J,t) \geq 2r\beta\eps{1} t] \leq 
e^{-K_{0}\beta t}\]
Now, if $A(J,t) < 2r\beta \eps{1} t$ for all $J$,
then $\phi(i)$ (for any node $i$) will increase by at most 
\[
\sum_{J=1}^{D} 
\left(\frac{J+1}{D}+\delta\right)
\frac{2r\beta \eps{1} t}{1+\delta} \]
\[= 
\left[\left(\sum_{J=1}^{D} \frac{J}{D} \right) +
\left(\sum_{J=1}^{D} \left(\frac{1}{D}
+ \delta\right)
\right)
\right]
\frac{2r\beta \eps{1} t}{1+\delta} \]
\[= 
\left[\frac{D(D+1)}{2D}+1+D\delta\right]
\frac{2r\beta \eps{1} t}{1+\delta} \]
Since $\eps{1}D=1$,
\[\left[\frac{1+2\delta + 3\eps{1}}{1+\delta}\right]
r \beta t
=
\left[1 + \frac{\delta}{1+\delta}+
\frac{3\eps{1}}{1+\delta}
\right]
r \beta t
\]
\[
=\left[\rh +
\frac{3r\eps{1}}{1+\delta}\right]
\beta t
\]
By Equation~\ref{exo 1 equation}, and since $\beta=\beta_{0}\gamma$,
\[\leq 
\left[\rh + \eps{0}/2 \right]
\beta_{0}\gamma t
\]
By the definition of $\beta_{0}$,
\[=\left[\rh + \eps{0} \right]
\gamma t
\]
Therefore, Equation~\ref{exogenous equation} holds.  
\EOP

Next, we will bound the expected maximal increase in $\phi(i)$ caused
by packets travelling in the ring.  Suppose a packet $z$ can reach
node $i$.  Suppose further that $z$ is a hot potato
travelling around the ring.  On every time step, if it doesn't depart
the ring and if it doesn't cross $i$, then $f(i,z)$ is strictly
increasing.  The next two lemmas show that this increase in $\phi(i)$
is negligible for all $i$.

\begin{lemma} \label{hot potato increase lemma}
\index{Lemma \ref{hot potato increase lemma}}
Assume we are in state $\sigma$, where
packet $z$ starts at node $j$, was inserted
at node $k$, and is being measured at node $i$,
and that $k\leq j < i$.
Suppose that $z$ remains on the ring for one step,
to state $\tau$.  Then
\[f(i,z,\tau) - f(i,z,\sigma) < \frac{1}{\delta N}\]
In other words, the largest possible one-step increase in $\phi$
contributed by a hot potato packet is less than $1/(\delta N)$.
\end{lemma}

\proof
Using Equation~\ref{hot potato definition equation},
we get
\[f(i,z,\tau)-f(i,z,\sigma)
=\frac{(1+\delta)N - (i-k)}{(1+\delta)N - ([j+1]-k)}-
\frac{(1+\delta)N - (i-k)}{(1+\delta)N - (j-k)}\]
Let $W=(1+\delta)N$.  And suppose, without loss
of generality, that $k=1$.  Then we have
\[=\frac{W-i+1}{(W-j)(W-j+1)}\]
For any fixed $i$, this equation is maximized when $j$ is large.
Since $j<i$, we get the restriction $j=i-1$.  Substituting,
\[\leq\frac{W-i+1}{(W-i+1)(W-i+2)}
=\frac{1}{W-i+2}\]
This equation is maximized when $i$ is large, so we
can set $i=N-1$ and get
\[\leq\frac{1}{(1+\delta)N - (N-1) + 2}=\frac{1}{\delta N+3}
<\frac{1}{\delta N}\]
\EOP

Now, let us calculate a bound on the expected change in $\phi$ from
hot potato packets.

\begin{lemma}[Hot Potatoes]  \label{hot potatoes lemma}
\index{Lemma \ref{hot potatoes lemma}}
Fix $r$ (and a corresponding $\delta$ and $\rh$), and choose any
$\eps{0}>0$.  Suppose we start in some fixed state $\sigma$ on an $N$
node ring.  Let $G(\gamma,i,t)$ be the event that in the next $t$ time
steps, the increase in $\phi(i)$ contributed by packets travelling in
the ring is greater than $\eps{0}\gamma t$, where $\gamma \geq 1$.
Then there exist constants $T_{0}$ and $K_{0}$ such that for any
$t\geq T_{0}$, there exists $N_{t}$, such that for any $N\geq N_{t}$,
and any $\gamma\geq 1$,
\begin{equation} \label{hp statement equation}
\Pr [\exists i \mbox{ such that }G(\gamma,i,t)] < e^{-K_{0}\gamma t}
\end{equation}
\end{lemma}

\proof
Let 
\begin{equation} \label{rough bound on epsilon equation}
\eps{1}<\delta \eps{0}/7
\end{equation}
We will determine an additional upper bound on $\eps{1}$ later in the
proof.  Divide the ring into $D=1/\eps{1}$ segments.  Each segment is
of length $L=\eps{1}N$.  Label the segments (in order) $1,\ldots, D$.
For simplicity, assume that $L$ and $D$ are integral; the general
analysis is pretty much the same.

Fix a node $i$ from which we will measure $\phi(i)$.  Without loss of
generality, let $i$ be in segment $D$.  Let us consider a hot potato
packet $z$ that can reach node $i$ (and hence contributes to
$\phi(i)$.)  Observe that we can approximately describe a hot potato
packet in terms of the segment it arrived in, and the segment it is
currently in.  If $\eps{1}$ is small enough, this information is
sufficient to get fairly close bounds on $f(z,i)$ and on the
probability of packet $z$ departing in a finite number of steps.

More exactly, suppose that packet $z$ is at node $j$ in segment $J$,
originating from node $k$ in segment $K$.  Please see
Figure~\ref{segment class figure}.
\begin{figure}[ht]
\centerline{\includegraphics[height=1.0in]{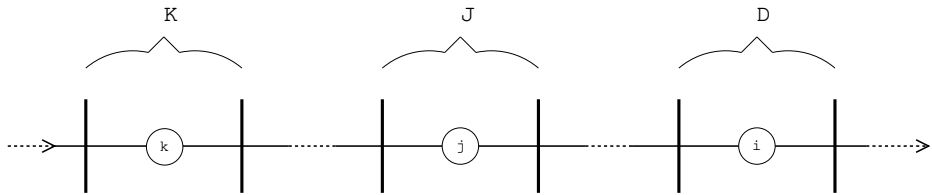} \hspace{.2in}  }
\caption{Segments on the ring}
\label{segment class figure}
\end{figure}
It's possible for a packet to enter segment $D$ twice; it can be
injected near the end of $D$, cross the rest of the ring, and enter a
node near the beginning.  For the sake of notational sanity, if this
happens (e.g.\ $K=D$, but the packet $z$ has left $D$ and may return
to it), then label the segment as $0$ rather than $D$.  Otherwise,
we'd have to write two versions of every equation below.

Consider the collection of hot potato packets in segment $J$ from
segment $K$.  I'll call these packets a \emph{segment class}.  (And
I'll distinguish the segments $0$ and $D$, as in the previous
paragraph.)  There are $(D+1)^{2}$ segment classes.

Assume that we have labelled the nodes so that $k\leq j \leq i$, and
$K \leq J \leq D$.  Then the probability that $z$ departs on the next
step is $1/(N-1-(j-k))$.  We can lower bound this
probability by
\[\frac{1}{N-1-(J-K-1)L} > \frac{1}{N-(J-K-1)L}\]
Note that the probability of a packet departing increases with each
step it spends travelling on the ring, so the probability that $z$ will
depart on the next $t$ time steps is at least
\[\frac{t}{N-(J-K-1)L}=\frac{t}{N[1-(J-K-1)\eps{1}]}\]
Let $C=1/[1-(J-K-1)\eps{1}]$, so our lower bound on the 
probability of a departure is $Ct/N$.

Suppose that there are at least $\eps{2}L$ packets in segment $J$ from
segment $K$, for some $\eps{2}>0$ (to be determined later in the proof).
Then the number of these packets departing over the next
$t$ steps can be lower bounded by a sum of at least $\eps{2}L$ Bernoulli
variables, each of probability $Ct/N$.  We can now use
Lemma~\ref{bernoulli lower bound lemma}, so for any sufficiently large
$t$, we have an exponential tail on the sum.  
More precisely, suppose that there are $Y$ packets in the segment
class (where $\eps{2}L\leq Y \leq L$), and $X$ of them depart
in $t$ time steps.  Then, for any $\eps{3}>0$,
there exists $T_{1}$ and $K_{1}$ such that for any $\gamma \geq 1$, if
$t\geq T_{1}$,
\begin{equation} \label{exp bernoulli hp bound 1 equation}
\Pr [X < (Ct/N)Y(1-\eps{3}\gamma)] \leq e^{-K_{1}\gamma t}
\end{equation}
(because the mean is bounded below by $(Ct/N)Y$.)

Now, by Equation~\ref{hot potato definition equation},
the contribution to $\phi(i)$ caused by $z$ is
\[\frac{N(1+\delta)-(i-k)}{N(1+\delta)-(j-k)}\]
So, the increase in $\phi(i)$ over $t$ time steps, should $z$
fail to depart, is
\[\frac{N(1+\delta)-(i-k)}{N(1+\delta)-(j+t-k)}
-\frac{N(1+\delta)-(i-k)}{N(1+\delta)-(j-k)}\]
Let $W=N(1+\delta)$.  Then
\[=[W-(i-k)]\left(\frac{1}{W-(j+t-k)}-\frac{1}{W-(j-k)} \right)\]
\[=[W-(i-k)]\left(\frac{t}{[W-(j+t-k)][W-(j-k)]} \right)\]
We can bound this by
\[\leq t \left( \frac{W-(D-K-1)L}{[W-(J-K+1)L -t][W-(J-K+1)L]} \right)\]
Assume that 
\begin{equation} \label{bound on N equation}
L>t
\end{equation}
(If $t$ and $\eps{2}$ is fixed and $N$ grows, then
eventually $L>t$ holds.)  
Instantiating $W$ back in, we can further bound the equation.
\[ \leq t \frac{N(1+\delta) - (D-K-1)\eps{1}N}
{[N(1+\delta)-(J-K+1+1)\eps{1}N]^{2}}\]
\begin{equation} \label{h.p. increase in t steps equation}
=\frac{t}{N}\left(\frac{1+\delta - (D-K-1)\eps{1}}
{[(1+\delta)-(J-K+2)\eps{1}]^{2}}\right)
\end{equation}
On the other hand,
the minimum value of $f(z,i)$ is
\[\frac{N(1+\delta)-(i-k)}{N(1+\delta)-(j-k)}
\geq \frac{N(1+\delta) - (D-K+1)L}{N(1+\delta) - (J-K-1)L}\]
\[=\frac{1+\delta - (D-K+1)\eps{1}}{1+\delta - (J-K-1)\eps{1}}\]
We can now combine the bounds above to get some bounds on the change
in $\phi(i)$ from all the packets in the same segment class as $z$
over the course of $t$ steps.  Let $s$ be the segment class of packet
$z$.  Suppose that there are $Y$ packets in the segment class, where
$Y\geq\eps{2}L$, and that there are $X$ departures, where $X\geq
\frac{Ct}{N}Y(1-\eps{3})$.  Let $\Delta_{1}^{s}$ be the change in $\phi(i)$
contributed by these hot potatoes.
Then $\Delta_{1}^{s}$ can be upper bounded by $Y$ times the maximum
increase in $\phi(i)$ in $t$ steps, minus $X$ times the minimum value
that $z$ can contribute to $\phi(i)$.  Plugging in, we find that the
increase in $\phi(i)$ is less than
\begin{eqnarray} 
\Delta_{1}^{s}& <& Y 
\left[
\frac{t}{N}\frac{1+\delta - (D-K-1)\eps{1}}
{[(1+\delta)-(J-K+2)\eps{1}]^{2}}
\right] -
\nonumber\\
& & X
\left[
\frac{1+\delta - (D-K+1)\eps{1}}{1+\delta - (J-K-1)\eps{1}}
\right]
\nonumber
\\ & < & Y 
\left[
\frac{t}{N}\frac{1+\delta - (D-K-1)\eps{1}}
{[(1+\delta)-(J-K+2)\eps{1}]^{2}}
\right] -
\nonumber\\
& &
\left[
\left(
\frac{1}{1-(J-K-1)\eps{1}}
\right)
\frac{t}{N}Y(1-\eps{3})
\right] 
\frac{1+\delta - (D-K+1)\eps{1}}{1+\delta - (J-K-1)\eps{1}}
\nonumber
\\ & = 
\frac{Yt}{N} &  \left[
\frac{1+\delta - (D-K-1)\eps{1}}
{[(1+\delta)-(J-K+2)\eps{1}]^{2}}-
\right.
\nonumber\\
& &
\left.
(1-\eps{3})
\frac{1+\delta - (D-K+1)\eps{1}}
{(1+\delta - (J-K-1)\eps{1})(1-(J-K-1)\eps{1})}
\right]
\label{h.p. increase amalgam equation}
\end{eqnarray}
Now, let $A=1+\delta-(D-K)\eps{1}$, and
$B=1+\delta-(J-K)\eps{1}$.  We can rewrite Equation~\ref{h.p. increase
amalgam equation} as
\[\frac{A+\eps{1}}{(B-2\eps{1})^{2}}
-(1-\eps{3})\frac{A-\eps{1}}{(B+\eps{1})(B-\delta+\eps{1})}
\]
\[=\frac{A+\eps{1}}{(B-2\eps{1})^{2}}
-(1-\eps{3})\frac{A-\eps{1}}{(B+\eps{1})^{2}}
\frac{B+\eps{1}}{B-\delta+\eps{1}}
\]
Now, $-\frac{B+\eps{1}}{B-\delta+\eps{1}}$
is maximized when $B=1+\delta$, so
\begin{equation} \label{tiny epsilon 1 limit equation}
<\frac{A+\eps{1}}{(B-2\eps{1})^{2}}
-(1-\eps{3})\frac{A-\eps{1}}{(B+\eps{1})^{2}}
\frac{1+\delta+\eps{1}}{1+\eps{1}}
\end{equation}
Now, in the limit as $\eps{1}\rightarrow 0$, we get
\[=\frac{A}{B^{2}} - \frac{A}{B^{2}}(1-\eps{3})(1+\delta)\]
Since $A\geq \delta$ and $B\leq 1+\delta$, then $A/B^{2}>0$.
Therefore,
\[=\frac{A}{B^{2}}[1- (1-\eps{3})(1+\delta)]\]
\[=\frac{A}{B^{2}}[\eps{3} - \delta +\eps{3}\delta]\]
So, suppose we take $\eps{3}<\delta/2$.  (There will be
an additional upper bound on $\eps{3}$ below.)  It follows that
$\eps{3} - \delta +\eps{3}\delta<-\delta/2 +\delta^{2}/2$
$=\frac{\delta}{2}[\delta-1]<0$, since $\delta<1$.
Therefore, if we take a sufficiently
small $\eps{1}$, we can make Equation~\ref{tiny epsilon 1 limit
equation} less than zero.  
Thus, Equation~\ref{tiny epsilon 1 limit equation} gives us
our second upper bound on $\eps{1}$.
In summary, we have:
\begin{eqnarray} \label{zero equation}
\Delta_{1}^{s}<0
\end{eqnarray}
In order to get exponential tails on the probabilities,
we have to analyze the behavior if $X$ is a bit smaller.
Suppose that $X\geq \frac{Ct}{N}Y(1-\eps{3}\gamma)$
for $\gamma >1$ (but $Y$ is still $>\eps{2}L$).
Let $\Delta_{\gamma}^{s}$ be the corresponding change in
$\phi(i)$ contributed by the class $s$ hot potatoes.  Then
\begin{eqnarray*}
\Delta_{\gamma}^{s} & \leq & 
\Delta_{1}^{s} + \eps{3}(\gamma-1)
\frac{Ct}{N}Y\frac{1+\delta-(D-K+1)\eps{1}}{1+\delta-(J-K+1)\eps{1}}
\\
& \leq & \Delta_{1}^{s} + \eps{3}(\gamma-1)
\frac{Ct}{N}Y\frac{1+\delta}{\delta}
\end{eqnarray*}
Since $C\leq \frac{1}{\eps{1}}=D$,
\[\leq
\Delta_{1}^{s} + \eps{3}(\gamma-1) \frac{Dt}{N}Y\frac{1+\delta}{\delta}\]
By Equation~\ref{zero equation}, 
\[<\eps{3}(\gamma-1) \frac{Dt}{N}Y\frac{1+\delta}{\delta}\]
Now, $Y/N\leq L/N=\eps{1}$, so
\[\leq\eps{3}(\gamma-1) tD\eps{1}\frac{1+\delta}{\delta}\]
Since $D\eps{1}=1$,
\[=\eps{3} \frac{1+\delta}{\delta}
(\gamma-1) t\]
\begin{equation} \label{Delta lower tail equation}
<\left[\eps{3} \frac{1+\delta}{\delta}\right]
\gamma t
\end{equation}
Suppose that we take
$\eps{3}<\frac{1}{Z(D+1)^{2}}\frac{\delta}{1+\delta}\frac{\eps{0}}{7}$,
where $Z$ will be determined below (and $Z$ will depend only
on $\delta$ and $\eps{0}$.)
(This is the second upper bound on $\eps{3}$.)  Then
\begin{equation} \label{Delta lower tail 2 equation}
<\frac{1}{Z(D+1)^{2}}\frac{\eps{0}}{7}\gamma t
\end{equation}
Let $\Delta^{s}$ be the change in $\phi(i)$ contributed
by the class $s$ hot potatoes if $Y > \eps{2}L$, with no
restriction on $X$.  Using the exponential tail in Equation~\ref{exp
bernoulli hp bound 1 equation} and the linearity of
Equation~\ref{Delta lower tail 2 equation}, we get
\begin{equation} \label{segment hp exponential down drift}
\Pr[\Delta^{s}>\frac{1}{Z(D+1)^{2}}
\frac{\eps{0}}{7}\gamma t] \leq e^{-K_{2}\gamma t}
\end{equation}
for some $K_{2}>0$ and all $t\geq T_{2}$, for some
$T_{2}$.
Let $\Delta$ be the change in $\phi(i)$ from the
hot potatoes in all the $(D+1)^{2}$ segment classes.
Then Equation~\ref{segment hp exponential down drift}
gives us
\begin{equation} \label{hp exponential down drift}
\Pr[\Delta>\frac{\eps{0}}{7Z}\gamma t] \leq e^{-K_{3}\gamma t}
\end{equation}
for some $K_{3}>0$ and all $t\geq T_{3}$, for some
$T_{3}$.

If we take a snapshot of all the hot potatoes in the 
system at time zero, and ask how their contribution
to $\phi(i)$ has changed by time $t$, then
Equation~\ref{hp exponential down drift}
can tell us 
the change.  However, during these $t$ time steps, other
new hot potatoes may enter a cell and begin travelling on 
the ring; these equations don't take those newer hot potatoes
into account.  

I will call these newly inserted hot potato packets
\emph{inserted hot potatoes}, as distinguished from the 
\emph{original hot potatoes}.  Recall that in order to satisfy
Equation~\ref{exp bernoulli hp bound 1 equation}, we needed $t\geq
T_{1}$.  However, we can always make $t$ bigger.  Let $T_{0}=ZT_{1}$
for some sufficiently large integer $Z$.  (We will determine $Z$
below).  Let us take intervals of $T_{0}$ steps (i.e.\ set $t=T_{0}$).

Consider one slot in the ring.  We will consider the time intervals
$(0,T_{1}]$, $(T_{1},2T_{1}]$, ...  $((Z-1)T_{1},ZT_{1}]$.  Consider
inserted hot potatoes at time $jT_{1}$ (for $j=0,1,\ldots,Z-1$).
Suppose there are $Y\geq \eps{2}L$ of them.  
Then, the increase in
$\phi(i)$ by those $Y$ packets during the time interval $(jT_{1},
(j+1)T_{1}]$ has a negative expected value with exponential tails, by
Equations~\ref{hp exponential down drift}.  
Let $\hat{\Delta}$ be the total change in $\phi(i)$ contributed
by hot potatoes travelling during these time intervals.
Adding together
all $Z$ time intervals, and using Equations~\ref{hp exponential down drift},
we get that for $\gamma \geq 1$,
\begin{equation} \label{main hp bound 1 equation}
\Pr[\hat{\Delta}_{1}>\frac{\eps{0}}{7}\gamma t] \leq e^{-K_{4}\gamma t}
\end{equation}
for some $K_{4}>0$ and all $t\geq T_{4}$, for some $T_{4}$.
Note that Equation~\ref{main hp bound 1 equation} holds
simultaneously for all nodes $i$.

Let us consider the increases in $\phi(i)$ from inserted hot
potatoes that I didn't account for above.  There are two cases. First,
there may be fewer than $\eps{2}L$ packets during a time interval
$(jT_{1}, (j+1)T_{1}]$.  Let $\hat{\Delta}_{2}$ be the change in
$\phi(i)$ from these packets.
Recall from Lemma~\ref{hot potato increase
lemma} that the maximum one-step increase in $\phi(i)$ from any packet
is $1/\delta N$.  Then the maximum increase in $\phi(i)$ over all $Z$
such time intervals, over all segment classes in segment $J$, is
at most
\[ T_{0} \frac{\eps{2}L}{\delta N} \]
\[=T_{0}\frac{\eps{0}\eps{1}}{7(D+1)^{2}}\]
Summing over all $D+1$ segments, we get
\begin{equation} \label{main hp bound 2 equation}
\hat{\Delta}_{2}<T_{0}\frac{\eps{0}\eps{1}}{7(D+1)}<\frac{\eps{0}}{7}T_{0}
\end{equation}

Second of all, we must account for the initial contributions from the
inserted hot potatoes.  When a hot potato is inserted, we only started
measuring its contributions to $\phi(i)$ from time $jT_{1}$ onward
(for some $j$).  Therefore, each inserted hot potato can travel for up
to $T_{1}$ time steps before we started measuring its contribution to
$\phi(i)$ in Equation~\ref{main hp bound 1 equation}.
Let $\hat{\Delta}_{3}$
be the contribution from all inserted
hot potatoes to $\phi(i)$ during these unmeasured steps.

How many inserted hot potatoes are there?
Well, if there were more than $N$ inserted hot potatoes during the
$T_{0}$ time steps, then some of the inserted hot potatoes must have
been inserted and then departed.  More precisely, if there were $N+m$
inserted hot potatoes, then there were at least $m$ inserted hot
potatoes that departed.  The probability of an inserted hot potato
departing in at most $T_{0}$ steps is at most $T_{0}/N$.  Let $W$ be
the total number of inserted hot potatoes.  We can use Lemma~\ref{sums
of bernoulli lemma} on the initial $N$ inserted hot potatoes to
conclude that for any $\beta>1$,
\begin{equation} \label{exp tails on the step function}
\Pr[(W-N)\geq \beta T_{0}]\leq e^{(1-\frac{1}{\beta}-\ln \beta)\beta T_{0}}
\end{equation}
Assume that $W-N<\beta T_{0}$ and that 
\begin{equation} \label{second bound on N equation}
N > T_{0}
\end{equation}
The net increase in $\phi(i)$ over all the uncounted time steps
is (by Lemma~\ref{hot potato increase lemma})at most
\[T_{1}\frac{1}{\delta N}W\]
\[<\frac{T_{1}}{\delta N}\left[ \frac{\beta T_{0}}{N}+1\right]N\]
Since $N>T_{0}$, and canceling, we get
\[<\frac{T_{1}}{\delta}[\beta + 1]\]
\[<\frac{2T_{1}}{\delta}\beta\]
\[=\frac{2}{Z\delta}\beta T_{0}\]
If we take $Z>\frac{2}{\delta}\frac{7}{\eps{0}}$, then
\[<\frac{\eps{0}}{7}\beta T_{0}\]
So, our total increase $\hat{\Delta}_{3}$ has exponential tails:
\begin{equation} \label{main hp bound 3 equation}
\Pr\left[\hat{\Delta}_{3} \geq \frac{\eps{0}}{7}\beta T_{0} \right]
<e^{-K_{5}\beta T_{0}}
\end{equation}
for some $K_{5}>0$ and all $T_{0}\geq T_{5}$ for some
$T_{5}$.

The analysis above pretty much accounts for all the significant
influences on $\phi(i)$ for any $i$.  To get a full bound on
$\phi(i)$, though, we must consider all the exceptional (and unlikely)
cases.

First of all, if a packet crosses node $i$ in $T_{0}$ time steps and
departs \emph{after} crossing $i$, the departure doesn't count (since
we're only counting reductions in $\phi(i)$ caused by packets leaving
the ring.)  Let $\hat{\Delta}_{4}$ be this contribution to
$\phi(i)$.
However, this effect is negligible: since the maximum
increase in a packet that crosses $i$ in $T_{0}$ time steps is
$T_{0}/(\delta N)$ (from Lemma~\ref{hot potato increase lemma}) and
there are at most $T_{0}$ such packets, the increase (for a fixed
$T_{0}$) can be bound by
\begin{equation} \label{main hp bound 4 equation}
\hat{\Delta}_{4}<\frac{T_{0}^{2}}{\delta N}=O(1/N)
\end{equation}
For sufficiently large $N$ this quantity can be made arbitrarily
small.

Next, there is a complication for packets in segment $D$, since the
packets before $i$ and after $i$ have different statistics.  I
represented this difference by distinguishing the $J=0$ and $J=D$
segment classes.  Let $\hat{\Delta}_{5}$ be the contribution
from these two segment classes to $\phi(i)$.
It's possible to show the increase in $\phi(i)$ for
every $i$ is well behaved, but it's easier just to consider the worst
case: there are at most $L$ such packets, each increasing $\phi(i)$ by
at most $T_{0}/(\delta N)$, leading to a maximum possible increase of
$T_{0}\eps{1}/\delta$.  By the definition of $\eps{1}$, we get the
bound
\begin{equation} \label{main hp bound 5 equation}
\hat{\Delta}_{5}<\frac{T_{0}\eps{1}}{\delta} < \frac{T_{0}\eps{0}}{7}
\end{equation}

Finally, suppose that nodes $i_{1}<i_{2}<i_{3}$ are in the same
segment (say $D$), and a packet from node $i_{2}$ is travelling on the
ring in another segment (say $J$).  Then the packet contributes to
$\phi(i_{1})$, but not $\phi(i_{3})$ (since it can't reach $i_{3}$
again.)  Observe, however, that if the nodes in segment $D$ are
$i_{0}, i_{0}+1, \cdots,i_{0}+L-1$, then every packet that can reach
segment $D$ crosses node $i_{0}$; some of them may cross $i_{0}+1$;
fewer of them may cross $i_{0}+2$, and so forth.  Fix a segment class
(of packets not in $D$).  Consider the last node of $D$, node
$i_{0}+L-1$.  Suppose there are $Y_{i}$ packets from the segment class
that can cross node $i_{0}+i$.  If $Y_{L-1}<\eps{2}L$, then we know
from above that the increase in $\phi(i_{0}+L-1)$ is inconsequential.
Let us continue backwards across the ring from node $i_{0}+L-1$ until
we hit the first node $i_{0}+w$ such that $Y_{w}
\geq\eps{2}L$.  If there is no such node, we're done,
because all the packets in the segment class only contribute
inconsequentially to the nodes in segment $D$.  Otherwise, if there
exists $w_{1}$ such that $Y_{w_{1}}\geq\eps{2}L$, then we can take
these $Y_{w_{1}}$ packets, and perform our bounding analysis from
above (and Equation~\ref{hp exponential down drift} applies).

Let us continue even further backwards along segment $D$ until we find
the nearest node $i_{0}+w_{2}$ such that $Y_{w_{2}}-
Y_{w_{1}}\geq\eps{2}L$.  We can then take the $Y_{w_{2}}- Y_{w_{1}}$
packets specified and perform our bounding analysis again.  
The nodes between $i_{0}+w_{2}$ and $i_{0}+w_{1}$ may still be
effected by this second batch of packets.  However, the effect is that
of less than $\eps{2}L$ packets per node, so it's inconsequential.  We
can continue this process all the way back to node $i_{0}$.  
Let $\hat{\Delta}_{6}$ be the contribution from
these packets to $\phi(i)$.
Since there are at most $L$ packets in the
segment class, and each jump is at least $\eps{2}L$, then there are at
most $1/\eps{2}$ such batches we need to consider, i.e.\ a finite
number of batches.  Therefore, for sufficiently large
$T_{0}$ (i.e.\ sufficiently large $T_{1}$ with fixed $Z$)
and any $\beta\geq 1$,
\begin{equation} \label{main hp bound 6 equation}
\Pr\left[\hat{\Delta}_{6} > \frac{\eps{0}}{7}T_{0}(1+\beta)T_{0}\right]
<e^{-K_{6}\beta T_{0}}
\end{equation}
for some $K_{6}>0$.  

We now have all the equations to complete the proof.  
We are trying to bound the probability that
there exists an $i$ such that $G(\gamma,i,t)$.
For any $t\geq T_{0}$ and $\gamma \geq 1$, we can bound this quantity by
\begin{eqnarray} \label{final hp equation}
\flu{\Delta} & = &
\Pr\left[
\hat{\Delta}_{1}>\frac{\eps{0}}{7}\gamma t
\right]
+
\Pr\left[
\hat{\Delta}_{2}>\frac{\eps{0}}{7}\gamma t
\right]
+
\Pr\left[
\hat{\Delta}_{3}>\frac{\eps{0}}{7}\gamma t
\right] 
\nonumber
\\ & + &
\Pr\left[
\hat{\Delta}_{4}>\frac{\eps{0}}{7}\gamma t
\right]
+
\Pr\left[
\hat{\Delta}_{5}>\frac{\eps{0}}{7}\gamma t
\right]
+
\Pr\left[
\hat{\Delta}_{6}>\frac{\eps{0}}{7}(1+\gamma) t
\right] 
\nonumber
\\
\end{eqnarray}
Equation~\ref{main hp bound 4 equation}
tells us that for sufficiently large $N_{T_{0}}$, 
$\Pr\left[
\hat{\Delta}_{4}>\frac{\eps{0}}{7}\gamma t
\right]$ is zero.
Equations~\ref{main hp bound 2 equation},
and~\ref{main hp bound 5 equation} tell us that
\[
\Pr\left[
\hat{\Delta}_{2}>\frac{\eps{0}}{7}\gamma t
\right]
+
\Pr\left[
\hat{\Delta}_{5}>\frac{\eps{0}}{7}\gamma t
\right]
=0\]
Using Equations~\ref{main hp bound 1 equation},
\ref{main hp bound 3 equation} and
\ref{main hp bound 6 equation}, we can conclude 
that there exists $K_{0}$ and $T_{0}$,
such that for any $t\geq T_{0}$, there 
exists $N_{t}$, such that for any
$N\geq N_{t}$,
\[\flu{\Delta}<e^{-K_{0}\gamma t}\]
which implies Equation~\ref{hp statement equation}, and
we're done.

Note that the size of $T_{0}$ is determined by
the size of $T_{1}$,
which is determined by the exponentials
in Equations~\ref{main hp bound 1 equation},
\ref{main hp bound 3 equation},
and \ref{main hp bound 6 equation}.
The same equations determine $K_{0}$.
The size of $N_{t}$ is determined by
Equations~\ref{bound on N equation}
and~\ref{second bound on N equation}.
\EOP

\begin{lemma}[Main Lemma] \label{main lemma}
\index{Lemma \ref{main lemma}}
Fix any nominal load $r<1$.  Then there exist
sufficiently large constants $T_{0}$, $K_{0}$, and $N_{0}$, and
sufficiently small $\eps{0}>0$ such that the following holds.

Suppose that we are on an $N$ node ring, for any $N \geq N_{0}$.
Suppose that the network starts in any state $\sigma$, and $T_{0}$ time
steps later is in state $\tau$, where $\tau$ is a random variable.  If
$\Phi(\sigma)>K_{0}N$ then
\begin{equation} \label{main regular ergodicity equation}
\Ex[\Phi(\tau)]-\Phi(\sigma)<-\eps{0}
\end{equation}
Moreover, we can find sufficiently large constants
$T_{1}, K_{1}$, and $N_{1}$ and sufficiently small $\eps{1}>0$, $\eta>1$,
such that if $\Phi(\sigma)>K_{1}N$ then
\begin{equation} \label{main geometric ergodicity equation}
\Ex[\eta^{\Phi(\tau)}]-\eta^{\Phi(\sigma)}<-\eps{1}\eta^{\Phi(\sigma)}
\end{equation}
\end{lemma}

\proof
If we combine Lemma~\ref{exogenous arrivals lemma}
and Lemma~\ref{hot potatoes lemma}, we get
the following the conclusion:  

Suppose we are in state $\sigma$, and take any $\eps{2}>0$
such that $0<2\eps{2}<1-\rh$.  Let
$d(i,t)$ be the increase in $\phi(i)$ over the next $t$ time steps
from exogenous arrivals and hot potatoes.  Let $D(\gamma,i,t)$ be the
event that $d(i,t)>(\rh+\eps{2})\gamma t$, for any $\gamma \geq 1$.
Then there exist $N_{2}, T_{2}$, and $K_{2}$ such that if $N\geq
N_{0}$, then for any $\gamma \geq 1$,
\begin{equation} \label{uber increase tail equation}
\Pr[\exists i \mbox{ such that } D(\gamma,i,T_{0})]
<e^{-K_{2}\gamma T_{0}}
\end{equation}
and (assuming we took our $T_{0}$ large enough in the Lemmas),
\begin{equation}\label{uber increase mean equation}
\Ex[\max_{i} d(i,T_{0})]<(\rh+2\eps{2})T_{0}
\end{equation}
Let $\eps{0}=1-(\rh+2\eps{2})$.  Note that our choice of $\eps{2}$ was
small enough to guarantee that $\eps{0}>0$.

Our first goal will be to establish Equation~\ref{main regular
ergodicity equation}.  Suppose that a node $i$ has at least $T_{0}$
packets waiting in its queue.  Then over the next $T_{0}$ time steps,
it will be guaranteed of ejecting $T_{0}$ packets.  Thanks to
Equation~\ref{uber increase mean equation},
the expected change
in $\phi(i)$ is at most 
\begin{equation} \label{uber mean negative drift equation}
-\eps{0}T_{0}<0
\end{equation}

These bounds are all well and good when node $i$ has a sufficiently
long queue, but $\Phi$ is the maximum over all $i$.  How can we
guarantee that every node with large values of $\phi$ also has a
queue of length at least $T_{0}$?

Lemma~\ref{trick lemma} will provide the trick we need.
Let $\beta>1$ and let
\[\alpha=1-\frac{1}{\beta(1+\delta)N}\]
(Note that $\alpha<1$.)
Suppose that $\phi(i)\geq \alpha \Phi$.  It follows, then,
that $\phi(i-1) \leq \frac{1}{\alpha}\phi(i)$ (since, by 
definition, $\phi(i-1)\leq \Phi$.  Therefore, Lemma~\ref{trick lemma}
implies that
\[Q_{i}\geq \phi(i)-\frac{1}{\zeta}\left[\frac{1}{\alpha}\phi(i) \right]-1\]
We would like to guarantee that $Q_{i}$ is at least, say,
$3T_{0}$.  To guarantee this bound, it is sufficient that
\[\phi(i)-\frac{1}{\zeta}\left[\frac{1}{\alpha}\phi(i)\right]-1
\geq 3T_{0}\]
hence
\begin{equation} \label{main 1 equation}
\phi(i)\geq \frac{3T_{0}+1}{1-\frac{1}{\zeta\alpha}}
\end{equation}
Note that 
$1-\frac{1}{\zeta\alpha}=(\beta-1)/[\beta(1+\delta)N -1]$
which is $\Theta(1/N)$, and the numerator is $\Theta(1)$,
so the right hand side of Equation~\ref{main 1 equation} is $O(N)$.

Now, if $\phi(i)<\alpha \Phi$, how much smaller is $\phi(i)$?
Well, $(1-\alpha)\Phi=\frac{1}{\beta[(1+\delta)N-1]}$, so
if $\Phi\geq 3T_{0}\beta[(1+\delta)N-1]$, then 
\begin{equation} \label{main 2 equation}
\Phi-\phi(i)>3T_{0}
\end{equation}

Let us define $\Phi_{N}$ as:
\[\Phi_{N}=\max\left\{\frac{1}{\alpha}
\frac{3T_{0}+\frac{1}{\zeta}}{1-\frac{1}{\zeta\alpha}},
3T_{0}\beta[(1+\delta)N-1] \right\}\]

Let us suppose, then, that $\Phi \geq \Phi_{N}$.
Observe that $\Phi_{N}$ is $\Theta(N)$. Therefore, we
can find a $K_{0}$ such that $\Phi_{N}\leq K_{0}N$,
as in the statement of this theorem.

Consider any $i$.  If $\phi(\sigma, i) < \alpha \Phi(\sigma)$, then
Equation~\ref{uber increase mean equation}
and Equation~\ref{main 2 equation} imply that
$\Ex[\phi(\tau,i)]<\Phi(\sigma)-2T_{0}$.  If, on the other hand,
$\phi(\sigma,i) \geq \alpha \Phi(\sigma)$, then $Q_{i}\geq 3T_{0}$, so
by Equation~\ref{uber mean negative drift equation},
$\Ex[\phi(\tau,i)]<\phi(\sigma,i)-\eps{0}T_{0}\leq
\Phi(\sigma)-\eps{0}T_{0}$.  Therefore, 
\begin{equation} \label{negative mean drift restated equation}
\Ex[\Phi(\tau)] -\Phi(\sigma) \leq -\eps{0}T_{0}
\end{equation}
Since $T_{0}$ is (much) larger than 1, then
\[\Ex[\Phi(\tau)] -\Phi(\sigma) \leq -\eps{0}\]
which establishes Equation~\ref{main regular ergodicity equation}.

To see that Equation~\ref{main geometric ergodicity equation} holds,
divide it by $\eta^{\Phi(\sigma)}$.
Then we need to prove
\[\Ex\left[\eta^{\Phi(\tau)-\Phi(\sigma)}\right]-1<-\eps{1}\]
(Note that
$\Ex\left[\eta^{-\Phi(\sigma)}\right]=\eta^{-\Phi(\sigma)}$.)  Observe
that Equation~\ref{uber mean negative drift equation} implies that
$\Ex\left[\eta^{\Phi(\tau)-\Phi(\sigma)}\right]$ is an analytic
function of $\eta$ in a neighborhood of $\eta=1$.  Observe that the
first derivative at $\eta=1$ is $\Ex[\Phi(\tau)]-\Phi(\sigma)$, which
we've just shown is negative (in Equation~\ref{negative mean drift
restated equation}).  Therefore, there exists a sufficiently small
$\eta>1$ such that Equation~\ref{main geometric ergodicity equation}
holds.
\EOP

\section{Ergodicity and Expected Queue Length} 
  \label{ergodicity results section} 

Once we have constructed a potential function with negative drift,
there are a number of powerful theorems we can draw on.
Section~\ref{comparison theorem section} reviews this material.  These
drift theorems allow us to translate Lemma~\ref{main lemma} into
statements about the ergodicity and expected queue length of the
system.  Let us begin with some immediate ergodicity results.

\begin{theorem} \label{ergodicity payoff theorem}
\index{Theorem \ref{ergodicity payoff theorem}}
The standard Bernoulli ring is ergodic if $r<1$, for all
sufficiently large $N$.  Moreover, it converges to its
stationary distribution exponentially rapidly.  Finally,
\begin{equation} \label{expected phi equation}
\Ex[\Phi] = O(N)
\end{equation}
\end{theorem}

\proof
To show that the Markov chain is ergodic, we can use Equation~\ref{main
regular ergodicity equation} and Foster's criterion
(Corollary~\ref{foster's criterion}).

Now, Equation~\ref{main geometric ergodicity equation} 
and Corollary~\ref{geometric ergodicity corollary}
allows us
to establish the stronger property of geometric ergodicity.
Exponential rates of convergence to stationarity follow.

Next, using Equation~\ref{main geometric ergodicity equation} again,
and, Theorem~\ref{geometric ergodicity corollary} (or Theorem 14.0.1
from Meyn and Tweedie~\cite{Meyn_and_Tweedie}), we can conclude that
\[\Ex[\eta^{\Phi}]<\infty\]
and hence
\[\Ex[\Phi]<\infty\]
for a fixed $N$.  Finally, we can use the Comparison
theorem (see Theorem~\ref{comparison theorem}) and the fact that the
negative drift holds for all states $\sigma$ with $\Phi(\sigma)>K_{1}N$
to conclude that
\[\Ex[\Phi]=O(N)\]
\EOP

Next, I'll show how to convert Equation~\ref{expected phi equation}
into a bound on the expected queue length.

\begin{theorem} \label{0(1) expected queue theorem}
\index{Theorem \ref{0(1) expected queue theorem}}
On a standard Bernoulli ring, the expected queue length per node is
$O(1)$.
\end{theorem}

\proof
Consider node $N-1$.  Then
\[ \Ex [\phi(N-1)]=\Ex \left[\sum_{j=1}^{N-1}\left(c_{j} 
	+ \frac{j}{N-1}Q_{j}\right)\right]\]
where $c_{j}$ is the expected contribution to $\phi(i)$
from the packet in service, and $Q_{j}$ is the length of
the $j$th queue.  
The $c_{j}$ terms add up to $rN/3$, but rather than calculate
that, I'll just drop the (non-negative) term:
\[\geq  \Ex \left[\sum_{j=1}^{N-1} \frac{j}{N-1}Q_{j}\right]\]
\[=  \sum_{j=1}^{N-1} \frac{j}{N-1}\Ex [Q_{j}]\]
Note that $\Ex[Q_{j}]=\Ex[Q]$, i.e.\ the
expected queue length per node is independent of the node
(because of cyclical symmetry).  
\[=  \Ex[Q]\sum_{j=1}^{N-1} \frac{j}{N-1}\]
\[= (N/2) E[Q]\]
Now, from Equation~\ref{expected phi equation} we can write
\[O(N)=\Ex[\Phi] \geq \Ex[\phi(N-1)] \geq (N/2)E[Q]\]
Therefore, dividing by $N$,
\[\Ex[Q]=O(1)\]
\EOP

Because of Little's Theorem, we can translate this result
into a tight bound on the expected delay of a packet.
\begin{corollary} \label{theta(N) packet delay corollary}
\index{Corollary \ref{theta(N) packet delay corollary}}
The expected delay per packet of an $N$-node standard Bernoulli
ring with nominal load $r<1$ is $\Theta(N)$.
\end{corollary}

\proof
The expected delay of a packet consists of the expected delay
while waiting in queue, plus the expected delay while
travelling along the ring.  Since destinations on the ring
are uniformly distributed from $1$ to $N-1$, then the
expected delay on the ring is $N/2$.  Therefore, the
total expected delay of a packet is $\Omega(N)$.

Little's Theorem (Theorem~\ref{little's theorem}) tells us
that the expected delay at a fixed queue is the expected queue
length times the arrival rate.  Since the arrival rate
is $p=2r/N$ and the expected queue length is $O(1)$ per
node, then the expected delay in queue is $O(N)$.  Adding
the $N/2$ expected delay in the ring, we have the expected
delay of a packet is $O(N)$.

Combining the upper and lower bounds, the expected delay
of a packet is $\Theta(N)$.
\EOP

\section{Other Bernoulli Rings}
  \label{nonstandard ring bounds section}

In this section, I will briefly discuss extensions of
Sections~\ref{main section} and~\ref{ergodicity results section}.  The
reasoning is closely related to that of the standard Bernoulli case,
so the proofs are only in outline.

\subsection{The Bidirectional Ring}
\index{bidirectional ring!$O(1)$ queue length}
There is nothing terribly special about the standard
Bernoulli ring, as indicated by the following theorem.
\begin{theorem}
Fix $\alpha>0$.  Suppose we have a family of nonstandard Bernoulli
rings, where the $N$th ring has $N$ nodes, and parameter $L=\lfloor
\alpha N \rfloor$.  Fix a nominal load $r<1$.  Then the expected queue
length per node is $O(1)$.
\end{theorem}
\proof
The arguments are identical to those for a standard Bernoulli ring,
since the parameter $L$ scales linearly with $N$.
\EOP

We can use this result to analyze a bidirectional ring.
\begin{corollary}[Bidirectional Ring]
For a fixed nominal load $r<1$, and a Bernoulli arrival process,
the expected queue length per node for a bidirectional ring is $O(1)$.
\end{corollary}

\proof
Decompose the ring into two unidirectional butterflies, where
$L=\lfloor (N-1)/2 \rfloor$, as in Section~\ref{bidirectional
section}.  We then get two $O(1)$ bounds on the expected
queue length, which we can add together by the linearity of
expectation.

There is a minor detail to worry about if $N$ is even, 
because there isn't a unique shortest path to the node 
$N/2$ hops away.  If we specify that these $N/2$
length paths are all (say) clockwise, then the decomposition
above works.  If we decide that the packet chooses between
the two paths with equal odds, then we need to make some
minor (and simple) adjustments, but the proof still
follows.
\EOP

\subsection{$N$ constant, $L \rightarrow \infty$}
The technique for the standard Bernoulli ring 
works fairly well if the size of the ring is fixed.

\begin{corollary}
Suppose we have an $N$-node nonstandard Bernoulli
ring with parameter $L$.  Suppose that $N$
is constant.  Then for any nominal load $r<1$,
the expected queue length is $\Theta(1)$ as
$L\rightarrow \infty$.
\end{corollary}

\proof
We can repeat the results of Section~\ref{main section}, as in the
standard Bernoulli case, but we need to take longer time steps.  That
is, the $T_{0}$ terms, which were $O(1)$ in the standard case, become
$O(L)$.  However, since there are only a constant number of nodes, we
still get an $O(L)$ bound for $\Ex[\Phi]$.  As in
Section~\ref{ergodicity results section}, we can translate this into
an $O(1)$ upper bound on the expected queue length per node.  This
bound matches the $\Omega(1)$ bound from Section~\ref{lower bound
section}, giving a tight $\Theta(1)$ bound.
\EOP

It's worth considering the following intuitive analysis of the system.
Suppose that, for a fixed $N$, we rescale time by speeding it up by a
factor of $L$, and let $L\rightarrow \infty$. Then the network begins
to resemble a single-queue network with $N$ servers, where each packet
has an amount of work uniformly distributed from $0$ to $1$.  The 
expected queue length is finite (because the variance is bounded).
Changing
the time scale of a network doesn't change the expected queue length,
so this limit would suggest an $O(1)$ limit for the expected queue
length of the original non-standard Bernoulli ring.

\subsection{$L$ constant, $N \rightarrow \infty$}

The technique doesn't work quite so nicely if the ring grows while the
parameter $L$ is fixed.

\begin{corollary}
Suppose we have an $N$-node nonstandard Bernoulli
ring with parameter $L$.  Suppose that $L$
is constant.  Then for any nominal load $r<1$,
the expected queue length is $O(\log N)$ as
$N\rightarrow \infty$.
\end{corollary}

\proof
We can proceed exactly as in Sections~\ref{main section} 
and~\ref{ergodicity results section}.  However, rather
than showing negative drift in $T_{0}=O(1)$ time steps,
we need $O(\log N)$ time steps.  Correspondingly,
the upper bound on the expected queue length is $O(\log N)$.
\EOP

This weaker result is to be expected; since we are maximizing the
$\phi(i)$ over all $i$, and the nodes are fairly independent, then we
would expect an order $\log (N)$ result for the maximum $\phi(i)$.  In
all likelihood, the lower bound is tight, and the expected queue
length per node is $\Theta(1)$.  There is probably a way to modify
$\Phi$ to get the $O(1)$ bound, but it's not obvious how.

\subsection{Bounded Queue Lengths}
If queues have a bounded maximum queue length in any (otherwise)
standard Bernoulli ring, how does this effect the queue length?
Well, suppose that the $n$th ring has an upper bound
of $B_{n}$ on the number of packets in any queue.  If a queue
is full, any excess exogenous arrivals are simply deleted.

If the $B_{n}$ are bounded by some $B$, then obviously the
expected queue length per node is $O(1)$.  But what
happens if the $B_{n}$ are unbounded?  

\begin{corollary}
Suppose we have a family of standard Bernoulli rings,
where the $n$th ring is has a maximum queue length of
$B_{n}$.  Then the expected queue length is $O(1)$.
\end{corollary}

\proof
We can use exactly the same Lyapunov function $\Phi$ to show drift in
this network; all the proofs are identical.  The only difference
in the analysis is that certain packets never arrive.  This change
only helps us to bound $\Phi$ (by strictly reducing the exogenous
arrival bound in Lemma~\ref{exogenous arrivals lemma}).  Therefore, for
any fixed nominal load $r<1$, the expected queue length is $O(1)$.
\EOP


\section{Future Work}
In order to prove the results in Section~\ref{main section}, I had to
prove exponential bounds on the tails of unlikely events.  These
bounds ultimately came from Theorems~\ref{sums of bernoulli lemma}
and~\ref{bernoulli lower bound lemma}.  It should be
possible to prove these sorts of results with any arrival process
with appropriate exponential tails.  For instance, if the number
of packet arrivals in one time step has a geometric
or Poisson distribution, then we ought to be able to show 
$O(1)$ bounds on the expected queue length by using the
same techniques from this chapter.

Can we extend these results to a continuous time ring?
Consider the number of arrivals to a node in $N$ steps.
As $N$ gets large, this distribution
converges to a Poisson distribution.  (The convergence
of a rescaled Bernoulli process to a Poisson
process is sometimes called a ``baby Bernoulli''
approximation.)\index{baby Bernoulli}
Perhaps, then, we could construct a continuous time version of the
GHP protocol, and show that under a Poisson arrival process,
the expected queue length per node is $O(1)$.  
(The natural continuous time version of the GHP
protocol is not obvious, unfortunately.)

Finally, let's consider higher dimensional variants on the ring.
Suppose we have a $d$-dimensional torus $T=N_{1}\times\cdots\times
N_{d}$.  Suppose that packets arrive according to a rate $p$ Bernoulli
process at every node, and destinations are uniformly distributed
throughout the torus.  Every node has out-degree $d$, so suppose we
allow a node to route packets along as many of these edges as it can.
The appropriate queueing theory model for this network, then, will be
its edge graph, $T_{e}$ (since each edge only routes at most one
packet per time step), but I'd like to translate the results back to
the original network $T$.  (In $T$, the edges queues wait at the
nodes; the queue at node $t\in T$ consists of all the queues $t_{e}\in
T_{e}$ representing edges originating at $t$.)

We still need to specify the protocol.  We can route packets using
dimensional routing, but not all possible conflicts are resolved.  We
need a refinement to the protocol.  Suppose that we consider a subring
$R$, and consider the packets travelling along it.  Let us give
precedence to hot potatoes travelling along the ring.  Exogenous
packets, and packets entering from another ring, all wait (in FIFO
order) in queue.  This protocol is a kind of higher dimensional GHP.

The techniques of this chapter should suffice to prove an $O(d)$
expected queue length per node for any torus.  (Because dimensional routing
is inherently asymmetric, it would be difficult to strengthen this to
an $O(1)$ expected queue length per node in $T_{e}$.  If we
symmetrized the dimensional routing, though, it would probably work.)
If the torus had the same size in all dimensions, i.e.\ if
$N_{i}=N_{j}$ for all $i,j$, then it should follow that the expected
delay per packet was $\Theta(dN)$.

%% file: fluid.tex
\chapter{Fluid Limits} \label{fluid chapter}

\section{Introduction}
In Chapters~\ref{exact chapter} and~\ref{bounds chapter}, I analyzed
specific models of packet routing on the ring.  That is, I specified
the arrival process (Bernoulli), the distributions of packet life
spans (uniform over some range), and the protocol (GHP, for the most
part).  Once these details were specified, I could attempt to prove
ergodicity results and bounds on the expected queue length.

Could we do more?  Might it be possible to prove the stability
of the ring under \emph{any} arrival process, with \emph{any} protocol?

The answer, more or less, is yes.  This chapter is devoted to the
development of a technique known as the \emph{fluid limit} approach.
It allows us to establish the ergodicity of vast classes of queueing
networks with relative ease.

The idea behind fluid limits is to take a stochastic 
process of interest (like the length of queues in a network), rescale
it in time and space (e.g.\ speed up time by a factor of $T$, while
simultaneously dividing the queue lengths by $T$), and take the
limit as the scaling goes to infinity (i.e.\ $T\rightarrow \infty$.)
It turns out that this process is well-defined in many cases of interest,
and the limit (called the ``fluid limit model'') can be
fairly simple to analyze.  Laws of large numbers convert the 
stochastic system into a deterministic one, and the discrete number of
customers in queues is transformed into a continuous (``fluid'') quantity.

Fluid limit models are of interest because they give information
about the original model.  In particular, if the fluid limit model is
stable, then the original stochastic system is stable.  (I will define
fluid stability later in this chapter.)  

There is a growing body of literature on fluid limits.  The seminal
paper establishing ergodicity by the fluid limit technique is by
Dai~\cite{Dai}.  The result was refined by Chen~\cite{Chen}, and 
extended to higher moments (e.g.\ finiteness of expected queue length)
by Dai and Meyn~\cite{Dai_and_Meyn}.  The fluid stability of 
the ring was proved by Dai and Weiss~\cite{Dai_and_Weiss}.

Given all the results I've just referenced, it may sound like there's
nothing left to do; we should just look up the ergodicity results and
rejoice.  Unfortunately, there is a complication.  The results apply
to continuous time, but specifically exclude discrete time systems.  I
rectify that problem in Section~\ref{drift stability section}, and
extend the fluid limit approach to discrete time.  We can now apply
the other ring stability results from the literature and make
conclusions about ergodicity on the (discrete time) ring under any
greedy protocol.

A more limiting restriction of the networks studied in the literature
is the dynamics of the stochastic processes.  Let us consider
the arrival process at a particular queue as an example.
In a traditional network of queues, we imagine that the interarrival
times (the amount of time between adjacent arrivals of packets
to the same class) form a series of i.i.d.\ random variables.

But are interarrival times in packet routing networks really
identically distributed?  In actual networks, like the internet, there
are brief periods with short interarrival times (i.e.\  lots of new
packets arrive), interspersed with long periods of relative silence.
These sorts of long-range correlations (and even self-similarity) have
been verified empirically.  See, for instance, the work of
Crovella in~\cite{crovella} and~\cite{crovella2}.

The situation can be even more dire.  Suppose that a malicious hacker
decides to destabilize the network.  He can inject packets from any
node whenever he wants, but if he simply floods the network, he'll be
detected and eliminated.  Therefore, he'll try to destabilize the
network not with brute force, but by timing his packet injections
carefully.  Because the packet injections are ultimately performed by
a computer program (written by the hacker), we can model the adversary
as a finite state machine, possibly using randomness in the choice of
states (i.e.\ a randomized FSM).  To simplify the problem, let us
assume that the machine's state is independent of the state of the
network (but simply executes according to its own internal
logic).\footnote{For the reader familiar with the bounded adversaries
of adversarial queueing theory,\index{adversarial queueing theory}
this model may sound faintly reminiscent.  There are two crucial
differences.  First, the complexity of the bounded adversary is
allowed to be arbitrarily great; for instance, the strategy need not
even be recursively computable.  Second of all, a bounded adversary
has some associated constant $B$, such that it can inject a limited
number of packets into any window of $B$ steps.  I view a randomized
FSM as providing a much more reasonable model of an adversary, in that
the complexity is bounded, but there is no artificial length $B$
window to consider.}

Unfortunately for web surfers everywhere, this concept is modelled
after a real-world phenomenon, the \emph{Distributed Denial of
Service}\index{DDOS attack}\index{distributed denial of service attack}
(DDOS) attack.  A DDOS attack involves a hacker taking over a large
number of computers on the internet, then instructing all of them to
download the same web page.  Although each computer only requests a
few downloads, the number of computers involved and their simultaneity
can crash major web pages.  A fairly high-profile example of this
occurred to Yahoo on February 7th, 2000 (see Richtel~\cite{richtel}).

To model these systems, we'd really like to allow more general
stochastic processes.  In Section~\ref{packet routing defs section}
and onward, I show how to extend the fluid limit technique to handle
hidden Markov processes (which allow us to simulate both long-range
correlations and randomized FSMs).  An immediate implication is the
ergodicity of the ring under any greedy protocol in this more general
setting.

[For the reader unfamiliar with fluid limits, or who doesn't trust that
these dramatic-sounding stability results really follow, I've included
Appendix~\ref{dummies chapter} as a self-contained primer.
Contentwise, the appendix amounts to proving a special case of Dai's
results.  The results are general enough to apply to the Bernoulli
ring, though.  The proofs themselves are different and much
simpler.]

\section{A Drift Criterion for Stability} \label{drift stability section}
The first property to consider in a Markov chain is its ergodicity.
Definitions of Markov chains and ergodicity can be found in Section~\ref{basic
probability section}.

Let us fix some notation:
\begin{definition}
We are considering a discrete time, irreducible, aperiodic Markov
chain $X(\cdot)$.  It has a countable state space, $\ssx$.  $X^{y}(t)$
is the state of the Markov chain at time $t$ when started in state
$y$.

If $t$ isn't an integer, then we interpret $X^{y}(t)$ as
$X^{y}(\lfloor t \rfloor)$
\end{definition}
I'm going to devote this section to proving that an apparently very
weak drift condition is sufficient to establish ergodicity.  First,
I'll need to define a bounded norm on a countable state space.

\begin{definition} \index{bounded norm}
Suppose we have a state space $\ssx$.  A \emph{bounded norm} is a
function $|\cdot|:\ssx\rightarrow \real^{+}$ such that for any integer
$k$, the set $\{x|\,\,|x|\leq k\}$ is finite.
\end{definition}

For example, in many queueing systems, the sum of the queue lengths forms
a bounded norm.

\begin{theorem} \label{lim stab}
\index{Theorem \ref{lim stab}}
Assume we have an irreducible, aperiodic discrete time Markov chain
with a countable state space and a bounded norm, $|\cdot|$.  Suppose
there exists $T>0$ such that
\begin{equation} \label{eq:limit stability}
\lim_{|x|\rightarrow \infty} \frac{1}{|x|} \Ex \left| X^{x}(
|x|T )\right| =0 
\end{equation}
Then $X(\cdot)$ is an ergodic Markov chain.
\end{theorem}

\note If you find the notation ``$\lim_{|x|\rightarrow\infty}$''
vague, then specify an enumeration of the state space: $x_{1},
x_{2},\ldots$. Since for every $k$ there are only finitely many $x_{n}$
with $|x_{n}|\leq k$, then $\lim_{n \rightarrow
\infty}|x_{n}|=\infty$.  We can then use this ordering $\{x_{n}\}$
above.

\note
A version of this theorem appears in Dai~\cite{Dai} as Theorem 3.1.
Dai's version assumes that the interarrival and service times are
unbounded and spread out (Equations 1.4 and 1.5,) which rules out
discrete time systems.  My proof removes that restriction (in discrete
time).  The only related discrete time theorem in the literature is by
Malyshev and Menshikov~\cite{malyshev}, and is substantially weaker.

The proof below also generalizes the role of the norm, which will
be useful later in this chapter.

\proof
The existence of the limit implies that for any $\epsilon$,
there exists a sufficiently large bound $L$ such that
\[\mbox{ if } |x|>L \mbox{ then }
\frac{1}{|x|} \Ex\left| X^{x}(|x|T)\right| 
\leq \epsilon\]
Let $\epsilon=1/2=1-\epsilon$, and take $L\geq 1$, so
\[\mbox{ if } |x|>L \mbox{ then }
\frac{1}{|x|} \Ex\left| X^{x}(|x|T)\right| 
\leq 1- \epsilon\]
Let $B=\{x\in \ssx \, | \, |x|\leq L\}$.  Then for any
$x\not\in B$,
\[\Ex\left| X^{x}(|x|T)\right| 
\leq (1-\epsilon)|x| = |x| - \epsilon|x| \]
Now, consider the following function:
\[ n(x)=\left\{ \begin{array}{ll} |x|T, &  \mbox{if }x\not\in B\\
			T, & \mbox{if } x\in B \end{array}
	\right. \]
Since we chose an $L$ such that $L\geq 1$, 
it follows that $n(x)\geq T$ for all $x$.  Therefore, for any 
$x$,
\begin{equation} \label{eq:stab1}
\Ex\left| X^{x}(n(x))\right| 
\leq |x| -\epsilon|x| + L_{1}1_{B}(x)
\leq |x| -\epsilon \frac{n(x)}{T} + L_{2}1_{B}(x)
\end{equation}
where $L_{1}, L_{2}$ are some (finite) constants.  To see that
$L_{1}$ is finite, first observe that 
$\Ex\left| X^{x}(n(x))\right| \leq
|x| +$ (expected number of arrivals in $n(x)$ steps),
which is finite for any fixed $x$.  Therefore,
we can just take the maximum of
$\Ex\left| X^{x}(n(x))\right| -|x| +\epsilon L$
over all $x\in B$.  Since $B$ is a finite set, this maximum exists.
We can take $L_{2}=L_{1}+\frac{L}{T}\epsilon$.

We now construct a new Markov chain (an ``embedded chain''), as follows.
We have our original transition probabilities, where $p_{i,j}$ is the
probability of changing from state $i$ to state $j$ in one step, and
$p_{i,j}^k$ is the probability of changing from $i$ to $j$ in exactly
$k$ steps.  Construct a new Markov chain on the same state space
with transition probabilities $\hat{p}_{i,j}=p_{i,j}^{n(i)}$.

The embedded chain (call it $\hat{X}(t)$) represents a particular
sampling of points from the original chain, namely 
\\ $\hat{X}(0)=X(0)$, 
\\ $\hat{X}(1)=X(n(X(0)))$,
\\ $\hat{X}(2)=X(n(X(n(X(0)))) + n(X(0)))$, and so on.  If we define $s(t)$
by $s(0)=0$, and $s(t+1)=n(X(s(t))) + s(t)$, then $\hat{X}(t)=X(s(t))$.
(Note that $s(t)$ isn't a deterministic function; 
it's a stochastic process.)

Let $\tau_{B}$ be the first return time to $B$ in $X(t)$; that is, $\tau_{B}$
is a stopping time defined as the least time $t\geq1$ such that
$X(t)\in B$.  Let $\hat{\tau}_{B}$ be the first return time to $B$ in
$\hat{X}(t)$.  Then observe that if $\hat{X}$ has returned to $B$
by time $t$, then $X$ has returned to $B$ by time $s(t)$ (and possibly
sooner).  So,
\begin{equation} \label{eq:embed}
s(\hat{\tau}_{B})=\sum_{k=0}^{\hat{\tau}_{B}-1}n(\hat{X}(k))
\geq \tau_{B}
\end{equation}

Considering the embedded chain, we can view Equation~\ref{eq:stab1} as
\[
\Ex\left| \hat{X}^{x}(1)\right| 
\leq |x| -\epsilon \frac{n(x)}{T} + L_{2}1_{B}(x)\]
which tells us about the one-step drift of $\hat{X}(t)$.  Using the
Comparison Theorem (Theorem~\ref{comparison theorem}), we conclude
that for any $x\in\ssx$, if we set $X(0)=\hat{X}(0)=x$, then
\[ \Ex\left[ \sum_{k=0}^{\hat{\tau}_{B}-1}\frac{\epsilon}{T}
n(\hat{X}(k))\right]
\leq |x| + L_{2}\]
dividing both sides by $\epsilon/T$ and using
Equation~\ref{eq:embed}, we conclude that
\[\Ex[\tau_{B}] \leq 
\frac{T}{\epsilon}(|x| + L_{2})\]
Thus, the expected return time to $B$ from any state is finite, so (by
Theorem~\ref{basic ergodic facts theorem}), $X(t)$ is a positive
recurrent Markov chain.
\EOP

A more careful examination of the preceding proof reveals that
we don't really need the limit to go to zero in 
Equation~\ref{eq:limit stability}-- it is sufficient that
\[\limsup_{|x|\rightarrow \infty} \frac{1}{|x|} \Ex \left| X^{x}(
|x|T  )\right| <1-\epsilon \]
for any fixed $\epsilon>0$.

Theorem~\ref{lim stab} is sufficient to establish a universal
stability result on the ring.

\begin{corollary} \label{ergodic ring corollary}
\index{Corollary \ref{ergodic ring corollary}}
Suppose we are routing on an $N$ node ring in discrete time under any
greedy protocol.  Suppose that the ring is a generalized Kelly network
(see page~\pageref{Kelly network}).  Finally, suppose that the nominal
loads are less than one at each node (see page~\pageref{nominal
load}), i.e.\ $r<1$ at every node.  Then the ring network is ergodic.
\end{corollary}

\proof
Suppose we replace the probabilistic model of the ring, with discrete
packet arrivals, with a deterministic model, with continous
quantities of packet ``fluid'' entering the system.  The rates
of flow in the fluid model are determined by the expected rates
in the probabilistic model.
Dai and Weiss~\cite{Dai_and_Weiss} show that the fluid model
of a ring is stable; that is, all the fluid eventually exits the
system.  

Dai~\cite{Dai} shows that if the fluid model is stable, then
Equation~\ref{eq:limit stability} holds. 
We can now plug into Theorem~\ref{lim stab}, and we're done.

If the reader is interested in a self-contained version of this for
the standard (or non-standard) Bernoulli ring, he or she can combine
the results of Appendix~\ref{dummies chapter} and Corollary~\ref{ring
stability by convexity corollary}.  The more general results of the
next sections, combined with Corollary~\ref{ring stability by
convexity corollary}, will give Corollary~\ref{ergodic ring corollary}
on any ring with constant mean service times (not just Bernoulli
ones.)
\EOP

The amount of time it takes a packet to cross an edge can be much more
general than the deterministic behavior of Corollary~\ref{ergodic ring
corollary}, and the result will still hold.  All we really need is
that the ring is a generalized Kelly network.

\section{Our Model of Packet Routing} \label{packet routing defs section}
I'm going to lay out all the assumptions I'll make about the packet routing
model, and fix some of my notation.  I'm making an effort to make
this framework very general, so I'm not assuming (for instance)
that the time it takes to cross an edge is one time step.
\begin{itemize}
\item We're operating in discrete time.
\item Packets travel on an $N$ node network, which can be an arbitrary 
directed graph.  Packets wait at nodes (rather than edges).  For
packet routing where packets queue on edges, we just consider the edge
graph and perform our analysis there.
\item Packets are members of a class.  There are $C$ classes (where
$C$ is finite).
A class usually contains information such as a packet's destination,
or possibly its destination and priority in the system.
By queueing theory convention, each class
occurs at only one node.  (This is not restrictive-- it's just a 
naming convention.)  

For a node $i$, let $C_{i}$ be the \emph{constituency} of $i$, i.e.\ 
the collection of classes that occur at node $i$.
\item Packets enter a node either from another node, or from outside
the network.  The first type of packet is called an \emph{internal
packet},\index{internal packet} and the second type is called an
\emph{exogenous packet}.\index{exogenous packet}

In a traditional network of queues, the exogenous packet arrivals
are determined by their ``interarrival times''.  That is,
there is a series of i.i.d.\ random variables that determine
how much time passes between each arrival of a class $c$ packet.

We're going to use a more general arrival process.  (I'll demonstrate
the reduction of the i.i.d.\ case to the following case later.)  The
exogenous packet arrivals are determined by a Markov process $A(t)$ on
a state space $\ssa$.  (``$A$'' stands for ``arrival''.)  For each
class of packet $c$ and each integer $i>0$, there is a set
$\ssa_{c}^{i}\subseteq\ssa$ such that whenever the Markov chain enters
state $\sigma$, if $\sigma\in\ssa_{c}^{i}$ then $i$ packets of class
$c$ arrive.  For a fixed $c$, the $\ssa_{c}^{i}$ are disjoint.  If
$c$ varies, the $\ssa_{c}^{i}$ are not necessarily disjoint.  
The $\ssa_{c}^{i}$ may be empty (if there are never $i$
exogenous arrivals to a particular class).

\item Next, we need to specify the behavior of the internal packets.
To avoid trivialities, I'm going to assume that at most one 
packet departs a node on each time step.  The interested
reader can generalize this appropriately.

We need to be a bit more careful with the packet routing than with the
arrival process.  We need to define a different Markov process for
each class $c$.

When a packet leaves a node, it must either follow one of the outgoing
edges to another node, or it must leave the network.  
These decisions are made by a Markov process $R_{c}(t')$ on 
a state space $\ssr_{c}$.  (``$R$'' stands for ``routing''.)
For every class $c$ at node $n$, and every node $m$ with
an edge from $n$ to $m$, there is a subset $\ssr_{c}^{m}\subseteq\ssr$.
The $\ssr_{c}^{m}$ are disjoint.

If a class $c$ packet leaves node $n$ and
we're in state $\sigma \in \ssr_{c}^{m}$, then the packet
travels to node $m$.  If 
$\sigma \not\in \ssr_{c}^{m}$ for any $m$, then the packet leaves
the system.

The $t'$ variable in the Markov process $R_{c}(t')$ is not the time,
but rather the number of packets that have been routed from class $c$
so far (so $tleq t'$).  In other words, time only advances for
$R_{c}(t')$ when packets are being routed; otherwise, the Markov chain
remains frozen in the same state.

\item
Next, we need to determine how long it takes a packet to cross a node.
As mentioned above, this amount of time is usually deterministically
one in the case of traditional packet routing networks, but can be a
collection of i.i.d.\ random variables of arbitary distribution in a
more general network of queues.

For each class $c$, we define a Markov
process $S_{c}(t')$ on the state space $\sss_{c}$.  


For each class $c$ and integer $i>0$, there's a subset $\sss_{c}^{i}
\subseteq \sss_{c}$ such that for a fixed $c$, the $\sss_{c}^{i}$
are disjoint.  If we are working on a class $c$ packet and enter class
$\sigma\in\sss_{c}^{i}$, then $i$ class $c$ packets leave the node
(although some may possibly reenter the node immediately, if there's
an edge from the node to itself).  If there are $i_{0}<i$ class $c$
packets in queue, then all $i_{0}$ will leave immediately, and the
next $i-i_{0}$ class $c$ packets that have work done on them will
immediately depart, i.e.\ have service times of zero.  This means that
it is possible for a packet to travel across several nodes in one time
step.  (If we set $\sss_{c}^{i}=\emptyset$ for any $i>1$ and all $c$,
then this strange node-hopping behavior never occurs.)

As with the routing process, the $t'$ counts the number of units of
time spent servicing class $c$ packets, so $t' \leq t$.

The state of $S_{c}(t')$ will only advance after $i$ packets of
class $c$ have left; otherwise, we remain frozen in the same state.

\item $A(t)$, $R_{c}(t)$ and the $S_{c}(t)$ are mutually independent,
irreducible, aperiodic, and ergodic.  For any state $\sigma$ in any of
these Markov chains, let $p(\sigma)$ be the stationary probability of
being in state $\sigma$.

The \emph{mean (exogenous) arrival rate}
\index{mean (exogenous) arrival rate}
 of class $c$ packets is:
\[\alpha_{c}=\sum_{i=1}^{\infty}\sum_{\sigma \in \ssa_{c}^{i}}
i p(\sigma)\]
which we will assume to be finite.

The \emph{mean service rate}
\index{mean service rate}
of class $c$ packets is:
\[\mu_{c}=\sum_{i=1}\sum_{\sigma \in \sss_{c}^{i}}
i p(\sigma)\]
which we will also assume to be finite.

The \emph{mean transition probability}
\index{mean transition probability}
of class $c$ packets (from node $k$) to node $l$ is:
\[P_{kl}=\sum_{\sigma \in \ssr_{c}^{l}}p(\sigma)\]
Let $P$ be the matrix formed from the $P_{kl}$.  Assume
that this matrix is transient, i.e.\
\[I+P+P^{2}+\cdots \mbox{ is convergent}\]
This implies that the expected number of visits to class $l$
by a class $k$ packet is finite, i.e.\ we have an open queueing
network.  It follows from this equation that
\[(I-P')^{-1}=(I+P+P^{2}+\cdots)'\]
Then, the \emph{effective arrival rate}
\index{effective arrival rate}
(in vector form) is:
\[ \lambda=(I-P')^{-1}\alpha\]
For a particular class $c$, the let $\lambda_{c}$ be the $c$th
coefficient.

Finally, the \emph{nominal load}
\index{nominal load} \label{nominal load}
at node $n$ is
\[\rho_{n}=\sum_{c\in C_{n}}\lambda_{c}/\mu_{c}\]
\item Our routing protocol is greedy, or work-conserving-- if a queue
is non-empty, it will always send \emph{some} packet across an edge.
\end{itemize}

For example, a standard Bernoulli ring meets all the requirements
listed above.  

Given a queueing system with all the features described above, it's
fairly easy to view it as a Markov chain; we just have to build the
state space.  The details of the state space are determined to some
extent by the protocol.  For instance, consider node $i$ under a
priority discipline\footnote{In general queueing systems with priority
disciplines, the issue of preemption arises.  Suppose we're working on
a class $c_{0}$ packet, and a higher-priority class $c_{1}$ packet
arrives.  Do we stop working on the $c_{0}$ packet immediately and
switch to the $c_{1}$ packet, or do we complete servicing the $c_{0}$
packet and then work on the $c_{1}$ packet?  These two options are
referred to as preemptive and non-preemptive priority disciplines,
respectively.  In discrete time systems with deterministic service
times of exactly one, these two classes coincide, so we don't have to
distinguish the two.}, where certain classes of packets get priority
over other classes of packets.  We need to store the number of packets
of each class currently at node $i$, to determine who should be
serviced.  If instead we were using FIFO to determine
packet priority, then we would need to keep an ordered list of all
packet arrivals, with the earliest arrivals at the front of the list.
In order to determine when the packet being serviced (if any) is ready
to leave, we need to keep track of our state in $\sss_{c}$ for every
$c$.  We need to store all this information for every node $i$.

To determine exogenous arrivals, routing choices and service times, we
need to keep track of the state of the various hidden Markov
processes.  (This information is not per node, but for the whole
network.)  Given all this information (the arrangement of classed
packets waiting in queues, along with $\sigma_{\ssa} \in \ssa$, and,
for all $c$, $\sigma_{\ssr_{c}} \in \ssr_{c}$, and $\sigma_{\sss_c} \in
\sss_{c}$), we have a Markov process $X(t)$ on state space $\ssx$.

We can now complete our definitions and notation for the queueing network.
\begin{itemize}
\item $\ssx$ is the queueing network's state space, as defined above.
\item Let 
$Q_{c}(t)$ be the total number of class $c$ packets in the system at
time $t$.  
(Remember, all class $c$ packets will be at the same node.)
\item If we start our Markov chain in state $x$, then
the state at time $t$ will be written $X^{x}(t)$.  In general,
whenever I want to know the value of a quantity at time $t$ when the
system started in state $x$, I'll denote this by a superscripted $x$.
\end{itemize}

Also, it will be useful to refer to time continuously in addition to
viewing it in discrete steps.  So, for $t$ a non-negative real, define
$X(t)=X(\lfloor t \rfloor )$, and similarly for the other quantities.

\section{Building a Bounded Norm} \label{bounded norm section}
I'd like to construct a bounded norm for $\ssx$.  To do this, I need
to make a brief digression about countable Markov chains in discrete
time.  Suppose we have any discrete-time Markov chain $M(t)$ on a
countable state space $\ssm$.  Suppose that $M(t)$ is aperiodic,
irreducible, and ergodic.

Let us select (any) fixed state $\sigma_{renew} \in \ssm$.  Consider
every possible (finite) path through $\ssm$ that begins and ends in
$\sigma_{renew}$, but doesn't return to it at any other point.  
Let $\gamma$ be such a loop.  Let
$p(\gamma)$ be the probability that, starting in state $\sigma_{renew}$,
the Markov chain follows path $\gamma$ back to $\sigma_{renew}$.
Because the Markov chain is ergodic, the
expected return time to $\sigma_{renew}$ is finite.  Therefore,
the probability of $M(t)$ never returning to $\sigma_{renew}$
is zero.

I'm going to construct a second Markov chain $M'(t)$ on
state space $\ssm'$.  Consider a loop $\gamma=
\sigma_{renew} \sigma_{1}\cdots\sigma_{k}\sigma_{renew}$
through $\ssm$.  I'll insert states $\sigma_{renew}^{\gamma},
\sigma_{1}^{\gamma},\cdots, \sigma_{k}^{\gamma}$ to $\ssm'$.
I'll also insert state transition probabilities such 
that the probability of changing from 
$\sigma_{i}^{\gamma}$ to $\sigma_{i+1}^{\gamma}$ is 1,
and the probability of changing from $\sigma_{k}^{\gamma}$
to $\sigma_{renew}^{\gamma}$ is 1.  (I haven't specified
the transitions out of $\sigma_{renew}^{\gamma}$ yet.)

Insert these $\sigma_{i}^{\gamma}$ and edges for all loops $\gamma$
with $p(\gamma)>0$.  Next, associate all the $\sigma_{renew}^{\gamma}$
into one node, called $\sigma_{renew}'$.  (Notice that this may induce
an edge from $\sigma_{renew}'$ to itself, in case there exists a
$\gamma'=\sigma_{renew}\sigma_{renew}$ with $p(\gamma')>0$.  If so,
remove it.)  Let the probability of travelling from $\sigma_{renew}'$
to $\sigma_{1}^{\gamma}$ be $p(\gamma)$.  (This may re-introduce
an appropriately weighted edge from $\sigma_{renew}'$ to itself.)

There is a function $f:\ssm'\rightarrow \ssm$ that takes every
node in $\ssm$ to the original node that induced it.  Observe
that we can stochastically couple the two processes such 
that $f(M'(t))=M(t)$.  

Now, consider the arrival process $A(t)$.  Suppose that we are in
state $\sigma\in\ssa$.  As defined in the previous section, there are
$i$ arrivals to class $c$ iff $\sigma\in\ssa_{c}^{i}$.  That is,
the arrival process is a hidden Markov process, where the
underlying process is $A(t)$.  So, there is some function 
$h_{c}:\ssa\rightarrow\natu$ such that the number of class $c$
arrivals at time step $t$ is $h_{c}(A(t))$.  From the previous paragraph,
this is equal to $h_{c}(f(A'(t)))$.  Therefore,
we might as well assume that the underlying Markov chain is
$A'(t)$, and the ``hiding'' function for determining arrivals 
is $h_{c}\circ f()$.  We can perform the same kind of change
to the service time and routing processes.

Consider two loops in $\ssm'$:
\[\gamma=\sigma_{renew}\sigma_{1}\cdots\sigma_{k}\sigma_{renew}
\mbox{ and }
\gamma'=\sigma_{renew}\tau_{1}\cdots\tau_{k}\sigma_{renew}\]
(notice that both loops are $k+2$ steps long).
Suppose that for $i=1,...,k$, $h_{c}(\sigma_{i})=h_{c}(\tau_{i})$.
If we were looking at the arrival process, for instance, then
these two loops would generate the same packet arrivals
on the same time steps, and renew in the same amount of time.
Since they are identical from the point of view of arrivals,
we will amalgamate them into the same loop.  (More precisely,
we will remove $\gamma'$, and add the $p(\gamma')$ to
the probability of selecting the edge into the $\gamma$ loop.)
This final change will determine our state space $\ssm'$.

If we compare $\ssm$ and $\ssm'$, it doesn't really look like we've
done anything very useful to the state space.  However, suppose we're
in state $\sigma'\in\ssm'$.  Then we can (deterministically) count how
many steps it will take until we enter $\sigma_{renew}'$ for the first
time.  (If we're \emph{in} state $\sigma_{renew}'$, this number is
zero.)  This will allow us to build a bounded norm.

\begin{definition}
Suppose we have a discrete time Markov chain $M(t)$ on $\ssm$.
Construct $M'(t)$ on $\ssm'$ as above from the paths through
$\ssm$. Then we can define a function
$g_{\ssm'}:\ssm'\rightarrow\natu$ such that $g^{1}_{\ssm'}(\sigma)$ is
the number of time steps until the Markov chain (first) returns to
$\sigma_{renew}'$.  If $\sigma=\sigma_{renew}'$, then $g^{1}_{\ssm'}$
is zero.

Suppose our Markov chain has a function $h:\ssm'\rightarrow\natu$ on
it.  (For the arrival and routing process, we can take
$h=\sum_{c}h_{c}$; for the service process, $h$=$h_{c}$.)  Consider
$\sigma_{0}\in\ssm'$ where the evolution of the state space is
$\sigma_{0} \sigma_{1} \cdots \sigma_{k} \sigma_{renew}'$ (and
$\sigma_{i}\neq \sigma_{renew}'$ for $i>0$).  Define
$g^{2}_{\ssm'}:\ssm\rightarrow\natu$ such that
$g^{2}_{\ssm'}(\sigma_{0})=\sum_{i=0}^{k}h_{c}(\sigma_{i})$.

Let $g_{\ssm'}=g^{1}_{\ssm'} + g^{2}_{\ssm'}$.

Let us now define a norm on the state space $\ssx$ defined in
Section~\ref{packet routing defs section}.  Recall that a state in
$\ssx$ is determined by:
\begin{itemize}
\item The arrival process state $\sigma_{\ssa'}$.
\item The routing process states $\sigma_{\ssr_{c}'}$
for each class $c$ of packet.
\item The service process states $\sigma_{\sss_{c}'}$
for each class $c$ of packet.
\item The queue lengths $Q_{c}$ of each class $c$.
\end{itemize}
Then the norm on $\ssx$ is:
\[ |\cdot|=g_{\ssa'}(\sigma_{\ssa'}) 
+\sum_{c\in C}g_{\ssr_{c}'}(\sigma_{\ssr_{c}'}) 
+\sum_{c\in C}g_{\sss_{c}'}(\sigma_{\sss_{c}'}) 
+\sum_{c\in C}Q_{c}
\]
\end{definition}

\begin{lemma}
The function $|\cdot|$ is a bounded norm on $\ssx'$.
\end{lemma}

\proof
Consider the arrival process for a moment.
Observe that if $g_{\ssa'}(\sigma_{\ssa'}) = B$, then
the loop $\gamma$ that $\sigma_{\ssa'}$ is on
is at most of length $B$; at most $B$ packets arrive
across $C$ classes,
so there are (as an upper bound) at most $(B+1)^{CB}$
possible packet arrivals that we will see.  Since
we are amalgamating loops with identical arrival
patterns, that means that there are at most
$(B+1)^{CB}$ possible values for $\sigma_{\ssa'}$,
which is finite.  Analogous arguments hold for the other
processes.

Therefore, the function $|\cdot|$ is a bounded norm on $\ssx'$.
\EOP

\section{Fluid Limits}

The purpose of the fluid limit technique is to find a practical way of
checking Equation~\ref{eq:limit stability} for a system of
queues.  Dai~\cite{Dai} solved this problem in the case of i.i.d.\
interarrival and service times and Bernoulli routing in
Section 4 of his paper.  In this section, I'll show how to
alter a few lemmas of his paper in order to translate the result to
hidden Markov processes.  

For the reader unfamiliar with fluid limits, please consider glancing
at Appendix~\ref{dummies chapter}.  It contains a formal exposition of
all the ideas behind taking a fluid limit.  The appendix applies to
the special case of a memoryless discrete-time system, as I feel that
that better illuminates the important parts of the theorems.

There are two general points worth making about Dai's theorems before
we begin.  First of all, because of the generality of the Renewal
Reward Theorem (Theorem~\ref{delayed renewal reward theorem}), it is
possible to extend many of Dai's results to hidden Markov processes
without changing his proofs at all.  Second of all, by delaying the
fluid limit (considering it only after time fluid $t=1$), we can obviate the
need for some of the results.  (Some of the ``initial conditions'' of
the limits wear off in a finite amount of time, allowing laws of large
numbers to take over; if we simply observe the fluid model after this
second regime has begun, the mathematics is much more pleasant.)  The
idea of delaying a fluid limit to simplify it is due to
Chen~\cite{Chen}.  For our problem, the ``initial conditions'' are
much more complicated than for Dai, and the delay is probably
necessary to make the fluid limits well-defined.

Let us begin.

\begin{definition} \index{u.o.c.} \index{uniformly on compact sets}
We say that a collection of functions $\{f_{n}\}$ converges
to $f$ \emph{uniformly on compact sets} (abbreviated \emph{u.o.c.})
if
\[\sup_{0 \leq s \leq t} |f_{n}(s)-f(s)|\rightarrow 0 
\mbox{ as }n\rightarrow \infty\]
where $f$ and the $\{f_{n}\}$ are right-continuous functions on
$\real^{+}$.
\end{definition}

Next, let us define some useful functions:
\begin{definition}
Let $\tot{A}_{c}^{x}(t)$ be the total number of arrivals
at time $t$ to class $c$, from an initial state of $x$.
(This is analogous to Dai's $E_{l}^{x}(t)$.)

Let $\tot{S}_{c}^{x}(t)$ be the total number of packet 
departures from class $c$ after $t$ units of service,
from an initial state of $x$.
(This is analogous to Dai's $S_{l}^{x}(t)$.)

Let $\tot{R}_{c,d}^{x}(t)$ be the total number of packet 
departures from class $c$ to class $d$ after $t$ packets have been routed,
from an initial state of $x$.
(This is analogous to Dai's $\Phi^{k}(t)$.)

Extend these functions to non-integral $t$ by rounding down $t$.
\end{definition}

We can now convert Dai's Lemma 4.2 into a form more applicable
to our model.

\begin{theorem} 
Let $\{x_{n}\}\subseteq \ssx$ with $|x_{n}|\rightarrow \infty$
as $n\rightarrow \infty$.  Assume that
\[\frac{1}{|x_{n}|}\tot{A}_{c}^{x_{n}}(1)=\bar{A}_{c}\]
\[\frac{1}{|x_{n}|}\tot{S}_{c}^{x_{n}}(1)=\bar{S}_{c}\]
\[\frac{1}{|x_{n}|}\tot{R}_{c,d}^{x_{n}}(1)=\bar{R}_{c,d}\]
Then as $n\rightarrow \infty$, for any $t\geq 1$, almost surely
\[\frac{1}{|x_{n}|}\tot{A}_{c}^{x_{n}}(t)=\alpha_{c}(t-1) + \bar{A}_{c}\]
\[\frac{1}{|x_{n}|}\tot{S}_{c}^{x_{n}}(t)=\mu_{c}(t-1) + \bar{S}_{c}\]
\[\frac{1}{|x_{n}|}\tot{R}_{c,d}^{x_{n}}(t)=P_{cd}(t-1) + \bar{R}_{c,d}\]
\end{theorem}

\proof
Surprisingly, Dai's proof runs through unchanged.  The key observation
(using the arrival process as an example) is that 
\[\lim_{t\rightarrow \infty}\frac{A^{\sigma_{renew}}_{c}(t)}{t}
=\alpha_{k}\] by the Renewal Reward Theorem (which Dai calls ``the
strong law of large numbers for renewal processes''; see
Theorem~\ref{delayed renewal reward theorem} in this thesis).
Our hidden Markov processes undergo renewals (because, by assumption,
they're ergodic), so we can apply the theorem.  Because of our norm, we are
guaranteed that the first renewal has occurred by time $|x_{n}|$, so
we don't have to account for the initial delay from $x_{n}$.
\EOP

Recall that $Q^{x}_{c}(t)$ is the number of class $c$ packets in queue
at time $t$, if we start in state $x$.
There is one final property of a network of queues that
we need to define:
\begin{definition}
Let $T^{x}_{c}(t)$ be the cumulative amount of time that has been
lavished on class $c$ packets by time $t$.  (Note that
$T$ is non-decreasing.)
\end{definition}

Dai's first main theorem, Theorem 4.1, is transformed into the following:
\begin{theorem} \label{dai's 4.1 theorem}
\index{Theorem \ref{dai's 4.1 theorem}}
For almost all sample paths $\omega$ and any sequence of initial
states $\{x_{n}\}\subseteq\ssx$ with $|x_{n}|\rightarrow \infty$,
there is a subsequence $\{x_{n_{j}}\}$ with 
$|x_{n_{j}}|\rightarrow \infty$ such that
\[\frac{1}{|x_{n_{j}}|}\left(Q^{x_{n_{j}}}_{c}(1),T^{x_{n_{j}}}_{c}(1)\right)
\rightarrow (\flu{Q}(1), \flu{T}(1))\]
and for any $t\geq 1$, 
\[\frac{1}{|x_{n_{j}}|}\left(Q^{x_{n_{j}}}_{c}(|x_{n_{j}}|t),
T^{x_{n_{j}}}_{c}(|x_{n_{j}}|t)\right)
\rightarrow (\flu{Q}(t), \flu{T}(t)) \mbox{ u.o.c.}\]
for some functions $\flu{T}(t)$ and $\flu{Q}(t)$.
\end{theorem}

\proof
By our norm,
\[\frac{1}{|x_{n}|}Q^{x_{n}}_{c}(0)\leq 1\]
so we could use compactness (on $[0,1]$) 
to find a convergent subsequence.  However,
we want to do this at $t=1$.  Observe, however, that 
we can bound the queue length at time $|x_{n}|$ by the
sum of all the packets in the system, plus all the new arrivals
(to all classes) in those $|x_{n}|$ steps.  There may be
a certain number of packets destined to arrive based on the 
initial state of the hidden Markov arrival process, but from
the definition of our norm, that will account for at most 
$|x_{n}|$ new packets.  Other new arrivals will be injected after
a renewal.  After a renewal, we can then use the strong law of 
large numbers to tell us that
\[\lim_{n\rightarrow\infty}\frac{1}{|x_{n}|}Q^{x_{n}}_{c}(|x_{n}|)
\leq B\]
almost surely for some $B$.  We can then use compactness (on $[0,B]$)
to find a convergent subsequence.

We can use the same reasoning on $\flu{T}$.
The rest of the proof follows along Dai's lines.
\EOP

We need one final definition:
\begin{definition} \index{fluid limit}\index{fluid limit model}
\index{fluid stability}
Let a queueing discipline be fixed.  Any limit $(\flu{Q}(t), \flu{T}(t))$
from Theorem~\ref{dai's 4.1 theorem} is a \emph{fluid limit}
of the discipline.  We say that a fluid limit model of
the queueing discipline is \emph{stable} if there exists a 
constant $t_{0}>0$ that depends on $\alpha$, $\mu$ and $P$ only,
such that for any fluid limit with $\flu{Q}(1)=1$,
and any $t\geq t_{0}$, $\flu{Q}(t)=0$.
\end{definition}

We can now state the main result.

\begin{theorem}
Let a queueing discipline be fixed.  Suppose we have a network of
queues with hidden Markov processes, and the resulting Markov process
$X(t)$, as defined in Section~\ref{packet routing defs section}.  If
(every) fluid limit model of the queueing discipline is stable, then
the $X(t)$ is ergodic.
\end{theorem}

\proof
With the modifications to the original lemmas that we've
just made, Dai's proof still works.
\EOP

Finally, let's show that these hidden Markov processes
are actually a generalization of the standard queueing theory
results.

\begin{lemma}
Discrete time i.i.d.\ interarrival times are a special case of arrivals
from a hidden Markov process.
\end{lemma}

\proof
This proof is similar to the arguments in Section~\ref{bounded norm
section}.

Suppose we want to generate i.i.d.\ arrivals such that
$\Pr($interarrival time at class $i=k)=d_{k}^{i}$.  Assume that
there exists $d_{k}^{i}$ and at most 1 packet arrives per time step.
Then consider a state space consisting of an infinite number of loops,
where loop $k$ has $k$ nodes along it.  Let all these loops share
exactly 1 node in common, called $z_{0}$.  When we enter $z_{0}$, we
insert a packet.  From $z_{0}$, we select the first node in loop $k$
with probability $d_{k}^{i}$.  The existence of a mean arrival rate is
equivalent to having finite expected return times to state $z_{0}$, so the
system is ergodic.  It's pretty clear how to genearlize this process
to allow batch arrivals (i.e.\ more than one arrival per turn to the
same class).  We have a Markov chain for each class, so if we take the
Cartesian product of all these Markov chains, we get one Markov chain
which generates all the exogenous arrivals for all classes.
\EOP

\section{Future Work}
In this section, I'm going to discuss some avenues for future research
that seem promising.

\begin{itemize}
\item
In Dai and Meyn~\cite{Dai_and_Meyn}, the authors show
that if the stochastic processes involved in the
fluid limits have finite $n$th moments, then the
$(n-1)$st moment of the expected queue length is stable.
For instance, if the interarrival and service times
all have finite variance, then the expected queue length
is finite.

It should be straightforward to apply these results to networks with
hidden Markov processes, too.  The relevant property is probably the
$n$th moment of the hidden process per renewal.  For instance, for the
arrival process, this variable is the total number of arrivals per
renewal period.  Continuing the example above, if the return time to
state $\sigma_{renew}\in\ssx'$ has finite variance for the arrival,
service, and routing processes, then the expected queue length should
be finite.

\item
Consider a countable family $\mathcal{F}$ of ring networks.  Different
networks may have different (greedy) protocols, and the rings may be
of different sizes.  Suppose that there is a maximum nominal load
$r<1$ for all nodes throughout the family.  Suppose, finally, that the
interarrival times and service times are i.i.d.\ and have finite
variance, and the same bound on the variances apply to all the
networks in $\mathcal{F}$.

If we look at the fluid stability result of Dai and 
Weiss~\cite{Dai_and_Weiss} on the ring (or look carefully
at Corollary~\ref{ring stability by convexity corollary}),
we'll realize that the amount of time it takes for any
fluid limit of any ring in the family to converge to zero
can be bounded as a function of $r$, independent of the
particular ring size or protocol.

It is tempting, then, to imagine taking a disjoint union of
the state spaces of the rings in $\mathcal{F}$.  From
any starting configuration, with any limit of initial
states stretching over all the rings, the fluid limit
will still converge to zero by a fixed point in time
dependent only on $r$.  Therefore, the whole family
of rings would have a universal bound that
would translate into an $O(1)$ bound on the expected
queue length.  It would follow that for a fixed
maximum load $r$, there is a universal maximum expected
queue length $Q_{r}$ for any ring, with any greedy protocol.

There is quite a bit of work to be done to show that this works.  The
fundamental problem is that the step in Theorem~\ref{lim stab} where
we select an $L$ fails to work; we have a series of
$L_{0},L_{1},\ldots$ which may diverge to infinite.  This
difficulty seems surmountable, but additional assumptions about
the family (or a more effective use of the bounded variance)
may be necessary.

\item
If the previous suggestion holds for ring networks, it should also
work on the ``convex routing'' networks discussed in
Section~\ref{convex routing section}.

\item
We could also consider taking a fluid limit on a ring with $N=\infty$.
To make this problem well-defined, we must change the norm accordingly;
rather than use the sum of the queue lengths, it may
be more useful to use their lim sup.  In any event, one might
optimistically hope to gain knowledge about the asymptotic 
behavior of large rings by leap-frogging to the infinite.

\item
Throughout this chapter, I had to assume that all the hidden Markov
processes driving the network were ergodic.  Ergodicity applies in a
more general setting than Markov chains-- see, for instance,
Dudley~\cite{Dudley}, Section 8.4.  Is it possible
to extend the fluid models to include these more general
processes?

Because the system doesn't have Markovian renewals, it's
probably better to prove stability through the
techniques of Dai and Meyn~\cite{Dai_and_Meyn} than
of Dai~\cite{Dai}.  If we try to generalize the proof, 
the first major difference we find is
determining what kind of norm to use on the state.
(The ``state'' now includes the state of the system
at all times in the past.)

After some thought about the purpose of the norm in fluid
limits, we probably want to define a norm such that
if we wait $|\sigma|$ time steps, the ergodicity
will have kicked in; in other words, for some
fixed $\epsilon>0$, the observed arrival
rates should have begun to converge within a factor of $1\pm\epsilon$
of their
expected values.  The ergodic theorem tells us that such
a value for $|\sigma|_{\epsilon}$ exists.

We can then continue with most of the proof.  Unfortunately, we are
eventually faced with proving uniform integrability results for this
system, and it's not clear how to proceed.  Proving this result seems
to require some new ideas for fluid limits that haven't been needed
before.
\end{itemize}

%% file: analytic.tex
\chapter{Analyticity} \label{analytic chapter}

\section{Analyticity and Absolute Monotonicity}
Consider the expected queue length at one node of a 3 node standard
Bernoulli ring.  The expected queue length can be expressed as a
function of $p$, where $p$ is the probability of a packet arriving at
a node on one time step.  From the results of Chapter~\ref{exact
chapter}, we know what this function is:
\[\frac{p^{2}}{2-3p}\]
Observe that this function is analytic, rational, strictly monotonic,
and convex.  If we consider the stationary probability of being in
any fixed state, this function is also a rational function of $p$.

It's natural to ask ourselves how many of these properties hold 
for other packet routing networks.  This chapter looks at some 
analyticity and monotonicity results that can be widely applied.

We mustn't be overly confident with our conjectures, though.
Consider the two node Markov chain in Figure~\ref{2 node pathology figure}.
\begin{figure}[ht]
\centerline{\includegraphics[height=1in]{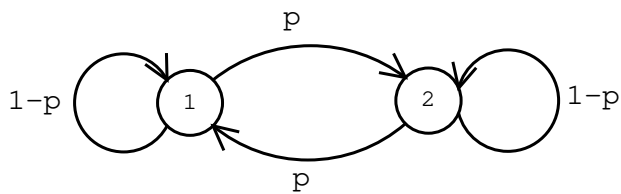} \hspace{.2in}  }
\caption{A pathological two node Markov chain}
\label{2 node pathology figure}
\end{figure}
The states are labeled 0 and 1, and we start in state 0.  We switch
states with probability $p$, and remain in the same state with
probability $1-p$.  If $0<p<1$, then the probability $\pi_{0}(p)$ of
being in state 0 is $1/2$.  But if $p=0$, then $\pi_{0}(0)=1$.
Therefore, $\pi_{0}(p)$ is discontinuous at $p=0$.  If we want to
prove general smoothness results, we'd better avoid cases like this
one.

It will be very useful to consider a strong type of monotonicity
from which we can deduce various other smoothness and monotonicity
results.

\begin{definition}[D] \label{abso mono D definition} 
\index{absolute monotonicity}
A function $f:\real\rightarrow\real^{+}$ is \emph{absolutely monotonic (D)}
in $[a,b)$ iff it has derivatives of all orders that satisfy 
\[f^{(k)}(x)\geq 0, x \in (a,b), k\in\natu \]
\end{definition}
Let $\Delta_{h}f(x)=f(x+h) - f(x)$ and 
$\Delta_{h}^{n}f(x)=\Delta_{h}(\Delta_{h}^{n-1}(f(x))$,
for $n=2,3,...$.
\begin{definition}[$\Delta$] \label{abso mono Delta definition}
A function $f:\real\rightarrow\real^{+}$ is 
\emph{absolutely monotonic ($\Delta$)} in $[a,b)$ iff
\begin{equation} \label{abso mono Delta equation}
\Delta_{h}^{n}f(x)=\sum_{k=0}^{n}(-1)^{n-k}\choose{n}{k}f(x+kh)\geq 0
\end{equation}
for all non-negative integers $n$ and for all $x$ and $h$ such that
\[a\leq x<x+h<\cdots<x+nh<b\]
\end{definition}

Absolute monotonicity is useful because of the following facts:
\begin{theorem}
Definitions~\ref{abso mono D definition} 
and~\ref{abso mono Delta definition} are equivalent.

Moreover, if $f$ is absolutely monotonic in $[a,b)$, then
$f$ is analytic on $[a,b)$.  In fact, $f$ can be analytically
continued on an open disk of radius $b-a$ centered at $a$ (so
the Taylor expansion of $f(x)$ at $x=a$ converges inside this disk).
\end{theorem}

\proof See Bernstein~\cite{Bernstein} and Widder~\cite{Widder}. \EOP

A slight change in a paper by Zazanis~\cite{zazanis} allows
us to get our first interesting conclusion.

\begin{theorem}[Zazanis] \label{zazanis theorem}
\index{Theorem \ref{zazanis theorem}}
Suppose we have a discrete time queueing system where exogenous
packets are inserted to classes according to a rate $p$ Bernoulli arrival
processes.  Suppose that the system is ergodic if $0\leq p <p_{0}$.
Then the expected queue length and the stationary
probabilities for any state are analytic functions on $0\leq p <p_{0}$.
\end{theorem}

\proof
Zazanis~\cite{zazanis} proves this result in continuous time for
Poisson arrivals.  His proof amounts to showing that a particular
function is absolutely monotonic, and hence analytic.  He begins his
Theorem 3 with his Equation (1):
\[\frac{dP_{\lambda,T}}{dP_{a,T}}=\left(\frac{\lambda}{a}\right)
                e^{-T(\lambda-a)}\]
(A derivation of this formula, which Zazanis merely quotes,
can be found in Br\'{e}maud~\cite{bremaud}, particularly pages 190-191.)
The discrete time analogue (where $p$ replaces $\lambda$) is 
\begin{equation} \label{neo-zazanis}
\frac{dP_{p,T}}{dP_{a,T}}=
\left(\frac{p(1-a)}{a(1-p)}\right)^{N_{T}}
\left(\frac{1-p}{1-a}\right)^{\lfloor T \rfloor} 
\end{equation}
The rest of the changes to Zazanis' proof follow immediately from
replacing his equation 1 with Equation~\ref{neo-zazanis} above.
\EOP

A special case is, of course, a Bernoulli ring:
\begin{corollary}
For any fixed $N$, the expected queue length per node of a standard
Bernoulli ring is an analytic function of $p$ on $p\in [0
,\frac{2}{N})$.  The stationary probability of being in any fixed state
is also an analytic function of $p$ on the same interval.
\end{corollary}

\proof
We established the ergodicity for $0\leq p <\frac{2}{N}$ in
Chapter~\ref{fluid chapter}, so we can use Theorem~\ref{zazanis
theorem}.
\EOP

\section{Light Traffic Limits} \label{light traffic section}

Note that the functions in Theorem~\ref{zazanis theorem}
are analytic at $p=0$.  This means that the
Taylor expansion around $p=0$ is well-defined and agrees with 
the actual function in some $\epsilon$-neighborhood.  
Calculations taken in the limit as $p\rightarrow 0$ are
sometimes called light traffic limits.  Theorem~\ref{zazanis theorem}
shows that for Bernoulli arrivals, the light traffic limits
are well defined.

\subsection{Product Form Results}

We can use these light traffic limits to prove that certain
stationary distributions are \emph{not} product form, answering 
a question posed in Section~\ref{exact future section}

\begin{theorem} \label{nonproduct form standard ring theorem}
\index{Theorem \ref{nonproduct form standard ring theorem}}
For $N\leq 3$, the stationary
distribution of a standard Bernoulli ring is product form.
For $N\geq 4$, the stationary distribution is 
not product form.
\end{theorem}

\proof
Chapter~\ref{fluid chapter} shows the existence of stationary
distributions for all $N$ when the nominal load $r<1$.  The results of
Chapter~\ref{exact chapter} showed that the stationary distribution is
product form for $N\leq 3$.

Assume that $N\geq 4$.  I will continue to use the state notation
of Chapter~\ref{exact chapter}, where
\[\cdots-\stq{n}{t}-\cdots\]
represents a node with $n$ packets in queue, and a hot potato
with $t$ steps left to travel (so $1\leq t \leq N-1$).  An empty
node (no packets in queue, and no hot potato packet in the ring)
is represented as
\[\cdots-\st{X}-\cdots\]
Imagine that packets travel from left to right.

Consider the state $\sigma$ where all nodes are empty except for two
adjacent nodes:
\[\cdots-\st{N-2}-\st{N-3}-\cdots\]
The stationary distribution of $\Pr[\sigma]$ can be Taylor
expanded around zero in the form:
\[a_{0}p^{0}+a_{1}p^{1}+a_{2}p^{2}+\cdots\]
Define $\sigma_{0}$ to be the ground state, where every
node is in state $\st{X}$.

In Theorem~\ref{computer theorem}, I show that in order 
to calculate $a_{i}$, we only need
to consider contributions from states that are reachable
from $\sigma_{0}$ by inserting at most $i$ packets.
Since it takes
two packet arrivals to get to $\sigma$ from the ground state,
then $a_{0}=a_{1}=0$.  Let's figure out what terms
contribute to $a_{2}$.

We can calculate the stationary probability of $\sigma$ by adding the
stationary probabilities flowing in to it.  However, since we're only
going to look at terms of order $p^{2}$ and lower, then
we need only look at states that are attainable with two or fewer packets:
\begin{eqnarray} 
\lefteqn{
\Pr\left( \cdots-\st{N-2}-\st{N-3}-\cdots \right)  
} 
& & \nonumber \\
 & \stackrel{p^{2}} = & (1-p)^{N}\Pr\left(
\cdots-\st{N-1}-\st{N-2}-\cdots \right) \nonumber \\
 & + &
p(1-p)^{N-1}\Pr\left( \cdots-\st{N-2}-\st{X}-\cdots \right) \nonumber
\\
 & + & p(1-p)^{N-1}\Pr\left( \cdots-\st{N-1}-\st{X}-\cdots \right)
\nonumber \\
 & + & \frac{1}{N-1}(1-p)^{N}\Pr\left(
\cdots-\stq{1}{N-2}-\st{X}-\cdots\right) \nonumber \\
 \label{uniform nonproduct 1 equation}
\end{eqnarray}
(By $\stackrel{p^{2}}{=}$, I mean that the $p^{2}$ and lower
terms of the Taylor expansion are equal.)
Let $\tau$ be the result of reversing $\sigma$, i.e.\
all nodes are empty except for two adjacent nodes:
\[\cdots-\st{N-3}-\st{N-2}-\cdots\]
If we consider the possible prior states reachable with two
or fewer packet arrivals, there is no analogue of the 
$\cdots-\stq{1}{N-2}-\st{X}-\cdots$ state, i.e.\
\begin{eqnarray}
\lefteqn{
\Pr\left( \cdots-\st{N-3}-\st{N-2}-\cdots \right) 
} & &
	\nonumber \\
& \stackrel{p^{2}}{=} &
    (1-p)^{N}\Pr\left( \cdots-\st{N-2}-\st{N-1}-\cdots \right) 
	\nonumber \\
& + & p(1-p)^{N-1}\Pr\left( \cdots-\st{X}-\st{N-2}-\cdots \right)
	\nonumber \\
& + & p(1-p)^{N-1}\Pr\left( \cdots-\st{X}-\st{N-1}-\cdots \right)
	\nonumber \\
	\label{uniform nonproduct 2 equation}
\end{eqnarray}
(This asymmetry is due to the fact that three packet arrivals are
necessary to get to a state like $\cdots-\stq{N-3}{N-1}-\cdots$.)

Note that the $p^{2}$ term in $\cdots-\stq{1}{N-2}-\st{X}-\cdots$
is $1/(N-1)$, which is non-zero.  Now, assume for a moment that
the distribution were product form.  Then the first three terms
on the right hand side of Equation~\ref{uniform nonproduct 1 equation}
would equal the three terms on the right hand side of 
Equation~\ref{uniform nonproduct 2 equation}.  Since
the $p^{2}$ term in $\cdots-\stq{1}{N-2}-\st{X}-\cdots$ is nonzero,
it follows that the $p^{2}$ term in $\Pr[\sigma]$ is different
from the $p^{2}$ term in $\Pr[\tau]$.  Because these are analytic
functions of $p$ at $p=0$, it follows that
$\Pr[\sigma]\neq\Pr[\tau]$, except at a finite number of points.
This fact contradicts the assumption that the distribution was
product form.
\EOP

Another interesting family of rings are the geometric rings of Coffman
et alia~\cite{Leighton95},~\cite{Leighton98}.  (The following theorem
doesn't use light traffic limits, but it makes a nice counterpoint to
Theorem~\ref{nonproduct form standard ring theorem}.)

\begin{theorem}
Fix $\lambda, \mu$, such that $0<\lambda<\mu$.  Suppose for any $N$,
we have an $N$-node ring where a packet arrives at each node with
probability $p=\lambda/N$, and departs on each step it travels with
probability $\mu/N$.  Then there are only finitely many $N$ such that
the $N$ node ring has a product form stationary distribution.
\end{theorem}

\proof
Suppose not.  Let us restrict our attention to the infinitely many $N$ with
product form stationary distributions.  Then the probability of the
ground state $\sigma$ (where all the nodes are empty) is the sum of all the
stationary probability flowing in to it.  A state $\tau$
preceding $\sigma$ has $t\leq N$ packets travelling in the ring
and no packets in queue, and becomes state $\sigma$ with
probability 
\[\left(\frac{\mu}{N}\right)^{t}(1-p)^{N}\]
Let $\Pr\st{X}$ be the marginal probability that a node is empty,
and $\Pr\st{1}$ be the marginal probability that a node has
one packet (travelling in the ring) in it.  Then by the product form,
\[\Pr[\sigma]=\left[\Pr\st{X}\right]^{N}\]
and by applying the product form to the possible previous states,
\[=\left[ \left(\Pr\st{X} + \frac{\mu}{N}\Pr\st{1}\right)(1-p)\right]^{N}\]
Taking the $N$th root and simplifying, we get
\[\Pr\st{1}=\frac{\lambda}{\mu}\frac{1}{1-p} \Pr\st{X}\]
The nominal load at any node is $r=\lambda/\mu$.
Now, by Little's theorem (Theorem~\ref{little's theorem}),
$\Pr\st{X}=1-r$.  Therefore, 
\[\Pr\st{1}=r\frac{1-r}{1-p}\]
If we take the limit of large $N$, we get
\[\lim_{N\rightarrow\infty}\Pr\st{1}= r(1-r)\]
Next, observe that the expected queue length per processor
is greater than the probability the queue is non-empty.
In the limit of large $N$, the probability of a non-empty
queue becomes
\[\lim_{N\rightarrow\infty}
1-\Pr\st{X}-\Pr\st{1}= 1-r-r(1-r)=r^{2}\] Therefore, the expected
queue length has an $\Omega(1)$ lower bound in $N$.  However,
Coffman et al.~\cite{Leighton95} shows that the expected queue length is
$o(1)$, which is a contradiction.
\EOP

\subsection{Explicit Calculations} \label{explicit subsection}

Suppose we perform a Taylor expansion in $s=(N-1)p$ at $s=0$.  Since
every packet insertion into the ring is one of $N-1$ equally likely
possibilities, then the coefficients of the Taylor expansion in $s$
are integral.  With some care, it's possible to write computer
programs to calculate these coefficients exactly, since there are no
rounding issues.  (I discuss the details in Appendix~\ref{computer
chapter}.)  I include two such calculations for the $N=4$ node
standard Bernoulli ring.

The expected queue length per node, for the first 18 coefficients,
is:
\begin{equation} \label{computed expected queue length equation}
\begin{array}{c}
9 s^{2}\\
+60 s^{3}\\
+360 s^{4}\\
+2178 s^{5}\\
+12786 s^{6}\\
+87036 s^{7}\\
+353364 s^{8}\\
+4334718 s^{9}\\
-1339320 s^{10}\\
+34239902 s^{11}\\
-2784053934 s^{12}\\
+53289152484 s^{13}\\
-706757636340 s^{14}\\
+10784818397940 s^{15}\\
-154169647942608 s^{16}\\
+2259931191910950 s^{17}\\
-32912356744493232 s^{18}
\end{array}
\end{equation}
The stationary probability of all the nodes being empty is:
\begin{equation} \label{computed stationary prob equation}
\begin{array}{c}
1 s^{0}\\
-24 s^{1}\\
+228 s^{2}\\
-1124 s^{3}\\
+3450 s^{4}\\
-8648 s^{5}\\
+18146 s^{6}\\
-57648 s^{7}\\
+1601326 s^{8}\\
-33833208 s^{9}\\
+507453786 s^{10}\\
-6464175792 s^{11}\\
+80039366294 s^{12}\\
-1052324918636 s^{13}\\
+14880952912160 s^{14}\\
-218279218629788 s^{15}\\
+3216382442758784 s^{16}\\
-47093125613982364 s^{17}\\
+686459780883843256 s^{18}
\end{array}
\end{equation}

We can deduce a few facts from these enormous polynomials.
\begin{theorem} \label{not abso mono theorem}
\index{Theorem \ref{not abso mono theorem}}
The expected queue length per node of the standard
Bernoulli ring is not always absolutely monotonic.
\end{theorem}

\proof
If $N=3$, then the expected queue length \emph{is} absolutely
monotonic.  However, if $N=4$, then observe that the $s^{10}$ term in
Equation~\ref{computed expected queue length equation} is negative,
contradicting absolute monotonicity.
\EOP

As we'll see in the next section, Markovian networks
have absolutely monotonic expected queue lengths, so
Theorem~\ref{not abso mono theorem} disproves a natural
hypothesis on standard Bernoulli rings.

Next, we analyze the rationality of these functions.
\begin{definition}
Given a rational function $a(x)/b(x)$, where
$a(x)$ is an $\alpha$ degree polynomial, and 
$b(x)$ is a $\beta$ degree polynomial, 
define the degree of $a(x)/b(x)$ as $\alpha+\beta$.
\end{definition}

\begin{theorem}
Neither Equation~\ref{computed expected queue length equation}
nor Equation~\ref{computed stationary prob equation} are 
rational functions of degree less than 18.
\end{theorem}

\proof
Suppose that we have a partial Taylor expansion of a rational
function, something like:
\[\frac{a_{0}+a_{1}x+\cdots+a_{\alpha}x^{\alpha}}
{b_{0}+b_{1}x+\cdots+b_{\beta}x^{\beta}}
=c_{0}+c_{1}x+\cdots+c_{\gamma}x^{\gamma} +O(x^{\gamma+1})\] 
(Or,
telegraphically, $a(x)/b(x)=c(x)+O(x^{\gamma+1})$.  For notational
simplicity, I interpret $a_{i}=b_{j}=0$ if $i>\alpha$ or $j>\beta$,
or if $i,j<0$.)
Now, if $\gamma$ is too small relative to $\alpha$ and $\beta$, we
have no hope of reconstructing $a(x)$ or $b(x)$; in other words, for a
fixed $\gamma$, we can only detect rationality if we assume that
$\alpha$ and $\beta$ are sufficiently small.  So, suppose that
$\gamma\geq \alpha + \beta + 1$.  Consider the $x^{\alpha+1}$
coefficient of $b(x)c(x)$.  It's
\[c_{\alpha+1}b_{0}+c_{\alpha}b_{1}+\cdots+c_{\alpha+1-\beta}b_{\beta}
=a^{\alpha+1}=0\]
We can perform a similar operation for the coefficient for
the $x^{\alpha+2}$ term, and so on, up to $x^{\alpha+\beta+1}$.
We get a resulting matrix equation:
\[ 
\underbrace{
\left(
\begin{array}{ccccc}
	c_{\alpha+1} & c_{\alpha} & c_{\alpha-1} & \cdots 
						& c_{\alpha+1-\beta}  \\
	c_{\alpha+2} & c_{\alpha+1} & c_{\alpha} & \cdots 
						& c_{\alpha+2-\beta}  \\
	\vdots       & \ddots & & & \\
	c_{\alpha+\beta+1} & \cdots &  &  
						& c_{\alpha+1}  \\
\end{array}
\right)
}_{C}
\left(
\begin{array}{c}
b_{0} \\ b_{1} \\ \vdots \\ b_{\beta}
\end{array}
\right)
=
\left(
\begin{array}{c}
0 \\ 0 \\ \vdots \\ 0
\end{array}
\right)
\]
So, given $c(x)$, we can construct the matrix $C$ for any
$\alpha+\beta<\gamma$.  If the resulting matrix
doesn't have an annihilating vector (i.e.\ is of full rank),
then $c(x)$ can not be a rational function with
numerator degree $\leq \alpha$ and denominator degree $\leq \beta$.
I checked the resulting matrices for 
Equations~\ref{computed expected queue length equation}
and~\ref{computed stationary prob equation}
for all $\alpha+\beta=17$ exhaustively\footnote{And exhaustingly.}
via computer, and found that every matrix had full rank.  

(A computational note: it is sufficient to reduce the 
matrix modulo a large prime and show that the resulting
matrix is nonsingular by modular arithmetic.)
\EOP

Finally, the curious reader may wonder what the first few places of
the Taylor expansion of the expected queue length per node looks like
as a function of $N$.  It is possible to calculate these values, and
it begins:
\[0p^{0}+0p^{1}+\frac{N-2}{2}p^{2}+O(p^{3})\]
or, in terms of the nominal load $r$,
\[\frac{1-(2/N)}{N}r^{2}+O(r^{3})\]
Observe that the coefficient to $r^{2}$ is $O(1/N)$, as one might suspect.
The
proof can be extended to the $r^{3}$ term, but even the
calculations for the $r^{2}$ term are too lengthy to include here.

\section{A Class of Absolutely Monotonic Networks}

We now turn our attention from multiclass networks to simpler
Markovian networks.  First, we need a little combinatorial result.
\begin{lemma} \label{abso mono combo lemma}
\index{Lemma \ref{abso mono combo lemma}}
If $l<n$, then
\begin{equation} \label{absolute mono equation}
\sum_{k=l}^{n}(-1)^{k} \choose{k}{l} \choose{n}{k} = 0
\end{equation}
\end{lemma}

\proof
Suppose we have $n$ distinct objects which we are allowed to paint
red, green, or blue.  We have the restriction that $l$ of the objects
must be red, and we weight each combination by $(-1)^{k}$ where $k$ of
the objects are green.  Then observe that the weighted sum of all
valid combinations of objects is exactly Equation~\ref{absolute mono
equation}.

I will prove the theorem by induction on $n$.
Observe that the theorem holds if $n=1$ (and hence $l=0$).

Assume, inductively, that the theorem holds on $n-1$.
Suppose that $l>0$.  We can sum all the weighted objects
as follows.  If the last object is red, then there must
be $l-1$ red objects among the other $n-1$ objects.  
If the last object is green or blue, there must be $l$ red 
objects among the other $n-1$ objects.  If it's
green,  though, we also must invert the weight of
the combination.  In equations,
\begin{eqnarray*}
\sum_{k=l}^{n}(-1)^{k} \choose{k}{l} \choose{n}{k} & = &
 \underbrace{\sum_{k=l-1}^{n-1}(-1)^{k} \choose{k}{l} \choose{n}{k}
}_{\mbox{last object red}} \\
& + & \underbrace{\sum_{k=l}^{n-1}(-1)^{k} \choose{k}{l} \choose{n}{k}
}_{\mbox{last object blue}} \\
& - & \underbrace{\sum_{k=l}^{n-1}(-1)^{k} \choose{k}{l} \choose{n}{k} 
}_{\mbox{last object green}}
\end{eqnarray*}
By induction,
\[=0+0-0=0\]
If $l=0$, then the last object can't be red, so the equations simplify:
\begin{eqnarray*}
\sum_{k=l}^{n}(-1)^{k} \choose{k}{l} \choose{n}{k} & = &
\underbrace{\sum_{k=l}^{n-1}(-1)^{k} \choose{k}{l} \choose{n}{k}
}_{\mbox{last object blue}} \\
& - & \underbrace{\sum_{k=l}^{n-1}(-1)^{k} \choose{k}{l} \choose{n}{k} 
}_{\mbox{last object green}}
\end{eqnarray*}
By induction,
\[=0-0=0\]
and we are done. \EOP

Next, I'm going to define a discrete version of a Taylor expansion,
and use Lemma~\ref{abso mono combo lemma} to find another
method of proving absolute monotonicity.

Let $f(x)$ be a function on $[0,P)$ that we suspect may
be absolutely monotonic.
Fix $x, h, n \geq 0$,
$n$ an integer,
such that $x+hn<P$.  Let $k=0,1,\cdots,\lfloor\frac{P-x}{h}\rfloor$.
\begin{itemize}
\item Let $f_{0}(x+kh)=f(x)$.  (So $f_{0}$ is a constant
function defined on $x, x+h, x+2h, \ldots, x+kh$.)
\item For $0<l\leq n$, let
\begin{equation} \label{f_n definition equation}
f_{l}(x+kh)=\left\{ \begin{array}{l}0 \mbox{ if } k < l \\
	        \left(f(x+lh)-\sum_{j=0}^{l-1}f_{j}(x+lh)\right)\choose{k}{l}
				\mbox{ else}
		   \end{array}
  	    \right.
\end{equation}
\end{itemize}

We can now prove the following lemma.
\begin{lemma} \label{abso mono taylor lemma}
\index{Lemma \ref{abso mono taylor lemma}}
The function $f(x)$ is absolutely monotonic iff 
$f_{n}(x+hn)\geq 0$ for all $n,x,h$ as above.
\end{lemma}

\proof
Observe that if $0\leq k \leq n$, then
\[f(x+kh)=\sum_{l=0}^{n}f_{l}(x+kh)\]
So, if we plug into Equation~\ref{abso mono Delta equation}, we get
\[
\Delta_{h}^{n}f(x)=\sum_{k=0}^{n}(-1)^{n-k}\choose{n}{k}f(x+kh)\]
\[=\sum_{k=0}^{n}(-1)^{n-k}\choose{n}{k}\sum_{l=0}^{n}f_{l}(x+kh)\]
\[=\sum_{l=0}^{n}\sum_{k=0}^{n} (-1)^{n-k}\choose{n}{k} f_{l}(x+kh)\]
\[=\sum_{l=0}^{n}\sum_{k=l}^{n} (-1)^{n-k}\choose{n}{k} 
\choose{k}{l}\left(f(x+lh)-\sum_{j=0}^{l-1}f_{j}(x+lh) \right)\]
For a fixed $l$, the term $\left(f(x+lh)-\sum_{j=0}^{l-1}f_{j}(x+lh) \right)$
is independent of $k$, so
\[=\sum_{l=0}^{n} \left(f(x+lh)-\sum_{j=0}^{l-1}f_{j}(x+lh) \right)
\sum_{k=l}^{n} (-1)^{n-k}\choose{n}{k} \choose{k}{l}\]
By Lemma~\ref{abso mono combo lemma}, 
when $l<n$, each of the terms of the second sum is equal to zero.
Since $f_{n}(x+hl)=0$ if $l<n$,
\[=f_{n}(x+hn)\]
Therefore, the problem of showing that $\Delta_{h}^{n}f(x) \geq 0$
is equivalent to showing that $f_{n}(x+hn)\geq 0$.
\EOP

We are now ready to prove our main result about Markovian networks.

\begin{theorem} \label{absolute monotonicity theorem}
\index{Theorem \ref{absolute monotonicity theorem}}
Suppose we have a (discrete time) Markovian network with $N$ nodes,
where each node has a Bernoulli arrival process of rate $p$.

Suppose that if $p<P$, then the maximum nominal load at a node is less
than one.  Then the network is stable for $p<P$, and the expected time
that a packet spends in the system is an absolutely monotonic function
of $p$, for $0 \leq p < P$.
\end{theorem}

\proof
The stability is immediate because the network is Markovian;
see the discrete time fluid limits from Chapter~\ref{fluid chapter}
of this thesis, and Section 5 of Dai~\cite{Dai}.

Let $f(p)$ be the expected delay in the system when the arrival rate
is $p$.  Define $f_{n}(p)$ as in Equation~\ref{f_n definition
equation}.

I'll use what is sometimes
called the ``method of collective marks'' (see,
e.g.\ Kleinrock~\cite{kleinrock}, Chapter 7).  We want to compare
$f(x)$ with $f(x+hl)$.  We need very fine control over our Bernoulli
arrival process.  We will get this control as follows.

Let $S_{0}=[0,x)$.  Let $S_{1}=[x,x+h)$, and generally,
$S_{i}=[x+(i-1)h,x+ih)$.  (Note that the $S_{i}$ are disjoint.)  On
each time step, at each node, we select a number $a$ from $[0,1]$
uniformly at random.  Suppose that we inject a packet if $a \in
\bigcup_{i=0}^{l}S_{i}$.  Then we have simulated a rate $x+lh$
Bernoulli arrival process.

If a packet arrives because $a \in S_{i}$, let us mark it with an $i$
(hence the name ``collective marks''.)
The packets are now members of class $i$.  
Suppose that we give priority to packets based on their mark, so that
packets with lower marks get priority over packets with higher marks.

The key observations to make are twofold.  First, since the mark $0$
packets have priority over all the other packets, they behave as
though they were travelling in a system with a rate $x$ Bernoulli
arrival processes.  Therefore, the expected delay of the mark $0$
packets is the same as the expected delay from a rate $x$ Bernoulli
process, namely $f(x)$.

Second of all, the increase in expected delays from inserting multiple
classes of marked packets is superadditive.  To clarify this point,
let me give a canonical example (with $l=2$).  Suppose that we compare
the system with arrivals when
\begin{enumerate} 
\item $a \in S_{0}$,
\item $a \in S_{0}\bigcup S_{m}$,
\item $a \in S_{0}\bigcup S_{\widehat{m}}$, or
\item $a \in S_{0}\bigcup S_{m}\bigcup S_{\widehat{m}}$,
\end{enumerate}
where $\widehat{m}<m$.  Let us call the expected delay
in the systems $D_{1}$, $D_{2}$, $D_{3}$ and $D_{4}$, respectively. 

In both cases 2 and 3, we have Bernoulli arrival processes of with the
same rate.  Therefore, $D_{2}=D_{3}$.  In particular, the increase in
expected delay from cases 1 to 2, and from cases 1 to 3 is identical.
In the fourth case, the class $m$ packets will be delayed by the class
0 packets, and \emph{additionally} delayed by the class $\hat{m}$
packets.  Therefore,
\[D_{4}\geq D_{1}+(D_{2}-D_{1})+(D_{3}-D_{1})=D_{1}+2(D_{2}-D_{1})\]

To see more formally why the delay is superadditive (i.e.\ that the
increase from $D_{1}$ to $D_{4}$ is at least
$(D_{2}-D_{1})+(D_{3}-D_{1})$), let us examine the packets' paths a
little more closely.  Whenever a packet is ejected from a node, it
selects its outgoing edge based on some distribution.  Let us fix
these decisions ahead of time, per class, rather than dynamically as
the system runs.  That is, at start up we decide that the $s$th packet
that is marked $i$ at node $n$ will take edge $e$, for all $s$, $n$
and $i$.  (If no edge is selected for some $s$, then the packet must
leave the system.)

Let us compare the system with the class $\widehat{m}$ and class $m$
packets, versus the system with only the class $m$ packets.  I claim
that the $s$th class $m$ packet ejected from node $n$ will be ejected
at the same time or later in the $\widehat{m},m$ system than in the
$m$ system.  The proof follows immediately by induction on time.
(This technique was introduced by Harchol-Balter~\cite{mor}.)
Clearly, if the system departures for the class $m$ packets occur no
sooner in the $\widehat{m},m$ system than in the $m$ system, then the
expected delay of class $m$ packets in the former system is at least
as great as in the latter system.  In other words, the delays are
superadditive.

Now, let us consider $f_{1}(x+hl)$.  The 
addition of an additional class of marked packets can
only increase the total expected delay.
Therefore, if
the arrival rate is $x+hl$, with $l>0$, then
\[f(x+kh) \geq f(x)\]
and hence,
\[\Delta_{h}f(x)=f_{1}(x+h)=f(x+kh)-f(x)\geq 0\]

Next, consider $f_{2}(x+hl)$.  Suppose we let $D_{i}$ be the expected
delay from arrivals caused by $a \in S_{0}\bigcup S_{i}$, and $D_{0}$
by arrivals caused by $a\in S_{0}$ alone, for $1\leq i \leq l$.  There
are $\choose{l}{1}=l$ such $S_{0}\bigcup S_{i}$ arrival processes.
First of all, the $D_{i}-D_{0}$ increase in delay that each of these
processes offers over the $S_{0}$ process is identical.  Second of
all, if we consider the arrival process $a \in
S_{0}\bigcup_{i=1}^{l}S_{i}$, we can simply add all the differences
(because the delays are superadditive.)  Now,
$D_{0}=f_{0}(x)=f_{0}(x+h)$ and $D_{i}-D_{0}=f_{1}(x+h)$, for any
$i$, so
\begin{equation} \label{abso mono 3 equation}
f(x+hl) \geq f_{0}(x+h)+lf_{1}(x+h)
\end{equation}

How do we show that this process continues for $\Delta_{h}^{n}$,
for arbitrarily large $n$?
Well, we know that if $0\leq k \leq n-1$, then
\begin{equation} \label{abso mono equation}
f(x+kh)=\sum_{j=0}^{n-1}f_{j}(x+kh)
\end{equation}
Suppose, in addition, that for any $0\leq k <\lfloor\frac{P-x}{h}\rfloor$,
\begin{equation} \label{abso mono 2 equation}
f(x+kh)\geq \sum_{j=0}^{n-1}f_{j}(x+kh)
\end{equation}
hence
\[f_{n}(x+kh) = f(x+kh)- \sum_{j=0}^{n-1}f_{j}(x+kh) \geq 0\]
We will prove Equation~\ref{abso mono 2 equation} by induction
on $n$.  We've already proved the base case ($n=1,2$).  Assume
it holds for all $l<n$.  

As discussed in the $f_{2}(x+2h)$ case, 
the delay in packets from $f(x+lh)$ is greater than the sum of
\begin{itemize}
\item The delay in the packets arriving because $a\in S_{0}$.
\item The increase in delay caused by the packets arriving
in $S_{0}S_{m_{1}}$, for $1\leq m_{1} \leq l$.  There are $\choose{l}{1}$
such $S_{0}S_{m_{1}}$ sets.
\item The increase in delay caused by the packets arriving
in $S_{0}S_{m_{1}}S_{m_{2}}$, for $1\leq m_{1} <m_{2} \leq l$.  
There are $\choose{l}{2}$ such $S_{0}S_{m_{1}}S_{m_{2}}$ sets.
\item ....
\item The increase in delay caused by the packets arriving
in $S_{0}\bigcup_{i=1}^{k}S_{m_{i}}$, for $1\leq m_{i} <m_{i+1} \leq l$.  
There are $\choose{l}{k}$ such sets.
\end{itemize}
where $k=0,\cdots,l$ and $l< n$.  This sum is precisely equal to
\[\sum_{j=0}^{n-1}f_{j}(x+kh)\]
So, since $f(x+kh)$ is an upper bound,
\[f(x+kh) -\sum_{j=0}^{n-1}f_{j}(x+kh) \geq 0 \]
giving Equation~\ref{abso mono 2 equation} as desired.
\EOP

If we look at the preceding proof a bit more carefully,
it is possible to show that the expected delay is
strictly increasing and strictly convex.  
We will deduce some more interesting corollaries
on expected queue lengths below.

\begin{corollary}
Suppose that we have a Markovian network with Bernoulli
arrivals of rate $p$, with nominal loads less than
one so long as $p<P$.  Then the expected number of packets
in the system is absolutely monotonic, as is the 
expected (total) number of packets in queue.
\end{corollary}

\proof
The expected delay in the system is equal to the sum of the expected
delay at each node.  By Little's Theorem (Theorem~\ref{little's
theorem}), the expected number of packets per node is the expected
delay per node multiplied by $p$.  Since every packet has
a rate $p$ Bernoulli arrival process, the expected
number of packets in the whole system is $p$ times
the expected delay.  Multiplying an absolutely
monotonic function by $p$ retains absolute monotonicity,
so we're done with the first half of the corollary.

Let $E[Q](p)$ be the expected total number of
packets in queue as a function of $p$, and $E[S](p)$ 
be the expected total number of packets in the system
as a function of $p$.

Little's Theorem also tells us that if a system is stable, then
expected queue length at a node differs from the expected number of
packets at that node by exactly the nominal load $r$.  Therefore, by
the linearity of expectation, $E[S](p)$ equals $E[Q](p)$ plus the sum
of the nominal loads.

The sum of the nominal loads are a linear multiple of $p$, say $Mp$.
Observe that $Mp$ is an analytic function.  Since absolute
monotonicity implies analyticity, then $E[S](p)$
is analytic.  Therefore, $E[Q](p)$
is the difference between two analytic functions, and hence analytic.

Consider a Taylor expansion of $E[Q](p)$ around zero.  The only
coefficient that differs from $E[S](p)$ (and hence the only
coefficient that could be negative) is the $p^{1}$ term.  If this
coefficient were negative, then for a sufficiently small
$p_{\epsilon}>0$, $E[Q](p_{\epsilon})$ would be negative.  However,
the $E[Q]$ is always non-negative (since it measures a non-negative
quantity.)  Therefore, the $p^{1}$ coefficient is non-negative, and
hence $E[Q](p)$ is absolutely monotonic.
\EOP.

Because an $N$ node ring is symmetric, is possible to translate from
expected total queue length to expected queue length per node; we
simply divide by $N$.  This fact gives us a final corollary:
\begin{corollary}
A geometric Bernoulli ring is Markovian, and hence
its expected queue length per node is
absolutely monotonic.
\end{corollary}

\section{Future Work}

Suppose we have a family of Markovian networks with Bernoulli
arrivals, $A_{i}$, for $i=0,1,...$.  Suppose $q_{i}(r)$ is the
expected number of packets in queue in network $A_{i}$ when the
nominal load is $r$ (this presupposes some notion of a system-wide
nominal load; for instance, the maximum nominal load on any node.)
From the results of the previous section, we know that $\sum_{i}
q_{i}(r)$ is absolutely monotonic for $0\leq r<1$.

Suppose, finally, that for any $r$, there exists $B_{r}$
such that $\sum_{i}q_{i}(r)<B_{r}$. Then
there exists a function $q(r)$ absolutely monotonic on
$0 \leq r < 1$ and a subsequence $i_{0}, i_{1},...$ such that
\[\lim_{j\rightarrow \infty}q_{i_{j}}(r)=q(r)\]

Now, since the $q_{i_{j}}$ are absolutely monotonic, it implies they
are monotonically increasing and convex, and hence that $q$ is also
monotonically increasing and convex.  Convexity implies continuity on
open intervals (see Rudin~\cite{rudin}, page 61), giving continuity on
$(0,1)$.  The monotonicity allows us to extend the continuity to
$[0,1)$.  By Dai's~\cite{Dai} Lemma 4.1, the $q_{i_{j}}$ converge
uniformly on compact sets, so it follows that $q$ is analytic, and the
Taylor coefficients are the limits of the coefficients of the
$q_{i_{j}}$.  Therefore, $q$ is absolutely monotonic.

Now, recall the known bounds for a standard Bernoulli ring:
\begin{itemize}
\item If $0\leq r < 1/2$, then $E[Q]$ is $O(1/N)$.
\item If $1/2\leq r < 1$, then $E[Q]$ is $O(1)$.
\end{itemize}

Suppose that the standard Bernoulli ring were absolutely monotonic.
Then the arguments above would let us conclude that the
expected queue length converges to an analytic function $E[Q]$,
which is identically zero on $0\leq r < 1/2$, and is analytic
on $1/2 \leq r < 1$.  By analytic continuation, it follows
that $E[Q]$ is zero on the whole interval $0\leq r < 1$,
i.e.\ the expected queue length per node would be $o(1)$!

Sadly, these arguments don't work.  A standard Bernoulli ring is not
Markovian, so Theorem~\ref{absolute monotonicity theorem} doesn't
apply; in fact, as we showed in Theorem~\ref{not abso mono theorem},
there exist $N$ for which the $N$ node standard Bernoulli ring is
provably not absolutely monotonic.  However, an interesting avenue of
future research would be to find some smoothness property analogous to
absolute monotonicity.  Using it, we might be able to make conclusions
about $r\geq 1/2$ based solely on analytic continuation arguments.

%% file: ringlike.tex
\chapter{Ringlike Networks} \label{ringlike chapter}

\section{Introduction}
A ring is the simplest possible network with feedback.  If we
wished to generalize results about the ring to other networks,
where should we begin?

One way of characterizing a ring is to observe that it is a regular
degree 1 directed graph where all nodes are identical.  (By identical,
I mean that there exists a graph automorphism that sends any node to
any other node.  This property allows us to calculate the expected
queue length per node simply by dividing the total expected queue
length by $N$.)  A natural first step in generalization is to increase
the degree of the graph, but maintain regularity.  I will discuss two
possibilities, the butterfly and the torus.  First, though, I will
define a more general class of networks, of which butterflies and tori
are members.

\begin{definition} \index{layered network}
A directed graph is \emph{layered} if its nodes can be partitioned
into $k$ disjoint sets $G_{1},\ldots,G_{k}$ such that any edge lies
between $G_{i}$ and $G_{i+1}$ for some $i$.  A layered network is also
called \emph{feedforward}.\index{feedforward network}

A \emph{wrapped layered}\index{wrapped layered network} network allows
edges from $G_{k}$ to $G_{1}$, too.
\end{definition}

Now, to define the two graphs of interest:

\begin{definition} \index{torus} \label{torus definition} 
 A $N_{1}\times N_{2}\times \cdots\times N_{d}$ \emph{torus}
is a directed graph consisting of $\prod_{i=1}^{d}N_{i}$ nodes.  A
node is labeled $(n_{1},\cdots,n_{d})$, where $n_{i}$ is an integer
between $0$ and $N_{i}-1$.  There is an edge from
$(n_{1},\cdots,n_{d})$ to $(m_{1},\cdots,m_{d})$ iff
there is an $i$ such that $(n_{i}+1) \bmod N_{i} = m_{i}$,
and for all $j\neq i$, $n_{j}=m_{j}$.

\label{hypercube definition}
If $N_{i}=2$ for all $i$, the torus is called a 
\emph{$d$-dimensional hypercube}.\index{hypercube}
\end{definition}
If $k$ divides $N_{i}$ for all $i$, then the torus
can be written as a wrapped layered network with $i$
layers.

Another popular network for packet routing is the butterfly graph.
 \begin{definition} \index{butterfly}
 A \emph{(standard) $d$-dimensional butterfly} 
 is a directed, layered graph defined as follows: 
 nodes fall into one of $d+1$ disjoint layers, numbered 0 through $d$. Each
 layer consists of $N=2^d$ nodes, which we label with the $N$ binary
 strings of length $d$.  (So, a node is specified by a binary string and a
 layer number.)
 Consider any length $d$ binary string, say $b=b_{1}b_{2}\cdots b_{d}$.
 For each $i$ such that $0\leq i <d$,
 there is a directed edge from node $b$ of layer $i$ to node 
 $b_{1}b_{2}\cdots b_{i-1}0 b_{i+1}\cdots b_{d}$ of layer $i+1$, and 
 another directed edge from node $b$ of layer $i$ to node 
 $b_{1}b_{2}\cdots b_{i-1}1 b_{i+1}\cdots b_{d}$ of layer $i+1$.
 The nodes on layer 0 are called the input nodes, and the nodes on layer 
 $d$ are the output nodes.

 A \emph{wrapped butterfly}\index{wrapped butterfly}
 is a directed graph where the 
 nodes on the last layer are associated with the nodes on the 
 first layer.
 \end{definition}
See Figure~\ref{basic butterfly figure} for a drawing of
a three-dimensional butterfly graph.
\begin{figure}[ht]
\centerline{includegraphics[height=2in]{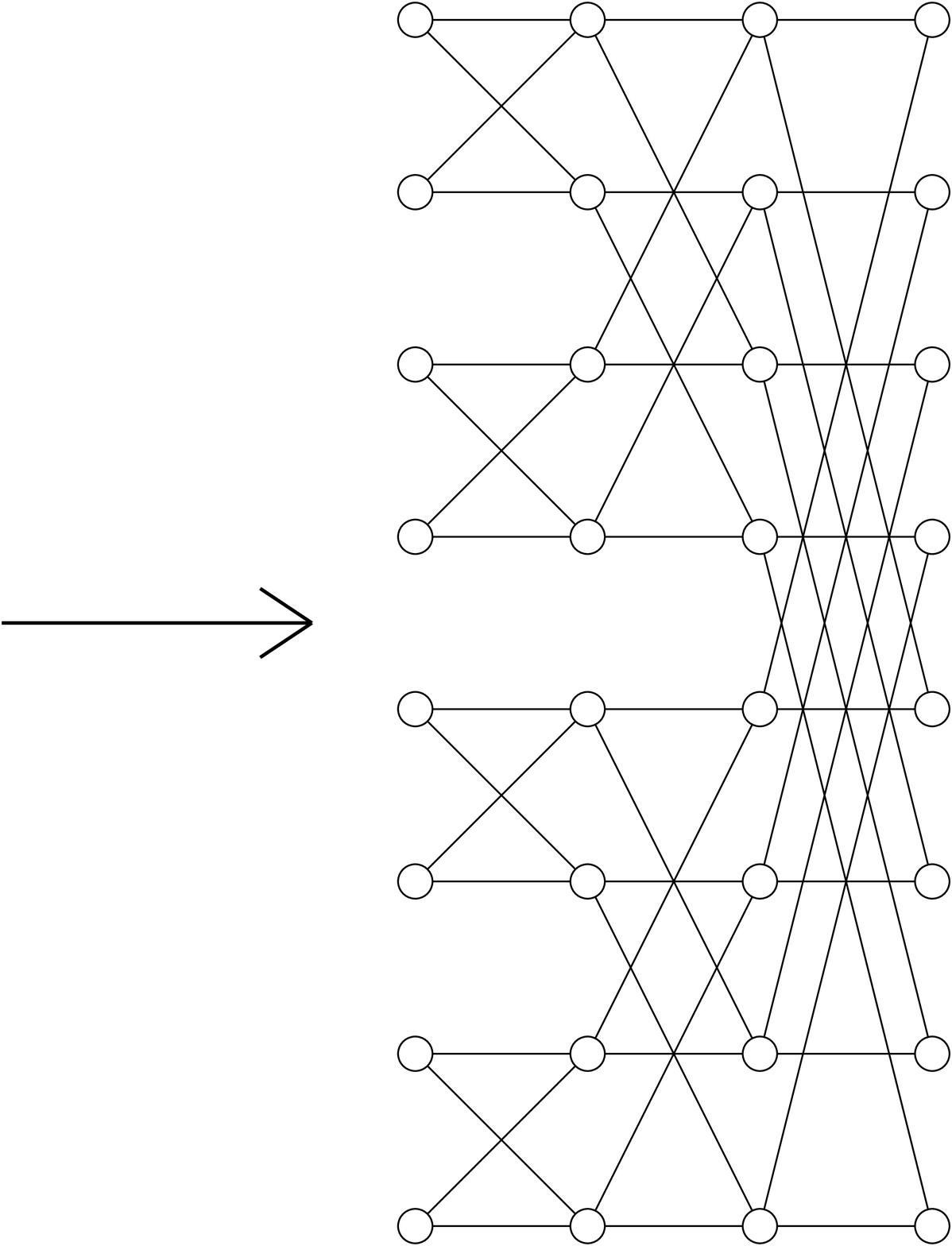} \hspace{.2in}  }
\caption{A 3 dimensional butterfly graph }
\label{basic butterfly figure}
\end{figure}
Note that tori and wrapped butterflies are both regular layered 
graphs where every node is identical.

It will be useful to keep these examples in mind during the next 
section.


\section{Convex Routing} \label{convex routing section}

\begin{definition}
Consider a node $n_{0}$ in a network.  Let $n_{1},...,n_{m}$
be the $m$ nodes with directed edges into $n_{0}$.  
Let $p_{i}$ be the probability that a packet 
travels from node $n_{i}$ to $n_{0}$, and
suppose that the probability is independent of
the class of the packet.

Suppose that 
\begin{equation} \label{convex routing equation}
\sum_{i=1}^{m}p_{i} \leq 1
\end{equation}
If this equation holds for all nodes
$n_{0}$, then we say that the network\footnote{Is
convex routing a property of the network or of the
protocol?  Although some may take issue with me, I 
view the selection of edges as a function of the
packet class, determined by the network.  The protocol,
on the other hand, selects which packet gets ejected,
not where it goes.\label{protocol footnote}}  has the
property of \emph{convex routing}.\index{convex routing}
\end{definition}

For example, if we are on a regular graph, and a packet
chooses its next edge uniformly at random, the network
has convex routing.

\begin{theorem} \label{convex routing theorem}
\index{Theorem \ref{convex routing theorem}}
Suppose we have a generalized Kelly network with convex routing.
Suppose further that we use any greedy protocol, and the network has
any topology.  If the nominal loads are less than one, then the
network is stable.

If all the interarrival and service times have finite variance, then
the expected queue length is finite, too.
\end{theorem}

\proof
I'll prove this by using the (delayed) fluid limit technique.

Let $n$ be a node.
I'm going to define a potential function $\phi(n)$ on the fluid model
to be the analogue of the expected
congestion at node $n$ (i.e.\ the expected number of times
that packets now in the system will use node $n$).

To make this precise in the fluid regime, consider a fluid class $c$
and a node $n$.  Suppose that we have a unit of class $c$ fluid,
and suppose that class $c$ resides at node $n_{c}$.
Suppose we have a path $\gamma$ through the network (not necessarily
node disjoint), beginning at class $c$'s node and ending at node $n$.
Then some fraction $f_{\gamma}$ will pass through node $n$ along
path $\gamma$.  (Note that $f_{\gamma}$ is independent of
the class by the convexity of the routing.)
Since we have an open queueing network, all packets almost surely
leave the system.  In the fluid domain, this means that,
summing over all paths $\gamma$ from $n_{c}$ to $n$,
\[\sum_{\gamma}f_{\gamma}<\infty\]

Suppose that there is $q_{c}$ quantity of fluid of class $c$.  
Then, since there are a finite number of classes $c$, we can let
\[\phi(n)=\sum_{c} q_{c} \sum_{\gamma}f_{\gamma}<\infty\]

Notice that $\phi(n)$ is the total amount of fluid that would pass
through node $n$ if no new fluid arrived in the system.  Let
$\Phi=\max_{n}\phi(n)$.

Let $q(n)$ be the number of packets in queue at node $n$
(from all classes resident at $n$).
Let $a_{n}$ be the nominal arrival rate at node $n$. 
Let $s_{n}$ be the nominal service rate at node $n$.  
(Note that $a_{n}<s_{n}$, since the nominal loads are less than one.)
Let
\[\epsilon=\min_{n}(s_{n}-a_{n})>0\]

Observe that $\phi(n)$ is a Lipschitz function.
To see this, note that fluid increases at most
at a rate $a_{n}$, and decreases at most at a rate $s_{n}$.  It
follows that $\Phi$ is Lipschitz.

Now, Lipschitz functions are absolutely continuous, and hence
continuous and differentiable almost everywhere.  (See, e.g.,
Rudin~\cite{rudin}).  If we can show that $\Phi>0$ implies that
$\frac{d}{dt}\Phi\leq-\epsilon$ a.e., then it implies fluid stability
(because all fluid will empty from the system by time $1/\epsilon$),
and we will be done.

Observe that if $q(n)>0$, then 
\[\frac{d}{dt}\phi(n) \leq -\epsilon\]
almost surely.
Therefore, if the maximum value of $\phi(n)$ is attained
at a node with $q(n)>0$, then it follows almost surely that
$\frac{d}{dt}\Phi \leq -\epsilon$.

It suffices, therefore, to show that so long
as there exists an $n$ with $q(n)>0$, then
\begin{equation} \label{nonempty queue equation}
\max_{n}\phi(n)=\max_{n, q(n)>0}\phi(n)
\end{equation}
and we will have proved fluid stability.

Assume that $\Phi>0$.
If $q(n)>0$ for all $n$, then Equation~\ref{nonempty queue equation}
holds.  Assume, then, that there exists a node $n_{0}$
such that $q(n_{0})=0$, but $\phi(n_{0})>0$.

I am going to construct a tree of all the possible paths $\gamma$
from node $n_{1}$ with $q_{n_{1}}>0$ to node $n$, where every 
intermediate node $n_{2}$ on the path has $q_{n_{2}}=0$.
Since $q(n_{0})=0$ but $\phi(n)>0$, then there must exist
some such path, so the tree has more than one node.

It's certainly possible that by the third level of the tree, a
node from the network may show up in more than one place in the tree,
because there may be multiple paths from the node to $n$.
We treat these as formally distinct nodes.  (For instance, if we have
a diamond shape, as in Figure~\ref{diamond figure}, node $n_{3}$ from
the network will split into two different nodes in the tree.)
\begin{figure}[ht]
\centerline{\includegraphics[height=1.5in]{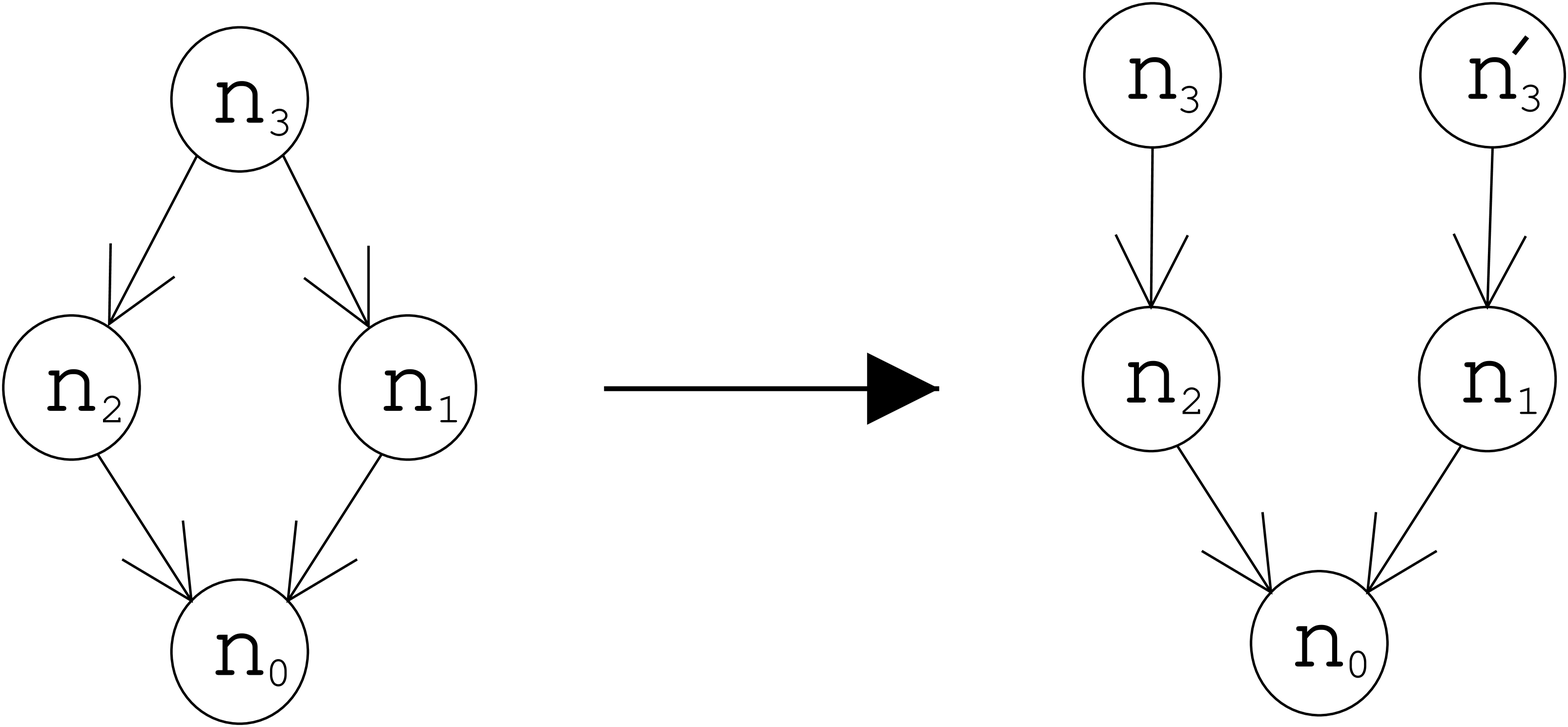} \hspace{.2in}  }
\caption{Converting the network into a tree of paths }
\label{diamond figure}
\end{figure}

The terminal nodes in the tree correspond to nodes in the network
with non-zero queues.  I will label the terminal nodes to represent
the amount of traffic that will follow a given path.  More precisely,
consider a node $n_{\gamma}$ in the tree, corresponding to 
node $n_{1}$ in the network, and the path $\gamma$ from
$n_{1}$ to $n_{0}$.
Let $\Gamma_{c}$ be the class of paths that start at the node where
class $c$ packets are located and end with $\gamma$.  Then
define
\[\psi(n_{\gamma})=\frac{1}{f_{\gamma}}
\sum_{c}q_{c}\sum_{\delta \in \Gamma_{c}} f_{\delta}\]
If $n_{\gamma}$ is a terminal node in the tree,
I will label it with $\psi(n_{\gamma})$.

Now, let $T$ be the set of terminal nodes in the tree.  
Observe that
\begin{equation} \label{convex combo equation}
\phi(n)=\sum_{\gamma\in T}f_{\gamma}\psi(n_{\gamma})
\end{equation}
Consider taking a random walk along the tree away from the
root node.  
Given the edges $e_{1}, \ldots, e_{m}$ leading
to node $n_{\gamma}$, let the probability of crossing
edge $e_{i}$ in our random walk be the probability
of crossing the edge $e_{i}$ into node $n_{\gamma}$.
Because the routing is convex, the sum of the
probabilities $p_{\gamma}$ is 
\[p_{\gamma}\leq 1\]
If $p_{\gamma}<1$, then with probability $1-p_{\gamma}$, we 
stop the walk in node $n_{\gamma}$.  Observe that
for any terminal node $n_{\gamma}$, 
the probability of stopping at node $n_{\gamma}$
is $f_{\gamma}$.  Since we have a distribution,
\begin{equation} \label{Kraft-like equation}
\sum_{\gamma : n_{\gamma}\in T}f_{\gamma}\leq 1
\end{equation}
(This inequality can also be proved from the Kraft
inequality of data compression theory.)

Equations~\ref{Kraft-like equation} and~\ref{convex combo equation}
combine to tell us that $\phi(n)$ is bounded by
a convex combination of the terminal nodes.

Note that if we have a convex combination of non-negative
reals $r_{i}$ that are all less than some bound $B$, then
there exists an $i$ such that the convex combination is
less than or equal to $r_{i}$.  Using the total work
in the system as a bound on $\psi()$, we can conclude that
\begin{equation} \label{convex bound equation}
\phi(n_{0}) \leq \psi(n_{\gamma})
\end{equation}
for some particular node $n_{\gamma}$.

Finally, observe that if node $n_{1}$ in the
network corresponds to terminal node $n_{\gamma}$ in the
tree, then (at least) $\psi(n_{\gamma})$ packets currently
in the system need to cross $n_{1}$.  Therefore,
\begin{equation} \label{psi as a bound equation}
\psi(n_{\gamma})\leq \phi(n_{1})
\end{equation}
Combining Equations~\ref{convex bound equation}
and~\ref{psi as a bound equation}, we get
\[\phi(n)\leq \phi(n_{1})\]
Note that
node $n_{1}$ has $q(n_{1})>0$ (because
$n_{\gamma}$ is a terminal node of the tree),
so we have established fluid stability.

Fluid stability, plus the finite variance of arrivals and service
times, implies finiteness of expected queue length.  (For details, see
Dai and Meyn~\cite{Dai_and_Meyn}).
\EOP

\begin{corollary} \label{ring stability by convexity corollary}
\index{Corollary \ref{ring stability by convexity corollary}}
Any ring network uses convex routing, and thus
is universally stable.
\end{corollary}

\note The fluid stability of the ring was first proved by Dai and 
Weiss~\cite{Dai_and_Weiss}.

These results on convex routing have some fairly natural applications
to load balancing.  Suppose we have a $d$-dimensional wrapped
butterfly where each nodes is a processor, performing some
computations.  Occasionally, a node will decide that it has too much
work, and will insert a packet into the system, representing one
quantum of work.  The processor would like to share its work fairly
uniformly across the other processors.  (For the moment, I won't worry
about aggregating the completed work of the system.)

Sharing the load can be accomplished fairly easily on a wrapped
butterfly.  At every node, there are two outgoing edges; if a packet
selects each edge with probability $\frac{1}{2}$, then in $d$ (or
more) steps, its probability of being at any point in its current
layer is uniform.  Thus, we have a multi-class convex routing
problem, and we can use Theorem~\ref{convex routing theorem}
to deduce stability.

There is an even more efficient method of sending the load through the
network.  Since there are two outgoing edges at every node, we can
send out up to two packets per time step.  Suppose that 
we select the particular
edge at random.  Then we maintain convex routing (and hence stability),
while still guaranteeing uniform distribution over the final layer
in $d$ time steps.

Generally speaking, if we have a $d$-regular graph, then by selecting
each of the outgoing edges with equal probability, we have
convex routing.  We can also send out $d$ packets instead of 
1 packet.  Most interestingly, since at most $d$ packets arrive, we can
give them precedence over the packets in queue, i.e.\ use the
Greedy Hot Potato algorithm.  This choice opens up the possibility
of using the techniques from Chapter~\ref{bounds chapter} to
get bounds on the expected queue length per node.

\section{Superconcentration on a Pair of Butterflies}


The remainder of this chapter will examine some problems in
node-disjoint circuit switching.  Unlike the stochastic results of the
rest of this thesis, these results are more graph-theoretical and
structural in flavor.  The motivating problem can be described as
follows.  Suppose we have a directed graph with $N$ input and $N$
output nodes, both labelled from 1 to $N$.  For each input node $v$,
we choose an output node $\pi(v)$ to be its destination, for some
permutation $\pi$.  The problem is to find a collection of $N$
node-disjoint paths which each run from $v$ to $\pi(v)$ for all $v$.
A directed graph that can route all permutations $\pi$ is called
\emph{rearrangeable}.\index{rearrangeable network}
  (For some real-world applications of
node-disjoint routing, see, for example,~\cite{optical_networks}.)

A classic example of rearrangeability is the Bene\v{s} 
network\index{Bene\v{s} network}
 (see~\cite{leighton}).  
This network (i.e.\ directed graph) consists of a ``forward'' butterfly 
adjoined to a 
``reversed'' butterfly.  A natural question to ask is: if we attach 
two ``forward'' butterflies, is this network (the 
double butterfly) still rearrangeable?
This problem has been open for several decades.  At least one proof is 
currently under review~\cite{cam}.  This suggests a more general
hypothesis.  Suppose that we have two graphs, each isomorphic to a
butterfly, but not necessarily identical to each other.  If we attach
the output nodes of the first to the input nodes of the second, is 
the resulting graph rearrangeable?

At the current time, proving this kind of result seems far too much to
hope for.  So, rather than show that these types of networks are rearrangeable,
I will prove various concentration and superconcentration results.
\begin{definition} \index{superconcentrator}
\index{concentrator}
Consider a directed graph $G$.  Fix $n$ input nodes and $n$ output nodes.
Suppose that between any $k$ input and $k$ output nodes there exist 
$k$ node-disjoint paths.
(By ``node-disjoint'', I mean that a path intersects
neither itself nor any other path.)  Then we say that $G$ is a
$k$-concentrator.  If $G$ is a $k$-concentrator
for all $k\leq n$, then we call $G$ a \emph{superconcentrator}.
\end{definition}
(Observe that every selected input and output node occurs on exactly one 
path.  Note also that node-disjointness implies edge-disjointness of the paths.
See \cite{pippenger} or \cite{hwang} for more on superconcentrators.)

Clearly, rearrangeability implies subset routing-- just choose a permutation
that respects $v\in A$ iff $\pi(v)\in B$.  However, the converse is not true
for arbitrary networks (see Figure~\ref{subset counterexample}).
\begin{figure}[ht]
\centerline{\includegraphics[height=1in]{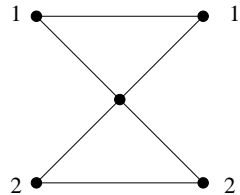} \hspace{.2in}  }
\caption{A non-rearrangeable superconcentrator (consider 
$1\rightarrow 2$, $2\rightarrow 1$) }
\label{subset counterexample}
\end{figure}
It's straightforward to show that a single butterfly does not route all 
subsets, so we need to use at least two butterflies to get interesting 
concentration results.

In the next several sections, I show that any concatenated pair of
$d$-dimensional butterflies (not necessarily identical to each other)
are $2^{k}$-con\-cen\-trators, for any $k\leq d$.  I can strengthen this
statement in a special case: if the butterflies are standard
butterflies with their layers shuffled (e.g.\ a Bene\v{s} network, or
the double buttefly in~\cite{cam}), the network is a
superconcentrator.

	The rest of this chapter is structured as follows.  Section 1
establishes some definitions and fixes notation.
Section 2 examines the structure of a graph related to a pair of
butterflies that highlights some of its connectivity properties.
Section 3 solves the problem in the case where $|A|=2^{m}$ for some
$m$, and proves a rearrangeability-type result when $|A|\leq ^{\lfloor
d/2 \rfloor}$ on certain networks.  Section 4 presents the main
result, except for one lemma that I postpone for section 5.  Section 6
is dedicated to closing remarks.

\section{Definitions and Notation}

Let us begin by defining and fixing notation for a butterfly.
A standard $d$-dimensional butterfly can be viewed as a network with 
$2^{d}$ nodes where we switch the first bit in the first layer
of edges, the second bit in the second layer of edges, and so forth.
If we choose to switch the bits in a different order, we get a
\emph{layer-permuted butterfly}.
\begin{definition}
A \emph{$d$-dimensional layer-permuted butterfly} 
is a directed, layered graph defined as follows: 
nodes fall into one of $d+1$ disjoint layers, numbered 0 through $d$. Each
layer consists of $N=2^d$ nodes, which we label with the $N$ binary
strings of length $d$.  Take some (fixed) permutation $\pi$ on $d$
objects.  Consider any such binary string, say $b=b_{1}b_{2}\cdots b_{d}$.
For each $i$ such that $0\leq i <d$,
there is a directed edge from node $b$ of layer $i$ to node 
$b_{1}b_{2}\cdots b_{\pi(i)-1}0 b_{\pi(i)+1}\cdots b_{d}$ of layer $i+1$, and 
another directed edge from node $b$ of layer $i$ to node 
$b_{1}b_{2}\cdots b_{\pi(i)-1}1 b_{\pi(i)+1}\cdots b_{d}$ of layer $i+1$.
\end{definition}
(Note that these butterflies are all graph-isomorphic to each other.)
Finally, the networks we'll be looking at consist of pairs of these
butterflies.
\begin{definition}
Suppose that we have two graphs, $G_{1}$ and $G_{2}$.  Suppose that 
$G_{1}$ has $n$ output nodes and $G_{2}$ has $n$ input nodes, each
numbered from 1 to $n$.  Then
we say that $G$ is the \emph{concatenation} of $G_{1}$ and $G_{2}$ if 
we form $G$ by associating the output node $i$ of $G_{1}$ with 
the input node $i$ of $G_{2}$.  
 
If $G_{1}$ and $G_{2}$ are each isomorphic to a standard butterfly
(but not necessarily identical to each other), we call $G$ 
a \emph{pair of butterflies}.  Similarly, if $G_{1}$ and $G_{2}$ are
layer-permuted butterflies, we have a 
\emph{pair of layer-permuted butterflies}.  Finally, if $G_{1}$ and 
$G_{2}$ are both standard butterflies, we have a \emph{double butterfly}.
\end{definition}
Note that these graphs have $2d+1$ layers of
nodes (0 through $2d$).  Since I imagine the paths from inputs to outputs
to be running from 
left to right, I will refer to the butterfly on layers 0 through $d$ as
the \emph{left butterfly}, and the one on layers $d$ through $2d$ as the
\emph{right butterfly}.  Note also an alternate way of specifying
a pair of butterflies: consider a network consisting of
two standard butterflies, but permute the labels of the output nodes
of the left butterfly.  
Observe that these two definitions give rise to the same class of 
graphs (up to isomorphism).

Over the course of this chapter, I construct directed node-disjoint 
paths from input nodes to output nodes.  So, for example, a path from an 
input
node to an output node on a pair of butterflies is exactly $2d+1$ nodes
long-- the path can't double back on the layers.  

Suppose I have a set of input nodes $A$ and a set of output nodes $B$ of the
same size (i.e.\ $|A|=|B|$)
Then if I specify a collection of node-disjoint paths from $A$ to $B$, observe
that I can extend these paths into a consistent setting of all the 
switches in the network.  These switches will induce $N$  node disjoint
paths from \emph{every} input to every output node, and retain the
feature that a path begins in $A$ iff it ends in $B$.  So, on a
switching network, node-disjoint routing of a subset implies there
exists a node-disjoint routing of a permutation $\pi$ such that 
$v\in A$ iff $\pi(v)\in B$.  Since this is an ``if and only if'' statement,
we get the following lemma:
\begin{lemma} \label{set complement}
\index{Lemma \ref{set complement}}
If we can find node-disjoint paths from $A$ to $B$ on a switching
network, then we can find node-disjoint paths from the complements
$A^{c}$ to $B^{c}$.
\end{lemma}
Throughout this chapter, I will use $A$ to represent a collection of
input nodes, $B$ a collection of output nodes, and assume that
$|A|=|B|$.

\section{The Sub-Butterfly Connectivity Graph} \label{connectivity}

Suppose we specify a path of length $m$ on a standard butterfly 
(for $m\leq d$) from an input node.
By choosing which edge to take, the path changes $m$ bits of its
location any way we want.  Suppose we select the first bits to be
$b=b_{1}b_{2}b_{3}\cdots b_{m}$.  Then from layers $m+1$ to $d$, 
the first $m$ bits will remain equal to $b$.  Let's specify 
the resulting sub-graph of the butterfly in the following definition:

\begin{definition}
Consider a $d$ dimensional standard butterfly.  Take an $m$ bit
binary string $b=b_{1}b_{2}b_{3}\cdots b_{m}$  ($m\leq d$).
Consider the sub-graph formed by the nodes on layers $m$ through $d$
(inclusive) whose first $m$ bits are $b$.  Observe that this graph
is (isomorphic to) a $(d-m)$-dimensional butterfly.  Let us call it 
the \emph{sub-butterfly $b*$}.

If we specify a suffix instead and consider layers 0 through $d-m$,
we get the \emph{sub-butterfly $*b$}.

If we have a graph isomorphic to a standard butterfly, the isomorphism
will induce (isomorphic) images of the sub-butterfly, so we can
meaningfully refer to sub-butterflies on any butterfly-isomorphic graph.
\end{definition}
I will be considering sub-butterflies in a pair of butterflies.  In
this context, $b*$ is the sub-butterfly residing on layers $m$
through $d$ (and stopping there), i.e.\ only in the left butterfly.
I'll also be interested in sub-butterflies on the right side.  
These inhabit layers $d$ through $2d-m$.

It will be useful to investigate the structure of the connections between
the $q$-dimensional sub-butterflies on the right and left sides of
a $d$-dimensional pair of butterflies, that is, the sub-butterflies of the
 form $x*$ or $*x$ where $x$ is a binary string of length $m$ (such that
$m+q=d$).  Note that these sub-butterflies inhabit layers $m=d-q$ through
 $d$, and $d$ through $d+q$.  Let us represent each 
sub-butterfly by a vertex in a bipartite graph; the vertex is on the left
side of the bipartite graph iff the sub-butterfly is on the left
side of the pair of butterflies.  I will label each vertex by its associated
sub-butterfly, abusing the label notation somewhat.  
Place an edge between two vertices $x*$ and $*y$ iff
the two sub-butterflies are connected, that is, iff $x*$ and $*y$
(as sub-butterflies) share at least
one common node on layer $d$ of the pair of butterflies.
Equivalently, there is an edge between the nodes in the bipartite graph
iff there exists a path from every layer $d-q$ input node 
of $x*$ to every layer $d+q$ output node of $*y$. 
I will refer to this graph as the \emph{$q$-dimensional 
sub-butterfly connectivity graph},
or just the \emph{connectivity graph}.
Observe that there are $2^{m}$ vertices
on either side of this graph.  How are the vertices connected?

I will consider progressively more specialized cases in order to derive 
various results in later sections.  Suppose, first, that we build a bipartite
connectivity graph, but if there are $x$ common nodes on layer $d$ 
between a sub-butterfly on the left and one on the right, 
we insert $x$ edges (instead of only 1 edge).  Let us call this the 
\emph{enriched connectivity graph}.  

\begin{lemma} \label{enriched graph is regular}
\index{Lemma \ref{enriched graph is regular}}
For any pair of butterflies, its enriched connectivity graph is regular.
\end{lemma}

\begin{proof} Since each sub-butterfly has $2^{q}$ output nodes, then all 
nodes in the enriched connectivity graph have degree $2^{q}$.
\end{proof}

Now we move our attention to the special case of layer-permuted butterflies.
First, let us analyze the structure of one connected compnent of
the $q$-dimensional connectivity graph.

\begin{lemma} \label{connected components}
\index{Lemma \ref{connected components}}
Each connected component in the connectivity graph of a layer-permuted
butterfly is a completely connected bipartite graph.
\end{lemma}

\begin{proof}
Suppose that the layer-permuted butterfly on the left has permutation
$\pi$, and the butterfly on the right has permutation $\sigma$.
Consider a sub-butterfly $b*$ in the left butterfly.  This corresponds to 
a sub-graph on layers $m$ through $d$ where the value of bit $\pi(i)$ 
is $b_{i}$.  Notice that a sub-butterfly $b*$ in the left butterfly 
connects to 
a  sub-butterfly $*c$ in the right butterfly if and only if
\begin{equation} \label{sub-butterfly connection}
\forall i < q, \forall j > m, \mbox{ if }\pi(i)=\sigma(j) 
\mbox{ then } b_{i} = c_{j} 
\end{equation}
Thus, each connected 
component is a complete bipartite graph (with the same
number of nodes on each side.)
\end{proof}

Next, suppose that we have a pair of layer-permuted butterflies.  
How does the graph change as we specify one more
layer?  That is, if we compare the connectivity graphs between $q$ 
and $q-1$ dimensional sub-butterflies, what happens?

Therefore, determining the structure of the connectivity graph on pairs
of layer-permuted butterflies reduces to determining the connected components.
Consider one connected component in the connectivity graph looking at
$q$-dimensional sub-butterflies.  When we advance to the 
$(q-1)$ dimensional sub-butterflies, each node becomes two nodes (because each
$q$ dimensional sub-butterfly splits into two $q-1$ dimensional 
sub-butterflies).  There are  essentially three cases that can occur.
\begin{itemize} \label{three cases}
\item \textbf{(No reused dimensions)}
Suppose that 
$\sigma(q-1) \neq \pi(j)$ for any $1 \leq j \leq m+1$ and
$\pi(m+1) \neq \sigma(j)$ for any $q-1 \leq j \leq d$.
Then if the $q$-dimensional sub-butterfly $b*$ is adjacent
to $*c$, it follows that $bb_{m_1}*$ is adjacent to 
$*c_{m+1}c$ for $b_{m-1}, c_{m-1} = 0,1$.  

In the connectivity graph, that means that the connected component doubles
the number of nodes, but remains completely connected.
\item \textbf{(One reused dimension)}
Suppose that there exists (exactly) one $i$ such that either 
\begin{itemize}
\item
$\sigma(q-1) = i = \pi(m+1)$, or
\item
$\sigma(q-1) = i = \pi(j)$ for some $1 \leq j \leq m+1$ and
$\pi(m+1) \neq \sigma(k)$ for any $q-1 \leq k \leq d$, or
\item
$\sigma(q-1) \neq \pi(j)$ for any $1 \leq j \leq m+1$ and
$\pi(m+1) = i = \sigma(k)$ for some $q-1 \leq k \leq d$
\end{itemize}
Then the connected component splits into two connected components,
based on the value of the $i$th bit.
\item  \textbf{(Two reused dimensions)}
Suppose that $\sigma(d-m-1) = i = \pi(j)$ for $1 \leq j \leq m+1$ and
$\pi(m+1) = l = \sigma(k)$ for $d-m-1 \leq k \leq d$, and $i \neq l$.
Then the connected component splits into four connected components, based
on the four possible values that the $i$ and $l$ bits can take.
\end{itemize}

\section{Subsets of Size $2^m$}
We want to select a collection of node-disjoint paths from input set $A$
to output set $B$ on a pair of butterflies.
Although I've expressed this problem in terms of paths, it's often easier
to express the proof in terms of packets travelling through the network.
In particular, if packets travel forward (node disjointly, and
without stopping) from every input node in $A$,
and backwards from every output node in $B$, and we can match up the 
packets on level $d$, then the paths traced by the packets give us the
collection of paths we're looking for.  I will switch between the path
and packet descriptions of the problem whenever it seems helpful.
\begin{lemma} \label{simple splitting}
\index{Lemma \ref{simple splitting}}
Suppose we have a set $A$ of input nodes on a butterfly.
By passing from layer 0 to layer 1 of a butterfly, there exist
paths that send
$\lceil |A|/2 \rceil$ of the packets to sub-butterfly $0*$, and 
$\lfloor |A|/2 \rfloor$ of the packets to sub-butterfly $1*$.
Similarly, we could send 
$\lceil |A|/2 \rceil$ of the packets to sub-butterfly $1*$, and 
$\lfloor |A|/2 \rfloor$ of the packets to sub-butterfly $0*$.
Mutatis mutandi, this applies to packets in output nodes
travelling backwards, by passing from layer $2d$ to $2d-1$.
\end{lemma}
\begin{proof}
The $N$ nodes on the first layer of the butterfly can be grouped into 
$N/2$ switches, where the nodes labelled $T_{0}=0t_{2}t_{3}\cdots t_{d}$
and $T_{1}=1t_{2}t_{3}\cdots t_{d}$ form one switch.  Observe that 
each switch can be set straight or crossed, that is, 
we have to send $T_{i}$ to $T_{i}$ on the next layer 
(for $i=$ both 0 and 1), or $T_{i}$ to $T_{1-i}$.  Setting switches in one of 
these two states guarantees that paths are node-disjoint, so I 
will always set them accordingly. 

	For all the switches such that $T_{0},T_{1}\in A$, half
of these packets get sent to $0*$, and half to $1*$.  If 
$T_{0},T_{1}\not\in A$, half of these (zero) packets get sent
to each sub-butterfly, too.  Consider all of the remaining
packets.  Each of these is the sole packet in the switch.  So,
by setting $\lceil |A|/2 \rceil$ of the switches to send the
packets to sub-butterfly $0*$, and $\lfloor |A|/2 \rfloor$ of them 
to $1*$, we prove the first part of the lemma.  The rest
follows by symmetry. 
\end{proof}

This lemma allows a surprisingly simple proof of $2^{m}$ concentration.
\begin{theorem} \label{powers of two}
\index{Theorem \ref{powers of two}}
Suppose $|A|=2^{m}=|B|$.  Then there exist node-disjoint paths
from any input set $A$ to any output set $B$ on a pair of butterflies.
(In other words, a pair of butterflies is a $2^{m}$-concentrator.)
\end{theorem}
\begin{proof}
Consider the left butterfly.
We can apply Lemma~\ref{simple splitting} recursively for $m$ steps.  On
step 1, we split $A$ so that $2^{m-1}$ packets go to $0*$ and 
$2^{m-1}$ go to $1*$.  Since $0*$ and $1*$ are themselves $d-1$ 
dimensional butterflies, we can apply the lemma again, on each
of them, giving us 4 sub-butterflies, each with $2^{m-2}$ 
paths.  After $m$ steps, we end up with $2^{m}$ sub-butterflies
(which is all of the $d-m$ dimensional sub-butterflies),
each of which has exactly 1 packet.  Now, on each of these butterflies,
we can send the packet along any path we want for the remainder
of the left butterfly (i.e.\ until we hit layer $d$); since it's the only 
packet on its sub-butterfly, there's no possibility of any other 
packet's path crossing its own.

We can perform the same construction on the output packets in $B$,
moving backwards toward the input layer.  When we reach layer $2d-m$,
there will be 1 packet per sub-butterfly.

At this point, observe that the sub-butterfly connectivity graph 
determines the connections between these butterflies.
By Lemma~\ref{enriched graph is regular},
this graph is a regular bipartite graph.
By Hall's theorem, there exists a 
perfect matching.  
This matching in the connectivity graph implies a matching in the
set of sub-butterflies, which implies a matching between the (unique)
packets in each sub-butterfly.  By construction of the connectivity
graph, there exists a path (not necessarily unique) between matched
packets.  As observed above, these paths are node-disjoint,
so we're done. 
\end{proof}

\subsection{Some Corollaries}

We get a very short corollary:
\begin{corollary}
Suppose $|A|=|B| = 2^{d} - 2^{m}$.  Then there exist node-disjoint paths
from $A$ to $B$ on a pair of butterflies.
\end{corollary}
\begin{proof}
Use Lemma~\ref{set complement} and Theorem~\ref{powers of two} on 
the complements of $A$ and $B$. 
\end{proof}

We can use Theorem~\ref{powers of two} to give us information about
a kind of rearrangeability on sufficiently small input and output sets.
\begin{corollary} \label{mini-rearrangeability}
\index{Corollary \ref{mini-rearrangeability}}
Suppose we have a pair of $d$-dimensional butterflies.
Suppose that there is a path between each node on layer 
$\lfloor d/2 \rfloor$ (in the left butterfly) and each node
on layer $2d - \lfloor d/2 \rfloor$ (in the right butterfly).
Then if we select any input set $A$ and output set $B$ with
$A=B \leq 2^{\lfloor d/2 \rfloor}$,  and any permutation $\rho$ from
$A$ to $B$, there exists a collection of node-disjoint paths
from $A$ to $B$ such that for every $a \in A$, the path from $a$
ends at $\rho(a)$.
\end{corollary}
\begin{proof}
If the corollary holds when $A=B = 2^{\lfloor d/2 \rfloor}$, then,
by using dummy packets to make up the difference, the corollary 
holds for $A=B \leq 2^{\lfloor d/2 \rfloor}$.  So, suppose that 
$A=B = 2^{\lfloor d/2 \rfloor}$.  We can use the same argument in
Theorem~\ref{powers of two} to split the packets until there is one
packet on each $\lceil d/2 \rceil$ dimensional sub-butterfly.
By the assumption in the corollary, the resulting connectivity graph
is a complete bipartite graph on \emph{all} nodes, so we can select
node-disjoint paths between the path originating at any $a$ and
send it to the path terminating at $\rho(a)$.
\end{proof}

Note that if we have a pair of standard butterflies, the corollary holds.
Also, suppose we have a pair of layer-permuted butterflies.
Suppose further that we insist that
\begin{itemize}
\item  if $i \leq \lfloor d/2 \rfloor$, then 
$\pi(i) \leq \lfloor d/2 \rfloor$ (where $\pi$ is the left layer permutation)
 on the left butterfly, and
\item
 if $i \geq \lceil d/2 \rceil$, then 
$\sigma(i) \geq \lceil d/2 \rceil$ (where $\sigma$ is the right layer 
permutation) on the right butterfly.
\end{itemize}
(In other words, we permute the layers but don't send any layer from the
left half of the butterfly to the right half.)
Then Corollary~\ref{mini-rearrangeability} holds.

\section{The General Case}

Proving node-disjoint 
subset routing for an arbitrary input and output set (of the same 
size) is somewhat more challenging.  However, 
for pairs of layer-permuted butterflies,
the same basic approach
from Theorem~\ref{powers of two} works.  Looking at the proof, there are
two parts: first, we split the packets into a number of sub-butterflies,
until we have one packet per sub-butterfly.  Then, we view the 
problem as an
exact matching problem on a particular bipartite graph, and show that
a matching exists.  

The proof for the general case runs the same way.  In order to find 
a matching, it's clearly necessary that each connected component
of the bipartite connectivity graph has as many packets 
on the left side as on the right.  
In the next section, I'll prove that this condition (roughly speaking)
is sufficient for the existence of a matching on the 
connectivity graph.
But assuming for now that it holds, we can prove the main result:
\begin{theorem} \label{main}
\index{Theorem \ref{main}}
For any input set $A$ and output set $B$ on a pair of $d$-dimensional 
layer-permuted butterflies,
such that $|A|=|B|$, there exist node-disjoint paths from $A$ to
$B$.
\end{theorem}

\begin{proof}
If $|A| = 2^{d}$, then we are done, by (for example) 
Theorem~\ref{powers of two}.  So throughout, we can assume that
$|A| < 2^{d}$.
Suppose that, in binary, $|A|=b_{m}b_{m-1}\cdots b_{1}$, where
$m\leq d$.  I will prove the lemma by induction on $m$.
The exact statement that I will be inducting on is: 

\emph{
Over the course of $m+1$ steps, 
we can recursively split the packets over the sub-butterflies, 
so that if sub-butterfly $x*$ has $p$ packets in it, then $x0*$ will
have $\lceil p/2 \rceil$ or $\lfloor p/2 \rfloor$ packets, and $x1*$ will
have $\lfloor p/2 \rfloor$ or $\lceil p/2 \rceil$ packets, respectively.
The same holds on the right butterfly.
(There will then be 0 or 1 packets in each sub-butterfly on level $m+1$
and level $2d-m-1$).  We can then select a matching between the 
sub-butterflies giving us node-disjoint paths from $A$ to $B$.
}

First, the base case: if $m$=0 or 1, then we are done (by 
Theorem~\ref{powers of two}).

Next, the inductive step.  Fix $m$ and assume the theorem holds
for all $m'<m$.  Let us try to reproduce the proof of 
Theorem~\ref{powers of two} with $2^{m+1}>|A| \geq 2^{m}$ 
packets to see where complications arise.  If $|A|\neq 2^{m}$, then
we will not be able to divide the packets evenly in half at every
sub-butterfly for $m$ steps.  A sub-butterfly $x*$ may have an odd number of
packets, so we must send the ``extra'' packet either to $x0*$ or $x1*$.
I will refer to this choice (the ``0'' or ``1'') as the \emph{rounding
decision}.  Note that there is no actual packet that is distinguished
as the ``extra'' one-- there's just a surplus of one more packet that
either goes to $x0*$ or $x1*$.  But it's helpful to imagine that one
of the packets is the extra one when describing the paths.

Choose $A' \subseteq A$ and $B' \subseteq B$ such that
$|A'|=|B'|=|A|-2^{m}$.   Let $m'$ be the integer such that 
$|A'|=b_{m'}b_{m'-1}\cdots b_{1}$ (so, $m'<m$).  
By induction, we can find node-disjoint paths
from $A'$ to $B'$.  The information that we keep from the
induction is not the actual paths themselves.
Instead, we keep the rounding decisions that every sub-butterfly
makes.  Note that even after
we've split $A'$ until there's only 1 packet per sub-butterfly,
we are still splitting with an extra packet; it's just that if
$p=1$, then $\lceil p/2 \rceil =1$, and $\lfloor p/2 \rfloor=0$.
Hence, the almost exact recursive splitting part of the inductive
hypothesis holds not just for the first $m'$ steps, but for the
first $m$ steps.  We need to keep this rounding
information, too.

Consider, now, the original sets $A$ and $B$.  Using 
Lemma~\ref{simple splitting} recursively for $m-1$ steps
on the right and left
butterflies, we can send the ``extra'' packet on each sub-butterfly
the same way on level $k<m$ as we did when routing $A'$ to $B'$.  
To do this, we need to know that the same sub-butterflies have an
odd number of packets in them.
Observe that if a sub-butterfly on level $k$ that has $t$ packets in it
in the $(A', B')$ case, then it  
has $t+2^{m-k}$ in the $(A,B)$ case.  
As long as $k<m$, then
$t$ and $t+2^{m-k}$ have the same parity; therefore, extra packets
exist in the same sub-butterflies.  When we reach step $m$, all
the $m$-level sub-butterflies that had one packet in them in the $(A',B')$
case now have 2 packets, and all the sub-butterflies that had no
packets now have 1.

We shift now to the matching problem on the sub-butterfly connectivity graph.
Consider one connected component of the connectivity graph.
Every node on the left hand side represents a sub-butterfly with
one or two packets on it, as does
every node on the right hand side.  By induction, the total number of
packets on each side is the same.  (If they weren't, the packets in
the $(A',B')$ case couldn't match up.)
Using Lemma~\ref{matching} of the 
next section, we can split the packets over level $m$ (and $2d-m$)
to get 0 or 1 packet per sub-butterfly, with the same number on the 
LHS and RHS of each of the connected components.  By 
Lemma~\ref{connected components} each component is completely connected,
and we're done. 
\end{proof}

Note that since we're using the direction of the ``extra'' packet, rather
than any particular path, the actual packets going into the upper or
lower sub-butterflies are not necessarily the same between the $(A',B')$
case and the $(A,B)$ case.  In particular, $A'$ will not necessarily still
be routed to $B'$.

\section{The Matching Lemma}
\begin{lemma} \label{matching}
\index{Lemma \ref{matching}}
Suppose we have a pair of $d$-dimensional layer-permuted butterflies.
Consider its $q$-dimensional
sub-butterfly connectivity graph, where the sub-butterflies
reside on layer $m=d-q$, and $2d-m$.  Suppose
that each node has 1 or 2 packets
on it.  Finally, assume that there are the same number of packets
on the LHS and the RHS of each connected component.

Then, when passing from the $q$-dimensional connectivity graph to the 
$(q-1)$-dimensional connectivity graph, we can send each packet to
a different sub-butterfly such that each connected component has
the same number of packets on the LHS and the RHS.
\end{lemma}

\begin{proof}
Since the behavior of the two-packet sub-butterflies is determined
(one packet goes to $x0*$, one to $x1*$),
this proof will eventually come down to making the correct 
rounding decision for the sub-butterflies with single packets.

There are three cases we have to consider, reflecting the three
possible behaviors of the connectivity graph as outlined on 
page~\pageref{three cases}.

\textbf{Case 1:  (No reused dimensions)}  If the connected components don't
split between the $q$ and $q-1$ dimensional sub-butterflies,
then the lemma is trivially true.

\textbf{Case 2:  (One reused dimension)}  Suppose that each connected component
splits into two connected components.  Consider one connected component
$C$ in the $q$-dimensional connectivity graph that splits into 
$C_{0}$ and $C_{1}$ in the $q-1$-dimensional connectivity graph.

Suppose that there are
$x$ nodes in $C$ with two packets on them, and $y$ nodes with one 
packet on them.  We must send $x$ packets to $C_{0}$ and $x$ to $C_{1}$
on both the left and the right sides
because the behavior of two-packet sub-butterflies is determined.
We can send the $y$ packets from one-packet nodes to either component;
we simply send
$\lfloor y/2 \rfloor$ to $C_{0}$ and $\lceil y/2 \rceil$ to $C_{1}$
on both the left and the right sides.  Then the lemma holds.

\textbf{Case 3:  (Two reused dimensions)} Suppose that each connected
component in the $q$-dimensional connectivity graph splits into four
connected components, e.g.\ $C$ splits into $_{0}C_{0}$, $_{1}C_{0}$, 
$_{0}C_{1}$, and $_{1}C_{1}$.  

Let us calculate how many
of the packets from the 2-packet butterflies arrive in each of these
splintered components.  

If we have a sub-butterfly $x*$ on level $m$ with 2 packets in it, then 
we must send exactly 1 packet to $x0*$ and one to 
$x1*$.  I will refer to these packets as \emph{constrained} packets.  
(By contrast, if a sub-butterfly $x*$ on level $m$ has only 1 packet in it,
we can send the packet either to $x0*$ or $x1*$; such a packet is 
a \emph{free} packet.)  We shift our view back to the corresponding
connectivity graph.
Let us label the number of constrained packets on each side of each 
$_{i}C_{j}$.
Observe first of all that because constrained packets come in pairs,
for a fixed $i=0$ or 1, there are as many constrained
packets on the LHS of $_{i}C_{0}$ as of $_{i}C_{1}$, and similarly
as many on the RHS of $_{0}C_{i}$ as of $_{1}C_{i}$.  
Let the number of packets on the LHS of
$_{0}C_{0}$ be $a_{1}$, and the number of packets on the LHS of $_{1}C_{0}$
 be $a_{2}$.
Let the number of packets on the RHS of 
$_{0}C_{0}$ be $b_{1}$, and the number of packets on the RHS of $_{0}C_{1}$
 be $b_{2}$.
See Figure~\ref{constrained}. 
\begin{figure}[ht] 
\centerline{\includegraphics[height=1.5in]{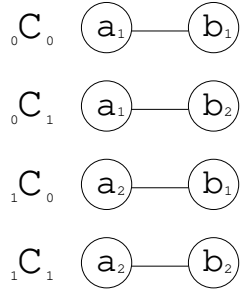}
 \hspace{.2in}} 
\caption{The number of constrained packets} 
\label{constrained}
\end{figure}

Observe that, since each sub-butterfly in layer $m$ has
either one free or two constrained packets, then the number of
packets on the LHS is 
\[\# \mbox{free}_{LHS} + \# \mbox{constrained}_{LHS}=
2^{m}+\frac{1}{2}(\# \mbox{constrained}_{LHS})\]
Since the analogous
equation holds on the RHS, and since the total number of packets are equal,
we get that 
\[2^{m}+\frac{1}{2}(\# \mbox{constrained}_{LHS})=
2^{m}+\frac{1}{2}(\# \mbox{constrained}_{RHS}),\]
so there's the same total number of constrained
packets on the RHS and the LHS.  Therefore, adding up the 
constrained packets in Figure~\ref{constrained} and dividing by two, we get
\[a_{1} + a_{2}=b_{1}+b_{2}\]
Also, any particular $a_{i}$ or $b_{i}$ can't be larger than 
$2^{m-1}$, so 
\[a_{i}, b_{i} \leq 2^{m-1}\]
Due to symmetry, we can assume w.l.o.g. that $a_{1}\geq a_{2}$,
$b_{1} \geq b_{2}$, and $a_{1}\geq b_{1}$.  Putting this together,
we can assume that
\[2^{m-1}\geq a_{1} \geq b_{1} \geq b_{2} \geq a_{2} \geq 0\]

Generally speaking, $a_{i}\neq b_{j}$, so there will not be the 
same number of constrained packets on the RHS and LHS of each
connected component of Figure~\ref{constrained}.  However, we
still have the free packets to allocate.  The situation is as drawn in
Figure~\ref{free 1}.
\begin{figure}[ht] 
\centerline{\includegraphics[height=2.5in]{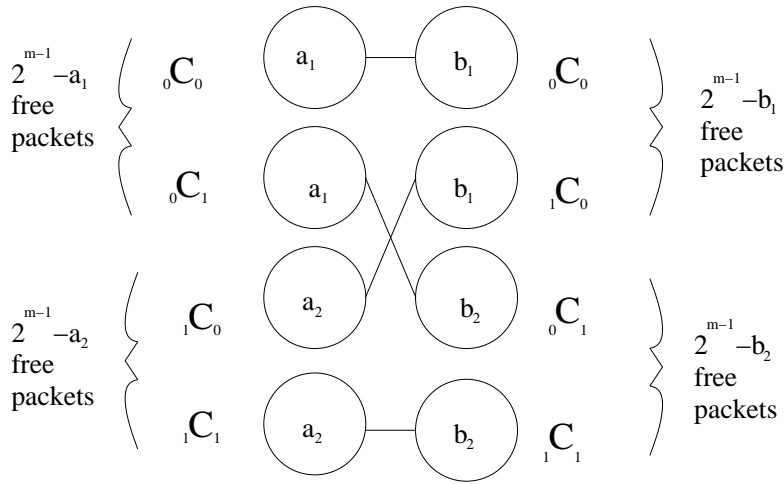} \hspace{.4in}  }
\caption{The number of free and constrained packets} 
\label{free 1}
\end{figure}
Since we assume that $a_{1}\geq b_{1}\geq b_{2} \geq a_{2}$, then 
in order to balance the packets on the LHS and RHS, we have to
add packets as in Figure~\ref{free 2}.
\begin{figure}[ht]
\centerline{\includegraphics[height=2.5in]{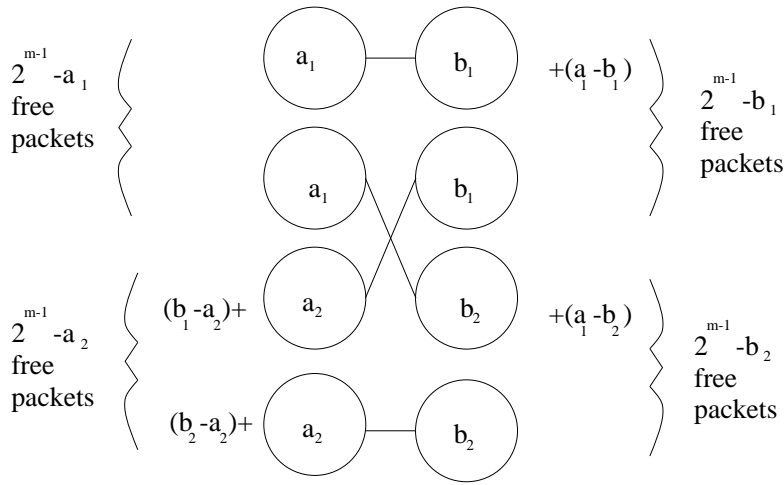} \hspace{.4in}  }
\caption{Adding free packets to balance the bipartite graph}
\label{free 2}
\end{figure}
We have to show that there are enough free packets to add.
There are three inequalities to check.  First,
for $_{1}C_{0}$ and $_{1}C_{1}$ on the LHS,
let us calculate how many free packets are required.
\[(b_{1}-a_{2})+(b_{2}-a_{2}) = b_{1}+b_{2}-2a_{2}=a_{1}+a_{2}-2a_{2}
=a_{1}-a_{2}\]
Now, $a_{1}\leq 2^{m-1}$, so we need no more than $2^{m-1}-a_{2}$ 
free packets, which we have.
For the other two cases, (namely $i=0$ and $i=1$), note that
\[a_{1}-b_{i}\leq 2^{m-1}-b_{i}\]
and in each case, there are $2^{m-1}-b_{i}$ free packets.
So, in all cases, we can use a subset of the free packets to
make the total number of packets on the RHS and LHS equal.
Since all the remaining unmatched free packets on the left are
connected to all the unmatched free packets on the right, we can
choose an exact matching to match these packets, send them to
the appropriate connected component,
and we're done.
\end{proof}

\section{Conclusion}
Are all pairs of butterflies superconcentrators?  Or only the layer-permuted
ones?  It's certainly natural to conjecture that the stronger statement is
true.  As a piece of support, Theorem~\ref{powers of two} can be extended
to prove that any pair of butterflies is a $(2^{m}+1)$ concentrator.
Unfortunately, the pathological cases (from unusual butterfly isomorphisms)
make the general analysis more complicated than I could solve.

The concentration and superconcentration 
results in this chapter all spring from a splitting
and matching approach.
This method holds out a tantalizing suggestion of a
proof of the rearrangeability of pairs of butterflies.
Theorem~\ref{main} can be viewed as follows: if we number each
input and output node 0 or 1, and have the same number of 
zeroes among the inputs and outputs, we can route a permutation
that sends $0\rightarrow 0$ and $1\rightarrow 1$.  
Suppose we labelled the input and output nodes
0,1,2, or 3, with the same size restraints.  The proofs above seem
likely to apply to this case, too.  If we could just continue
doubling the number of labels up to $N=2^{d}$, we'd have proved
rearrangeability.  Getting the proofs to work for an arbitrary
$2^{m}$ seems pretty challenging, though.

Another natural network to try these methods on is the hypercube.
Typically, rearrangeability on the hypercube requires that each
edge is used at most once, ever, and concerns edge-disjointness, rather
than node-disjointness.  A result analogous to 
Theorem~\ref{main} would be more likely to apply to a hypercube
that uses each edge at most once per time step, but possibly multiple
times over several time steps.  However, edge-disjointness might
be strengthened to node-disjointness.  Unfortunately, the translation
to a hypercube is not trivial.

Proving that a graph is a superconcentrator can also be viewed as a
max flow/ min cut problem; thus, Theorem~\ref{main} can be viewed as 
saying that
for any collection of $k$ input and $k$ output nodes, it is 
necessary to delete at least $k$ edges to prevent any (single-pass)
paths from the input to the output sets.  One might optimistically hope
that these results might translate to other max flow problems, at least
on switching networks.

On a possibly more practical note, it's interesting to observe that the
Theorem~\ref{main} makes use of the size of the 
input set, rather than the set itself (i.e.\ $|A|$, not $A$).  It follows that
once you calculate the rounding decisions for a particular sized
input set, the same rounding decisions solve the problem for all
input sets of the same size.  This also suggests another method for
proving concentration results.

%% file: probability.appendix.tex
\chapter{Analysis and Probability} \label{probability chapter}

\section{Markov Chains} \label{basic probability section}
The most common stochastic object in this thesis is the Markov chain.

\begin{definition} \index{Markov chain}
A \emph{countable discrete time Markov chain} is a stochastic process
$X(t)$ defined on a countable state space $\ssx$ at discrete moments in time
$t\in\inte$.  It has the property that the distribution of states at
time $t>t_{0}$ is independent of the distribution of states at time
$t<t_{0}$, conditional on the state at time $t_{0}$.

A Markov chain is \emph{time independent},\index{time independence}
if the probability of transfering from state $x$ to $y$
in one time step is independent of the time $t$.

Suppose that for any pair of states $x$ and $y$, there is a nonzero
probability of travelling from state $x$ to $y$ in a finite number of
steps, and from $y$ to $x$ in a finite number of steps.  We call such
a Markov chain\emph{irreducible}.\index{irreducible}

A Markov chain is \emph{periodic with period $p$}
\index{periodic Markov chain}
if the number of time steps
it takes to get from any node $x$ back to itself is always a multiple
of $p$, for $p>1$.  A Markov chain is called
\emph{aperiodic}\index{aperiodic Markov chain} if it is not periodic.
\end{definition}

The Markov chains I will be studying will always be irreducible
and aperiodic.

The first property of interest in studying Markov chains is
their stationary distributions.
\begin{definition} \index{stationary distribution}
Suppose we have an irreducible, aperiodic discrete time Markov
chain with a countable state space.  Take any state $x$.
Start the Markov chain in state $x$, and
let $f_{x}(t)$ be the amount of time that the Markov
chain has spent in state $x$ during time $<t$.  Define
\[\pi(x)=\lim_{t\rightarrow \infty}\Ex \left[\frac{f_{x}(t)}{t}\right]\]

Suppose that $\pi$ forms a distribution on $\ssx$, i.e.\
\[\sum_{x\in\ssx}\pi(x)=1\]
Then we call $\pi$ a \emph{stationary distribution}
of the Markov chain.

If the Markov chain has a stationary distribution, then it is called
\emph{ergodic}\index{ergodicity}.  It is also called
\emph{stable}\index{stability} or \emph{positive
recurrent}\index{positive recurrence}.
\end{definition}

It turns out (see Lawler~\cite{lawler}) that either $\pi(x)=0$ for all
$x$, or $\pi$ is a distribution on $\ssx$.  Here are some 
basic facts about ergodicity and Markov chains:
\begin{theorem} \label{basic ergodic facts theorem}
\index{Theorem \ref{basic ergodic facts theorem}}
The stationary distribution for an irreducible, aperiodic Markov
chain, if it exists, is unique.

Let $S_{\sigma}(t)$ be the amount of time spent in state $\sigma$ between
time 0 and $t$.  If we have a ergodic, irreducible, aperiodic Markov chain,
with stationary distribution $\pi$, then
\[\lim_{t\rightarrow \infty}\frac{S_{\sigma}(t)}{t} = \pi (\sigma) \]
almost surely.  (i.e.\ $\pi(\sigma)$ equals the average fraction of
time spent in state $\sigma$ a.s.)

Suppose we have an irreducible, aperiodic Markov chain, and a state
$x\in\ssx$.  Then the Markov chain is ergodic iff the expected
number of steps between visits to $x$ is finite.
\end{theorem}

\proof See Lawler~\cite{lawler}. \EOP

Let $p(x,y)$ be the probability of travelling from state $x$ to state
$y$ in one time step.  Suppose we have a distribution $\pi$ on $\ssx$
such that for any $x$,
\begin{equation} \label{kolmogorov equations}
\pi(x)=\sum_{y\in\ssx}p(y,x)\pi(y)
\end{equation}
Then there exists a stationary distribution, and the distribution is
$\pi$.  Equation~\ref{kolmogorov equations} is actually a family of
equations, one for each $x\in\ssx$; these are sometimes called the
Kolmogorov equations.  (See Lawler~\cite{lawler}).

\begin{definition}\index{product form}
Suppose we have a stationary distribution defined on an $N$ node
queueing network.  The state of the network can be specified 
by the state of each of its nodes.  We can write this as
$\sigma=(\sigma_{1},\cdots,\sigma_{N})$.

Suppose that 
\[\Pr(\sigma)=\prod
_{i=1}^{N}\Pr(\sigma_{i})\]
for every state $\sigma$, i.e.\ the marginal probabilities multiply
together as though they were independent.  Then we say
that the stationary distribution is of \emph{product form}.
\end{definition}

\section{Tail Bounds}
First, let's construct an exponential upper bound on the tail of sums
of Bernoulli random variables.  (Bounds of this type sometimes go
under the name of ``Hoeffding inequalities''.)\index{Hoeffding
inequality}
\begin{lemma} \label{sums of bernoulli lemma}
\index{Lemma \ref{sums of bernoulli lemma}}
Given a collection of $n$ independent Bernoulli random variables
$X_{1}, X_{2},\ldots,X_{n}$, where $\Pr[X_{k}=1]\leq P_{k}$
for $1\leq k \leq n$, then
\[\Pr[X \geq \beta P] \leq e^{(1-\frac{1}{\beta}-\ln \beta)\beta P}\]
where $\beta>1$, $X=X_{1}+\cdots+X_{n}$, and $P=P_{1}+\cdots+P_{n}$.
\end{lemma}
\proof
See Leighton~\cite{leighton}, page 168. 
Incidentally, $\beta>1$ implies that $1-\frac{1}{\beta}-\ln \beta<0$,
as can be seen by taking the derivative, so the bound in the
theorem is non-trivial.
\EOP

Next, a lower bound on sums of Bernoulli random variables.
\begin{lemma} \label{bernoulli lower bound lemma}
\index{Lemma \ref{bernoulli lower bound lemma}}
Given a collection of $n$ independent Bernoulli random variables
$X_{1}, X_{2},\ldots,X_{n}$, where $\Pr[X_{k}=1]\geq P_{k}$
for $1\leq k \leq n$, then
\begin{equation} \label{bern lower bound equation}
\Pr[X \leq \beta P] \leq e^{(1-\frac{1}{\beta}+\ln \beta )\beta P}
\end{equation}
where $0<\beta<1$, $X=X_{1}+\cdots+X_{n}$, and $P=P_{1}+\cdots+P_{n}$.
\end{lemma}

\note
Equation~\ref{bern lower bound equation} implies that if $\gamma>0$, then
\[\Pr[X \leq (1-\gamma) P] \leq e^{-\gamma P}\]

\proof
It's tempting to try to prove this result by using Lemma~\ref{sums of
bernoulli lemma}.  Since $1-X_{i}$ is also a Bernoulli random variable,
then the upper bound of Lemma~\ref{sums of bernoulli lemma} translates
into a lower bound.  Unfortunately, if we consider the behavior
for large $n$, our value of $\beta$ will be order $1+\frac{1}{n}$,
and it becomes difficult to analyze exactly what's going to happen.

Instead, I'll prove this using a fresh moment generating function.
(This technique is almost identical to the proof of Lemma~\ref{sums
of bernoulli lemma} from Leighton~\cite{leighton}.)

First, observe that for any $\lambda>0$,
\begin{eqnarray*}
\Ex\left[e^{-\lambda X_{k}}\right] & = & 
	\Pr[X_{k}=1]e^{-\lambda}+1-\Pr[X_{k}=1]	\\
& = &	1-\Pr[X_{k}=1](1-e^{-\lambda}) 		\\
&\leq &	1-P_{k}(1-e^{-\lambda}) 		\\
&\leq &	e^{-P_{k}(1-e^{-\lambda})}
\end{eqnarray*}
since $e^{-\lambda}<1$ and $1-x\leq e^{x}$ for all $x$.  Since
the $X_{k}$'s are independent, it follows that
\begin{eqnarray*}
\Ex\left[e^{-\lambda X}\right] & = & 
\Ex\left[e^{-\lambda X_{1}}\cdots e^{-\lambda X_{n}}\right] \\
&=& \Ex\left[e^{-\lambda X_{1}}\right]
\cdots \Ex\left[e^{-\lambda X_{n}}\right] \\
&\leq& 
e^{-P_{1}(1-e^{-\lambda})} \cdots e^{-P_{n}(1-e^{-\lambda})} \\
&\leq& 
e^{-P(1-e^{-\lambda})}
\end{eqnarray*}
By Markov's inequality,
\begin{eqnarray*}
 \Pr[e^{-\lambda X} \geq e^{-\lambda \beta P}] &\leq&
	\frac{\Ex[e^{-\lambda x}]}{e^{-\lambda \beta P}} \\
&\leq&
e^{-P(1-e^{-\lambda})+\lambda \beta P}
\end{eqnarray*}
If we set $\lambda= -\ln \beta$, which minimizes the bound, then
\begin{eqnarray*}
\Pr[X \leq \beta P] 
& = &
\Pr[e^{\lambda X}\leq e^{\lambda \beta P}] \\
& = &
\Pr[e^{-\lambda X}\geq e^{-\lambda \beta P}] \\
&\leq & 
e^{-P(1-\beta)- \beta \ln\beta P}  \\
&\leq & 
e^{(1-\frac{1}{\beta} +\ln\beta)\beta P}
\end{eqnarray*} 
Note that by taking derivatives,
it is straightforward to show that if
$0<\beta <1$, then
\[1-\frac{1}{\beta} +\ln \beta < 0\]
so the bound in the theorem is non-trivial.
\EOP

\section{The Comparison Theorem and Drift} 
  \label{comparison theorem section} The Comparison theorem is a
powerful theorem that allows us to say, roughly: if a real,
non-negative function of the state space has expected negative drift,
then the expected return times of the system are finite.  The theorem
follows from Dynkin's formula.  This whole exposition is stolen,
pretty much whole hog, from Meyn and Tweedie's
book~\cite{Meyn_and_Tweedie}.  This theorem works equally well for
continuous and discrete time.

We consider a stochastic process $X(t)$ giving the state of 
a Markov chain at time $t=0,1,2,\ldots$.  Let $Z$ be a 
function from $\ssx$ to the non-negative reals.
(This can 
be made more general, but it's not useful to do so.)  For example,
$Z(X(t))$ might be the total queue length of $X(t)$.

For any stopping time $\tau$, define
\[\tau^{n} = \min\{n, \tau, \inf\{k\geq 0 : Z(X(k))\geq n\}\}\]

\begin{theorem}[Dynkin's Formula]
\index{Dynkin's formula}
For each $x\in\ssx$ and non-negative integer $n$,
suppose that $X(0)=x$.  Then
\[\Ex[Z(X(\tau^{n}))] = \Ex[Z(X(0))] +
\Ex \left[ \sum_{i=1}^{\tau^{n}} (\Ex[Z(X(i)) \, | \, X(i-1)] - Z(X(i-1))\right]\]
\end{theorem}

\proof  
For each $n\in \inte_{+}$,
\[Z(X(\tau^{n})) = Z(X(0)) + \sum_{i=1}^{\tau^{n}}
(Z(X(i))-Z(X(i-1)))\]
\[=Z(X(0)) + \sum_{i=1}^{n}1_{\{\tau^{n} \geq i\}}(Z(X(i))-Z(X(i-1)))\]
Taking expectations and noting that $\Ex[1_{\{\tau^{n} \geq i\}}|X(i-1)]
=\Ex[1_{\{\tau^{n} \geq i\}}]$, we get
\[\Ex[Z(X(\tau^{n}))] = \Ex[Z(X(0))] + 
\Ex\left[\sum_{i=1}^{n}\Ex[Z(X(i))-Z(X(i-1)) | X(i-1)]
1_{\{\tau^{n} \geq i\}}\right] \]
\[=\Ex[Z(X(0))] + \Ex\left[\sum_{i=1}^{\tau_{n}}\Ex[Z(X(i))|X(i-1)] - Z(X(i-1))\right] \]
\EOP

We can use Dynkin's formula to analyze drift in a system.
As a corollary of this, we get our main result.
\begin{theorem}[The Comparison Theorem] \label{comparison theorem}
\index{Comparison Theorem}
\index{Theorem \ref{comparison theorem}}
Suppose that $Z, f,$ and $s$ are functions from $\ssx$ to
the non-negative reals.  Suppose further that, for $X(0)$ 
equal to any fixed $x$,
\begin{equation} \label{eq:drift}
\Ex Z(X(1)) \leq Z(x) - f(x) + s(x)
\end{equation}
Then for any stopping time $\tau$,
\[\Ex \left[ \sum_{k=0}^{\tau-1} f(X(k)) \right]
\leq Z(x) + \Ex\left[ \sum_{k=0}^{\tau -1} s(X(k))\right] \]
\end{theorem}

\note The functions $Z, f$, and $s$ in the Comparison theorem should
be interpreted as follows.  $Z$ is the function whose drift we're 
considering, e.g.\ the total queue length.  The amount
that $Z$ drifts down by in one step is $f$.  Finally, $s$ is an
exception parameter; typically, it equals a constant on some
``bad'' set of (finitely many) exceptions, and zero everywhere else.

\proof
Since we have a Markov chain, our assumption in Equation~\ref{eq:drift}
actually tells us that for all $t$,
\[\Ex Z(X(t+1)|X(t)) \leq Z(X(t)) - f(X(t)) + s(X(t))\]
Fix some integer $N>0$ and note that
\[\Ex [Z(X(t+1)) | X(t)] \leq Z(X(t)) - f(X(t))\wedge N + s(X(t))\]
(where $a\wedge b=\min \{ a,b \}$.)  By Dynkin's formula, 
\[0\leq \Ex [Z(X(\tau^{n}))]\leq Z(X(0)) +
\Ex\left[ \sum_{i=1}^{\tau^{n}}(s(X(i-1)) -[f(X(i-1)) \wedge N])\right] \]
and hence by adding the finite term
\[\Ex \left[ \sum_{k=1}^{\tau^{n}}[f(X(k-1))\wedge N] \right] \]
to each side we get
\[\Ex \left[ \sum_{k=1}^{\tau^{n}}[f(X(k-1))\wedge N] \right] 
\leq Z(X(0)) +
\Ex\left[ \sum_{i=1}^{\tau^{n}}s(X(i-1)) \right] \]
\[ \leq
Z(X(0)) +
\Ex\left[ \sum_{i=1}^{\tau}s(X(i-1)) \right] \]
Letting $n\rightarrow \infty$, and then $N\rightarrow \infty$ gives the 
result by the monotone convergence theorem.
\EOP

Special cases of the Comparison Theorem give some frequently
used stability criteria.

\begin{corollary}[Foster's Criterion] \label{foster's criterion}
\index{Corollary \ref{foster's criterion}}
\index{Foster's criterion}
Suppose we have a discrete time irreducible, aperiodic Markov chain
with a countable state space.  Suppose we construct a potential
function $Z:\ssx \rightarrow \real^{+}$.  Suppose that for any real
$B$, there are only finitely many states $x\in\ssx$ with $Z(x)<B$.

Finally, suppose that there exists a $B_{0}$ and
an $\epsilon>0$ such that
if $Z(X(0))>B_{0}$, then
\[\Ex[Z(X(1))]\leq Z(X(0)) - \epsilon\]
Then the Markov chain is stable.
\end{corollary}

\proof
This result follows from the Comparison Theorem.  It was originally
proved (in a slightly weaker form) by Foster~\cite{foster}.  
\EOP

It's interesting to node that the converse is also true.  Suppose we
have a stable Markov chain.  Fix some state $x_{0}\in\ssx$.  
Define $Z(x_{0})=0$, and otherwise let $Z(x)$ be the expected number
of time steps until $x$ returns to $x_{0}$.  (By stability, this
expectation is finite.)  If $Z(x)>1$, then the expected change is
exactly $-1$ in one time step.

Foster's Criterion amounts to taking a constant $f(x)$ in the
Comparison Theorem.  To prove stronger results, we need to let $f(x)$
grow.  For example, suppose we wanted to show that the
expected queue length was finite, where $q(x)$
is the length of the queue of state $x$.  Then
it would suffice to find a $B_{0}$ and $f(x)$ such
that
\[\mbox{if } f(x)>B_{0},\mbox{ then }f(x)\geq q(x)\]
If we choose a $f(x)$ that grows even faster, we
can prove stronger results.

\begin{definition}
Consider a discrete time Markov chain.  Take a $\beta>0$ and
$f:\ssx\rightarrow \real^{+}$ such that $f(x)\geq 1$ for all $x$.  Let
$\Delta f(x)$ be the expected change in $f(x)$ after one time step.
Suppose there is a bound $B$ such that for any $x\in\ssx$ with
$f(x)>B$, we have
\[\Delta f(x) \leq -\beta f(x)\]
Then we say that the Markov chain is \emph{geometrically
ergodic}.\index{geometric ergodicity}
\end{definition}

There are a host of interesting facts about geometrically
ergodic Markov chains.  The two I use in this thesis are
the following:

\begin{corollary} \label{geometric ergodicity corollary}
\index{Corollary \ref{geometric ergodicity corollary}}
If we have a geometrically ergodic Markov chain,
then the return time to any state has
an exponential tail.  Also,
\[\Ex[f(x)]<\infty\]
\end{corollary}

\proof
See Meyn and Tweedie~\cite{Meyn_and_Tweedie}, Chapter 15.
One can also prove
\[\Ex[f(x)]<\infty\]
by a more immediate appeal to the Comparison Theorem.
\EOP

\section{Uniform Integrability Facts} 
\label{uniform integrability appendix}
Suppose we have a sequence of functions $f_{n}(t)$ that converge
to $f(t)$.  It would be nice if $\Ex [f_{n}(t)]$ converged to
$\Ex[f(t)]$.  Uniform integrability allows one to
make conclusions like that.  More precisely:

A collection of random variables $\{ Y_{n} \}$ is called 
\emph{uniformly integrable} if
\[\lim_{M\rightarrow \infty} \sup_{n} \Ex [|Y_{n}| 1_{\{|Y_{n}|>M\}}] =0\]

We can now state the main theorem of significance to us:
\begin{theorem}  \label{main unif int}
\index{Theorem \ref{main unif int}}
For $\{ Y_{n} \}$ and $Y$ in $\mathcal{L}^{1}$, $\lim_{n\rightarrow
\infty} \Ex |Y_{n}-Y| = 0$ if and only if both $Y_{n}\rightarrow Y$ in
probability and $\{Y_{n}\}$ are uniformly integrable.
\end{theorem}
A proof of this result
can be found in a number of sources, e.g.\ Dudley~\cite{Dudley}, Theorem
10.3.6. 

We can make the following 
\begin{corollary} \label{lim exp}
\index{Corollary \ref{lim exp}}
Suppose that $\{ Y_{n} \}$ are a sequence of real-valued 
random variables.  
Suppose further that there exists a bound $d$ such that for
every $n$, $Y_{n}\in[-d,d]$.  Suppose finally that 
$\lim_{n} Y_{n}=0.$  Then
\[\lim_{n\rightarrow \infty} \Ex [Y_{n}]=0\]
\end{corollary}
\proof
The bound $d$ tells us that the $\{ Y_{n} \}$ are in $\mathcal{L}^{1}$,
and also that they are uniformly integrable.  Since the $\{ Y_{n} \}$
converge pointwise to 0, then in particular they converge to 0 in probability.
Therefore, Theorem~\ref{main unif int} gives us the result. 
\EOP

A more straightforward proof of Corollary~\ref{lim exp} follows from Lebesgue's
dominated converge theorem (see, for instance, Rudin~\cite{rudin},
page 26).  However, generalizations of Corollary~\ref{lim exp} for
unbounded random variables need to use Theorem~\ref{main unif int},
so I include it here.

\section{Queueing Theory} \label{queueing theory section}
\begin{theorem} \label{birth-death theorem}
\index{Theorem \ref{birth-death theorem}}
Suppose we have a discrete time single server queue.  Suppose
that packets are inserted by a Bernoulli
process such that a packet arrives with probability 
$\hat{A}$.  Suppose that service times are geometrically
distributed such that a packet departs with
probability $\hat{D}$.  

Define $A=\hat{A}(1-\hat{D})$ and $D=\hat{D}(1-\hat{A})$.
($A$ is the chance of a net gain of one packet; $D$ is 
the chance of a net departure of one packet.)

Suppose that we measure the 
system after new packets have arrived and before
old packets have departed.  Then the stationary distribution
is:
\[\Pr[\mbox{0 packets}]=\frac{\hat{D}-\hat{A}}{\hat{D}}\]
\[\Pr[\mbox{1 packet}]=\frac{\hat{D}-\hat{A}}{\hat{D}}
\frac{\hat{A}}{D}\]
and for $n>1$,
\[\Pr[\mbox{$n$ packets}]=
\left(\frac{A}{D}\right)^{n-1}\frac{\hat{A}}{\hat{D}}
\frac{D-A}{D}\]
so the expected queue length is:
\[\frac{A\hat{A}}{D(D-A)}(1-\hat{A})
=\frac{\hat{A}^{2}(1-\hat{D})}{\hat{D}(\hat{D}-\hat{A})}\]

Suppose that we measure the system after packets have departed
and before new packets have arrived.  Then the stationary distribution is:
\[\Pr[\mbox{$n$ packets}]=
\left(\frac{A}{D}\right)^{n}\frac{D-A}{D}\]
so the expected queue length is:
\[\frac{A^{2}}{D(D-A)}\]
\end{theorem}

\proof
Given the stationary distributions above, it is straightforward
to plug them in and verify that they work.  
\EOP

\begin{theorem}[Discrete Time Pollaczek-Khinchin Formula]
\index{Pollaczek-Khinchin formula}
\label{p-k theorem}
\index{Theorem \ref{p-k theorem}}
If arrivals are Bernoulli with parameter $\lambda$, 
and $Z$ is the distribution of service times, then
\[E[\mbox{queue length}] = 
\frac{\lambda^{2}(E[Z^{2}] - E[Z])}{2(1-\lambda E[Z])}\]
\end{theorem}

\proof The continuous time version can be found in Gallager~\cite{kn:b3},
pages 85-87.  The conversion to discrete time is fairly straightforward.
\EOP

Why does the regular Pollaczek-Khinchin formula work need to be 
changed at all, since discrete time queues should be a special case of
continuous time?  Well, the values of the expected queue length
are sampled at discrete intervals, so this skews the formula
slightly.

\begin{theorem}[Little's Theorem] \label{little's theorem}
\index{Little's theorem}
\index{Theorem \ref{little's theorem}}
The time average number of packets at a node is equal to the time
average waiting time multiplied by the mean arrival rate.  Similarly,
the time average number of packets in queue equals the time average
waiting in the queue times the mean arrival rate.

This result implies that if a node has nominal load $r$, then the
probability of that node being idle is $1-r$, assuming ergodicity.
\end{theorem}
\proof A proof can be found in Gallager~\cite{kn:b3}. \EOP

In the single node case, the $1-r$ probability is particularly useful.
If $r<1$, then the system is ergodic, so Little's Theorem holds.
Therefore, the stationary probability that there are no packets in
queue is $1-r$.  If we have $N>1$ nodes in our network, though, it
does not generally follow that the probability that all of them are
empty is $(1-r)^{N}$ (unless, for instance, the stationary
distribution is product form.)

\begin{corollary}
For a fixed node $n$ in an ergodic network,
\[\Ex[\mbox{total packets at node } n] = 
\Ex[\mbox{packets in queue at node } n] + r\]
\end{corollary}

\proof
\[\Ex[\mbox{total packets at node } n] = 
\sum_{i=1}^{\infty}n\Pr[\mbox{$n$ packets total}]\]
\[=\left(\sum_{i=1}^{\infty}(n-1)\Pr[\mbox{$n$ packets total}]\right)+
\left(\sum_{i=1}^{\infty}\Pr[\mbox{$n$ packets total}]\right)\]
Now, $\sum_{i=1}^{\infty}\Pr[\mbox{$n$ packets total}]=1-
\Pr[\mbox{0 packets total}]$.  By Little's Theorem, this
equals $1-(1-r)=r$.  The result follows.
\EOP

Finally, let us turn to renewal reward functions.

\begin{definition}\index{renewal reward function}
\index{delayed renewal process}
Suppose we have discrete time, aperiodic, irreducible, countable state
space Markov chain $X(t)$.  Let us specify a special state
$\sigma_{renew}$, and consider the system to undergo a \emph{renewal}
when it enters that state.  If we start this Markov chain in state
$\sigma_{renew}$, it is called a \emph{renewal process}; if we start
it in an arbitrary state, it's called a \emph{delayed renewal
process}.  (The ``delay'' refers the time until the first entrance to
state $\sigma_{renew}$.)

Let $\omega$ be the path of states leading up to the current state;
so, at time $t$,
\[\omega=(X(0),X(1),...,X(t))\]
Let $\Omega$ be the collection of all such paths.

A \emph{renewal reward function} 
is a function $R:\Omega\rightarrow \real$ such that
$R$ depends only on the last renewal period.

In other words, suppose that $\omega_{1}=(X(0), X(1), ..., X(t_{1}))$
and $\omega_{2}=(X'(0), X'(1), ..., X'(t_{2}))$.  Suppose
that there exists a $k$ such that $X(t_{1}-k)=X'(t_{2}-k)=\sigma_{renew}$
and for any $k'<k$, $X(t_{1}-k')=X'(t_{2}-k')$.
Then $R(\omega_{1})=R(\omega_{2})$.
\end{definition}

The Renewal Reward Theorem can be expressed in the following way:
the time average value of the a renewal reward function $R$
is the expected value per renewal, divided by the expected
length of a renewal.  Here's the precise version.

\begin{theorem}[Renewal Reward Theorem] \label{delayed renewal reward theorem}
\index{renewal reward theorem}
\index{Theorem \ref{delayed renewal reward theorem}}
(This theorem goes by several other names, such as the
\emph{Key Renewal Theorem} or the \emph{Strong Law
for Delayed Renewal Processes}.)

Suppose we have a delayed renewal process $X$ and a renewal
reward function $R$.  Let $\omega_{1}$ be a path consisting
of exactly one renewal.
Formally speaking, let $\omega_{1}=(X(0),...,X(t))$ be a sample path
with $X(0)=X(t)=\sigma_{renew}$, where $x(k)\neq \sigma_{renew}$
for $k=1,...,t-1$.
Let $\omega_{2}=(X(0),...,X(t-1))$.
Then $\Ex[R(\omega_{2})]$ is the expected value of $R$
over one renewal.  

Let $\bar{X_{2}}$ be the expected number of time steps until a
renewal.  Let $\hat{\omega}(t)$ be the path at time $t$.  Then
\begin{equation} \label{delayed reward renewal equation}
\lim_{t\rightarrow \infty}\frac{1}{t}\sum_{\tau=0}^{t}
R(\hat{\omega}(t))
=\frac{\Ex[R(\omega_{2})]}{\bar{X_{2}}}
\mbox{ \,\,\, almost surely}
\end{equation}
\end{theorem}

\proof
See Gallager~\cite{kn:b3}, Sections 3.4-3.7.
\EOP

%% file: dummies.appendix.tex
\chapter{A Primer on Fluid Limits}
\label{dummies chapter}

\section{Introduction} \label{dummies intro section}
Suppose we're routing a finite class of packets, in discrete time,

with Bernoulli arrival processes and constant service times.  In this
appendix, I prove a stripped down version of the fluid limit results
of Dai~\cite{Dai} that apply to this case.  By assuming Bernoulli
arrivals and deterministic service times, the proofs become much
simpler, shorter, and more self-contained than in Dai's paper.
However, all the basic ideas of the fluid limit are present.  This
Appendix is intended to be a shadow to Chapter~\ref{fluid chapter},
which discusses ergodicity in a much more general setting.

Actually, it is possible to prove quite a bit more than stability for
Bernoulli arrivals.  For instance, both the expected queue length and
the variance of the queue length are finite.  However, I won't be
proving that result here; the interested reader may consult
Dai and Meyn~\cite{Dai_and_Meyn}.

I'm going to make the following assumptions about the network:
\begin{itemize}
\item Some classes may have exogenous (external) arrivals, as opposed
to the internal packets that arrive from other classes.  I will assume that
these arrivals form a Bernoulli process, i.e.\ for each class $c$ 
there exists $0\leq \alpha_{c} \leq 1$ such that the probability a packet 
arrives in class $c$ on each time step is $\alpha_{c}$.  
(If there are no exogenous arrivals, $\alpha_{c}=0$.)  Note that
the expected arrival rate of class $c$ packets is $\alpha_{c}$.

Let $a_{i}(s)=1$ if there is an arrival to class $i$ on time step $s$,
for $s=1,2,\ldots$, and 0 otherwise.  I assume that if $s\neq s'$, then
$\{a_{1}(s),\ldots,a_{C}(s)\}$ is independent of 
$\{a_{1}(s'),\ldots,a_{C}(s')\}$ (i.e.\ the Bernoulli arrivals are 
independent in time).  Note that $a_{a}(s)$ and $a_{b}(s)$ 
\emph{can} depend on each other on the same time step.  So, for example,
if classes $a$ and $b$ both arrive at node $i$, you could guarantee that
they never both arrive simultaneously.

\item Let $\phi_{l,k}(s)=1$ if the $s$th packet departing
class $l$ enters class $k$, and 0 otherwise.  For a fixed
$l$ and $k$, this forms another Bernoulli process.

\item It takes a packet (exactly) one time step to cross an edge, and
only one packet crosses a particular edge on one time step.
\item There are a finite number of classes.  (This number can
be countably infinite, but I won't deal with that case here.)
\item Observe that since the arrivals are determined by the 
sum of a finite number of Bernoulli processes, then there is
a maximum arrival rate.  Since at most one packet crosses
each edge on one time step, there is a maximum departure rate.
Let $B$ be the maximum of these two numbers.
\end{itemize}

As a simple, but hopefully sufficient
introduction to fluid limit theorems, I offer the following proofs.

\section{An Analytic Fact}
I'm going to need a result from analysis.
\begin{theorem} \label{lipschitz}
\index{Theorem \ref{lipschitz}}
Suppose we have a family of functions $f_{j}(t)$, where the
$f_{j}(t)$ are Lipschitz, and all have the same
Lipschitz coefficient (i.e.\ there exists some
bound $\hat{B}$ such that  for any $j$, 
\[|f_{j}(t)-f_{j}(s)| \leq \hat{B}|t-s| \]
holds.)  Suppose further that for any $j$, $f_{j}(0)=0$.
Then there exists a subsequence $\{f_{j_{k}}\}$ and a function
$f$ such that for any $t$, 
\[\lim_{k\rightarrow \infty} f_{j_{k}}(t)=f(t)\]
Also, $f(t)$ is a Lipschitz function with Lipschitz
coefficient $\hat{B}$.
\end{theorem}

\proof Order the rationals $q_{1}, q_{2}, \ldots$.  Observe that for any $t$,
$|f_{j}(t)|\leq \hat{B}t$.  The set of values for $f_{j}(t)$
for a fixed $t$ and for all $j$ is contained in a compact set (namely
$[-\hat{B}t,\hat{B}t]$.)  Therefore, there exists a subsequence
$j_{1}^{1}, j_{2}^{1},\ldots$ such that
\[\lim_{i\rightarrow \infty} f_{j_{i}^{1}}(q_{1})=f(q_{1})\]
for some constant $f(q_{1})$.  

Next, we look for a sub-subsequence $f_{1}^{2}, f_{2}^{2},\ldots
\subseteq \{f_{k}^{1}\}$ such that
\[\lim_{i\rightarrow \infty} f_{j_{i}^{2}}(q_{2})=f(q_{2})\]
Observe that since this is a sub-subsequence, we still have that
\[\lim_{i\rightarrow \infty} f_{j_{i}^{2}}(q_{1})=f(q_{1})\]
We can continue on in this way for all the rationals, constructing
$\{j_{i}^{k}\}$ for $k=1,2,\ldots$.  

Unfortunately,
$\cap_{k=1}^{\infty} \{j_{i}^{k}\}$ may be empty, so we're not
done yet.  However, consider the diagonal sequence $j_{i}$
defined by $j_{i}=j_{i}^{i}$ for $i=1,2,\ldots$.  Observe that
this sequence is infinite (and non-empty), and that for any $i$,
\[ \lim_{k\rightarrow \infty} f_{j_{i}}(q_{i})=f(q_{i})\]
(because after the first $i$ terms, $\{j_{i}\}$ is contained in
$\{j_{k}^{i}\}$.)  Therefore, $\{f_{j_{i}}\}$ converges on all
rationals.

Now, take any $t$ (not necessarily rational).  For any rational
$q_{k}$, for any $f_{j_{i}}$, we have that
\begin{equation} \label {eq:lipsch}
|f_{j_{i}}(q_{k}) - f_{j_{i}}(t)| \leq \hat{B}|q_{k}-t| 
\end{equation}
So, by taking a series of rationals that converge to $t$, we
can show both $\limsup_{i}f_{j_{i}}(t)$ and 
$\liminf_{i}f_{j_{i}}(t)$ are finite and equal to each
other.  Therefore, $f(t)$ exists for all $t$.  Equation~\ref{eq:lipsch}
also tells us that $f(t)$ is Lipschitz with the same Lipschitz
coefficients.
\EOP

\section{Fluid Limits}

The purpose of this section (and fluid limit models in general) is to
show that Equation~\ref{eq:limit stability} (page~\pageref{eq:limit
stability}) can be proved by certain very natural rescalings of the
underlying stochastic process.  It takes a while to establish all the
limits, but never fear: all will come together in Theorems~\ref{fluid
limit} and~\ref{stabilities}.

Let $x$ be the initial state of the system, i.e.\ $X(0)=x$.
To indicate that a process starts in state $x$, I will
stick a superscript $x$ on the process, e.g.\ $X^{x}(t)$ is
the state at time $t$ when it started in state $X^{x}(0)=x$.

We can always view a discrete time process as embedded in continuous
time.  As an example, let us consider $Q(t)$, the length of a queue at
time $t$.  One natural way of converting to continuous time is to use
step functions: $Q(t)=Q(\lfloor t \rfloor)$.  However, in order to get
the simplest possible proofs, I'm going to ``connect the dots''
between integer values: if $\lfloor t \rfloor < t < \lceil t \rceil$,
then $Q(t)=(t-\lfloor t \rfloor)Q(\lceil t \rceil) + (\lceil t \rceil
- t)Q(\lfloor t \rfloor)$.  (In order to preserve Markovity, then, the
value at $Q(\lceil t \rceil)$ must be known when we are at time
$t_{0}>X(\lfloor t \rfloor)$.  One simple way to do this is to have
$Q(t)$ be the queue length at time $t-1$.)  Observe that since the
discrete time process changes by at most $B$ packets in one time step,
then this continuous version is Lipschitz, with coefficient $B$.

First, some definitions:
\begin{itemize}
\item
$A_{l}(t)$ is the total number of exogenous arrivals to class $l$
by time $t$, i.e.\
	\[A_{l}(t)=\sum_{s=1}^{\lfloor t \rfloor} a_{l}(s)\]
(Remember, $a_{l}$ was defined in Section~\ref{dummies intro section}.)
\item
$S_{l}(t)$ is the total number of departures from class $l$
after $t$ units of service, i.e.\
\[S_{l}(t)=\sum_{s=1}^{\lfloor t \rfloor} 1 = \lfloor t \rfloor
\mbox{ or }\lfloor t \rfloor\]
(That is, if a processor has not been idle for $t$ units of time, then
exactly $\lfloor t \rfloor$ packets will have been emitted in that 
period.)
\item
$\Phi_{l,k}(t)$=total number of transitions from class
$l$ to class $k$ at the time of the $\lfloor t \rfloor$th
transition out of class $l$, i.e.\
\[\Phi_{l,k}(t)=\sum_{s=1}^{\lfloor t \rfloor} \phi_{l,k}(s)\]
\end{itemize}
Note that all three of these quantities are independent of the 
initial state, e.g.\ $A^{x}_{l}(t)=A_{l}(t)$.

\begin{lemma}
Let $\{x_{n}\}\subseteq \ssx$ be such that $\lim_{n\rightarrow \infty}
|x_{n}|=\infty$.  Then almost surely, for any (fixed) $t\geq 0$,
\[\lim_{n\rightarrow\infty} \frac{1}{|x_{n}|} 
A_{l}(t|x_{n}|) = \alpha_{l}t\]
\[\lim_{n\rightarrow\infty} \frac{1}{|x_{n}|} 
S_{l}(t|x_{n}|) = t\]
\[\lim_{n\rightarrow\infty} \frac{1}{|x_{n}|} 
\Phi_{l,k}(t|x_{n}|) = P_{l,k}t\]
\end{lemma}

\proof Fix $t$.  By the strong law of large numbers, for any i.i.d.\ random
variables $X_{i}$,
\[\lim_{n\rightarrow \infty}\frac{1}{n}\sum_{i=1}^{\lfloor nt \rfloor}X_{i}
= t \lim_{n\rightarrow \infty}\frac{1}{tn}\sum_{i=1}^{\lfloor nt \rfloor}X_{i}
= t \Ex [X_{1}]\]
almost surely.  The lemma follows immediately in all three cases.
\EOP

A few more definitions:
\begin{itemize}
\item Suppose that class $l$ occurs at node $i$.  Then
define $T_{l}^{x}(t)$ as the cumulative amount of time node
$i$ has spent on class $l$ packets by time $t$, starting
in state $x$.  (``$T$'' stands for ``throughput''.)
\item For a node $i$, define
$I_{i}^{x}(t)$ to be the cumulative amount of time that node
$i$ is idle, i.e.\
\[I_{i}^{x}(t)=t-\sum_{l \in C_{i}} T_{l}^{x}(t)\]
\end{itemize}
(Note that the subscript for throughput refers to a \emph{class},
whereas the subscript for the idleness refers to a \emph{node}.)
Now let's try to calculate the total queue length of class $l$ packets at
a particular node at time $t$.  We can figure this out by taking the
class $l$ packets that we start out with (in state $x$), adding all the 
packets entering class $l$, and subtracting all the packets that
leave it.  The departures are equal to the total amount of time
spent processing class $l$ packets, fed into the function telling us
how many packets depart in that amount of time.  This is simpler as an 
equation:
\begin{equation} \label{eq:4.9}
Q_{l}^{x}(t)=Q_{l}^{x}(0) + A_{l}(t)
+ \left[ \sum_{k=1}^{C} \Phi_{k,l} (S_{k}(T_{k}^{x}(t))) \right]
-S_{l}(T_{l}^{x}(t))
\end{equation}
Some other useful facts that follow immediately are
\begin{equation} \label{eq:4.10}
Q_{l}^{x}(t)\geq 0
\end{equation}
\begin{equation} \label{eq:4.11}
T_{l}^{x}(t) \mbox{ is nondecreasing and } T_{l}^{x}(0)=0
\end{equation}
\begin{equation}\label{eq:4.12}
I_{i}^{x}(t)=t-\sum_{l \in C_{i}} T_{l}^{x}(t)
\mbox{ is nondecreasing and } I_{i}^{x}(0)=0
\end{equation}
Observe that $T_{l}^{x}(t)$ and $I_{i}^{x}(t)$ are Lipschitz
functions, hence they are continuous and almost-everywhere 
differentiable (see Rudin~\cite{rudin}, page 146).

The protocols that I'm interested in are all work-conserving, 
or greedy.  That is, a node $i$ should be idle \emph{only if} 
its queue is empty (of any class of packets).  One way of expressing 
this is by saying that at node $i$,
\begin{equation} \label{eq:4.13}
\dot{I}_{i}^{x} \left( \sum_{k \in C_{i}}Q_{k}^{x}(t) \right)=0
\mbox{ a.e.}
\end{equation}
(By the comment above, the derivative $\dot{I}_{i}^{x}(t)$ exists
almost everywhere.)

Finally, there might be some additional constraints specific to the 
routing protocol we're using.  For completeness, I'll list
\begin{equation} \label{eq:4.14}
\mbox{some additional conditions specific to the routing protocol}
\end{equation}

Now we come to our first major fluid limit theorem:
\begin{theorem} \label{fluid limit}
\index{Theorem \ref{fluid limit}}
I am going to construct three functions, $\flu{Q}_{l}(t), \flu{T}_{l}(t)$,
and $\flu{I}_{i}(t)$, that satisfy certain limits below.

Consider a greedy routing protocol.  For almost all sample paths
and any sequence of initial states $\{x_{n}\} \subseteq \ssx$
with $|x_{n}|\rightarrow \infty$, there is a subsequence
$\{z_{n}\} \subseteq \{x_{n}\}$ with 
$|z_{n}|\rightarrow \infty$ such that, for all classes $l$,
the following finite limits exist
\begin{equation} \label{eq:4.15}
\lim_{n\rightarrow \infty} Q_{l}^{z_{n}}(0) = \flu{Q}_{l}(0) 
\end{equation}
\begin{equation} \label{eq:4.16}
\lim_{n\rightarrow \infty} T_{l}^{z_{n}}(|z_{n}|t) = \flu{T}_{l}(t)
\end{equation}
Furthermore, $\flu{Q}_{l}(t)$ and $\flu{T}_{l}(t)$ satisfy the following:
\begin{equation} \label{eq:4.17}
\flu{Q}_{l}(t)=\flu{Q}_{l}(0) + \alpha_{l}t +\sum_{k\in C}P_{k,l}
\flu{T}_{l}(t)
\end{equation}
\begin{equation} \label{eq:4.18}
\flu{Q}_{l}^{t}\geq 0
\end{equation}
\begin{equation}\label{eq:4.19}
\flu{T}_{l}^{x}(t) \mbox{ is nondecreasing and } \flu{T}_{l}^{x}(0)=0
\end{equation}
Let class $l$ be served at node $i$.  Then
\begin{equation} \label{eq:4.20}
\flu{I}_{i}^{x}(t)=t-\sum_{l \in C_{i}} \flu{T}_{l}^{x}(t)
\mbox{ is nondecreasing and } \flu{I}_{i}^{x}(0)=0
\end{equation}
\begin{equation} \label{eq:4.21}
\dot{\flu{I}}_{i}^{x} \left( \sum_{k \in C_{i}}\flu{Q}_{k}^{x}(t) \right)=0
\mbox{ a.e.}
\end{equation}
\begin{equation} \label{eq:4.22}
\mbox{some additional conditions specific to the routing protocol}
\end{equation}
\end{theorem}
\proof
Observe that for any $x_{n}$, any $l$,
\[0 \leq \frac{1}{|x_{n}|}Q_{l}^{x_{n}}(0) \leq 1\]
Since $[0,1]$ is compact, there exists a convergent subsequence that
converges to some value between 0 and 1.  Similarly, since $[0,1]^{C}$
is compact, we can choose $x_{n_{j}}$ that simultaneously converge for
any $l=1,\ldots,C$.  

Next, let $f_{j}(t)=\frac{1}{|x_{n_{j}}|}T(|x_{n_{j}}|t)$.  Observe
that $f_{j}(t)$ is non-decreasing, and grows no faster than $t$, i.e.\
for any $s,t\geq 0$,
\[ |f_{j}(t)-f_{j}(s)|\leq t-s \]
Hence, the $f_{j}$ are a collection of Lipschitz functions with the same
Lipschitz coefficient (namely 1).  Also, $f_{j}(0)=0$.  This allows
us to use Lemma~\ref{lipschitz} to find a subsequence $\{ z_{h}\}
\subseteq \{x_{n_{j}} \}$ such that for any $t$,
\[\lim_{h\rightarrow \infty} \frac{1}{|z_{h}|}T(|z_{h}|t)
=\flu{T}_{l}(t) \mbox{ a.s.}\]
for some function $\flu{T}_{l}$.  This implies equations~\ref{eq:4.15}
and \ref{eq:4.16}.  This also tells us that all these objects are
Lipschitz functions, and hence continuous and almost-everywhere 
differentiable.  In particular, $\dot{\flu{I}}_{i}(t)$ exists
almost everywhere.

Now, observe that all of our 
equations can be expressed in terms of $\flu{Q}_{l}(0)$ and
$\flu{T}_{l}(t)$, so equations \ref{eq:4.9},
\ref{eq:4.10}, \ref{eq:4.11}, and \ref{eq:4.12} imply equations
\ref{eq:4.17}, \ref{eq:4.18}, \ref{eq:4.19}, and \ref{eq:4.20},
respectively.  

Next, we're left with the ``greedy'' condition, formula~\ref{eq:4.21}.
Suppose that at some node $i$, for some $t>0$, 
$\sum_{k\in C_{i}} \flu{Q}_{k}(t)=\epsilon>0$.  Therefore, for
some sufficiently large bound $d$, we have
\[
\mbox{if $h\geq d$ then }
\frac{1}{|z_{h}|}
\sum_{k\in C_{i}} Q_{k}(|z_{h}|t)>\epsilon /2\]
Since at most $B$ packets can arrive or leave the system on any given
time step, then 
\[ \frac{1}{|z_{h}|}
\sum_{k\in C_{i}} Q_{k}(|z_{h}|(\delta + t))>\epsilon /2 -B|\delta| 
-1/|z_{h}|
\]
So, if we take $d'\geq d$ sufficiently large, we can guarantee that
for $h\geq d'$,
\[ \frac{1}{|z_{h}|}
\sum_{k\in C_{i}} Q_{k}(|z_{h}|(\delta + t))>\epsilon /2 -2B\delta \]
and hence that if $s\in G=(\min \{ 0 , t- \epsilon/4B \},
t + \epsilon/4B)$, then
\[ \frac{1}{|z_{h}|}
\sum_{k\in C_{i}} Q_{k}(s|z_{h}|)>0\]
Therefore, by Equation~\ref{eq:4.13}, $(1/|z_{h})I_{i}(s|z_{h}|)$
will be constant for every $h\geq d'$ and $s\in G$.  Therefore, the limit
$\flu{I}_{i}(s)$ will be constant on $G$, so $\dot{\flu{I}}_{l}(s)$
is identically equal to zero a.e. on the interval.  In particular,
\[\dot{\flu{I}}_{l}(t)\sum_{k\in C_{i}}\flu{Q}_{k}(t)=0\]

Proving conditions in Equation~\ref{eq:4.22} must be done on a
case-by-case basis, of course.  Some examples will be discussed in 
the next section.
\EOP

\begin{definition}\index{fluid limit model}\index{fluid stability}
 Fix a greedy routing protocol.  The limits
$\flu{T}$ and $\flu{Q}$ from equations~\ref{eq:4.16}
and~\ref{eq:4.15}, respectively, are referred to as the
fluid limit of the protocol.  Any solution to~\ref{eq:4.17}
through~\ref{eq:4.22} is referred to as a fluid model
of the protocol.  (So, the fluid limits are a subset 
of the fluid models.)  We say that a fluid limit
model (respectively fluid model) is stable if there
exists a constant $T$, depending only on the $\alpha_{l}$
and the $P_{k,l}$ such that for any fluid limit (resp.
fluid models) with $\sum_{k=1}^{C}\flu{Q}_{k}=1$, we 
have for any $k$, $Q_{k}(t)=0$ for any $t\geq T$.
\end{definition}

Finally, we're ready to illuminate the relationship between
stability for the fluid models and stability for the underlying
stochastic models.  
\begin{theorem} \label{stabilities}
\index{Theorem \ref{stabilities}}
Fix a greedy routing protocol.  If the fluid limit model of the
queueing discipline is stable, then the original
Markov chain $X$ is positive recurrent.
\end{theorem}

\proof 
Assume that the fluid model is stable.  
For any sequence
of initial states $\{x_{n}\}$, by Theorem~\ref{fluid limit},
there exists a subsequence $\{x_{n_{j}}\}$ such that for 
any class $l$,
\[\lim_{j\rightarrow \infty} 
\frac{1}{|x_{n_{j}}|}
	Q_{l}^{x_{n_{j}}}(|x_{n_{j}}|T)=\flu{Q}_{l}(T)=0\]
where this equals zero by fluid stability.

Observe that because we assumed that there was a maximum rate of arrivals,
then $0\leq\frac{1}{|x_{n_{j}}|} Q_{l}(|x_{n_{j}}|T)\leq B(T+1)$.
We can then use Corollary~\ref{lim exp} from 
Section~\ref{uniform integrability appendix} to conclude that 
\[\lim_{j\rightarrow \infty} 
\frac{1}{|x_{n_{j}}|}
	\Ex[Q_{l}^{x_{n_{j}}}(|x_{n_{j}}|T)]=0\]
(Note that the $\frac{1}{|x_{n_{j}}|}$ is a constant for each $j$,
so we can pull it out of the expectation.)

Now, our choice of sequences $\{x_{n}\}$ was arbitrary.  In particular,
let $\{ \hat{x}_{n} \}=\ssx$ in some ordering.  Consider
\[\limsup_{n \rightarrow \infty} \frac{1}{|\hat{x}_{n}|} 
\Ex \sum_{k=1}^{C}Q_{k}^{\hat{x}_{n}}
(|\hat{x}_{k}| T)\]
The Lipschitz condition implies that the equation above is bounded,
so the lim sup is finite.
Let us take a subsequence $\{x_{n}\} \subseteq \{\hat{x}_{n}\}$ that
has the lim sup as the limit, i.e.\
\[\lim_{n} \frac{1}{|x_{n}|} \Ex \sum_{k=1}^{C}Q_{k}^{x_{n}}
(|x_{k}| T)\]
Therefore, if we take the subsequence $\{ x_{n_{j}} \}$ as above,
we get that
\begin{eqnarray}
\limsup_{n} \frac{1}{|\hat{x}_{n}|} \Ex \sum_{k=1}^{C}Q_{k}^{\hat{x}_{n}}
(|\hat{x}_{k}| T) & = &
\lim_{n} \frac{1}{|x_{n}|} \Ex \sum_{k=1}^{C}Q_{k}^{x_{n}}
(|x_{k}| T) \nonumber \\ & = &
\lim_{n_{j}} \frac{1}{|x_{n_{j}}|} \Ex \sum_{k=1}^{C}Q_{k}^{x_{n_{j}}}
(|x_{k}| T) \nonumber \\ 
& = & 0 \nonumber
\end{eqnarray}
where the second equality follows because we are taking a subsequence
of a sequence with a limit.  Since $Q_{k}$ is non-negative, we 
conclude that 
\[\lim_{n} \frac{1}{|\hat{x}_{n}|} \Ex \sum_{k=1}^{C}Q_{k}^{\hat{x}_{n}}
(|\hat{x}_{k}| T)=0\]
so by Theorem~\ref{lim stab}, the underlying Markov chain is positive
recurrent.
\EOP

One final comment on the relationship between fluid models and 
fluid limit models.  In principal, it would suffice to prove
results about fluid limit models, then use Theorem~\ref{stabilities}
to push the result back to the stochastic case.  Unfortunately,
it is practically impossible to distinguish the fluid models that
arise from fluid limits from the fluid models that exist only
as solutions to equations~\ref{eq:4.17} through~\ref{eq:4.22}.
Therefore, in practice, one proves results about \emph{all} fluid
models.  The results will then apply to the subset of fluid
models that happen to be fluid limits.

\section{Specific Protocols}
Theorem~\ref{fluid limit}, minus Equation~\ref{eq:4.22}, works
perfectly well.  For example, if you want to show that every
greedy protocol on a ring is stable, you're ready to go.  However,
you may wish to prove that a particular protocol (possibly on
a particular network) is stable.  In that case, you need
to instantiate Equation~\ref{eq:4.22} with some equation
appropriate to the protocol.  I'm going to mention just two
protocols here; generally speaking, dealing with these sort
of limits is technically necessary but not the hard part
of the problem.

\subsection{FIFO}
The FIFO protocol demands that the packet that has been waiting
for the longest time
at a node be the next served.  Ties are broken in any 
manner; in our packet routing model, it doesn't make any difference.
To write an equation capturing FIFO, define $D_{k}^{x}(t)$ as the
total number of departures from class $k$ by time $t$.  Define
$W_{i}^{x}(t)$ as the total amount of work left to do at node $i$
at time $t$, that is, $W_{i}^{x}(t)=\sum_{k\in C_{i}} Q_{k}^{x}(t)$.  Then
if class $k$ packets live at node $i$, the equation
\[D_{k}^{x}(t+W_{i}^{x}(t))=Q_{k}^{x}(0)+A_{k}(t)\]
specifies FIFO.  The fluid limit of this follows pretty easily.

More on this can be found in Bramson~\cite{Bramson96}, which also shows that
FIFO is always stable on a packet routing network as I've defined it.
(Because all edge crossings take one time step, it's a generalized
Kelly network.)  Note that this contradicts the intuition given
by the adversarial result that FIFO is not always stable (against
an adversary).

Note also that FIFO either requires an infinite number of classes,
or (as is standard) that the packets at each queue are ordered.

\subsection{Priority Disciplines}
A priority discipline always gives precedence to certain classes
over others.  We have to define, in addition to the regular
throughput, a special throughput for all the classes that 
effect class $k$ packets, i.e.\ all the packets of greater than
or equal priority.  Let $H_{k}$ be the set of packets of priority
greater than or equal to $k$'s priority that
are served at the same node as $k$.  (Note: the only class
with priority equal to $k$'s is $k$ itself.) Then, define
\[T_{k}^{x,+}(t)=\sum_{k\in H_{k}} T_{k}^{x}(t)\]
We then can define
\[I_{k}^{x,+}=t-T_{k}^{x,+}(t)\]
\[Q_{k}^{x,+}(t)=\sum_{k\in H_{k}} Q_{k}^{x}(t)\]
We then get a new greediness condition:
\[\dot{I}_{k}^{x,+}(t)Q_{k}^{x,+}(t)=0\]
almost everywhere.  The fluid limits to these equations
follow pretty readily-- the new greediness condition
can be proved the same way we proved the old one.



%% file: examples.appendix.tex
\chapter{Fluid Limit Examples} \label{fluid limit examples chapter}

This appendix is a summary of the relevant known results about
stability (i.e.\ ergodicity) on networks.  These results apply to fluid
models, so they work with any kind of fluid limit.  In other words,
whether we use the results of Dai~\cite{Dai}, the simpler version of
Appendix~\ref{dummies chapter}, or the more general results of
Chapter~\ref{fluid chapter}, the stability results still apply.

I also include some counterexamples in Appendix~\ref{fluid limit
counterexamples chapter}, to dampen our hopes.

This appendix and the next one will just be a long list of known cases,
one after another.  If the result is known, I'll reference the 
author; otherwise, I'll prove the case myself.  All the 
counter-examples are by other authors, except the last two.
Note: the fluid limit theorems don't have converses,
so an unstable fluid model does not imply the existence of
an unstable stochastic model.  I'll clarify what's 
known about the counter-examples in each particular case.

\begin{definition} \label{simple definition} \index{simple network}
A network is \emph{simple} if a packet never
returns to the same node twice.
\end{definition}
\begin{definition}\index{Kelly network}\index{generalized Kelly network}
\label{Kelly network}
 A \emph{generalized Kelly network} is a network
where all packets serviced at the same node have the same expected
service time.
A network with \emph{uniform expected service
times} is a generalized Kelly network where any packet served at any
node takes the same expected amount of service time.
\end{definition}
(Incidentally, a regular Kelly network assumes that the arrival
process is Poisson and all the service times are exponentially
distributed.  In that situation, under FIFO routing, the
network offers a particular nice product-form solution.)

When I say that a network is stable, I implicitly assume that the
nominal loads are all less than one.  The examples I look at are the
following:
\begin{enumerate}
\item A generalized Kelly ring network is stable under any greedy protocol.
\item A layered (i.e.\ feedforward) network is stable under any greedy
protocol.  The hypercube under dimensional routing is an example.
\item Suppose we have a collection of networks $N_{1},\ldots,N_{m}$
that have stable fluid limits.  Then if we add directed edges such
that $e$ crosses from $N_{i}$ to $N_{j}$ only if $i<j$, then the
resulting network is stable.  (This is sort of a meta-feedforward
network.)  In particular, tori are stable under dimensional routing.
\item Any network, with any greedy protocol, is stable under convex routing.
\item Generalized Kelly networks under FIFO are stable.
\item Suppose we have a simple network and we rank all the
possible paths in a fixed priority list.  Then the network is stable.
For example, if a packet's route is determined solely by its current
location and final destination, then prioritizing packets based on
destination will give a stable network.  (Other natural examples are
given.)
\item On any simple network, Farthest To Go (FTG) is stable.
\item On any simple network, Closest To Origin (CTO) is stable.
\item Longest in System (LIS) is stable.
\item Shortest in System (SIS) is stable.
\item For a re-entrant line, Nearest to Go (NTG) is always stable.
\item A host of round-robin type protocols are stable on all
networks.
\item ``Leaky buckets'' can stabilize any network under
any greedy protocol.
\item There is a general method in which adversarial stability results can
be translated into stochastic stability results (via a fluid model).
\end{enumerate}

After the good news, I turn to known counter-examples in
Appendix~\ref{fluid limit counterexamples chapter}:
\begin{enumerate}
\item The stability need not be monotonic in the arrival rate.
(That is, there exist networks with Bernoulli arrivals that are
unstable, but become stable by \emph{increasing} the arrival rate.)
\item There exists a simple network unstable under NTG.
\item There exists a generalized Kelly network with uniform service times
that is unstable under NTG.
\item There exists a generalized Kelly network without immediate feedback
that is unstable.
\item FIFO can be unstable.
\item There exists a network and protocol that is stable against a 
stochastic process but unstable against a bounded adversary.
\item There exists a network and protocol that is stable against a 
bounded adversary but unstable against a stochastic process.
\end{enumerate}

\section{The Ring} \label{ring example section}
An $N$-node unidirectional ring is a... well, you ought to know what a
ring network is by this point.  If the network is also a generalized
Kelly network, then any greedy protocol is stable on it.  There's a
proof of the fluid limit portion of this result in Dai and
Weiss~\cite{Dai_and_Weiss}.  (The title claims that the paper concerns
re-entrant lines, which the ring is not; nevertheless, it appears as
Theorem 6.2.)

\section{Layered Networks}
\begin{theorem}
A layered network is stable under any greedy protocol.
\end{theorem}

A proof of this result can be found in Dai~\cite{Dai}, Section 6.
Also, for the world's shortest conceivable proof of this, 
note that the output process of a stable queueing system
is itself a hidden Markov process; therefore, the fluid
stability of a single node network under the hidden Markov
fluid model implies stability for layered networks.

A good example of this technique is the hypercube network
defined on page~\pageref{hypercube definition}.  But how should we 
route packets on a hypercube?

\begin{definition}\index{dimensional routing}
Consider a $d$-dimensional torus.  The network\footnote{See the
footnote on page~\pageref{protocol footnote}} uses \emph{dimensional
routing} if packets proceed as follows: to get from
$(x_{1},\ldots,x_{d})$ to $(y_{1},\ldots,y_{d})$, we find the first
$i$ such that $x_{i} \neq y_{i}$, travel around that dimension for
$y_{i} - x_{i} \bmod l_{i}$ steps, and repeat for all successive
dimensions (in order).
\end{definition}

Because the hypercube is a special case of the torus, we can use
dimensional routing on it, too.  If we consider the edge graph of the 
hypercube under dimensional routing, it is layered, and hence 
the network is stable.  Formally,

\begin{corollary}
Under any i.i.d.\ or hidden Markov arrival process, under any greedy
protocol, the hypercube is stable under dimensional routing.
\end{corollary}

\section{Meta-layered Networks}

We can generalize the layered results to a wider class of
networks.

\begin{definition} \index{meta-layered network}
Suppose we have a collection of networks $N_{1},\ldots,N_{m}$.
Suppose that we add directed edges $e$ such that $e$ crosses from a
node in $N_{i}$ to a node in $N_{j}$ only if $i<j$.  Then we
will call this network a \emph{meta-layered network}
where we call $N_{i}$ the $i$th layer.
\end{definition}

\begin{theorem} \label{meta-layered theorem}
\index{Theorem \ref{meta-layered theorem}}
Suppose we have a meta-layered network, each of whose layers
has a stable fluid limit.  Then the meta-layered network
is stable.
\end{theorem}

\proof
This result follows immediately from the hidden Markov stability
results.  Or, one can reason as follows: Since nothing enters $N_{1}$
except for its original arrivals, and we know that the fluid model is
stable, then the fluid drains from $N_{1}$ by some finite amount of
time $T_{1}$.  In the first $T_{1}$ time steps, the fluid that has
pooled in $N_{2}$ can only have worsened by a finite, bounded amount.
From $T_{1}$ onward, the fluid flowing in to $N_{2}$ is the same as if
$N_{1}$ weren't attached, and $N_{2}$ just had a higher arrival rate.
Since we assume that it's fluid model is stable, there exists a time
$T_{2}$ such that it has emptied by time $T_{2}$.  Proceeding in this
fashion, by a finite amount of time $T_{n}$ the whole system will have
emptied.
\EOP

We can apply this result to the torus:

\begin{corollary}
Under any i.i.d.\ or hidden Markov arrival process, under any greedy
protocol, the torus is stable under dimensional routing.
\end{corollary}

\proof
Viewing each layer of a torus as a ring, which 
we know has a stable fluid limit from Section~\ref{ring example section},
we can
use Theorem~\ref{meta-layered theorem}.
\EOP

\section{Convex Routing}

See Section~\ref{convex routing section}.  The stability
of wrapped butterflies (under convex routing) follows from this result.

\section{Generalized Kelly Networks with FIFO}
\index{generalized Kelly network!FIFO stability}
This stability result is in Bramson~\cite{Bramson96}.  The parts
relevant for implying stochastic stability are a relatively small
subset of the paper: Sections 1 through 5, plus one lemma from Section
6, suffice to imply stability when the nominal loads are all less than
1.

\section{Prioritizing All Paths} \label{sec:priority}

\begin{theorem} \label{priority}
\index{Theorem \ref{priority}}
Suppose that we have a simple network, and each class of packet
follows a deterministic, fixed path.  Suppose that we prioritize all
the (finitely many) paths, so that packets travelling along
higher priority paths have precedence.  Then the network is stable.
\end{theorem}

\proof
List the classes from highest priority to lowest priority, as
$c_{1},\ldots,c_{n}$.  Suppose that the total arrival rate in 
the whole network of class $c_{l}$ packets is $\lambda_{l}^{total}$.
The class $c_{1}$ packets will see a 
feed-forward network (because the network is simple), and hence
all the fluid will drain by some time $t_{1}$ regardless of the
initial fluid configuration.  In this amount of time, the class
$c_{2}$ packets will have fluid volume at most 1+$\lambda_{2}^{total} t_{1}$
(which is finite).  

Now, there will never be any more class $c_{1}$ fluid in queue.  However,
there is still some processing capacity taken up by the nominal load
of this class.  Let me make this precise.  Suppose that at some node
$k$, the nominal arrival rate of class $l$ packets is $\lambda_{l}$,
the mean service time is $\mu_{l}$, so the nominal load contributed
by class $l$ packets is $r_{l}=\lambda_{l}\mu_{l}$.  The nominal arrival
rate at $k$ is, of course, $r=\sum r_{l}<1$.  

When there is no longer any class $c_{1}$ fluid in queue, node $k$
behaves as though $\mu_{l}$ had been replaced by $\mu / (1-r_{1})$ for
all classes $l$, and analogously at each node $k$.  We can now
repeat the whole feed-forward argument above on classes $c_{2},\ldots,
c_{n}$, and continue (by induction), and we're done.
\EOP

\note This analysis is similar to the ``push starts'' 
\index{push starts} of Dai and Vande Vate~\cite{Dai_Vande_Vate}.
They are interested in non-simple networks, but can only analyze
networks with at most two (!) nodes.  For packet routing purposes, the
above theorem is obviously much more relevant.

There are some immediate corollaries of great use for packet routing.
\begin{corollary}
Suppose that to route a packet to its destination, only its destination
and current node are necessary to find the path, and all paths
are simple.  Suppose that we
rank the destinations in any order, and let packets with higher
ranked destinations have precedence over packets with lower ranked
destinations.  Then the network is stable.
\end{corollary}

\proof 
This follows immediately from Theorem~\ref{priority}.
\EOP


\begin{corollary}  Suppose we have a simple network, and 
packets follow fixed paths.  When each packet is created, label it
with an integer between 1 and $P$.  This integer is called the
\emph{priority} of the packet.  If two packets $x$ and $y$ contend for
service, and they have priorities $i$ and $j$, respectively, then $x$
has precedence if $i<j$, and $y$ has precedence if $j<i$.  If $i=j$,
then the contention is resolved in an arbitary fashion (but fixed for
that $i$).  If all the nominal loads are less than one, then this
protocol is stable.
\end{corollary}

\note This is my attempt to imitate Ranade's ghost 
packet\index{ghost packet algorithm} algorithm on an arbitrary
network.  See Leighton~\cite{leighton}, Section 3.4.6
for details.

\proof We natively have, say, $n$ classes arriving, namely
$c_{1},\ldots,c_{n}$.  By sticking on the priority flag, we get
$Pn$ classes, $c_{1}^{1},\ldots,c_{n}^{P}$.  To get the corollary,
we just have to rank the classes such that $c_{l}^{i}<c_{m}^{j}$
if $i<j$.  The arbitrary resolution if $i=j$ is determined by the
arbitary ranking between $c_{l}^{i}$ and $c_{m}^{i}$ for all the 
$l$ and $m$.  The result follows by Theorem~\ref{priority}.
\EOP

\note Suppose we take $P$ to be even.  Then,
for the ``arbitary resolution'' if $i=j$, 
suppose we rank $c_{l}^{2i+1}<c_{m}^{2i+1}$ 
iff $c_{l}^{2i}>c_{m}^{2i}$.  Then the 
probability that a class $l$ packet is of higher rank 
than class $m$ is exactly $\frac{1}{2}$.  This may mitigate some
worry about the arbitrariness of the $i=j$ case.

\section{Farthest To Go} \index{FTG protocol!fluid stability}
In the FTG protocol, the packet farthest from its destination gets
precedence.  For this statement to be meaningful, a packet must be
created with a fixed number of steps to cross.  (For example,
if packets are born with destinations, this property holds.)
If two packets are equidistant, there are several
possible interpretations of the protocol.  For our purposes,
I'm going to assume that there's an arbitrary but fixed resolution;
for example, if packets are born with a path to travel, we can
place an arbitrary priority on the destinations, or origin/destination
pairs, and use that to resolve ties.  (But see the note below
for a FIFO generalization.)

\begin{theorem}
FTG is stable on all simple networks.
\end{theorem}

\proof 
A proof is in Chen and Yao~\cite{stoch_net}.  I offer another here:

Consider the fluid model.  Because we have a simple network, there is
a longest possible path that a packet can take, say of length $l$.
Consider all classes of packets that have $l$ steps to take.  These
are of highest priority.  They may conflict with each other, but these
conflicts are resolved according to a fixed priority discipline.
Section~\ref{sec:priority} shows that this is stable.  As in
Theorem~\ref{priority}, once the fluid in queue from these classes
drops to zero, we can renormalize the service times and remove the
classes from the network.  Repeating for $l-1, l-2,\ldots,1$, we prove
stability.
\EOP

\note If we have a generalized Kelly network, then we can resolve
ties between equidistant packets with FIFO, and the network
will still be stable.

Why can't we just use Theorem~\ref{priority}?  Well, consider
a network containing a subgraph like Figure~\ref{FTG figure}.
\begin{figure}[ht]
\centerline{\includegraphics[height=1.2in]{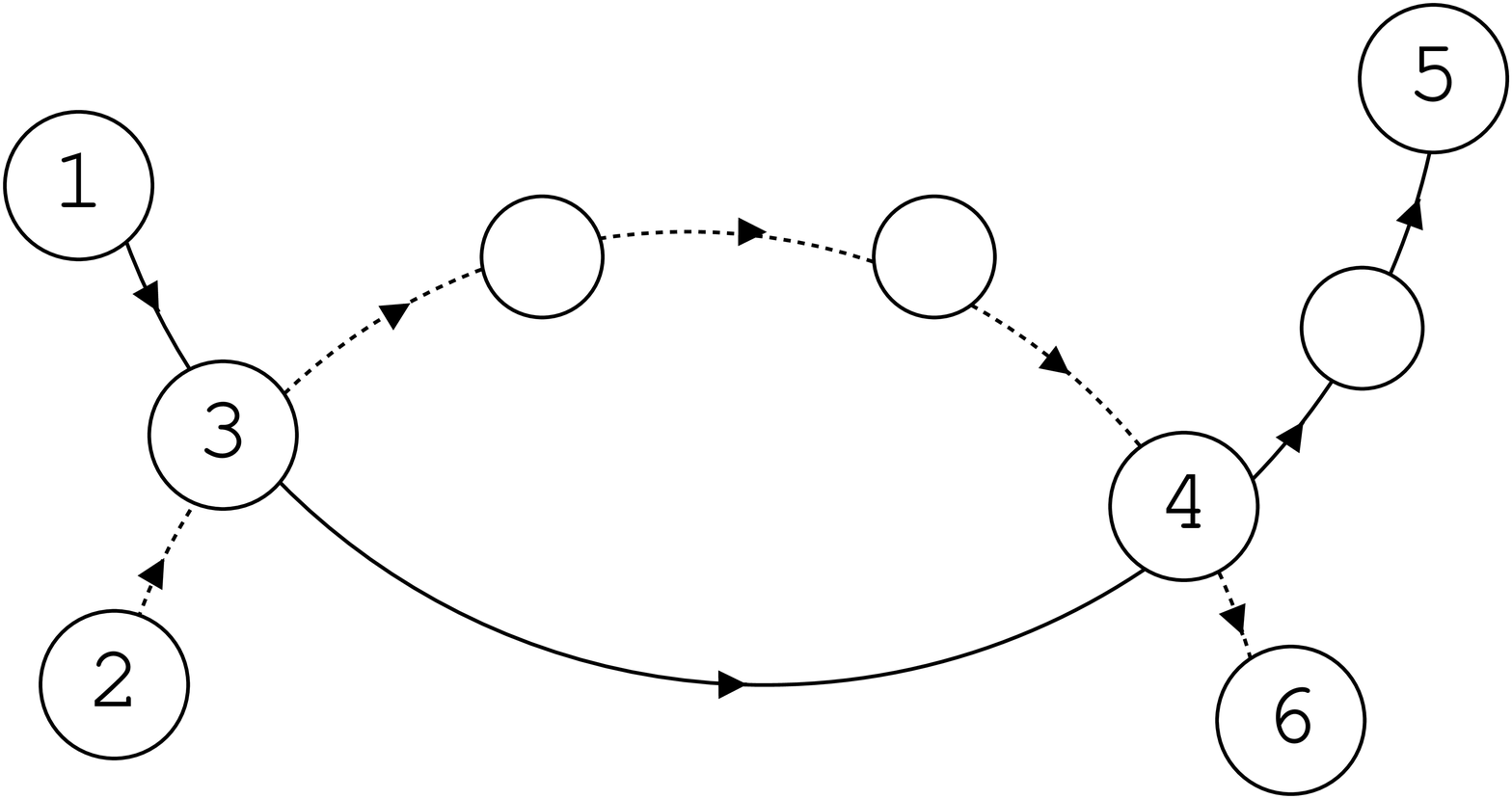} \hspace{.2in}  }
\caption{Non-prioritizable paths in FTG}
\label{FTG figure}
\end{figure}
Consider a packet $z_{1}$ travelling from node 1 to node 5
along the solid line, and a packet $z_{2}$ travelling from
2 to 6 along the dotted line.  At node 3, $z_{2}$ has priority
over $z_{1}$, but at node 4, it's the reverse.  Therefore,
we can't consistently prioritize the paths, so we can't use 
Theorem~\ref{priority}.

\section{Closest to Origin} \index{CTO protocol!fluid stability}
Closest to Origin (CTO) proceeds almost exactly as FTG does.  
Packets are ranked according to their distance from their origin, 
with packets closer to their origin getting priority.  Ties
can be resolved by arbitrary priority, or, in the case of
a generalized Kelly network, by FIFO.

\begin{theorem}
CTO is stable on all simple networks.
\end{theorem}

\proof  We consider all packets that are at their origin;
as in FTG, we show that these are stable, and then 
renormalize the mean service times.  We then consider 
packets that are distance 1 from their origin, then
distance 2, and so forth.  By simpleness, there exists
a finite $l$ such that no packets travel farther than
$l$ steps from their origin, and we're done.
\EOP

\note  This proof is exactly like FTG, except we induct
in the opposite direction.





\section{Longest In System} \label{LIS stability section}
\index{LIS protocol!fluid stability}

It is rumored that Maury Bramson has a proof of the stability of LIS
for some class of networks; I haven't been able to find it, so I offer
this proof.

First, a technical point.
For time-based protocols, like LIS and SIS (see the next section), 
the reader may become worried about the structure of the state
space.  If we keep a time-stamp on every
packet, the system is clearly not going to be stable-- since time
keeps increasing, we would never get to close to returning
to the same state.  To solve this problem, every time the system is 
emptied of all packets, just reset the system clock to zero.
This resetting doesn't change the protocol.

\begin{theorem} \label{LIS theorem}
\index{Theorem \ref{LIS theorem}}
All networks are stable under the Longest In System (LIS) protocol.
\end{theorem}

\proof
Take the fluid limit, but change the norm slightly; rather than take
the sum of the queue at each node, take the sum of the remaining 
expected work of each packet.  Because the expected work
of each packet is finite, this change still yields a bounded
norm, and we can take a fluid limit.

Suppose that we place 1 fluid unit of work in the system, and 
don't let any new fluid enter.  So long as there is a non-empty
queue, that queue will be performing work at a rate of one unit
per time step.  Therefore, by time $t\leq 1$, the system
will have emptied of all fluid.

So, suppose we place 1 fluid unit of work in the system, and
allow new fluid to enter.
Because older fluid has priority in LIS, then this fluid
behaves as though no 
new fluid had entered the system, and will empty by time $t$.
Because the nominal loads are less than one (say $r<1$), then
at most $r$ more units of fluid will enter the system during
that initial $t<1$ time interval.  This fluid will, in turn,
empty in at most $r$ time steps.  Continuing this process,
the entire system will empty by, at the latest,
\[1+r+r^{2}+r^{3}+\cdots=\frac{1}{1-r}<\infty\]
\EOP

\section{Shortest In System} \label{SIS stability section}
\index{SIS protocol!fluid stability}

\begin{theorem}
All networks are stable under the Shortest In System (SIS) protocol.
\end{theorem}

\proof
Let us create a work-based norm, as in Theorem~\ref{LIS theorem},
and take the fluid limit.

New fluid arriving doesn't see any of the (older) fluid queues, and
thus immediately exits the system.  Therefore, no new fluid is added
to queues.  Because the nominal loads are less than one, there is
a least $\epsilon>0$ such that every node does work at a rate
of at least $\epsilon$ when it has a non-empty queue.  Therefore,
if we start the system with one unit of fluid, it will
empty by time $\frac{1}{\epsilon}<\infty$.
\EOP

\section{Nearest To Go for Re-entrant Lines}
\index{NTG protocol!fluid stability}
A re-entrant line is a network where every packet follows the same
path.  However, there may be multiple classes present at each node, so
priority is important.  Clearly, if the network is also simple, we
have a linear array (which is layered, and hence stable), so this
problem is only non-trivial if we have a non-simple network.  The
stability of NTG in this case (and FTG, for that matter) can be found
in Dai and Weiss~\cite{Dai_and_Weiss}.

\section{Round Robin} \label{round robin stability section}
\index{round robin protocol!fluid stability}
Normally speaking, a node uses a ``round robin'' protocol if it
switches between all the non-empty classes present in that node in
some order.  The (more general) protocols I'm going to be considering
might better be called ``weighted round robin'', because certain
classes might appear several times in the same cycle.

\begin{theorem}[Bramson]
Consider a fluid model.  Let $r_{c}^{n}$ be the nominal
load of class $c$ packets at node $n$.  Suppose that
there exists an $\epsilon>0$ such that
node $n$ always dedicates at least $r_{c}^{n}+\epsilon$
of its resources to class $c$, for every $n$, whenever
class $c$ is non-empty.  Then the protocol is stable.
\end{theorem}

A proof of this result can be found in Bramson~\cite{Bramson98}.

This gives us some interesting corollaries.

\begin{corollary}
Suppose that we route packets according to the following 
protocol:  at each node $n$, for classes $c_{1},\ldots c_{m}$
that pass through $n$, we spend $s_{c_{1}}^{n}$ steps passing class 
$c_{1}$, then $s_{c_{2}}^{n}$ steps passing class $c_{2}$
packets, and so on.  If there are no more packets of class
$c_{i}$, we (immediately) move to the next class.

If 
\[\frac{s_{c_{i}}}{\sum_{j=1}^{m}s_{c_{j}}} > r_{c_{j}}^{n}\]
for all $c_{i}$ at all nodes $n$, then the system is stable.
Observe that if the nominal loads are less than one,
then such a choice of $s_{c_{i}}$ always exists (and
easy to figure out.)
\end{corollary}

\begin{corollary}
Suppose that we route packets according to the following protocol: Let
$c_{1},\ldots ,c_{m}$ be the classes \emph{currently} present at node
$n$.  Suppose that the nominal loads of these classes at node $n$ are
$r_{c_{1}},\ldots,r_{c_{m}}$.  When selecting the next packet to go,
choose one at random according to some fixed distribution (determined
by the classes that are currently present at the node) such that
\[\Pr[\mbox{class } c_{i} \mbox { is chosen}]
> r_{c_{i}}\]
Then the system is stable.

If the nominal loads are less than one, such distributions
always exists.
\end{corollary}

\section{Leaky Buckets}
The idea of a ``leaky bucket'' is to reduce the burstiness of the
packets travelling in a network.  Formally, for every class transition
from class $c_{1}$ to class $c_{2}$, we insert a new, single class
node $n_{c_{1},c_{2}}$.  The packets from $c_{1}$ must travel to
node $n_{c_{1},c_{2}}$ before bouncing back to the location of
the $c_{2}$ packets.

It turns out that it is possible to stabilize any network with
the judicious use of leaky buckets.  See Bramson~\cite{Bramson98}
for details.

\section{The Utility of Adversarial Results}
Just as one can prove ergodicity by taking fluid
limits, it is possible to prove stability against
bounded adversaries by taking a slightly different
kind of fluid limit.  The details were worked out by
Gamarnik~\cite{Gamarnik}.

The resulting class of functions that can be the limits of adversarial
networks is larger than the functions generated by stochastic fluid
limits.  Optimistically, then, one might hope that stability results
for adversarial queues might have clear fluid analogues, which would
then apply to the special case of stochastic fluid limits.
Unfortunately, I don't know of any instances of this technique
actually producing new theorems yet.

\chapter{Fluid Limit Counterexamples} 
\label{fluid limit counterexamples chapter}

This appendix is the twin to Appendix~\ref{fluid limit examples
chapter}.  It consists of surprising examples of instability
in queueing networks.

\section{Nonmonotonic stability}
Uriel Fiege~\cite{feige} has some fascinating results
showing how pathological a stability region can be.
He constructs a 20 node network with a
simple and natural adaptive greedy routing protocol.
Packets are injected according to Bernoulli arrival processes
at rate $q$.  He shows that the system is stable
iff $q\in [0,1/3)\cup (2/3,1]$, but unstable for
the $[1/3,2/3]$ region in the middle.

\section{Virtual Stations and Instability} \index{virtual stations}
There is a very clever general technique for generating unstable
queueing networks even when the nominal loads are less than one.  If
two classes at two distinct nodes are never simultaneously in service,
then they act as though they were sharing service in the same node,
forming a ``virtual station''.  By including virtual stations in a
network, it gives extra restrictions on stability, analogous to the
restrictions on the nominal load.  If these restrictions are violated,
the network can easily be shown to be unstable.  See Dai and Vande
Vate~\cite{Dai_Vande_Vate} or Bertsimas, Gamarnik and
Tsitsiklis~\cite{Bertsimas_et_al} for more on this.

The counterexamples in the next three sections all rely
on virtual stations for their instability.

\section{NTG can be Unstable}\index{NTG protocol!instability}
There is a simple two-node network where Nearest To Go (NTG) is
unstable.  See Dai and Weiss~\cite{Dai_and_Weiss}, Figure 4.  (I mean
``simple'' in the sense of Definition~\ref{simple definition}, not
colloquially.)

\section{Uniformly Generalized Kelly Networks can be Unstable}
\index{generalized Kelly network!instability}
See Dai and Weiss~\cite{Dai_and_Weiss}, section 6, remark 2.
Observe that their two-node network is a re-entrant line, and not simple.

\section{A Generalized Kelly Network without Immediate Feedback can be
Unstable}
A network has immediate feedback if it is possible for a packet to
return to a node without travelling to any intervening nodes.  An
immediate feedback-free generalized Kelly network can be found in Dai
and Weiss~\cite{Dai_and_Weiss}, Figure 5.

If you have as much difficulty looking up the reference given by Dai
and Weiss as I did, you may prefer to consider the following system.
Suppose we have a generalized Kelly network with immediate feedback,
and insert extra stations along the edges with immediate feedback.
Let each new node have the same mean service time as the (unique) node
preceding it.  Observe that if we consider the fluid limit and don't
place any initial fluid on these new nodes, no fluid will ever queue
there.  Therefore, the fluid model will evolve identically to the
fluid model with immediate feedback, which we can make unstable.

\section{FIFO can be Unstable}\index{FIFO protocol!instability}
Check out Bramson~\cite{Bramson_fifo_1},~\cite{Bramson_fifo_2},
or ~\cite{Bramson_fifo_3}.  For a simple and short, but stochastically
unsettling account, check out Seidman~\cite{Seidman}.

\section{Adversarially Unstable, Stochastically Stable}
Since all generalized Kelly networks are stable under FIFO
(see Bramson~\cite{Bramson96}), then the 
counterexamples showing FIFO to be adversarially unstable
(see Andrews et al.~\cite{Andrews}) show this.

\section{Stochastically Unstable, Adversarially Stable}
Consider a one node network (i.e.\ a single queue)
in discrete time.
Consider the stochastic arrival process where with
probability $\frac{1}{2}$, no packets arrive, and with probability
$\frac{1}{2}$, two packets arrive.  Each packet takes 1 time step
to leave the network.

If we consider the state space generated from the ``new packets
arrive, then packets depart'' cycle, it's easy to see that
all states are equally likely.  (A state is determined entirely 
by the number of packets in queue.)  Therefore, the
network is unstable (but null-recurrent).

If we consider a rate $(1,w)$ adversary on a single node, i.e.\
for every window of $w$ steps, no more than $w$ packets can 
arrive, then it's easy to show that there can never be more than
$w$ packets in the system, so the system is stable against
a bounded adversary.

%% file: computers.appendix.tex
\chapter{Analytic Computing} \label{computer chapter}

As a first step towards understanding the
behavior of packet routing networks, many researchers
find it useful to write programs that can simulate
the behavior of the systems.  Coffman et al.~\cite{Leighton95},
for instance, based Hypothesis~\ref{ring hypothesis} on
the results of massive simulations.

Such work has a certain value in making hypotheses plausible.
However, from a mathematical point of view, it doesn't prove anything.
How pleasant it would be, though, to perform exact, error-free
analytic calculations on a computer!  It almost seems to be too much
to hope for.

Surprisingly, it is possible to calculate a great deal
of information about packet routing networks exactly,
and with no rounding errors.  This appendix explains
how, focussing on the mathematically interesting parts.
(In this spirit I have not included the source code,
as no one would want to read it.)

\section{Exact Information about Stationary Distributions}
Consider a state $\sigma$ in an $N$ node standard Bernoulli ring.
From the results of Chapter~\ref{analytic chapter}, we know that the
stationary probability of being in state $\sigma$
is an analytic function of the arrival rate, $p$, and 
can be Taylor expanded around 0.  Let us see how
to calculate these Taylor coefficients.

Suppose we wished to calculate the first $k+1$ coefficients
of $\sigma$, i.e.\
\[\Pr(\sigma)=a_{0}+a_{1}p + a_{2}p^{2}+\cdots+a_{k}p^{k} +O(p^{k+1})\]
Let us start the network in the ground state,
$\sigma_{0}$, where no packets are in any of the nodes.

Suppose that it takes more than $k$ packet arrivals to get
from $\sigma_{0}$ to $\sigma$.  For each packet
arrival, the contribution to $\Pr(\sigma)$ picks up
an extra factor of $p$.  Therefore, $a_{i}=0$ for all
$i\leq k$.

Now, if there are more than $k$ packets in $\sigma$, total, then
clearly more than $k$ packet arrivals are needed to get from
$\sigma_{0}$ to $\sigma$.  Therefore, the only states that have
non-zero coefficients of order less than $k$ must have fewer than $k$
packets in them.  There is a finite set of such states.

Suppose we restrict our attention to that finite set of non-zero
states.  We can view the coefficients as time progresses through
the system.  At time $t=0$, we have $\Pr(\sigma_{0})=1$,
and for all $\sigma\neq\sigma_{0}$, $\Pr(\sigma)=0$.  At
time $t=1$, the states adjacent to $\sigma_{0}$ may have 
non-zero coefficients.  As time goes on, the probabilities
are converging to their steady state values, so we might
hope that the Taylor coefficients are converging, too.

At this point, one might expect us to take the limit as time grows
large, and try to bound the error in the evolution of the coefficients.
Shockingly, the coefficients converge in a finite (and
explicitly calculable) amount of time.

Why does this happen?  Let's sketch a proof.

\begin{theorem} \label{computer theorem}
\index{Theorem \ref{computer theorem}}
Assume that there is a maximum path length in the network.  Then in a
finite amount of time, the $k$th degree Taylor coefficients of the
stationary probabilities will converge to their final value.
\end{theorem}

\proof
We can prove this result by a double induction.  First, we induct on
the degree of the coefficients.  If we consider the $p^{0}$ order
term, observe that it is always one for state $\sigma_{0}$, and always
zero for all other states.  This establishes the base case.  For an
arbitrary degree $k$, there are several different terms that
contribute to it.  The probability of state $\sigma$ is the weighted
sum of all the different possible paths into it.  Now, take the
collection of states $S$ that are reachable with $k$ packet arrivals.
Because packets have maximum path lengths, the state space must empty
in a bounded amount of time.  Therefore, there are no loops among the
states in $S$ (or the state space could cycle through the loop for an
unbounded amount of time.)  It follows that $S$ forms a directed
acyclic graph.

Using the natural ordering on DAGs gives a partial well-ordering,
so if we can deal with all the base cases, we can induct on 
the structure.  (This is the second induction in the proof.)
The base case consists of the states in $S$ that can only be reached
by states with $k-1$ or fewer arrivals.  By induction,
the degree $k-1$ and  lower coefficients will converge in
a finite amount of time; since a packet arrival amounts
to multiplying the probability by $p$, which shifts over
the coefficients, then it follows that the degree
$k$ terms in the base cases will converge in a finite amount
of time.

For the other states in $S$, observe that they are reachable either by
states lower in the partial well-ordering, or by the insertion of new
packets.  By our inductive assumption on $S$, the coefficients of the
states lower in the partial well-ordering converge in a finite amount
of time.  By our induction on the degree of the coefficient, the prior
states that require packet insertions also contribute coefficients
converge in a finite amount of time.  Therefore, their sum will
converge.

Unfortunately, we're not quite done.  As $t\rightarrow \infty$,
our sample path will converge to the expected value; this follows
from ergodicity.  If the convergence were uniform in some
neighborhood of zero, then the result above would immediately
give us the value of the Taylor coefficients of the expected value
itself.  However, it's not clear how to prove uniformity of
convergence, so I'll use a different approach.

As mentioned above, the zero order coefficients are always correct;
the coefficient is one if there are no packets, and zero otherwise.
Using the same double-induction as above, we can show that all
the probabilities of the stationary distribution equal the 
values we have calculated in finite time.
\EOP

If we examine the preceding proof more carefully, it is possible to
estimate the speed of convergence for a ring fairly tightly.
The convergence is quite rapid, so calculating the Taylor
coefficients is practical.

What do these coefficients look like, anyway?  Well,
suppose that we define a new variable, $s=\frac{p}{N-1}$,
and expand in $s$, instead of $p$.  (Clearly, 
Theorem~\ref{computer theorem} applies to $s$, too.)
Whenever a packet arrives, it chooses its particular destination
with probability $s$.  With a little thought, it becomes
clear that the coefficients will all be integers.

We are now sitting in the catbird seat, computationally speaking.
We can set up the finite number of states that have
non-zero coefficients of degree less than $k+1$, store only
this finite set of coefficients per state, and calculate
the probability for a finite amount of time.  Since
the coefficients will all be integers, there won't be
any rounding errors.  We will then have calculated
our stationary probabilities exactly!

Moreover, we can calculate other quantities, like the expected queue
length.  If we're interested in the first $k+1$ terms of the expected
total queue length, we can just calculate the stationary probability
for the finite number of states with non-zero coefficients of degree
less than $k+1$, and then add them together (weighed by their queue
length).

\section{The Payoff}

After proving Theorem~\ref{computer theorem} and writing a body of
code to perform the calculations inherent in the proof, I
was able to determine the Taylor expansions that show
up in Subsection~\ref{explicit subsection}.  As I demonstrate
in that subsection, I can use the results to make various
conclusions about the stationary distributions.

However, the conclusions are mostly negative (the expected
queue length is not absolutely monotonic, nor is it
a small-degree rational function.)  Is it possible to 
get more positive results from these values?

It certainly is.  I calculated the Taylor expansions for the
stationary distribution of several states in a standard 3-node
Bernoulli ring.  By observing the Taylor coefficients, I recognized
some of the rational functions that show up in Chapter~\ref{exact
chapter}.  Without exact Taylor coefficients, it would have been
impossible to guess the functions.  Running the program on larger
nonstandard rings, with $L=2$, I noticed that the marginal
distributions were unchanged.  These results lead me to guess 
Equation~\ref{stationary equations} for the stationary distribution.

Once I had guessed the stationary distribution, I still had a fair
amount of work to do in proving that it held for all $N$.  However, I
never would have tried to prove something like Theorem~\ref{L=2 GHP
theorem} without the evidence from the Taylor expansions pointing 
the way.

\section{Real World Details}

In practice, these calculations took about 1 gigabyte of memory, ran a
couple of days on 500 MHz processors, and gave, for instance, $k=18$
places of accuracy for the 4 node standard Bernoulli ring.  The
limiting resource for my efforts was always the available memory (RAM)
of the machine I was working on.  Therefore, it was important to reduce
the size of the state space, and the information held at each state.

The most dramatic method of reducing the state space is by 
not specifying the destination of packets that are still waiting
in queue.  Because the Greedy Hot Potato algorithm never returns
a packet to queue, queued packets are all stochastically
identical.  Leaving them unspecified amounts to an exponential
reduction in state size.

There are a host of small issues to deal with (for instance, how
should I deal with overflow, when the coefficients become larger than
the $2^{31}$ bit signed integers on a typical machine?), but from a
mathematical point of view, they aren't really interesting enough to
relate.